\documentclass[11pt,twoside]{report}

\title{Discrete analogues of Kakeya problems}
\author{Marina Iliopoulou}
\date{2013}

\usepackage[phd,hyperref,fancyhdr]{edmaths}
\usepackage{color}

\newcommand{\C}{\mathbb{C}}
\newcommand{\R}{\mathbb{R}}
\newcommand{\N}{\mathbb{N}}

\usepackage{amssymb,amsmath,amsthm,enumerate}
\usepackage{graphicx}

\usepackage{etoolbox}
\newtheorem{theorem}{Theorem}[section]
\newtheorem{lemma}[theorem]{Lemma}
\newtheorem{proposition}[theorem]{Proposition}
\newtheorem{corollary}[theorem]{Corollary}
\newtheorem{definition}[theorem]{Definition}
\newtheorem{conjecture}[theorem]{Conjecture}

\newtheorem{claim}[theorem]{Claim}

\preto{\chapter}{}
\preto{\section}{}
\preto{\subsection}{}

\makeatletter
\@addtoreset{theorem}{chapter}
\@addtoreset{theorem}{section}
\@addtoreset{theorem}{subsection}
\makeatother

\usepackage{url}
\begin{document}

\maketitle

\pagenumbering{roman}
\declaration
\newpage
\thispagestyle{plain}
\cleardoublepage

\acknowledgements

First of all, I would like to deeply thank my supervisor, Anthony Carbery, for his continuous guidance and support throughout the last years. For searching for mathematical questions that would fit to my taste, such as the ones appearing in this thesis, for inspiring me with his mathematical insight, and for giving me strength with his good nature and modesty, perhaps without being aware of it.

I am also deeply grateful to my friends Jes\'us, Sanja, Paul, Spiros, Mairi, Tim, Rachel, George, Steven, Eric, Valerie, Hari, Zoe, Pamela, Patrick, Noel, Amy and Kevin, mostly from the School of Mathematics of the University of Edinburgh, whose friendship has given me a lot of happiness during my years as a student there. In particular, I would like to thank Jes\'us, for a lot of time-consuming discussions about real algebraic geometry that were very useful to me. Moreover, special thanks are due to Paul, for inspiring me with his love for mathematics, and, to a more practical level, for introducing me to many useful facts about finite fields that I am using in this thesis; also, for making this last year my happiest one in Edinburgh.

I especially thank my very close friends Myrto Manolaki, Georgia Christodoulou and Anda Diamandopoulou, who have really supported me through good, bad and some very bad times.

I am also very grateful to Vasileios Nestoridis, Apostolos Giannopoulos and Christos Athanasiadis, three of my undergraduate lecturers in the Department of Mathematics of the University of Athens, whose love for mathematics deeply inspired me and drove me to analysis and combinatorics, while they played a very important role in the fact that I am doing my PhD at the University of Edinburgh.

I would not have been able to even start my research in this thesis, if not for the financial support from the School of Mathematics of the University of Edinburgh; for this, and for the wonderful working environment, I am deeply grateful.

Last, but certainly not least, I want to thank my parents. I owe them more than I can express. In particular, to my mother, a mathematician herself, for the endless hours she has spent for and with me in her life, and for putting aside everything for me, I would like to dedicate this thesis.

\newpage
\thispagestyle{plain}
\cleardoublepage

\dedication{To my mother}

\newpage
\thispagestyle{plain}
\cleardoublepage

\chapter*{Lay Summary}
\addcontentsline{toc}{chapter}{Lay Summary}
The aim of this thesis is to shed some light on the geometric aspects of certain problems that lie in the heart of harmonic analysis. 

In particular, being interested in harmonic analytic problems that involve tubes, we shrink the tubes to lines, and thus formulate discrete versions of the original problems, avoiding analytical factors, such as volumes, and focusing mainly on the geometry of the problems, such as the directions of the lines.

More specifically, given a collection of lines in three-dimensional space, we are interested in the points at which at least three non-coplanar lines of our collection meet. The problem of controlling the number of such points by the number of the lines is a natural analogue of a very important, open harmonic analytic problem. This is the main question we are addressing in this thesis, together with certain variants of it.

\newpage
\thispagestyle{plain}
\cleardoublepage

\begin{abstract}
This thesis investigates two problems that are discrete analogues of two harmonic analytic problems which lie in the heart of research in the field.

More specifically, we consider discrete analogues of the maximal Kakeya operator conjecture and of the recently solved endpoint multilinear Kakeya problem, by effectively shrinking the tubes involved in these problems to lines, thus giving rise to the problems of counting joints and multijoints with multiplicities. In fact, we effectively show that, in $\R^3$, what we expect to hold due to the maximal Kakeya operator conjecture, as well as what we know in the continuous case due to the endpoint multilinear Kakeya theorem by Guth, still hold in the discrete case. 

In particular, let $\mathfrak{L}$ be a collection of $L$ lines in $\R^3$ and $J$ the set of joints formed by $\mathfrak{L}$, that is, the set of points each of which lies in at least three non-coplanar lines of $\mathfrak{L}$. It is known that $|J|=O(L^{3/2})$ (first proved by Guth and Katz). For each joint $x\in J$, let the multiplicity $N(x)$ of $x$ be the number of triples of non-coplanar lines through $x$. We prove here that
\begin{displaymath}\sum_{x\in J} N(x)^{1/2}=O(L^{3/2}),
\end{displaymath}
while we also extend this result to real algebraic curves in $\R^3$ of uniformly bounded degree, as well
as to curves in $\R^3$ parametrized by real univariate polynomials of uniformly bounded degree.

The multijoints problem is a variant of the joints problem, involving three finite collections of lines in $\R^3$; a multijoint formed by them is a point that lies in (at least) three non-coplanar lines, one from each collection.

We finally present some results regarding the joints problem in different field settings and higher dimensions.

\end{abstract}
\newpage
\thispagestyle{plain}
\cleardoublepage

\pagenumbering{arabic}

\tableofcontents
\addcontentsline{toc}{chapter}{Contents}
\newpage
\thispagestyle{plain}
\cleardoublepage

\chapter*{Notation}
\addcontentsline{toc}{chapter}{Notation}
Any expression of the form $A \lesssim B$ or $A = O(B)$ means that there exists a non-negative constant $M$, depending only on the dimension, such that $A \leq M \cdot B$, while any expression of the form $A\lesssim_{b_1,...,b_m} B$ means that there exists a non-negative constant $M_{b_1,...,b_m}$, depending only on the dimension and $b_1$, ..., $b_m$, such that $A \lesssim M_{b_1,...,b_m}\cdot B$. In addition, any expression of the
form $A \gtrsim B$ or $A \gtrsim_{b_1,...,b_m} B$ means that $B \lesssim A$ or $B \lesssim_{b_1,...,b_m} A$, respectively. Finally, any expression of the form $A \sim B$ means that $A \lesssim B$ and $A \gtrsim B$, while expression of the form $A \sim_{b_1,...,b_m} B$ means that $A \lesssim_{b_1,...,b_m} B$ and $A \gtrsim_{b_1,...,b_m} B$.
\newpage
\thispagestyle{plain}
\cleardoublepage

\chapter{Introduction} \label{1}

This chapter is a presentation of the two problems we will address in this thesis. Apart from their prominent geometric and combinatorial structure, the importance of these problems also lies in the fact that they all constitute discrete analogues of harmonic analytic problems. In fact, turning harmonic analytic questions into discrete ones, even in different field settings, is a tactic attracting a lot of attention recently. The reason for this is that it allows us to deprive our initial problems, some of which are still open, of their underlying analytical nature, and thus eliminate factors that could distract us from understanding their essential geometric structure.

More particularly, our questions mainly arise from the study of Kakeya sets in $\R^n$.

\begin{definition} \emph{\textbf{(Kakeya set)}} A compact subset $K$ of $\R^n$ is called a \emph{Kakeya set}, if it contains a unit line segment in each direction. \end{definition}

Even though it has been proved that there exist Kakeya sets in $\R^n$ with Lebesgue measure zero, it is conjectured that they are still quite large, in the following sense.

\begin{conjecture}\emph{(\textbf{Kakeya set conjecture})} The Hausdorff dimension of a Kakeya set in $\R^n$ is equal to $n$, for all $n \geq 2$. \end{conjecture}

In fact, the truth of the Kakeya set conjecture would be implied by the truth of the following, harder conjecture.

\begin{conjecture} \emph{\textbf{(Maximal Kakeya operator conjecture)}} If $T_{\omega}$, for $\omega \in \Omega \subset S^{n-1}$, are tubes in $\R^n$ with length 1 and cross section an $(n-1)$-dimensional ball of radius $\delta$, such that their directions $\omega \in \Omega$ are $\delta$-separated, then
\begin{displaymath}\int_{x \in \R^n} \Bigg(\sum_{\omega \in \Omega} \chi_{T_{\omega}}(x)\Bigg)^{\frac{n}{n-1}}{\rm d} x \leq C_{n}\log\frac{1}{\delta}\; \sum_{\omega \in \Omega}|T_{\omega}|,\end{displaymath} 
for all $n \geq 2$, where $C_n$ is a constant depending only on $n$.
 \end{conjecture}

The maximal Kakeya operator conjecture, and therefore the Kakeya set conjecture, have been resolved in the case $n=2$, but not in any other. However, a variant of the maximal Kakeya operator conjecture, and in particular its multilinear version, has recently been proved by Guth in \cite {Guth_10} for all dimensions, with the use of profound algebraic techniques:

\begin{theorem} \emph{\textbf{(Endpoint multilinear Kakeya theorem)}} Let $\{e_1, ..., e_n\}$ be a set of orthonormal vectors in $\R^n$, $n\geq 2$, and $\mathbb{T}_1$, ..., $\mathbb{T}_n$ finite families of doubly infinite tubes in $\R^n$, with cross section an $(n-1)$-dimensional unit ball, such that, for each $i=1,...,n$, the direction of each of the tubes in $\mathbb{T}_i$ lies in a fixed $\frac{c}{n}$-cap around $e_i$, for some explicit constant $c$. Then, there exists some constant $C_{n}$, depending only on $n$, such that
\begin{displaymath}\int_{x \in \R^n}\prod_{i=1}^{n}\Bigg(\sum_{T_i \in \mathbb{T}_i}\chi_{T_i}(x)\Bigg)^{\frac{1}{n-1}}{\rm d}x \leq C_{n} \cdot (|\mathbb{T}_1| \cdots |\mathbb{T}_n|)^{\frac{1}{n-1}}. \end{displaymath} 
\end{theorem}

In our work, we consider the maximal Kakeya operator conjecture and the endpoint multilinear Kakeya problem, and effectively shrink the tubes to lines, eventually formulating the corresponding discrete questions in $\R^n$; namely, the problem of counting joints with multiplicities and the multijoints problem, respectively. In fact, we show that, for $n=3$, what we expect to be true due to the maximal Kakeya operator conjecture and what we know due to the endpoint multilinear Kakeya theorem in the continuous case also hold in the discrete case, constituting another indication that the conjecture holds. These two problems form the main part of the thesis, while we also investigate them in higher dimensions and different field settings, as well as in the case where the lines are replaced with more general, appropriate curves.

We are now ready to introduce our problems in more detail.

\section{The joints problem with multiplicities}

A point $x \in \R^n$ is a joint for a collection $\mathfrak{L}$ of lines in $\R^n$ if there exist at least $n$ lines in $\mathfrak{L}$ passing through $x$, whose directions span $\R^n$. 

The problem of bounding the number of joints by a power of the number of the lines forming them first appeared in \cite{Chazelle_Edelsbrunner_Guibas_Pollack_Seidel_Sharir_Snoeyink_1992}, where it was proved that if $J$ is the set of joints formed by a collection of $L$ lines in $\R^3$, then $|J| =O(L^{7/4})$. Successive progress was made in improving the upper bound of $|J|$ in three dimensions, by Sharir, Sharir and Welzl, and Feldman and Sharir (see \cite{MR1280600}, \cite{MR2047237}, \cite{MR2121298}). 

Wolff  had already observed in \cite{MR1660476} that there exists a connection between the joints problem and the Kakeya problem, and, using this fact, Bennett, Carbery and Tao found an improved upper bound for $|J|$, with a particular assumption on the angles between the lines forming each joint (see \cite{MR2275834}). 

Eventually, Guth and Katz provided a sharp upper bound in \cite{Guth_Katz_2008}; they showed that, in $\R^3$, $|J|=O( L^{3/2})$. The proof was an adaptation of Dvir's algebraic argument  in \cite{MR2525780} for the solution of the finite field Kakeya problem, which involves working with the zero set of a polynomial. Dvir, Guth and Katz induced dramatic developments with this work, because they used for the first time the polynomial method to approach problems in incidence geometry. Further work was done by Elekes, Kaplan and Sharir in \cite{MR2763049}, and finally, a little later, Kaplan, Sharir and Shustin (in \cite{MR2728035}) and Quilodr\'{a}n (in \cite{MR2594983}) independently solved the joints problem in $n$ dimensions, using again algebraic techniques, simpler than in \cite{Guth_Katz_2008}.

In particular, Quilodr\'{a}n and Kaplan, Sharir and Shustin showed that, if $\mathfrak{L}$ is a collection of $L$ lines in $\R^n$, $n \geq 2$, and $J$ is the set of joints formed by $\mathfrak{L}$, then 

\begin{equation} |J| \leq c_n \cdot L^{\frac{n}{n-1}}, \label{eq:basic} \end{equation} 

where $c_n$ is a constant depending only on the dimension $n$. 

In this setting, we define the multiplicity $N(x)$ of a joint $x$ as the number of $n$-tuples of lines of $\mathfrak{L}$ through $x$, whose directions span $\R^n$; we mention here that we consider the $n$-tuples to be unordered, although considering them ordered would not cause any substantial change in what follows. 

From \eqref{eq:basic} we know that $ \sum_{x \in J}1 \leq c_n \cdot L^{\frac{n}{n-1}}$. A question by Anthony Carbery is if one can improve this to get

\begin{equation} \sum_{x \in J}N(x)^{\frac{1}{n-1}} \leq c_n' \cdot L^{\frac{n}{n-1}}, \label{eq:unsolved} \end{equation} 

where $c_n'$ is, again, a constant depending only on $n$. We clarify here that the choice of $\frac{1}{n-1}$ as the power of the multiplicities $N(x)$ on the left-hand side of \eqref{eq:unsolved} does not affect the truth of \eqref{eq:unsolved} when each joint has multiplicity $1$, while it is the largest power of $N(x)$ that one can hope for, since it is the largest power of $N(x)$ that makes \eqref{eq:unsolved} true when all the lines of $\mathfrak{L}$ are passing through the same point and each $n$ of them are linearly independent (in which case the point is a joint of multiplicity $\binom{L}{n} \sim L^n$). Also, \eqref{eq:unsolved} obviously holds when $n=2$; in that case, the left-hand side is smaller than the number of all the pairs of the $L$ lines, i.e. than $\binom{L}{2}\sim L^2$.

In fact, as we have already mentioned, the above question can also be seen from a harmonic analytic point of view (again, see \cite{MR1660476}). Specifically, if $T_{\omega}$, for $\omega \in \Omega \subset S^{n-1}$, are tubes in $\R^n$ with length 1 and cross section an $(n-1)$-dimensional ball of radius $\delta$, such that their directions $\omega \in \Omega$ are $\delta$-separated, then the maximal Kakeya operator conjecture asks for a sharp upper bound of the quantity
\begin{displaymath}\int_{x \in \R^n} \Bigg(\sum_{\omega \in \Omega} \chi_{T_{\omega}}(x)\Bigg)^{\frac{n}{n-1}}{\rm d}x= \int_{x \in \R^n} \#\{ \text{tubes }T_{\omega}\text{ through }x\}^{\frac{n}{n-1}}{\rm d}x.\end{displaymath}
On the other hand, in the case where a collection $\mathfrak{L}$ of lines in $\R^n$ has the property that, whenever $n$ of the lines meet at a point, they form a joint there, then, for all $x \in J$, $N(x)\sim \#\{$lines of $\mathfrak{L}$ through $x\}^n$, and thus the left-hand side of \eqref{eq:unsolved} is
\begin{displaymath}\sim \sum_{x \in J}\#\{\text{lines of }\mathfrak{L}\text{ through }x\}^{\frac{n}{n-1}}.\end{displaymath}
Therefore, in both cases, the problem lies in bounding analogous quantities, and thus the problem of counting joints with multiplicities is a discrete analogue of the maximal Kakeya operator conjecture.

We will indeed show that \eqref{eq:unsolved} holds in $\R^3$: 

\begin{theorem} \emph{\textbf{(Iliopoulou, \cite[Theorem 1.1]{Iliopoulou_12})}}\label{1.1} Let $\mathfrak{L}$ be a collection of $L$ lines in $\R^3$, forming a set $J$ of joints. Then,
 \begin{displaymath} \sum_{x \in J}N(x)^{1/2} \leq c \cdot L^{3/2}, \end{displaymath}
where $c$ is a constant independent of $\mathfrak{L}$. \end{theorem}

In Chapter \ref{6} we generalise the statement of Theorem \ref{1.1}, for joints formed by real algebraic curves in $\R^3$ of uniformly bounded degree, as well as curves in $\R^3$ parametrised by real univariate polynomials of uniformly bounded degree (Theorem \ref{4.2.1} and Corollary \ref{jointsforcurves}, respectively).

The basic tool for the proof of Theorem \ref{1.1}, as well as its generalisation, will be the Guth-Katz polynomial method, developed by Guth and Katz in \cite{Guth_Katz_2010}, which we present in Chapter \ref{2}. Note that the proof of \eqref{eq:unsolved} for three dimensions that we are providing cannot be applied for higher dimensions, as a crucial ingredient of it is that the number of critical lines of a real algebraic hypersurface in $\R^3$ is bounded, a fact which we do not know if is true in higher dimensions.

Finally, in the somewhat independent Chapter \ref{7} we consider the joints problem in different field settings and higher dimensions. In particular, we investigate the extent to which an application, in different field settings, of the techniques we use in euclidean space to tackle the problem, seems possible. We also prove that \eqref{eq:basic} holds in $\mathbb{F}^n$ as well, for any field $\mathbb{F}$ and any $n \in \N$ (Theorem \ref{carberyjoints}), while we provide a further result, hoping to shed some light on the way the lines forming a particular set of joints are distributed in $\mathbb{F}^n$. Finally, the combination of these two results gives us an estimate on the number of joints counted with multiplicities in $\mathbb{F}^n$, that is, however, weaker than the one in Theorem \ref{1.1}.

\section{The multijoints problem}

Let $\mathfrak{L}_1$, ..., $\mathfrak{L}_n$ be finite collections of $L_1$, ..., $L_n$, respectively, lines in $\R^n$. We say that a point $x \in \R^n$ is a multijoint for these collections of lines if, for all $i=1,...,n$, there exists a line $l_i\in \mathfrak{L}_i$, such that $x \in l_i$ and the directions of the lines $l_1$, ..., $l_n$ span $\R^n$ (in other words, such that $x$ is a joint for the collection $\{l_1,...,l_n\}$ of lines). 

The multijoints problem lies in bounding the number of multijoints by a number depending only on the cardinalities of the collections of lines forming them. 

More specifically, let $J$ be the set of multijoints formed by collections $\mathfrak{L}_1$, ..., $\mathfrak{L}_n$ of lines in $\R^n$. For each multijoint $x \in J$, we define $N'(x):=\{(l_1,...,l_n) \in \mathfrak{L}_1\times\cdots\times\mathfrak{L}_n$: $x \in l_i$ for all $i=1,...,n$, and the directions of the lines $l_1$, ..., $l_n$ span $\R^n\}$.

Anthony Carbery has conjectured that, for all $n \geq 3$ (for $n=2$ it is obvious),
\begin{equation} \label{eq:mult1}|J| \lesssim_n (L_1\cdots L_n)^{1/(n-1)}, \end{equation}
as well as that
\begin{equation} \label{eq:mult1000}\sum_{x \in J}N'(x)^{1/(n-1)} \lesssim_n (L_1\cdots L_n)^{1/(n-1)}. \end{equation}

In fact, in the particular case where the collections $\mathfrak{L}_1$, ..., $\mathfrak{L}_n$ of lines have the property that, whenever a point $x \in \mathbb{R}^n$ lies on the intersection of at least one line $l_i$ from each collection $\mathfrak{L}_i$, $i=1,2,...,n$, the directions of the lines $l_1$, ..., $l_n$ span $\mathbb{R}^n$, then the multijoints problem gives rise to a discrete analogue of the endpoint multilinear Kakeya problem. Indeed, if we denote by $J$ the set of multijoints formed by such collections of lines (a subset of the set of joints formed by the collection $\mathfrak{L}_1\cup...\cup\mathfrak{L}_n$), while, for every $x \in J$ and $i=1,...,n$, $N_i(x)$ denotes the number of lines of $\mathfrak{L}_i$ passing through $x$, then, under the above additional assumption on the transversality properties of the collections $\mathfrak{L}_1$, ..., $\mathfrak{L}_n$, \eqref{eq:mult1000} becomes
\begin{equation} \label{eq:mult2}\sum_{x \in J}(N_1(x) \cdots N_n(x))^{1/(n-1)} \lesssim_n (L_1\cdots L_n)^{1/(n-1)}. \end{equation}
Note that \eqref{eq:mult2} clearly shows the connection between the multijoints problem and the endpoint multilinear Kakeya problem. Indeed, Guth's work on the latter demonstrated that, whenever $\mathbb{T}_1$, ..., $\mathbb{T}_n$ are $n$ essentially transverse families of doubly-infinite tubes in $\mathbb{R}^n$, with cross section an $(n-1)$-dimensional unit ball (where, by the expression ``essentially transverse", we mean that, for all $i=1,...,n$, the direction of each tube in the family $\mathbb{T}_i$ lies in a fixed $\frac{c}{n}$-cap around the vector $e_i \in \R^n$, where the vectors $e_1$, ..., $e_n$ are orthonormal), it holds that
\begin{displaymath}\int_{x \in \R^n}\prod_{i=1}^{n}\Bigg(\sum_{T_i \in \mathbb{T}_i}\chi_{T_i}(x)\Bigg)^{\frac{1}{n-1}}{\rm d} x \lesssim_{n} (|\mathbb{T}_1| \cdots |\mathbb{T}_n|)^{\frac{1}{n-1}}, \end{displaymath} i.e.
\begin{displaymath}\int_{x \in \R^n}\big(\#\{\text{tubes of }\mathbb{T}_1\text{ through }x\}\cdots \#\{\text{tubes of }\mathbb{T}_n\text{ through }x\}\big)^{\frac{1}{n-1}}{\rm d}x \lesssim_{n} \end{displaymath}
\begin{displaymath}
\lesssim_n (|\mathbb{T}_1| \cdots |\mathbb{T}_n|)^{\frac{1}{n-1}},\end{displaymath}
an expression whose discrete analogue is \eqref{eq:mult2}.

Using, as for the solution of the joints problem, algebraic methods similar to the ones developed by Guth and Katz in \cite{Guth_Katz_2010}, we have proved that \eqref{eq:mult1} holds for $n=3$, i.e. that the following is true.

\begin{theorem}\label{theoremmult2} Let $\mathfrak{L}_1$, $\mathfrak{L}_2$, $\mathfrak{L}_3$ be finite collections of $L_1$, $L_2$ and $L_3$, respectively, lines in $\R^3$. Let $J$ be the set of multijoints formed by the collections $\mathfrak{L}_1$, $\mathfrak{L}_2$ and $\mathfrak{L}_3$. Then,
\begin{equation} \label{eq:proved}|J|\leq c\cdot  (L_1L_2L_3)^{1/2}, \end{equation}
where $c$ is a constant independent of $\mathfrak{L}_1$, $\mathfrak{L}_2$ and $\mathfrak{L}_3$.
\end{theorem}

In fact, Theorem \ref{theoremmult2} is an immediate corollary of the following, stronger proposition.

\begin{proposition} \label{multsimple} Let $\mathfrak{L}_1$, $\mathfrak{L}_2$, $\mathfrak{L}_3$ be finite collections of $L_1$, $L_2$ and $L_3$, respectively, lines in $\R^3$. For all $(N_1,N_2,N_3) \in \R_+^3$, let $J'_{N_1,N_2,N_3}$ be the set of those multijoints formed by $\mathfrak{L}_1$, $\mathfrak{L}_2$ and $\mathfrak{L}_3$, with the property that, if $x \in J'_{N_1,N_2,N_3}$, then there exist collections $\mathfrak{L}_1(x) \subseteq \mathfrak{L}_1$, $\mathfrak{L}_2(x) \subseteq \mathfrak{L}_2$ and $\mathfrak{L}_3(x) \subseteq \mathfrak{L}_3$ of lines passing through $x$, such that $|\mathfrak{L}_1(x)|\geq N_1$, $|\mathfrak{L}_2(x)|\geq N_2$ and $|\mathfrak{L}_3(x)|\geq N_3$, and, if $l_1 \in \mathfrak{L}_1(x)$, $l_2 \in \mathfrak{L}_2(x)$ and $l_3 \in \mathfrak{L}_3(x)$, then the directions of the lines $l_1$, $l_2$ and $l_3$ span $\R^3$. Then,
\begin{displaymath} |J'_{N_1,N_2,N_3}|\leq c \cdot \frac{(L_1L_2L_3)^{1/2}}{(N_1N_2N_3)^{1/2}}, \; \forall\;(N_1,N_2,N_3) \in \R_{+}^3, \end{displaymath}
where $c$ is a constant independent of $\mathfrak{L}_1$, $\mathfrak{L}_2$ and $\mathfrak{L}_3$.
\end{proposition}

Note that Theorem \ref{theoremmult2} follows from Proposition \ref{multsimple}, as, if $J$ is the set of multijoints formed by finite collections $\mathfrak{L}_1$, $\mathfrak{L}_2$ and $\mathfrak{L}_3$ of lines in $\R^3$, then, for each $x \in J$, there exist lines $l_1(x) \in \mathfrak{L}_1$, $l_2(x) \in \mathfrak{L}_2$ and $l_3(x) \in \mathfrak{L}_3$, passing through $x$, with the property that their directions span $\R^3$. Therefore, with the notation of Proposition \ref{multsimple}, $J=J'_{1,1,1}$, and thus $|J| \lesssim \frac{(L_1L_2L_3)^{1/2}}{(1 \cdot 1 \cdot 1)^{1/2}}\sim (L_1L_2L_3)^{1/2}$.

Even though it may be possible, we have not yet managed to take advantage of Proposition \ref{multsimple} to show \eqref{eq:mult1000} for $n=3$. However, we have effectively shown that \eqref{eq:mult2} holds for $n=3$, in the following sense.

\begin{theorem} \label{theoremmult1} Let $\mathfrak{L}_1$, $\mathfrak{L}_2$, $\mathfrak{L}_3$ be finite collections of $L_1$, $L_2$ and $L_3$, respectively, lines in $\R^3$, such that, whenever a line of $\mathfrak{L}_1$, a line of $\mathfrak{L}_2$ and a line of $\mathfrak{L}_3$ meet at a point, they form a joint there. Let $J$ be the set of multijoints formed by the collections $\mathfrak{L}_1$, $\mathfrak{L}_2$ and $\mathfrak{L}_3$. Then,

\begin{equation} \label{eq:proved}\sum_{\{x \in J:N_m(x)>10^{12}\}}(N_1(x)N_2(x)N_3(x))^{1/2} \leq c \cdot (L_1L_2L_3)^{1/2}, \end{equation}

where $m \in \{1,2,3\}$ is such that $L_m=\min\{L_1,L_2,L_3\}$, and $c$ is a constant independent of $\mathfrak{L}_1$, $\mathfrak{L}_2$ and $\mathfrak{L}_3$. \end{theorem}

\textbf{Remark.} Note that we prove the statement of Theorem \ref{theoremmult2} under the assumption that, for all $i=1,2,3$, $\mathfrak{L}_i$ contains only one copy of each line $l \in \mathfrak{L}_i$; the reason is that the proof we are providing for the theorem takes advantage of the Szemer\'edi-Trotter theorem, which is not scale invariant. However, we have no reason to expect that Theorem \ref{theoremmult1} and, in fact, the more general inequality \eqref{eq:mult1000}, for all $n \geq 2$, should not hold when the finite collections of lines forming the multijoints contain more than one copy of the same line.

$\;\;\;\;\;\;\;\;\;\;\;\;\;\;\;\;\;\;\;\;\;\;\;\;\;\;\;\;\;\;\;\;\;\;\;\;\;\;\;\;\;\;\;\;\;\;\;\;\;\;\;\;\;\;\;\;\;\;\;\;\;\;\;\;\;\;\;\;\;\;\;\;\;\;\;\;\;\;\;\;\;\;\;\;\;\;\;\;\;\;\;\;\;\;\;\;\;\;\;\;\;\;\;\;\;\;\;\;\;\;\;\;\;\;\;\;\;\;\;\;\;\;\;\;\;\;\;\;\;\;\;\;\;\blacksquare$

We would like to emphasise here that the constant $10^{12}$ which we consider as a lower bound on $N_m(x)$ for the joints $x \in J$ that contribute to the sum in \eqref{eq:proved} above is, actually, the smallest possible constant with that property that we can acquire from the proof we are providing. However, we believe that \eqref{eq:proved} is actually true when all joints in $J$ are contributing to the sum, i.e that \eqref{eq:mult2} is true for $n=3$. In fact, as will be demonstrated from our proof, the only thing missing to conclude that \eqref{eq:mult2} is true for $n=3$ (of course in the case where, whenever three lines, one from each of our three collections, meet at a point, they form a joint there) is to show that 
\begin{equation}\label{eq:dream}\sum_{\{x \in J:N_m(x)=1\}}(N_1(x)N_2(x)N_3(x))^{1/2} \lesssim (L_1L_2L_3)^{1/2}, \end{equation}
for any collections $\mathfrak{L}_1$, $\mathfrak{L}_2$ and $\mathfrak{L}_3$ of $L_1$, $L_2$ and $L_3$, respectively, lines in $\R^3$, such that, whenever a line of $\mathfrak{L}_1$, a line of $\mathfrak{L}_2$ and a line of $\mathfrak{L}_3$ meet at a point, they form a joint there, and $L_m=\min\{L_1,L_2,L_3\}$.

Note that, although \eqref{eq:dream} seems simpler than \eqref{eq:proved}, it cannot be proved using the same reasoning. Even though this will be explained later in detail, we would like to point out that the main reason for this is that, if $\mathfrak{L}$ is a finite collection of lines in $\R^n$, and $\mathcal{P}$ a collection of points in $\R^n$ such that at least two lines of $\mathfrak{L}$ pass through each point of $\mathcal{P}$, then we can bound $|\mathcal{P}|$ from above by the number of the pairs of the lines of $\mathfrak{L}$, i.e. by $\binom{|\mathfrak{L}|}{2} \sim |\mathfrak{L}|^2$. However, if at most one line of $\mathfrak{L}$ passes through each point of $\mathcal{P}$, then there exists no upper bound on $|\mathcal{P}|$ that depends on the number of the lines of $\mathfrak{L}$; for example, we can have arbitrarily many points of $\mathcal{P}$ lying on one of the lines. And this, essentially, is the obstruction in extending the proof of \eqref{eq:proved} to the case where exactly one line of $\mathfrak{L}_m$ passes through each multijoint.

\textbf{Remark.} We have proved Theorem \ref{theoremmult1} in the case where the collections $\mathfrak{L}_1$, $\mathfrak{L}_2$ and $\mathfrak{L}_3$ in $\R^3$ have the property that whenever three lines, one from each collection, meet at a point, they form a joint there. However, it is not hard to see that, by adapting our arguments in the lines of those in the proof of Proposition \ref{multsimple}, we can drop this extra assumption on the transversality properties of the collections $\mathfrak{L}_1$, $\mathfrak{L}_2$ and $\mathfrak{L}_3$, to deduce that, if $\mathfrak{L}_1$, $\mathfrak{L}_2$ and $\mathfrak{L}_3$ are any collections of $L_1$, $L_2$ and $L_3$, respectively, lines in $\R^3$, then, with the notation of Proposition \ref{multsimple},
\begin{displaymath}\sum_{(N_1,N_2,N_3) \in \mathcal{C}}|J'_{N_1,N_2,N_3}|(N_1N_2N_3)^{1/2} \lesssim (L_1L_2L_3)^{1/2},
\end{displaymath}
where $\mathcal{C}:=\{(N_1,N_2,N_3) \in \R_{\geq 1}^3:N_1=(1+10^{-8})^{\lambda_1}, N_2=(1+10^{-8})^{\lambda_2}$ and $N_3=(1+10^{-8})^{\lambda_3}$, for some $\lambda_1$, $\lambda_2$, $\lambda_3 \in \N$ such that $N_m >10^{12}\}$, where $m \in \{1,2,3\}$ is such that $L_m=\min\{L_1,L_2,L_3\}$.

We will not analyse this more in this thesis, as it is achieved by a small perturbation of the somewhat already complicated arguments appearing in the proof of Theorem \ref{theoremmult1}, while we have not yet been able to use it to prove \eqref{eq:mult1000} in the case where $n=3$.

$\;\;\;\;\;\;\;\;\;\;\;\;\;\;\;\;\;\;\;\;\;\;\;\;\;\;\;\;\;\;\;\;\;\;\;\;\;\;\;\;\;\;\;\;\;\;\;\;\;\;\;\;\;\;\;\;\;\;\;\;\;\;\;\;\;\;\;\;\;\;\;\;\;\;\;\;\;\;\;\;\;\;\;\;\;\;\;\;\;\;\;\;\;\;\;\;\;\;\;\;\;\;\;\;\;\;\;\;\;\;\;\;\;\;\;\;\;\;\;\;\;\;\;\;\;\;\;\;\;\;\;\;\;\blacksquare$

Finally, in Chapter \ref{6} we generalise the statement of Theorem \ref{theoremmult2}, for multijoints formed by real algebraic curves in $\R^3$ of uniformly bounded degree, as well as curves in $\R^3$ parametrised by real univariate polynomials of uniformly bounded degree (Theorem \ref{theoremmult3} and Corollary \ref{multijointsforcurves}, respectively).

Note that, as in the case of the joints problem, the proofs of \eqref{eq:mult1} and \eqref{eq:mult2} for three dimensions that we are providing cannot be applied for higher dimensions, as a crucial ingredient of them is that the number of critical lines of a real algebraic hypersurface in $\R^3$ is bounded, a fact which we do not know if is true in higher dimensions.

\newpage
\thispagestyle{plain}
\cleardoublepage

\chapter{The geometric background} \label{2}

As we have already mentioned, our basic tool for the solution of the joints and multijoints problems in $\R^3$ will be the Guth-Katz polynomial method, as it appears in \cite{Guth_Katz_2010}. We therefore go on to present this method, together with certain other facts which will prove useful to our goal.

\section{The Guth-Katz polynomial method} \label{section2.1}

Given a finite set of points $\mathfrak{G}$ in $\R^n$ and a quantity $d>1$, the Guth-Katz polynomial method results in a decomposition of $\R^n$, and consequently of the set $\mathfrak{G}$, by the zero set of a polynomial. Such a decomposition enriches our setting with extra structure, allowing us to derive information about the set $\mathfrak{G}$. The method is fully explained in \cite{Guth_Katz_2010}, but we are presenting here the basic result and the theorems leading to it.

In particular, the Guth-Katz polynomial method is based on the Borsuk-Ulam theorem.

\begin{theorem} \emph{\textbf{(Borsuk-Ulam)}} Let $f:S^n\rightarrow \R^n$ be an odd and continuous map, for $n \in \N$. Then, there exists $x \in S^n$, such that $f(x)=0$.
\end{theorem}

Indeed, the Guth-Katz polynomial method is the discrete version of the following result by Stone and Tukey, known as the polynomial ham sandwich theorem, which, in turn, is a consequence of the Borsuk-Ulam theorem. In particular, we say that the zero set of a polynomial $p \in \R[x_1,...,x_n]$ bisects a Lebesgue-measurable set $U \subset \R^n$ of finite, positive volume, when the sets $U \cap \{p>0\}$ and $U \cap \{p<0\}$ have the same volume.

\begin{theorem}\emph{\textbf{(Stone, Tukey, \cite{MR0007036})}}\label{2.1.1} Let $d \in \N^*$, and $U_1,$ ..., $U_M$ be Lebesgue-measurable subsets of $\R^n$ of finite, positive volume, where $M=\binom{d+n}{n}-1$. Then, there exists a non-zero polynomial in $\R[x_1,...,x_n]$, of degree $\leq d$, whose zero set bisects each $U_i$.\end{theorem}

\begin{proof} We associate each polynomial in $\R[x_1,...,x_n]$, of degree at most $d$, with its sequence of $\binom{d+n}{n}$ coefficients, thus identifying it with an element of $\R^{M+1}$ (where $M=\binom{d+n}{n}-1$). So, we can view the set of polynomials in $\R[x_1,...,x_n]$ of degree at most $d$ as the space $\R^{M+1}$ with the usual metric, and thus define the map
\begin{displaymath} f:S^{M}\rightarrow \R^M,
\end{displaymath}
such that, for any polynomial $p \in S^M$,

\begin{displaymath}f(p)=\Bigg(\int_{U_1\cap\{p>0\}}1-\int_{U_1\cap\{p<0\}}1,\;... \;,\;\int_{U_M\cap\{p>0\}}1-\int_{U_M\cap\{p<0\}}1\Bigg).
\end{displaymath}

Now, $f$ is an odd and continuous map, and therefore, by the Borsuk-Ulam theorem, there exists a polynomial $p\in S^M$ such that $f(p)=0$, i.e. such that $vol(U_i \cap \{p>0\})=vol(U_i\cap \{p<0\})$, for all $i=1,...,M$. Thus, $p$ is a non-zero polynomial in $\R[x_1,...,x_n]$, of degree at most $d$, whose zero set bisects $U_i$, for all $i=1,...,M$. 

\end{proof}

In analogy to the above, if $S$ is a finite set of points in $\R^n$, we say that the zero set of a polynomial $p \in \R[x_1,...,x_n]$ bisects $S$ if each of the sets $S \cap \{p>0\}$ and $S \cap \{p<0\}$ contains at most half of the points of $S$. Now, using the polynomial ham sandwich theorem above, Guth and Katz proved the following.

\begin{corollary}\emph{\textbf{(Guth, Katz, \cite[Corollary 4.4]{Guth_Katz_2010})}}\label{2.1.2} Let $d \in \N^*$, and $S_1,$ ..., $S_M$ be disjoint, finite sets of points in $\R^n$, where $M=\binom{d+n}{n}-1$. Then, there exists a non-zero polynomial in $\R[x_1,...,x_n]$, of degree $\leq d$, whose zero set bisects each $S_i$. \end{corollary}

\begin{proof} For each $\delta > 0$, let $U_{i,\delta}$ be the union of the $\delta$-balls centred at
the points of $S_i$, $i=1,...,M$.  By the polynomial ham sandwich theorem, for each $\delta >0$ there exists a non-zero
polynomial $p_{\delta}\in \R[x_1,...,x_n]$, of degree $\leq d$, whose zero set bisects $U_{i, \delta}$, for all $i=1,...,M$. 

Now, identifying each polynomial in $\R[x_1,...,x_n]$, of degree at most $d$, with its sequence of $\binom{d+n}{n}$ coefficients, which is an element of $\R^{M+1}$ (where $M=\binom{d+n}{n}-1$) we can view the set of polynomials in $\R[x_1,...,x_n]$ of degree at most $d$ as the space $\R^{M+1}$ with the usual norm. In particular, for all $\delta >0$, we normalise $p_{\delta}$ with respect to this norm, so that it belongs to $S^M$ (note that such a process does not inflict any change on the zero set of $p_{\delta}$).

However, $S^M$ is a compact subset of $\R^{M+1}$ with respect to the usual metric, therefore there exists an $n$-variate real polynomial $p \in S^M$ (which is therefore non-zero), such that $p_{\delta _n} \rightarrow p$ in the usual metric as $n \rightarrow \infty$, for a sequence $(\delta_n)_{n \in \N}$ of positive numbers that converges to 0. This, in turn, means that the coefficients of $p_{\delta _n}$ converge to the coefficients of $p$ as $n \rightarrow \infty$, which implies that $p_{\delta _n}$ converges to $p$ uniformly on bounded subsets of $\R^n$.

We will now show that the zero set of $p$ bisects $S_i$, for all $i=1,...,M$.

Indeed, let us assume that the zero set of $p$ does not bisect $S_i$, for some $i \in \{1,...,M\}$. Then, either $p>0$ on more than half of the points of $S_i$, or $p<0$ on more than half of the points of $S_i$. Suppose that $p > 0$ on more than half of the points of $S_i$. Let $S_i^+$ be the subset of $S_i$ on which $p$ is positive. Since $p$ is a continuous function on $\R^n$ and $S_i^+$ is a finite set of points, there exists some $\epsilon >0$, such that $p>0$ on the union of the $\epsilon-$balls centred at $S_i^+$. Now, since $p_{\delta_n}$ converges to $p$ uniformly on compact sets, there exists $n \in \N$, such that $p_{\delta _n}>0$ on the union of the $\delta _n$-balls centred at $S_{i}^+$ and the $\delta _n$-balls centred at $S_i^+$ are disjoint. However, $S_i^+$ contains more than half of the points of $S_i$, and thus the zero set of $p_{\delta _n}$ does not bisect the union of the $\delta _n$-balls centred at $S_{i}$, i.e. the set $U_{i, \delta _n}$, which is a contradiction. Similarly, we are led to a contradiction if we assume that $p<0$ on more than half of the points of $S_i$. Therefore, the zero set of $p$ bisects $S_i$, for all $i=1,...,M$.

\end{proof}

Another proof of Corollary \ref{2.1.2} appears in \cite{KMS}, using \cite{MR1988723}.

\textbf{Remark.} Note that, for $d=1$, the quantity $\binom{d+n}{n}-1$ equals $n$. Therefore, it follows from Corollary \ref{2.1.2} that for any finite, disjoint sets of points $S_1$, ..., $S_k$ in $\R^n$, where $k \leq n$, there exists a polynomial of degree 1, whose zero set bisects each $S_i$; in other words, there exists a hyperplane of $\R^n$ that bisects each $S_i$.

$\;\;\;\;\;\;\;\;\;\;\;\;\;\;\;\;\;\;\;\;\;\;\;\;\;\;\;\;\;\;\;\;\;\;\;\;\;\;\;\;\;\;\;\;\;\;\;\;\;\;\;\;\;\;\;\;\;\;\;\;\;\;\;\;\;\;\;\;\;\;\;\;\;\;\;\;\;\;\;\;\;\;\;\;\;\;\;\;\;\;\;\;\;\;\;\;\;\;\;\;\;\;\;\;\;\;\;\;\;\;\;\;\;\;\;\;\;\;\;\;\;\;\;\;\;\;\;\;\;\;\;\;\;\blacksquare$

The Guth-Katz polynomial method consists of successive applications of this last corollary. We now state the result of the application of the method, while its proof is the method itself.

\begin{theorem}\emph{\textbf{(Guth, Katz, \cite[Theorem 4.1]{Guth_Katz_2010})}}\label{2.1.3} Let $\mathfrak{G}$ be a finite set of $S$ points in $\R^n$, and $d>1$. Then, there exists a non-zero polynomial $p \in \R[x_1,...,x_n]$, of degree $\leq d$, whose zero set decomposes $\R^n$ in $\sim d^n$ cells, each of which contains $\lesssim S/d^n$ points of $\mathfrak{G}$.\end{theorem}

\begin{proof}We find polynomials $p_1,$ ..., $p_J \in \R[x_1,...,x_n]$, in the following way.

By Corollary \ref{2.1.2} applied to the finite set of points $\mathfrak{G}$, there exists a non-zero polynomial $p_1 \in \R[x_1,...,x_n]$, of degree $\lesssim 1^{1/n}$, whose zero set $Z_1$ bisects $\mathfrak{G}$. Thus, $\R^n\setminus Z_1$ consists of $2^1$ disjoint cells (the cell $\{p_1>0\}$ and the cell $\{p_1<0\}$), each of which contains $\lesssim S/2^1$ points of $\mathfrak{G}$.

By Corollary \ref{2.1.2} applied to the disjoint, finite sets of points $\mathfrak{G} \cap \{p_1>0\}$, $\mathfrak{G} \cap \{p_1<0\}$, there exists a non-zero polynomial $p_2 \in \R[x_1,...,x_n]$, of degree $\lesssim 2^{1/n}$, whose zero set $Z_2$ bisects $\mathfrak{G} \cap \{p_1>0\}$ and $\mathfrak{G} \cap \{p_1<0\}$. Thus, $\R^n\setminus (Z_1 \cup Z_2)$ consists of $2^2$ disjoint cells (the cells $\{p_1>0\}\cap\{p_2>0\}, \{p_1>0\}\cap\{p_2<0\}, \{p_1<0\}\cap\{p_2>0\}$ and $\{p_1<0\}\cap\{p_2<0\}$), each of which contains $\lesssim S/2^2$ points of $\mathfrak{G}$.

We continue in a similar way; by the end of the $j$-th step, we have produced non-zero polynomials $p_1,$ ..., $p_j$ in $\R[x_1,...,x_n]$, of degrees $\lesssim 2^{(1-1)/n},$ ..., $\lesssim 2^{(j-1)/n}$, respectively, such that  $\R^n \setminus (Z_1 \cup ...\cup Z_j)$ consists of $2^j$ disjoint cells, each of which contains $\lesssim S/2^j$ points of $\mathfrak{G}$.

We stop this procedure at the $J$-th step, where $J$ is such that the polynomial $p:=p_1 \cdots p_J$ has degree $\leq d$ and the number of cells in which $\R^n \setminus (Z_1 \cup...\cup Z_J)$ is decomposed is $\sim d^n$ (in other words, we stop when $2^{(1-1)/n}+2^{(2-1)/n}+...+2^{(J-1)/n} \lesssim d$ and $2^J \sim d^n$, for appropriate constants hiding behind the $\lesssim$ and $\sim$ symbols). The polynomial $p$ has the properties that we want (note that its zero set is the set $Z_1 \cup...\cup Z_J$). 

\end{proof}

\textbf{Remark.} Due to the Remark following Corollary \ref{2.1.2}, the polynomials $p_1$ and $p_2$ that correspond to the first two steps of the Guth-Katz polynomial method in $\R^n$ can be taken to be linear. In fact, we assume that this is the case whenever we apply the Guth-Katz polynomial method from now on. It is also easy to see that then the zero sets of $p_1$ and $p_2$ can be considered to be two intersecting hyperplanes, which will of course be contained in the zero set of the polynomial $p$ we end up with after the application of the Guth-Katz polynomial method. This means that either the zero set of $p$ is a single hyperplane, or each line in $\R^n$ that does not lie in the zero set of $p$  certainly intersects the zero set of $p$.

$\;\;\;\;\;\;\;\;\;\;\;\;\;\;\;\;\;\;\;\;\;\;\;\;\;\;\;\;\;\;\;\;\;\;\;\;\;\;\;\;\;\;\;\;\;\;\;\;\;\;\;\;\;\;\;\;\;\;\;\;\;\;\;\;\;\;\;\;\;\;\;\;\;\;\;\;\;\;\;\;\;\;\;\;\;\;\;\;\;\;\;\;\;\;\;\;\;\;\;\;\;\;\;\;\;\;\;\;\;\;\;\;\;\;\;\;\;\;\;\;\;\;\;\;\;\;\;\;\;\;\;\;\;\blacksquare$

\textbf{Remark.} Note that the cells we refer to in Theorem \ref{2.1.3}, and in which the zero set of the polynomial arising from the application of the Guth-Katz polynomial method decomposes $\R^n$, are open subsets of $\R^n$, as they can be defined as subsets of $\R^n$ where finitely many polynomials are positive. Moreover, their union is the complement, in $\R^n$, of the zero set of the polynomial.

Therefore, when we say from now on that, after applying the Guth-Katz polynomial technique, the zero set of a polynomial decomposes $\R^n$ in cells, we mean that $\R^n$ is the union of those open cells and the zero set of the polynomial. Moreover, in some cases we will refer to points as being in the interiors of such cells; we would like to clarify here that, since the cells are open sets, by such expressions we aim to emphasise that the points do not lie on the zero set of the polynomial arising from the application of the Guth-Katz polynomial method.

$\;\;\;\;\;\;\;\;\;\;\;\;\;\;\;\;\;\;\;\;\;\;\;\;\;\;\;\;\;\;\;\;\;\;\;\;\;\;\;\;\;\;\;\;\;\;\;\;\;\;\;\;\;\;\;\;\;\;\;\;\;\;\;\;\;\;\;\;\;\;\;\;\;\;\;\;\;\;\;\;\;\;\;\;\;\;\;\;\;\;\;\;\;\;\;\;\;\;\;\;\;\;\;\;\;\;\;\;\;\;\;\;\;\;\;\;\;\;\;\;\;\;\;\;\;\;\;\;\;\;\;\;\;\blacksquare$

\section{Computational results on algebraic hypersurfaces}

The great advantage of applying the Guth-Katz polynomial method to decompose $\R^n$ and, at the same time, a finite set of points $\mathfrak{G}$ in $\R^n$, does not only lie in the fact that it allows us to have a control over the number of points of $\mathfrak{G}$ in the interior of each cell; it lies in the fact that the surface that decomposes $\R^n$ is the zero set of a polynomial. This immediately gives us a control over many quantities, especially in three dimensions. In particular, the following holds.

\begin{theorem}\emph{\textbf{(Guth, Katz, \cite[Corollary 2.5]{Guth_Katz_2008})}}\label{2.2.1} \emph{(Corollary of B\'ezout's theorem)} Let $p_1$, $p_2 \in \R[x,y,z]$. If $p_1$, $p_2$ do not have a common factor, then there exist at most $\deg p_1 \cdot \deg p_2$ lines simultaneously contained in the zero set of $p_1$ and the zero set of $p_2$. \end{theorem}

An application of this result enables us to bound the number of critical lines of a real algebraic hypersurface in $\R^3$.

\begin{definition} Let $p \in \R[x,y,z]$ be a non-zero polynomial of degree $\leq d$. Let $Z$ be the zero set of $p$. 

We denote by $p_{sf}$ the square-free polynomial we end up with, after eliminating all the squares appearing in the expression of $p$ as a product of irreducible polynomials in $\R[x,y,z]$.

A \emph{critical point} $x$ of $Z$ is a point of $Z$ for which $\nabla{p_{sf}}(x)=0$. Any other point of $Z$ is called a \emph{regular point} of $Z$. A line contained in $Z$ is called a \emph{critical line} if each point of the line is a critical point of $Z$.
\end{definition}

Note that, for any $p \in \R[x,y,z]$, the polynomials $p$ and $p_{sf}$ have the same zero set.

Moreover, if $x$ is a regular point of the zero set $Z$ of a polynomial $p \in \R[x,y,z]$, then, by the implicit function theorem, $Z$ is a manifold locally around $x$ and the tangent space to $Z$ at $x$ is well-defined; it is, in fact, the plane perpendicular to $\nabla{p_{sf}}(x)$ that passes through $x$.

An immediate corollary of Theorem \ref{2.2.1} is the following.

\begin{proposition}\emph{\textbf{(Guth, Katz, \cite[Proposition 3.1]{Guth_Katz_2008})}}\label{2.2.3} Let $p \in \R[x,y,z]$ be a non-zero polynomial of degree $\leq d$. Let $Z$ be the zero set of $p$. Then, $Z$ contains at most $d^2$ critical lines. \end{proposition}

\begin{proof} Since there are no squares in the expansion of $p_{sf}$ as a product of irreducible polynomials in $\R[x,y,z]$, it follows that $p_{sf}$ and $\nabla{p_{sf}}$ have no common factor. In other words, if $p_{sf}=p_1\cdots p_k$, where, for all $i \in \{1,...,k\}$, $p_i$ is an irreducible polynomial in $\R[x,y,z]$, then, for all $i\in\{1,...,k\}$, there exists some $g_i \in \Big\{\frac{\partial p_{sf}}{\partial x},\frac{\partial p_{sf}}{\partial y},\frac{\partial p_{sf}}{\partial z}\Big\}$, such that $p_i$ is not a factor of $g_i$. 

Now, let $l$ be a critical line of $Z$. It follows that $l$ lies in the zero set of $p_{sf}$, and therefore in the union of the zero sets of $p_1$, ..., $p_k \in \R[x,y,z]$; so, there exists $j \in \{1,...,k\}$, such that $l$ lies in the zero set of $p_j$. However, since $l$ is a critical line of $Z$, it is also contained in the zero set of $\nabla{p_{sf}}$, and thus in the zero set of $g_j$ as well. Therefore, $l$ lies simultaneously in the zero sets of the polynomials $p_j$ and $g_j\in \R[x,y,z]$.

It follows from the above that the number of critical lines of $Z$ is equal to at most $\sum_{i=1,...,k}L_i$, where, for all $i \in \{1,...,k\}$, $L_i$ is the number of lines simultaneously contained in the zero set of $p_i$ and the zero set of $g_i$ in $\R^3$. And since the polynomials $p_i$ and $g_i \in \R[x,y,z]$ do not have a common factor, Theorem \ref{2.2.1} implies that $L_i\leq \deg p_i \cdot \deg g_i \leq \deg p_i \cdot d$, for all $i \in \{1,...,k\}$. Thus, the number of critical lines of $Z$ is $\leq \sum_{i=1,...,k}\deg p_i \cdot d\leq \deg p \cdot d=d^2$.

\end{proof}

What is more, due to Theorem \ref{2.2.1}, we have some control on the number of flat lines of a real algebraic hypersurface in $\R^3$.

\begin{definition} Let $Z$ be the zero set of a polynomial $p \in \R[x,y,z]$. A point $x \in \R^3$ is a \emph{flat point} of $Z$ if it is a regular point of $Z$, lying in at least three co-planar lines of $Z$.
\end{definition}

Now, the second fundamental form of the zero set $Z$ of a polynomial $p\in \R[x,y,z]$ at a regular
point $x$ of $Z$ is defined as $Adu^2 + 2Bdudv + Cdv^2$, where $r = r(u, v)$ is a parametrization of Z locally around $x$, and $$A = r_{uu} \cdot n\text{, }B = r_{uv}\cdot n\text{, }C = r_{vv} \cdot n,$$ where $n =\frac{\nabla{p}(x)}{\|\nabla{p}(x)\|}$ is the unit normal to Z at x.

\begin{proposition} \label{secondfundamentalform} \emph{\textbf{(Elekes, Kaplan, Sharir, \cite{MR2763049})}} Let $Z$ be the zero set of a polynomial $p \in \R[x,y,z]$. If $x \in \R^3$ is a \emph{flat point} of $Z$, then the second fundamental form of $Z$ at $x$ is 0.
\end{proposition}

\begin{definition} Let $Z$ be the zero set of a polynomial $p \in \R[x,y,z]$. A line $l$ in $\R^3$ is a \emph{flat line} of $Z$ if all the points of $l$, except perhaps for finitely many, are regular points of $Z$ on which the second fundamental form of $Z$ vanishes.
\end{definition}

It is obvious from the above that a critical point of the zero set $Z$ of a polynomial $p \in \R[x,y,z]$ cannot simultaneously be a flat point of $Z$, while also a critical line of $Z$ cannot simultaneously be a flat line of $Z$.

In their paper \cite{MR2763049}, Elekes, Kaplan and Sharir explain that the second fundamental form of the zero set $Z$ of a polynomial $p \in \R[x,y,z]$ vanishes at a regular point $a$ of $Z$ if and only if $\Pi_j(p)(a)=0$ for all $j=1,2,3$, where, for all $j=1,2,3$, $\Pi_j(p) \in \R[x,y,z]$ is the polynomial defined by
$$\Pi_j(p)(u) := (\nabla{p}(u) \times e_j)^TH_p(u)(\nabla{p}(u) \times e_j),
$$
where $e_1=(1,0,0)$, $e_2=(0,1,0)$ and $e_3=(0,0,1)$, while 
$$H_p=\begin{pmatrix}
   p_{xx} & p_{xy} & p_{xz} \\
   p_{yx} & p_{yy} & p_{yz} \\
   p_{zx} & p_{zy} & p_{zz}
  \end{pmatrix}.
$$

Note that, for all $j=1,2,3$, the degree of the polynomial $\Pi_j(p)$ is $\leq (\deg p-1)+(\deg p-2)+(\deg p-1)=3\deg p -4$.

\begin{proposition} \emph{\textbf{(Elekes, Kaplan, Sharir, \cite{MR2763049})}} Let $Z$ be the zero set of a polynomial $p \in \R[x,y,z]$. If a line $l$ in $\R^3$ contains at least $3d-3$ points of $Z$ on which the second fundamental form of $Z$ vanishes, then $l$ is a flat line of $Z$. 
\end{proposition}

\begin{proof} Let $l$ be a line in $\R^3$, containing more than $3d-3$ points of $Z$ on which the second fundamental form of $Z$ vanishes. 

It follows that $l$ contains more than $3d-3$ points of $Z$ on which the polynomials $\Pi_j(p)$, $j=1,2,3$, vanish. However, for all $j=1,2,3$, $\Pi_j(p)$ is a polynomial in $\R[x,y,z]$, of degree at most $3d-4$. Therefore, the restriction of $\Pi_j(p)$ on the line $l$ is a univariate real polynomial, of degree at most $3d-4$, vanishing at at least $3d-3$ points, i.e. more times than its degree. Thus, $\Pi_j(p)$ vanishes on the whole line $l$, for all $j=1,2,3$.

Moreover, all the points of $l$, except perhaps for finitely many, are regular points of $Z$. Indeed, $l$ contains at most $d$ critical points of $Z$, as otherwise it would be a critical line, i.e. all of its points would be critical, which is not true, since $l$ contains at least $3d-3$ flat points.

Therefore, all the points of $l$, except perhaps for finitely many, are regular points of $Z$, on which $\Pi_1(p)$, $\Pi_2(p)$ and $\Pi_3(p)$ vanish, i.e. on which the second fundamental form of $Z$ vanishes. So, $l$ is a flat line of $Z$.

\end{proof}

It immediately follows that, if $Z$ is the zero set of a polynomial $p \in \R[x,y,z]$, and a line $l$ in $\R^3$ contains more than $3d-3$ flat points of $Z$, then $l$ is a flat line of $Z$. 

Now, in \cite{MR2763049}, Elekes, Kaplan and Sharir use Theorem \ref{2.2.1} to bound the number of flat lines contained in the zero set of a real trivariate polynomial with no linear factors.

\begin{proposition}\emph{\textbf{(Elekes, Kaplan, Sharir, \cite[Proposition 7]{MR2763049})}} \label{elekes} Let $p \in \R[x,y,z]$ be a non-zero polynomial, of degree $\leq d$, that has no linear factors. Let $Z$ be the zero set of $p$. Then, $Z$ contains at most $3d^2-4d$ flat lines.
\end{proposition}

We therefore easily obtain the following.

\begin{corollary} \label{elekes'} Let $p \in \R[x,y,z]$ be a non-zero polynomial, of degree $\leq d$. Let $Z$ be the zero set of $p$, and $\Pi$ the union of the planes contained in $Z$. Then, there exist $\lesssim d^2$ flat lines of $Z$ not lying in $\Pi$.\end{corollary}

\begin{proof} We write $p=p_1\cdot p_2$, where $p_1$, $p_2 \in \R[x,y,z]$ and $p_2$ is the product of the linear factors of $p$. Note that $\Pi$ is the zero set of $p_2$, while $p_1$ has no linear factors. 

Now, let $l$ be a flat line of $Z$ that does not lie in $\Pi$. We will show that $l$ is a flat line of the zero set of $p_1$. 

Indeed, since $l$ does not lie in the finite union $\Pi$ of planes, only finitely many points of $l$ may lie on $\Pi$. Therefore, if $\mathcal{P}$ denotes the union of this finite set of points with the finite set of points of $l$ on which the second fundamental form of $Z$ is non-zero, then, for every $x \in l\setminus \mathcal{P}$, the polynomial $p_2$ does not vanish locally around $x$, and thus, locally around $x$, $Z$ is the zero set of $p_1$. However, $x$ is a regular point of $Z$, on which the second fundamental form of $Z$ vanishes; therefore, $x$ is a regular point of the zero set of $p_1$, on which the second fundamental form of the zero set of $p_1$ vanishes. So, $l$ is a flat line of the zero set of $p_1$.

It follows that the number of flat lines of $Z$ not lying in $\Pi$ is equal to at most the number of flat lines of the zero set of $p_1$, which is $\lesssim (\deg p_1)^2$ by \cite[Proposition 7]{MR2763049} (Corollary \ref{elekes} above), and thus $\lesssim d^2$.

\end{proof}

\section{The Szemer\'edi-Trotter theorem}

The Szemer\'edi-Trotter theorem plays a very important role in our proofs of the joints and multijoints problems with multiplicities, we therefore dedicate this section to stating and proving it.

\begin{definition} Let $\mathcal{P}$ be a collection of points and $\mathfrak{L}$ a collection of lines in $\R^n$. We say that the pair $(p, l)$, where $p \in \mathcal{P}$ and $l\in \mathfrak{L}$, is an \emph{incidence between $\mathcal{P}$ and $\mathfrak{L}$}, if $p \in l$. We denote by $I_{\mathcal{P}, \mathfrak{L}}$ the number of all the incidences between $\mathcal{P}$ and $\mathfrak{L}$.
\end{definition}

\begin{theorem}\emph{\textbf{(Szemer\'edi, Trotter, \cite{MR729791})}}\label{2.2.5} Let $\mathfrak{L}$ be a collection of $L$ lines in $\R^2$ and $\mathcal{P}$ a collection of $P$ points in $\R^2$. Then,
\begin{displaymath}I_{\mathcal{P}, \mathfrak{L}} \leq C \cdot ( P^{2/3}L^{2/3}+P+L), \end{displaymath}
where $C$ is a constant independent of $\mathfrak{L}$ and $\mathcal{P}$. \end{theorem}

This theorem first appeared in \cite{MR729791}; other, less complicated proofs have appeared since (see \cite{MR1464571} and \cite{KMS}). In particular, in \cite{KMS} the Szemer\'edi-Trotter theorem is proved  with the use of the Guth-Katz polynomial method. In fact, we would like to demonstrate this proof now, as it will be referenced in later parts of the thesis.

\textit{Proof of the Szemer\'edi-Trotter theorem in \cite{KMS} via the Guth-Katz polynomial method.} 

There are $\leq (P-1)P \leq P^2$ incidences between $\mathcal{P}$ and the lines of $\mathfrak{L}$ each of which contains at least 2 points of $\mathcal{P}$; the reason for this is that, through any fixed point of $P$, there are at most $P-1$ lines in $\mathfrak{L}$ with at least $2$ incidences with $\mathcal{P}$, as, for every $2$ distinct points in $\R^2$, there exists at most 1 line passing through both of them. On the other hand, there are $\leq L$ incidences between $\mathcal{P}$ and the lines of $\mathfrak{L}$ each of which contains at most 1 point of $\mathcal{P}$, as there exist at most $L$ lines in $\mathfrak{L}$, each containing at most 1 point of $\mathcal{P}$. Therefore, $ I_{\mathcal{P}, \mathfrak{L}} \lesssim P^2+L$. 

In addition, there are $\leq P$ incidences between $\mathfrak{L}$ and the points of $\mathcal{P}$ that each lie in at most 1 line of $\mathfrak{L}$, as there exist at most $P$ lines in $\mathcal{P}$, each lying in at most 1 line of $\mathfrak{L}$. On the other hand, there are $\lesssim \binom{L}{2} \sim L^2$ incidences between $\mathfrak{L}$ and the points of $\mathcal{P}$ that each lie on at least 2 lines of $\mathfrak{L}$. Therefore, $I_{\mathcal{P}, \mathfrak{L}} \lesssim P+L^2$.

Therefore, we may assume that $P < L^2$ and $L < P^2$, as otherwise $I_{\mathcal{P}, \mathfrak{L}} \lesssim P+L$. We thus set $r:= P^{4/3}/L^{2/3}$ ($ > 1$), and applying the Guth-Katz polynomial method we deduce that there exists some non-zero polynomial $p \in \R[x,y]$, of degree $\leq \sqrt{r}$, whose zero set $Z$ decomposes $\R^2$ in $\sim r$ cells, each containing $\lesssim P/r$ points of $\mathcal{P}$. 

Let $\mathcal{P}_0$ be the set of points of $\mathcal{P}$ which lie in $Z$, $\mathfrak{L} _0$ the set of lines of $\mathfrak{L}$ which lie in $Z$, $C_1,...,C_s$ the interiors of the cells in which $Z$ decomposes $\R^2$, $\mathcal{P}_i$ the set of points of $\mathcal{P}$ which lie in $C_i$, and $\mathfrak{L}_i$ the set of lines in $\mathfrak{L}$ intersecting the interior of $C_i$, $i =1,...,s$.

Since $ I_{\mathcal{P}, \mathfrak{L}} = I_{\mathcal{P}_0, \mathfrak{L}_0} + I_{\mathcal{P}_0, \mathfrak{L} \setminus \mathfrak{L} _0}+\sum_{i=1}^{s}{I_{\mathcal{P}_i, \mathfrak{L}_i}}$, it suffices to bound each of these three quantities. 

Indeed, $|\mathfrak{L}_0| \leq \sqrt{r}+1$. The reason for this is that a generic line in $\R^2$ intersects all the lines of $\mathfrak{L}_0$, and thus intersects the zero set $Z$ of the polynomial $p$ at least $|\mathfrak{L}_0|$ times. So, if the cardinality of $\mathfrak{L}_0$ was larger than $\sqrt{r}+1$, a generic line of $\R^2$ would intersect $Z$ more times than the degree of $p$, and thus it would itself lie in $Z$, which means that $Z$ would be the whole of $\R^2$, and $p$ would be the zero polynomial. Therefore,
\begin{displaymath} I_{P_0, \mathfrak{L} _0} \lesssim |\mathcal{P}_0| + |\mathfrak{L}_0|^2\lesssim P + \sqrt{r}^2 \sim P+r \lesssim P. \end{displaymath}
Moreover, it holds that
\begin{displaymath}I_{\mathcal{P}_0, \mathfrak{L}\setminus \mathfrak{L}_0}\leq L \cdot \sqrt{r}=P^{2/3}L^{2/3},\end{displaymath}
as each line of $\mathfrak{L} \setminus \mathfrak{L}_0$ does not lie in the zero set $Z$ of the polynomial $p$, and thus it intersects $Z$ at most $\deg p \leq \sqrt{r}$ times. 

Finally, for all $i \in \{1,...,s\}$, a line in $\mathfrak{L}_i$ can intersect at most $\deg p +1\lesssim \sqrt{r}$ of the sets $C_1$, ..., $C_s$, as otherwise it would intersect $Z$ more than $\deg p$ times, and would thus lie in $Z$. Therefore, $\sum_{i=1}^{s}|\mathfrak{L}_i| \lesssim L \cdot \sqrt{r}\sim P^{2/3}L^{2/3}$. On the other hand, $\sum_{i=1}^{s}|\mathcal{P}_i|^2\leq \max\big\{|\mathcal{P}_i|:i \in \{1,...,s\}\big\}\cdot \sum_{i=1}^{s}|\mathcal{P}_i|\leq P/r \cdot P=P^2/r=P^{2/3}L^{2/3}$, and thus
\begin{displaymath}\sum_{i=1}^{s}{I_{\mathcal{P}_i, \mathfrak{L}_i}}\lesssim \sum_{i=1}^{s}(|\mathfrak{L}_i|+|\mathcal{P}_i|^2) \sim \sum_{i=1}^{s}|\mathfrak{L}_i| +\sum_{i=1}^{s}|\mathcal{P}_i|^2\lesssim P^{2/3}L^{2/3}.
\end{displaymath}
It has therefore been proved that $I_{\mathcal{P}, \mathfrak{L}} \lesssim P^{2/3}L^{2/3}+P+L$.

\qed

An immediate consequence of the Szemer\'edi-Trotter theorem is the following.

\begin{corollary}\emph{\textbf{(Szemer\'edi, Trotter, \cite{MR729791})}}\label{2.2.6} Let $\mathfrak{L}$ be a collection of $L$ lines in $\R^2$ and $\mathcal{\mathfrak{G}}$ a collection of $S$ points in $\R^2$, such that each of them intersects at least $k$ lines of $\mathfrak{L}$, for $k \geq 2$. Then,
\begin{displaymath}S\leq c_0 \cdot (L^2 k^{-3}+Lk^{-1}), \end{displaymath}
where $c_0$ is a constant independent of $\mathfrak{L}$, $\mathfrak{G}$ and $k$. \end{corollary}

\begin{proof} We denote by $I_{\mathfrak{G}, \mathfrak{L}}$ the number of incidences between $\mathfrak{G}$ and $\mathfrak{L}$. Since at least $k$ lines of $\mathfrak{L}$ are passing through each point of $\mathfrak{G}$, it follows that
\begin{displaymath} I_{\mathfrak{G}, \mathfrak{L}} \geq Sk. \end{displaymath}
On the other hand, by the Szemer\'edi-Trotter theorem,
\begin{displaymath} I_{\mathfrak{G}, \mathfrak{L}} \leq C \cdot (S^{2/3}L^{2/3}+S+L), \end{displaymath}
for some constant $C$, independent of $\mathfrak{L}$ and $\mathfrak{G}$ (and, of course, $k$).

Therefore, it holds that $Sk \leq C \cdot(S^{2/3}L^{2/3}+S+L)$. We thus have that either $Sk \leq \frac{C}{3} \cdot S^{2/3}L^{2/3}$, or $Sk \leq \frac{C}{3} \cdot L$, or $Sk \leq \frac{C}{3} \cdot S$, which means that either $S \leq \frac{C^3}{3^3}\cdot L^2k^{-3}$, or $S \leq \frac{C}{3} \cdot Lk^{-1}$, or $k \leq \frac{C}{3}$, in which last case we have that $S \lesssim L^2k^{-3}$, as $S \leq L^2$, since at least 2 lines of $\mathfrak{L}$ are passing through each point of $\mathfrak{G}$.

Therefore,
\begin{displaymath}S\leq c_0 \cdot (L^2 k^{-3}+Lk^{-1}), \end{displaymath}

where $c_0$ is a constant independent of $\mathfrak{L}$, $\mathfrak{G}$ and $k$.

\end{proof}

Note that Theorem \ref{2.2.5} and Corollary \ref{2.2.6} hold not only in $\R^2$, but in $\R^n$ as well, for all $n \in \N$, $n >2$, by projecting $\R^n$ on a generic plane.

\newpage
\thispagestyle{plain}
\cleardoublepage

\chapter{Counting joints with multiplicities} \label{3}

In this chapter we prove Theorem \ref{1.1}.

\textbf{Theorem \ref{1.1}.} \textit{Let $\mathfrak{L}$ be a collection of $L$ lines in $\R^3$, forming a set $J$ of joints. Then,}
\begin{displaymath} \sum_{x \in J}N(x)^{1/2} \leq c \cdot L^{3/2}, \end{displaymath}
\textit{where $c$ is a constant independent of $\mathfrak{L}$.}

We start by making certain observations.

\begin{lemma}\label{3.1} Let $x$ be a joint of multiplicity $N$ for a collection $\mathfrak{L}$ of lines in $\R^3$, such that $x$ lies in $\leq 2k$ of the lines. If, in addition, $x$ is a joint of multiplicity $\leq \frac{N}{2}$ for a subcollection $\mathfrak{L}'$ of the lines, or if it is not a joint at all for the subcollection $\mathfrak{L}'$, then there exist $\geq \frac{N}{1000 \cdot k^2}$ lines of $\mathfrak{L} \setminus \mathfrak{L}'$ passing through $x$. \end{lemma}

\begin{proof} Since the joint $x$ lies in $\leq 2k$ lines of $\mathfrak{L}$, its multiplicity $N$ is $\leq \binom{2k}{3}\leq 8k^3$.

Now, let $A$ be the number of lines of $\mathfrak{L} \setminus \mathfrak{L}'$ that are passing through $x$. We will show that $A \geq \frac{N}{1000 \cdot k^2}$. Indeed, suppose that $A \lneq \frac{N}{1000 \cdot k^2}$. Then,

$N=\Big|\Big\{\{l_1, l_2, l_3\}:$ $l_1$, $l_2$, $l_3$ $\in \mathfrak{L}$ are passing through $x$, and their directions span $\R^3 \Big\}\Big|=$\newline
$=\Big|\Big\{\{l_1, l_2, l_3\}:$ $l_1$, $l_2$, $l_3$ $\in \mathfrak{L}'$ are passing through $x$, and their directions span $\R^3 \Big\}\Big|+$\newline
$+\Big|\Big\{\{l_1, l_2, l_3\}:$ $l_1$, $l_2$, $l_3$ $\in \mathfrak{L} \setminus \mathfrak{L}'$ are passing through $x$, and their directions span $\R^3 \Big\}\Big|+$\newline
$+\Big|\Big\{\{l_1, l_2, l_3\}:$ the lines $l_1$, $l_2$ $\in \mathfrak{L}'$, $l_3$ $\in \mathfrak{L} \setminus \mathfrak{L}'$ are passing through $x$, and their directions span $\R^3 \Big\}\Big|+\Big|\Big\{\{l_1, l_2, l_3\}:$ the lines $l_1$, $l_2$ $\in \mathfrak{L} \setminus \mathfrak{L}'$, $l_3$ $\in \mathfrak{L}'$ are passing through $x$, and their directions span $\R^3 \Big\}\Big|\leq$\newline
$\leq \frac{N}{2}+\binom{A}{3} + \binom{A}{2} \cdot 2k + \binom{2k}{2} \cdot A\leq$\newline
$\leq \frac{N}{2} + A^3 + A^2 \cdot 2k + (2k)^2 \cdot A\leq$\newline
$\leq \frac{N}{2}+\big(\frac{N}{1000 \cdot k^2}\big)^3 + \big(\frac{N}{1000 \cdot k^2}\big)^2 \cdot 2k + (2k)^2 \cdot \frac{N}{1000 \cdot k^2}\leq$\newline
$\leq \frac{N}{2} + \frac{N}{8} + \frac{N}{8} + \frac{N}{8} \lneq N$ (what we use here is the fact that $N \leq 8k^3$). 

So, we are led to a contradiction, which means that $A \geq \frac{N}{1000 \cdot k^2}$.

\end{proof}

\begin{lemma}\label{3.2} Let $x$ be a joint of multiplicity $N$ for a collection $\mathfrak{L}$ of lines in $\R^3$, such that $x$ lies in $\leq 2k$ of the lines. Then, for every plane containing $x$, there exist $\geq \frac{N}{1000 \cdot k^2}$ lines of $\mathfrak{L}$ passing through $x$, which are not lying in the plane. \end{lemma}

\begin{proof} Let $\mathfrak{L}'$ be the set of lines in $\mathfrak{L}$ passing through $x$ and lying in some fixed plane. By Lemma \ref{3.1}, we know that there exist $\geq \frac{N}{1000 \cdot k^2}$ lines of $\mathfrak{L} \setminus \mathfrak{L}'$ passing through $x$, and, by the definition of $\mathfrak{L}'$, these lines do not lie in the plane. Therefore, there indeed exist at least $\frac{N}{1000 \cdot k^2}$ lines of $\mathfrak{L}$ passing through $x$ and lying outside the plane.

\end{proof}

Now, for a set $J$ of joints formed by a collection of lines in $\R^3$, we consider, for all $N \in \N$, the subset $J_N$ of $J$, defined as follows:

$J_N:=\{x \in J: N \leq N(x) <2N\}$.

In addition, we define, for all $N$ and $k \in \N$, the following subset of $J_N$:

$J_N^{k}:=\{x \in J_N$: $x$ intersects at least $k$ and fewer than $2k$ lines of $\mathfrak{L}\}$. 

Now, Theorem \ref{1.1} will follow from Proposition \ref{1.2} that follows (details will be explained after the proof of the proposition).

\begin{proposition}\label{1.2} If $\mathfrak{L}$ is a collection of $L$ lines in $\R^3$ and $N$, $k\in \N$, then 
\begin{displaymath}|J^{k}_{N}| \cdot N^{1/2} \leq c \cdot \bigg(\frac{L^{3/2}}{k^{1/2}} + \frac{L}{k}\cdot N^{1/2}\bigg), \end{displaymath}
where $c$ is a constant independent of $\mathfrak{L}$, $N$ and $k$.\end{proposition}

\begin{proof} Our argument will be based on the Guth-Katz polynomial method, and also, to a large extent, on the proof of \cite[Theorem 4.7]{Guth_Katz_2010}. The following presentation, though, is self-contained, and it will be made clear whenever the techniques of \cite{Guth_Katz_2010} are being repeated.

The proof will be done by induction on the number of lines of $\mathfrak{L}$. Indeed, let $L \in \N$. For $c$ a (non-negative) constant which will be specified later:

- For any collection of lines in $\R^3$ that consists of 1 line, \begin{displaymath} |J^k_{N}|\cdot N^{1/2} \leq c \cdot \bigg(\frac{1^{3/2}}{k^{1/2}} + \frac{1}{k}\cdot N^{1/2}\bigg), \; \forall \; N, \;k \in \N\end{displaymath} (this is obvious, in fact, for any $c \geq 0$, as in this case $J_{N}=\emptyset$, $\forall \; N \in \N$). 

- We assume that \begin{displaymath} |J_N^{k}| \cdot N^{1/2} \leq c \cdot \bigg(\frac{L'^{3/2}}{k^{1/2}} + \frac{L'}{k}\cdot N^{1/2}\bigg), \; \forall \; N,\; k \in \N,\end{displaymath} for any collection of $L'$ lines in $\R^3$, for any $L' \lneq L$.

- We will now prove that  \begin{equation} |J_N^{k}| \cdot N^{1/2} \leq c \cdot \bigg(\frac{L^{3/2}}{k^{1/2}} + \frac{L}{k}\cdot N^{1/2}\bigg),\; \forall \; N, \; k \in \N \label{eq:final}\end{equation}
for any collection of $L$ lines in $\R^3$.

We emphasise here that this last claim should and will be proved for the same constant $c$ as the one appearing in the first two steps of the induction process, provided that that constant is chosen to be sufficiently large.

Indeed, let $\mathfrak{L}$ be a collection of $L$ lines in $\R^3$, and fix $N$ and $k$ in $\N$. Also, for simplicity, let \begin{displaymath}\mathfrak{G}:=J_{N}^k \end{displaymath} and \begin{displaymath} S:=|J_{N}^k|\end{displaymath} for this collection of lines.

We now proceed in effectively the same way as in the proof of \cite[Theorem 4.7]{Guth_Katz_2010}.

Each point of $\mathfrak{G}$ has at least $k$ lines of $\mathfrak{L}$ passing through it, so, by the Szemer\'edi-Trotter theorem, $S\leq c_0 \cdot (L^2 k^{-3}+Lk^{-1})$, where $c_0$ is a constant independent of $\mathfrak{L}$, $N$ and $k$. Therefore:

If $\frac{S}{2}\leq c_0\cdot  Lk^{-1}$, then $S \cdot N^{1/2} \leq 2c_0 \cdot \frac{L}{k}\cdot N^{1/2}$ (where $2c_0$ is independent of $\mathfrak{L}$, $N$ and $k$). 

Otherwise, $\frac{S}{2}> c_0 \cdot Lk^{-1}$, so, by the Szemer\'edi-Trotter theorem, $\frac{S}{2}< c_0 \cdot L^2 k^{-3}$, which gives $S < 2c_0 \cdot  L^2 k^{-3}$. 

Therefore, $d:=AL^2S^{-1}k^{-3}$ is a quantity $> 1$ whenever $A\geq 2c_0$; we thus choose $A$ to be large enough for this to hold, and we will specify its value later. Now, applying the Guth-Katz polynomial method for this $d>1$ and the finite set of points $\mathfrak{G}$, we deduce that there exists a non-zero polynomial $p$, of degree $\leq d$, whose zero set $Z$ decomposes $\R^3$ in $\sim d^3$ cells, each of which contains $\lesssim Sd^{-3}$ points of $\mathfrak{G}$. We can assume that this polynomial is square-free, as eliminating the squares of $p$ does not inflict any change on its zero set. 

If there are $\geq 10^{-8}S$ points of $\mathfrak{G}$ in the union of the interiors of the cells, we are in the cellular case. Otherwise, we are in the algebraic case.

\textbf{Cellular case:} We follow the arguments in the proof of \cite[Lemma 4.8]{Guth_Katz_2010}, to fix $A$ and deduce that $S \cdot N^{1/2} \lesssim L^{3/2}k^{-1/2}$. More particularly:

There are $\gtrsim S$ points of $\mathfrak{G}$ in the union of the interiors of the cells. However, we also know that there exist $\sim d^3$ cells in total, each containing $\lesssim Sd^{-3}$ points of $\mathfrak{G}$. Therefore, there exist $\gtrsim d^3$ cells, with $\gtrsim Sd^{-3}$ points of $\mathfrak{G}$ in the interior of each. We call the cells with this property ``full cells". Now:

$\bullet$ If the interior of some full cell contains $\leq k$ points of $\mathfrak{G}$, then $Sd^{-3} \lesssim k$, so $S \lesssim L^{3/2}k^{-2}$,  and since $N \lesssim k^3$, we have that $S \cdot N^{1/2} \lesssim L^{3/2}k^{-1/2}$.

$\bullet$ If the interior of each full cell contains $\geq k$ points of $\mathfrak{G}$, then we will be led to a contradiction by choosing $A$ so large, that there will be too many intersections between the zero set $Z$ of $p$ and the lines of $\mathfrak{L}$ which do not lie in $Z$. Indeed:

Let $\mathfrak{L}_Z$ be the set of lines of $\mathfrak{L}$ which are lying in $Z$. Consider a full cell and let $S_{cell}$ be the number of points of $\mathfrak{G}$ in the interior of the cell, $\mathfrak{L}_{cell}$ the set of lines of $\mathfrak{L}$ that intersect the interior of the cell and $L_{cell}$ the number of these lines. Obviously, $\mathfrak{L}_{cell} \subset \mathfrak{L} \setminus \mathfrak{L}_Z$.

Now, each point of $\mathfrak{G}$ has at least $k$ lines of $\mathfrak{L}$ passing through it, therefore each point of $\mathfrak{G}$ lying in the interior of the cell has at least $k$ lines of $\mathfrak{L}_{cell}$ passing through it. Thus, since $S_{cell} \geq k$, we get that $ L_{cell} \geq k + (k-1) + (k-2) +...+1 \gtrsim k^2$, so 
\begin{displaymath} L^2_{cell}k^{-3} \gtrsim L_{cell} k^{-1}.\end{displaymath} 
But $k \geq 3$, so, by the Szemer\'edi-Trotter theorem,
\begin{displaymath} S_{cell} \lesssim L^2_{cell}k^{-3} + L_{cell} k^{-1}.\end{displaymath}
Therefore, $S_{cell} \lesssim L^2_{cell}k^{-3}$, so, since we are working in a full cell, $Sd^{-3} \lesssim L^2_{cell}k^{-3}$, and rearranging we see that 
\begin{displaymath} L_{cell} \gtrsim S^{1/2}d^{-3/2}k^{3/2}.\end{displaymath}

But each of the lines of $\mathfrak{L}_{cell}$ intersects the boundary of the cell at at least one point $x$, with the property that the induced topology from $\R^3$ to the intersection of the line with the closure of the cell contains an open neighbourhood of $x$; therefore, there are $ \gtrsim S^{1/2}d^{-3/2}k^{3/2}$ incidences of this form between $\mathfrak{L}_{cell}$ and the boundary of the cell (essentially, if a line $l$ intersects the interior of a cell, we can choose one arbitrary point of the intersection of the line with the interior of the cell and move along the line starting from that point until we reach the boundary of the cell for the first time; if $x$ is the point of the boundary that we reach through this procedure,
then the pair $(x, l)$ can be the incidence between the line and the boundary of the particular cell that we take into account; we do not count incidences between this line and the boundary of the particular cell, with the property that locally around the intersection point the line lies outside the cell). 

On the other hand, if $x$ is a point of $Z$ which belongs to a line intersecting the interior of a cell, such that the induced topology from $\R^3$ to the intersection of the line with the closure of the cell contains an open neighbourhood of $x$,  then there exists at most one other cell whose interior is also intersected by the line and whose boundary contains $x$, such that the induced topology from $\R^3$ to the intersection of the line with the closure of that cell contains an open neighbourhood of $x$. This, in fact, is the reason why we only considered a particular type of incidences. More particularly, we are not, in general, able to bound nicely the number of all the cells whose boundaries all contain a point $x$ and whose interiors are all intersected by a line $l$ containing $x$, as the line could enter the interior of each of the cells only far from the point $x$. We know, however, that there exist at most two cells whose boundaries contain $x$ and such that $l$ lies in both their interiors locally around $x$. And the union of the boundaries of all the cells is the zero set $Z$ of the polynomial $p$. 

So, if $I$ is the number of incidences between $Z$ and $\mathfrak{L} \setminus \mathfrak{L}_Z$, $I_{cell}$ is the number of incidences between $\mathfrak{L}_{cell}$ and the boundary of the cell, and $\mathcal{C}$ is the set of all the full cells (which, in our case, has cardinality $\gtrsim d^3$), then the above imply that
\begin{displaymath} I \gtrsim \sum_{cell \in \mathcal{C}} I_{cell} \gtrsim (S^{1/2}d^{-3/2}k^{3/2})\cdot d^3 \sim S^{1/2}d^{3/2}k^{3/2}.\end{displaymath}
On the other hand, if a line does not lie in the zero set $Z$ of $p$, then it intersects $Z$ at $\leq d$ points. Thus,
\begin{displaymath} I \leq L \cdot d. \end{displaymath}
This means that 
\begin{displaymath} S^{1/2}d^{3/2}k^{3/2} \lesssim L \cdot d, \end{displaymath}
which in turn gives $A\lesssim 1$. In other words, there exists some constant $C$, independent of $\mathfrak{L}$, $N$ and $k$, such that $A \leq C$. By fixing $A$ to be a number larger than $C$ (and of course $\geq 2c_0$, so that $d> 1$), we have a contradiction.

Therefore, in the cellular case there exists some constant $c_1$, independent of $\mathfrak{L}$, $N$ and $k$, such that  \begin{displaymath} S \cdot N^{1/2} \leq c_1 \cdot \frac{L^{3/2}}{k^{1/2}}.\end{displaymath}

\textbf{Algebraic case:} Let $\mathfrak{G}_1$ denote the set of points in $\mathfrak{G}$ which lie in $Z$. Here, $|\mathfrak{G}_1|>(1-10^{-8})S$. We now analyse the situation.

Since each point of $\mathfrak{G}_1$ intersects at least $k$ lines of $\mathfrak{L}$, \begin{displaymath} I_{\mathfrak{G}_1,\mathfrak{L}} > (1-10^{-8})Sk. \end{displaymath}
Now, let $\mathfrak{L}'$ be the set of lines in $\mathfrak{L}$ each of which contains $\geq \frac{1}{100}SkL^{-1}$ points of $\mathfrak{G}_1$. Each line of $\mathfrak{L} \setminus \mathfrak{L}'$ intersects fewer than $\frac{1}{100}SkL^{-1}$ points of $\mathfrak{G}_1$, thus
\begin{displaymath} I_{\mathfrak{G}_1,\mathfrak{L}\setminus \mathfrak{L}'} \leq |\mathfrak{L}\setminus \mathfrak{L}'| \cdot \frac{Sk}{100L} \leq \frac{1}{100}Sk.\end{displaymath}
Therefore, since $I_{\mathfrak{G}_1, \mathfrak{L}}=I_{\mathfrak{G}_1,\mathfrak{L} \setminus \mathfrak{L}'}+I_{\mathfrak{G}_1,\mathfrak{L}'}$, it follows that 
\begin{displaymath}I_{\mathfrak{G}_1,\mathfrak{L}'}> (1-10^{-8}-10^{-2})Sk. \end{displaymath}
Thus, there are $\gtrsim Sk$ incidences between $\mathfrak{G}_1$ and $\mathfrak{L}'$; this, combined with the fact that there exist $\leq S$ points of $\mathfrak{G}$ in total, each intersecting $\leq 2k$ lines of $\mathfrak{L}$, implies that there exist $\gtrsim S$ points of $\mathfrak{G}_1$, each intersecting $\gtrsim k$ lines of $\mathfrak{L}'$. 

Let us now take a moment to look for a practical meaning of this: $\sim S$ of our initial points each lie in $\sim k$ lines of $\mathfrak{L}'$, which is a subset of our initial set of lines $\mathfrak{L}$. Thus, if $\mathfrak{L}'$ is a strict subset of $\mathfrak{L}$, and if many of these points are joints for $\mathfrak{L}'$ with multiplicity $\sim N$, we can use our induction hypothesis for $\mathfrak{L}'$ and solve the problem if $|\mathfrak{L}'|$ is significantly smaller than $L$; however, before being able to tackle the problem in the rest of the cases, we need to extract more information. 

To that end, we will need to use appropriate, explicit constants now hiding behing the $\gtrsim$ symbols, which we therefore go ahead and find.

More particularly, let $\mathfrak{G}'$ be the set of points of $\mathfrak{G}_1$ each of which intersects $\geq \frac{1-10^{-8}-10^{-2}}{2}k$ lines of $\mathfrak{L}'$.

Then, \begin{displaymath}I_{\mathfrak{G}_1 \setminus \mathfrak{G}', \mathfrak{L}'} \leq |\mathfrak{G}_1 \setminus \mathfrak{G}'|\cdot  \frac{1-10^{-8}-10^{-2}}{2}k \leq  \frac{1-10^{-8}-10^{-2}}{2}Sk,\end{displaymath}

therefore, since $I_{\mathfrak{G}_1, \mathfrak{L}'}= I_{\mathfrak{G}_1 \setminus \mathfrak{G}', \mathfrak{L}'}+I_{\mathfrak{G}', \mathfrak{L}'}$, it follows that
\begin{displaymath}I_{\mathfrak{G}', \mathfrak{L}'} > \frac{1-10^{-8}-10^{-2}}{2}Sk. \end{displaymath}
And obviously,  $I_{\mathfrak{G}', \mathfrak{L}'} \leq |\mathfrak{G}'|\cdot 2k$. Therefore, $ \frac{1-10^{-8}-10^{-2}}{2}Sk< |\mathfrak{G}'|\cdot 2k$, and thus
\begin{displaymath} |\mathfrak{G}'|\geq \frac{1-10^{-8}-10^{-2}}{4}S; \end{displaymath}
in other words, there exist at least $\frac{1-10^{-8}-10^{-2}}{4}S$ points of $\mathfrak{G}_1$, each intersecting $\geq \frac{1-10^{-8}-10^{-2}}{2}k$ lines of $\mathfrak{L}'$. 

Now, each point of $\mathfrak{G}_1$ lies in $Z$, so it is either a regular or a critical point of $Z$. Let $\mathfrak{G}_{crit}$ be the set of points of $\mathfrak{G}_1$ that are critical points of $Z$, and  $\mathfrak{G}_{reg}$ the set of points of $\mathfrak{G}_1$ that are regular points of $Z$; then, $\mathfrak{G}_1=\mathfrak{G}_{crit} \sqcup \mathfrak{G}_{reg}$.

We are in one of the following two subcases.

\textbf{The regular subcase:} At least $\frac{10^{-8}}{4}S$ points of $\mathfrak{G}_1$ are regular points of $Z$ ($|\mathfrak{G}_{reg}| \geq \frac{10^{-8}}{4}S$).

What we actually need to continue is that $Z$ contains $\gtrsim S$ points of $\mathfrak{G}$ that are regular. Now, if $x \in \mathfrak{G}$ is a regular point of $Z$, there exists a plane through it, containing all those lines through the point that are lying in $Z$ (otherwise, the point would be a critical point of $Z$). And, since $x$ is a joint for $\mathfrak{L}$, of multiplicity $\geq N$, lying in $\leq 2k$ lines of $\mathfrak{L}$, by Lemma \ref{3.2} there exist $\gtrsim \frac{N}{k^2}$ lines of $\mathfrak{L}$ passing through $x$, which are not lying on the plane; this means that these lines are not lying in $Z$, and thus each of them contains $\leq d$ points of $\mathfrak{G}_1$. Therefore, the number of incidences between $\mathfrak{G}_1$ and $\mathfrak{L} \setminus \mathfrak{L}_{Z}$ is $\gtrsim S \cdot \frac{N}{k^2}$, but also $\leq |\mathfrak{L} \setminus \mathfrak{L}_{Z}| \cdot d \leq L \cdot d$. Thus, $S \cdot \frac{N}{k^2} \lesssim L\cdot d$, which implies that $S \cdot N^{1/2} \lesssim L^{3/2}k^{-1/2}$.

Therefore, \begin{displaymath} S \cdot N^{1/2} \leq c_2 \cdot \frac{L^{3/2}}{k^{1/2}}, \end{displaymath} 
for some constant $c_2$ independent of $\mathfrak{L}$, $N$ and $k$.

\textbf{The critical subcase:} Fewer than $\frac{10^{-8}}{4}S$ points of $\mathfrak{G}_1$ are regular points of $Z$ ($|\mathfrak{G}_{reg}| < \frac{10^{-8}}{4}S$). Now, either $|\mathfrak{L'}| \geq \frac{L}{100}$ or $|\mathfrak{L'}| < \frac{L}{100}$.

$\bullet$ \textbf{Suppose that} $\boldsymbol{|\mathfrak{L'}| \geq \frac{L}{100}}$ \textbf{.} 

(The basic arguments for the proof of this case appear in the proof of \cite[Proposition 4.7]{Guth_Katz_2010}.)

We notice that, if $\frac{1}{200}SkL^{-1} \leq d$, then we obtain $S \lesssim L^{3/2}k^{-2}$ by rearranging, so $S \cdot  N^{1/2} \lesssim L^{3/2}k^{-1/2}$ (as $N \lesssim k^3$). 

Therefore, we assume from now on that $\frac{1}{200}SkL^{-1} \geq d+1$. Then, each line of $\mathfrak{L}'$ contains at least $d+1\geq \deg p+1$ points of the zero set $Z$ of $p$ (as $\mathfrak{G}_1$ lies in $Z$), and thus each line of $\mathfrak{L}'$ lies in $Z$.

Now, we know that each line of $\mathfrak{L}'$ contains $\geq \frac{1}{100} SkL^{-1}$ points of $\mathfrak{G}_1$. Therefore, it either contains $\geq \frac{1}{200} SkL^{-1}$ points of $\mathfrak{G}_{crit}$ or $\geq \frac{1}{200} SkL^{-1}$ points of $\mathfrak{G}_{reg}$. But $|\mathfrak{L}'| \geq \frac{L}{100}$, so, if $\mathfrak{L}_{crit}$ is the set of lines in $\mathfrak{L}'$ each containing $\geq \frac{1}{200} SkL^{-1}$ points of $\mathfrak{G}_{crit}$ and $\mathfrak{L}_{reg}$ is the set of lines in $\mathfrak{L}'$ each containing $\geq \frac{1}{200} SkL^{-1}$  points of $\mathfrak{G}_{reg}$, then either $|\mathfrak{L}_{crit}| \geq \frac{L}{200}$ or $|\mathfrak{L}_{reg}|\geq \frac{L}{200}$.

Let us suppose that, in fact, $|\mathfrak{L}_{reg}|\geq \frac{L}{200}$. This means that the incidences between $\mathfrak{L}$ and the points in $\mathfrak{G}$ which are regular points of $Z$ number at least $\frac{L}{200} \cdot \frac{1}{200} SkL^{-1} = \frac{1}{4 \cdot 10^{4}}  Sk$. However, there exist fewer than $\frac{10^{-8}}{4}S$ points of $\mathfrak{G}$ which are regular points of $Z$, and therefore they contribute fewer than $\frac{10^{-8}}{4}S \cdot 2k=\frac{1}{2 \cdot 10^{8}}Sk \lneq \frac{1}{4 \cdot 10^{4}}  Sk$ incidences with $\mathfrak{L}$; so, we are led to a contradiction. Therefore, $|\mathfrak{L}_{reg}|\lneq \frac{L}{200}$.

Thus, $|\mathfrak{L}_{crit}| \geq \frac{L}{200}$. Now, each line of $\mathfrak{L}_{crit}$ contains $\geq \frac{1}{200} SkL^{-1} \gneq d$ critical points of $Z$, i.e. $\gneq d$ points where $p$ and $\nabla{p}$ are zero. However, both $p$ and $\nabla{p}$ have degrees $\leq d$. Therefore, if $l \in \mathfrak{L}_{crit}$, then $p$ and $\nabla{p}$ are zero across the whole line $l$, so each point of $l$ is a critical point of $Z$; in other words, $l$ is a critical line of $Z$. So, the number of critical lines of $Z$ is $\geq |\mathfrak{L}_{crit}| \geq \frac{L}{200}$. On the other hand, the number of critical lines of $Z$ is $\leq d^2$ (Proposition \ref{2.2.3}). Therefore,
\begin{displaymath} \frac{L}{200} \leq d^2, \end{displaymath} which gives $S \lesssim L^{3/2}k^{-3}$ after rearranging. Thus, $S N^{1/2} \lesssim L^{3/2}k^{-3/2}$ ($\lesssim  L^{3/2}k^{-1/2}$).

In other words, \begin{displaymath} S N^{1/2} \leq c_3 \cdot \frac{L^{3/2}}{k^{1/2}}, \end{displaymath} 
for some constant $c_3$ independent of $\mathfrak{L}$, $N$ and $k$.

$\bullet$ \textbf{Suppose that} $\boldsymbol{|\mathfrak{L'}| < \frac{L}{100}}$ \textbf{.}

Since fewer than $\frac{10^{-8}}{4}$ points of $\mathfrak{G}_1$ are regular points of $Z$, the same holds for the subset $\mathfrak{G}'$ of $\mathfrak{G}_1$. So, at least $\frac{1-2 \cdot 10^{-8}-10^{-2}}{4}S$ points of $\mathfrak{G}'$ are critical points of $Z$.

Now, each of the points of $\mathfrak{G}'$ is a joint for $\mathfrak{L}$ with multiplicity in the interval $[N,2N)$, so it is either a joint for $\mathfrak{L}'$ with multiplicity in the interval  $[N/2, 2N)$, or it is a joint for $\mathfrak{L}'$ with multiplicity $< N/2$, or it is not a joint for $\mathfrak{L}'$. Therefore, one of the following two subcases holds.

\textbf{1st subcase:} There exist at least $\frac{1-2 \cdot 10^{-8}-10^{-2}}{8}S$ critical points in $\mathfrak{G}'$ each of which is either a joint for $\mathfrak{L}'$ with multiplicity $<N/2$ or not a joint at all for $\mathfrak{L}'$. Let $\mathfrak{G}_2$ be the set of those points.

By Lemma \ref{3.1}, for each point $x \in \mathfrak{G}_2$ there exist $\geq \frac{N}{1000 \cdot k^2}$ lines of $\mathfrak{L} \setminus \mathfrak{L}'$ passing through $x$. 

Now, let $\mathfrak{L}_3$ be the set of lines in $\mathfrak{L} \setminus \mathfrak{L}'$, such that each of them contains $\leq d$ critical points of $Z$. Then, one of the following two holds.

\textbf{(1)} There exist $\geq \frac{1-2 \cdot 10^{-8}-10^{-2}}{16}S$ points of $\mathfrak{G}_2$ such that each of them has $\geq \frac{N}{2000 \cdot k^2}$ lines of $\mathfrak{L}_3$ passing through it. Then, 
\begin{displaymath}S \cdot \frac{N}{k^2} \lesssim I_{\mathfrak{G}_2, \mathfrak{L}_3} \leq |\mathfrak{L}_3| \cdot d \leq L \cdot d.  \end{displaymath}  
Rearranging, obtain $S \cdot N^{1/2} \lesssim L^{3/2}k^{-1/2}$.

\textbf{(2)} There exist $\geq \frac{1-2 \cdot 10^{-8}-10^{-2}}{16}S$ points of $\mathfrak{G}_2$ such that each of them has $\geq \frac{N}{2000 \cdot k^2}$ lines of $\big(\mathfrak{L} \setminus  \mathfrak{L}'\big) \setminus \mathfrak{L}_3$ passing through it. Each line of $\big(\mathfrak{L} \setminus  \mathfrak{L}'\big) \setminus \mathfrak{L}_3$ contains $<\frac{1}{100}SkL^{-1}$ points of $\mathfrak{G}_1$. Also, it contains $>d$ critical points of $Z$, so it is a critical line. But $Z$ contains $\leq d^2$ critical lines in total (by Proposition \ref{2.2.3}). Therefore,
\begin{displaymath}S \cdot \frac{N}{k^2} \lesssim I_{\mathfrak{G}_2, \big(\mathfrak{L} \setminus  \mathfrak{L}'\big) \setminus \mathfrak{L}_3} \leq d^2 \cdot \frac{1}{100}SkL^{-1}, \end{displaymath}
so $S\cdot  N^{1/2} \lesssim L^{3/2}k^{-1/2}$, by rearranging.

Thus, in this 1st subcase,
\begin{displaymath} S \cdot N^{1/2} \leq c_4 \cdot \frac{L^{3/2}}{k^{1/2}},  \end{displaymath} 
where $c_4$ is a constant independent of $\mathfrak{L}$, $N$ and $k$.

We are now able to define the constant $c$ appearing in our induction process; we let $c:=\max\{2c_0, c_1, c_2, c_3, c_4\}$. Note that, in any case that has been dealt with so far, 
\begin{displaymath} S \cdot N^{1/2} \leq c \cdot \bigg(\frac{L^{3/2}}{k^{1/2}}+\frac{L}{k}\cdot N^{1/2}\bigg),\end{displaymath}
and $c$ is, indeed, an explicit, non-negative constant, independent of $\mathfrak{L}$, $N$ and $k$.

\textbf{2nd subcase:} At least $\frac{1-2 \cdot 10^{-8}-10^{-2}}{8}S$ points of $\mathfrak{G}'$ are joints for $\mathfrak{L}'$ with multiplicity in the interval $[\frac{N}{2},2N)$. Then, either \textbf{(1)} or \textbf{(2)} hold.

\textbf{(1)} At least $\frac{1-2 \cdot 10^{-8}-10^{-2}}{16}S$ points of $\mathfrak{G}'$ are joints for $\mathfrak{L}'$ with multiplicity in the interval $[N, 2N)$. However, each point of $\mathfrak{G}'$ intersects at least $\frac{1-10^{-8}-10^{-2}}{2} k$ and fewer than $2k$ lines of $\mathfrak{L}'$. Therefore, either ($1i$), ($1ii$) or ($1iii$) hold.

($1i$) At least $\frac{1-2 \cdot 10^{-8}-10^{-2}}{48}S$ points of $\mathfrak{G}'$ are joints for $\mathfrak{L}'$ with multiplicity in the interval $[N,2N)$, such that each of them lies in at least $k$ and fewer than $2k$ lines of $\mathfrak{L'}$. Then, since $|\mathfrak{L}' |< \frac{L}{100}\lneq L$, it follows from our induction hypothesis that
\begin{displaymath} \frac{1-2 \cdot 10^{-8}-10^{-2}}{48}S \cdot N^{1/2} \leq c \cdot \bigg( \frac{|\mathfrak{L}'|^{3/2}}{k^{1/2}} + \frac{|\mathfrak{L}'|}{k} \cdot N^{1/2} \bigg) \leq \end{displaymath}
\begin{displaymath}\leq  c \cdot \bigg( \frac{(L/100)^{3/2}}{k^{1/2}} + \frac{(L/100)}{k} \cdot N^{1/2} \bigg).  \end{displaymath}
However, \begin{displaymath} \frac{48}{1-2 \cdot 10^{-8}-10^{-2}} \cdot \frac{1}{100^{3/2}} <1\end{displaymath} and
\begin{displaymath}\frac{48}{1-2 \cdot 10^{-8}-10^{-2}} \cdot \frac{1}{100}<1, \end{displaymath}

therefore \begin{displaymath} S \cdot N^{1/2} \leq  c \cdot \bigg( \frac{L^{3/2}}{k^{1/2}} + \frac{L}{k} \cdot N^{1/2} \bigg).\end{displaymath}

($1ii$) At least $\frac{1-2 \cdot 10^{-8}-10^{-2}}{48}S$ points of $\mathfrak{G}'$ are joints for $\mathfrak{L}'$ with multiplicity in the interval $[N,2N)$, such that each of them lies in at least $\frac{1-10^{-8}-10^{-2}}{2}k$ and fewer than $(1-10^{-8}-10^{-2})k$ lines of $\mathfrak{L'}$. So, since $|\mathfrak{L}' |< \frac{L}{100}\lneq L$, it follows from our induction hypothesis that
\begin{displaymath} \frac{1-2 \cdot 10^{-8}-10^{-2}}{48}S \cdot N^{1/2} \leq c \cdot \Bigg( \frac{|\mathfrak{L}'|^{3/2}}{\big(\frac{1-10^{-8}-10^{-2}}{2}k \big)^{1/2}} + \frac{|\mathfrak{L}'|}{\big(\frac{1-10^{-8}-10^{-2}}{2}k \big)} \cdot N)^{1/2} \Bigg) \leq \end{displaymath}
\begin{displaymath}\leq  c \cdot \Bigg( \frac{(L/100)^{3/2}}{\big(\frac{1-10^{-8}-10^{-2}}{2}k \big)^{1/2}} + \frac{(L/100)}{\big(\frac{1-10^{-8}-10^{-2}}{2}k \big)} \cdot N^{1/2} \Bigg) . \end{displaymath}
However,
\begin{displaymath} \frac{48}{1-2 \cdot 10^{-8}-10^{-2}} \cdot \frac{1}{100^{3/2}} \cdot \frac{2^{1/2}}{(1-10^{-8}-10^{-2})^{1/2}} <1 \end{displaymath} and
\begin{displaymath}\frac{48}{1-2 \cdot 10^{-8}-10^{-2}} \cdot \frac{1}{100}\cdot \frac{2}{1-10^{-8}-10^{-2}}<1, \end{displaymath}
therefore 
\begin{displaymath} S \cdot  N^{1/2} \leq  c \cdot \bigg( \frac{L^{3/2}}{k^{1/2}} + \frac{L}{k} \cdot N^{1/2} \bigg).\end{displaymath}

($1iii$) At least $\frac{1-2 \cdot 10^{-8}-10^{-2}}{48}S$ points of $\mathfrak{G}'$ are joints for $\mathfrak{L}'$ with multiplicity in the interval $[N,2N)$, such that each of them lies in between $(1-10^{-8}-10^{-2})k$ and $2 \cdot (1-10^{-8}-10^{-2})k$ lines of $\mathfrak{L'}$. So, since $|\mathfrak{L}' |< \frac{L}{100}\lneq L$, it follows from our induction hypothesis that
\begin{displaymath} \frac{1-2 \cdot 10^{-8}-10^{-2}}{48}S \cdot N^{1/2} \leq c \cdot \Bigg( \frac{|\mathfrak{L}'|^{3/2}}{\big((1-10^{-8}-10^{-2})k)^{1/2}} +\end{displaymath}
\begin{displaymath}+ \frac{|\mathfrak{L}'|}{(1-10^{-8}-10^{-2})k} \cdot N^{1/2} \Bigg) \leq \end{displaymath}
\begin{displaymath}\leq  c \cdot \Bigg( \frac{(L/100)^{3/2}}{\big((1-10^{-8}-10^{-2})k \big)^{1/2}} + \frac{(L/100)}{(1-10^{-8}-10^{-2})k} \cdot N^{1/2} \Bigg) . \end{displaymath}
However,
\begin{displaymath} \frac{48}{1-2 \cdot 10^{-8}-10^{-2}} \cdot \frac{1}{100^{3/2}} \cdot \frac{1}{(1-10^{-8}-10^{-2})^{1/2}} <1 \end{displaymath} and
\begin{displaymath}\frac{48}{1-2 \cdot 10^{-8}-10^{-2}} \cdot \frac{1}{100}\cdot \frac{1}{1-10^{-8}-10^{-2}}<1, \end{displaymath}
therefore 
\begin{displaymath} S \cdot N^{1/2} \leq  c \cdot \bigg( \frac{L^{3/2}}{k^{1/2}} + \frac{L}{k} \cdot N^{1/2} \bigg).\end{displaymath}

\textbf{(2)} At least $\frac{1-2 \cdot 10^{-8}-10^{-2}}{16}S$ points of $\mathfrak{G}'$ are joints for $\mathfrak{L}'$ with multiplicity in the interval $[\frac{N}{2},N)$. However, each point of $\mathfrak{G}'$ intersects at least $\frac{1-10^{-8}-10^{-2}}{2}\cdot k$ and fewer than $2k$ lines of $\mathfrak{L}'$. Therefore, either ($2i$), ($2ii$) or ($2iii$) hold.

($2i$)  At least $\frac{1-2 \cdot 10^{-8}-10^{-2}}{48}S$ points of $\mathfrak{G}'$ are joints for $\mathfrak{L}'$ with multiplicity in the interval $[\frac{N}{2},N)$, such that each of them lies in at least $k$ and fewer than $2k$ lines of $\mathfrak{L'}$. Then, since $|\mathfrak{L}' |< \frac{L}{100}\lneq L$, it follows from our induction hypothesis that
\begin{displaymath} \frac{1-2 \cdot 10^{-8}-10^{-2}}{48}S \cdot \bigg(\frac{N}{2}\bigg)^{1/2} \leq c \cdot \Bigg( \frac{|\mathfrak{L}'|^{3/2}}{k^{1/2}} + \frac{|\mathfrak{L}'|}{k} \cdot \bigg(\frac{N}{2}\bigg)^{1/2} \Bigg) \leq \end{displaymath}
\begin{displaymath}\leq  c \cdot \Bigg( \frac{(L/100)^{3/2}}{k^{1/2}} + \frac{(L/100)}{k} \cdot \bigg(\frac{N}{2}\bigg)^{1/2} \Bigg).  \end{displaymath}
However, \begin{displaymath} \frac{48}{1-2 \cdot 10^{-8}-10^{-2}} \cdot \frac{1}{100^{3/2}} \cdot 2^{1/2} <1\end{displaymath} and
\begin{displaymath}\frac{48}{1-2 \cdot 10^{-8}-10^{-2}} \cdot \frac{1}{100}<1, \end{displaymath}
therefore 
\begin{displaymath} S \cdot N^{1/2} \leq  c \cdot \bigg( \frac{L^{3/2}}{k^{1/2}} + \frac{L}{k} \cdot N^{1/2} \bigg).\end{displaymath}

($2ii$) At least $\frac{1-2 \cdot 10^{-8}-10^{-2}}{48}S$ points of $\mathfrak{G}'$ are joints for $\mathfrak{L}'$ with multiplicity in the interval $[\frac{N}{2},N)$, such that each of them lies in at least $\frac{1-10^{-8}-10^{-2}}{2}\cdot k$ and fewer than $(1-10^{-8}-10^{-2})k$ lines of $\mathfrak{L'}$. So, since $|\mathfrak{L}' |< \frac{L}{100}\lneq L$, it follows from our induction hypothesis that
\begin{displaymath} \frac{1-2 \cdot 10^{-8}-10^{-2}}{48}S \cdot \bigg(\frac{N}{2}\bigg)^{1/2} \leq c \cdot \Bigg( \frac{|\mathfrak{L}'|^{3/2}}{\big(\frac{1-10^{-8}-10^{-2}}{2}k \big)^{1/2}} + \frac{|\mathfrak{L}'|}{\big(\frac{1-10^{-8}-10^{-2}}{2}k \big)} \cdot \bigg(\frac{N}{2}\bigg)^{1/2} \Bigg) \leq \end{displaymath}
\begin{displaymath}\leq  c \cdot \Bigg( \frac{(L/100)^{3/2}}{\big(\frac{1-10^{-8}-10^{-2}}{2}k \big)^{1/2}} + \frac{(L/100)}{\big(\frac{1-10^{-8}-10^{-2}}{2}k \big)} \cdot \bigg(\frac{N}{2}\bigg)^{1/2} \Bigg) . \end{displaymath}
However,
\begin{displaymath} \frac{48}{1-2 \cdot 10^{-8}-10^{-2}} \cdot \frac{1}{100^{3/2}} \cdot \frac{2^{1/2}}{(1-10^{-8}-10^{-2})^{1/2}} \cdot 2^{1/2} <1 \end{displaymath} and
\begin{displaymath}\frac{48}{1-2 \cdot 10^{-8}-10^{-2}} \cdot \frac{1}{100}\cdot \frac{2}{1-10^{-8}-10^{-2}}<1, \end{displaymath}
therefore 
\begin{displaymath} S \cdot N^{1/2} \leq  c \cdot \bigg( \frac{L^{3/2}}{k^{1/2}} + \frac{L}{k} \cdot N^{1/2} \bigg).\end{displaymath}

($2iii$) At least $\frac{1-2 \cdot 10^{-8}-10^{-2}}{48}S$ points of $\mathfrak{G}'$ are joints for $\mathfrak{L}'$ with multiplicity in the interval $[\frac{N}{2},N)$, such that each of them lies in at least $(1-10^{-8}-10^{-2})k$ and fewer than $2 \cdot (1-10^{-8}-10^{-2})k$ lines of $\mathfrak{L'}$. So, since $|\mathfrak{L}' |< \frac{L}{100}\lneq L$, it follows from our induction hypothesis that
\begin{displaymath} \frac{1-2 \cdot 10^{-8}-10^{-2}}{48}S \cdot \bigg(\frac{N}{2}\bigg)^{1/2} \leq c \cdot \Bigg( \frac{|\mathfrak{L}'|^{3/2}}{\big((1-10^{-8}-10^{-2})k)^{1/2}} + \end{displaymath}
\begin{displaymath}+\frac{|\mathfrak{L}'|}{(1-10^{-8}-10^{-2})k} \cdot \bigg(\frac{N}{2}\bigg)^{1/2} \Bigg) \leq \end{displaymath}
\begin{displaymath}\leq  c \cdot \Bigg( \frac{(L/100)^{3/2}}{\big((1-10^{-8}-10^{-2})k \big)^{1/2}} + \frac{(L/100)}{(1-10^{-8}-10^{-2})k} \cdot \bigg(\frac{N}{2}\bigg)^{1/2} \Bigg) . \end{displaymath}
However,
\begin{displaymath} \frac{48}{1-2 \cdot 10^{-8}-10^{-2}} \cdot \frac{1}{100^{3/2}} \cdot \frac{2^{1/2}}{(1-10^{-8}-10^{-2})^{1/2}} <1 \end{displaymath} and
\begin{displaymath}\frac{48}{1-2 \cdot 10^{-8}-10^{-2}} \cdot \frac{1}{100}\cdot \frac{1}{1-10^{-8}-10^{-2}}<1, \end{displaymath}
therefore 
\begin{displaymath} S \cdot N^{1/2} \leq  c \cdot \bigg( \frac{L^{3/2}}{k^{1/2}} + \frac{L}{k} \cdot N^{1/2} \bigg).\end{displaymath}

We have by now exhausted all the possible cases; in each one,

\begin{displaymath} S \cdot N^{1/2} \leq  c \cdot \bigg( \frac{L^{3/2}}{k^{1/2}} + \frac{L}{k} \cdot N^{1/2} \bigg), \end{displaymath}

where $c$ is, by its definition, a constant independent of $\mathfrak{L}$, $N$ and $k$.

Therefore, as $N$ and $k$ were arbitrary, \eqref{eq:final} holds for this collection $\mathfrak{L}$ of lines in $\R^3$. And since $\mathfrak{L}$ was an arbitrary collection of $L$ lines, \eqref{eq:final} holds for any collection $\mathfrak{L}$ of $L$ lines in $\R^3$.

Consequently, the proposition is proved.

\end{proof}

Now, Theorem \ref{1.1} will easily follow.

\textbf{Theorem \ref{1.1}.} \textit{Let $\mathfrak{L}$ be a collection of $L$ lines in $\R^3$, forming a set $J$ of joints. Then,}
\begin{displaymath} \sum_{x \in J}N(x)^{1/2} \leq c \cdot L^{3/2},\end{displaymath} 
\textit{where $c$ is a constant independent of $\mathfrak{L}$.}

\begin{proof}

The multiplicity of each joint in $J$ can be at most $ \binom{L}{3}\leq L^3$. Therefore,
\begin{displaymath} \sum_{x \in J}N(x)^{{1/2}}\leq 2 \cdot \sum_{\{\lambda \in \N:\; 2^{\lambda} \leq L^3\}}|J_{2^{\lambda}}| \cdot (2^{\lambda})^{1/2}.\end{displaymath}
However, if $x$ is a joint for $\mathfrak{L}$ with multiplicity $N$, such that fewer than $2k$ lines of $\mathfrak{L}$ are passing through $x$, then $N \leq \binom{2k}{3} \leq (2k)^3$, and thus $k \geq \frac{1}{2}N^{1/3}$. Therefore, for all $\lambda \in \N$ such that $2^{\lambda} \leq L^3$, 
\begin{displaymath} |J_{2^{\lambda}}|=\sum_{\big\{\mu \in \N:\; 2^{\mu}\geq \frac{1}{2}(2^{\lambda})^{1/3} \big\}}|J^{2^\mu}_{2^{\lambda}}|,\end{displaymath} thus
\begin{displaymath} |J_{2^{\lambda}}|\cdot (2^{\lambda})^{1/2} =\sum_{\big\{\mu \in \N:\; 2^{\mu}\geq \frac{1}{2}(2^{\lambda})^{1/3}\big\}}|J^{2^\mu}_{2^{\lambda}}|\cdot (2^{\lambda})^{1/2},\end{displaymath}
a quantity which, by Proposition \ref{1.2}, is 
\begin{displaymath} \leq \sum_{\big\{\mu \in \N:\; 2^{\mu}\geq \frac{1}{2}(2^{\lambda})^{1/3}\big\}} c \cdot \Bigg(\frac{L^{3/2}}{(2^{\mu}) ^{1/2}} + \frac{L}{2^{\mu}}\cdot (2^{\lambda})^{1/2}\Bigg) \leq \end{displaymath}
\begin{displaymath} \leq c' \cdot \Bigg(\frac{L^{3/2}}{\big((2^{\lambda})^{1/3}\big)^{1/2}} + \frac{L}{(2^{\lambda})^{1/3}}\cdot (2^{\lambda})^{1/2}\Bigg), \end{displaymath} 
where $c'$ is a constant independent of $\mathfrak{L}$, $k$ and $\lambda$.

Therefore, 
\begin{displaymath} \sum_{x \in J} N(x)^{1/2}\leq 2c'\cdot \sum_{\{\lambda \in \N:\; 2^{\lambda}\leq L^3\}} \Bigg(\frac{L^{3/2}}{(2^{\lambda})^{1/6}} + L\cdot (2^{\lambda})^{1/6}\Bigg) \leq c'' \cdot \big(L^{3/2} + L\cdot L^{1/2}\big)= c''\cdot L^{3/2}, \end{displaymath}
where $c''$ is a constant independent of $\mathfrak{L}$.

The proof of Theorem \ref{1.1} is now complete.

\end{proof}

\newpage
\thispagestyle{plain}
\cleardoublepage

\chapter{Counting multijoints} \label{4}

In this chapter we prove Theorem \ref{theoremmult2} and Theorem \ref{theoremmult1}. Let us remember their statements.

\textbf{Theorem \ref{theoremmult2}.} \textit{Let }$\mathfrak{L}_1$, $\mathfrak{L}_2$, $\mathfrak{L}_3$\textit{ be finite collections of lines of $L_1$, $L_2$ and $L_3$, respectively, lines in }$\R^3$\textit{. Let }$J$\textit{ be the set of multijoints formed by the collections }$\mathfrak{L}_1$, $\mathfrak{L}_2$\textit{ and }$\mathfrak{L}_3$\textit{. Then,}
\begin{displaymath}|J| \leq c \cdot (L_1L_2L_3)^{1/2}, \end{displaymath}
\textit{where }$c$\textit{ is a constant independent of }$\mathfrak{L}_1$, $\mathfrak{L}_2$\textit{ and }$\mathfrak{L}_3$.

\textbf{Theorem \ref{theoremmult1}.} \textit{Let }$\mathfrak{L}_1$, $\mathfrak{L}_2$, $\mathfrak{L}_3$\textit{ be finite collections of $L_1$, $L_2$ and $L_3$, respectively, lines in }$\R^3$\textit{, such that, whenever a line of }$\mathfrak{L}_1$\textit{, a line of }$\mathfrak{L}_2$\textit{ and a line of }$\mathfrak{L}_3$\textit{ meet at a point, they form a joint there. Let }$J$\textit{ be the set of multijoints formed by the collections }$\mathfrak{L}_1$, $\mathfrak{L}_2$\textit{ and }$\mathfrak{L}_3$\textit{. Then,}
\begin{displaymath}\sum_{\{x \in J:\; N_m(x)> 
10^{12}\}}(N_1(x)N_2(x)N_3(x))^{1/2} \leq c\cdot (L_1L_2L_3)^{1/2}, \end{displaymath}
\textit{where }$m \in \{1,2,3\}$\textit{ is such that }$L_m=\min\{L_1,L_2,L_3\}$\textit{, and }$c$\textit{ is a constant independent of $\mathfrak{L}_1$, $\mathfrak{L}_2$ and $\mathfrak{L}_3$.}

Our proof of Theorem \ref{theoremmult1} is achieved in three steps, the first of which ensures the truth of Theorem \ref{theoremmult2}.

In the first step, we prove the following proposition, by induction on $L_1$, $L_2$ and $L_3$.

\textbf{Proposition \ref{multsimple}.} \textit{Let $\mathfrak{L}_1$, $\mathfrak{L}_2$, $\mathfrak{L}_3$ be finite collections of $L_1$, $L_2$, and $L_3$, respectively, lines in $\R^3$. For all $(N_1,N_2,N_3) \in \R_+^3$, let $J'_{N_1,N_2,N_3}$ be the set of those multijoints formed by $\mathfrak{L}_1$, $\mathfrak{L}_2$ and $\mathfrak{L}_3$, with the property that, if $x \in J'_{N_1,N_2,N_3}$, then there exist collections $\mathfrak{L}_1(x) \subseteq \mathfrak{L}_1$, $\mathfrak{L}_2(x) \subseteq \mathfrak{L}_2$ and $\mathfrak{L}_3(x) \subseteq \mathfrak{L}_3$ of lines passing through $x$, such that $|\mathfrak{L}_1(x)|\geq N_1$, $|\mathfrak{L}_2(x)|\geq N_2$ and $|\mathfrak{L}_3(x)|\geq N_3$, and, if $l_1 \in \mathfrak{L}_1(x)$, $l_2 \in \mathfrak{L}_2(x)$ and $l_3 \in \mathfrak{L}_3(x)$, then the directions of the lines $l_1$, $l_2$ and $l_3$ span $\R^3$. Then,}
\begin{displaymath} |J'_{N_1,N_2,N_3}|\leq c \cdot \frac{(L_1L_2L_3)^{1/2}}{(N_1N_2N_3)^{1/2}}, \; \forall\;(N_1,N_2,N_3) \in \R_{+}^3, \end{displaymath}
\textit{where $c$ is a constant independent of $\mathfrak{L}_1$, $\mathfrak{L}_2$ and $\mathfrak{L}_3$.}

Theorem \ref{theoremmult2} obviously follows (as we have already mentioned, it is an application of Proposition \ref{multsimple} for $(N_1,N_2,N_3)=(1,1,1)$).

In the second step, independently of the first step, we show the following.

\begin{proposition} \label{pointslarge} Let $\mathfrak{L}_1$, $\mathfrak{L}_2$, $\mathfrak{L}_3$ be finite collections of $L_1$, $L_2$ and $L_3$, respectively, lines in $\R^3$. For all $x \in \R^3$ and $i=1,2,3$, we denote by $N_i(x)$ the number of lines of $\mathfrak{L}_i$ passing through $x$. Also, for each $k \in \{1,2,3\}$, let $c_k$ be a constant such that $c_k \cdot L_k^{1/2}\gneq 1$. Then,

\begin{displaymath} \sum_{\big\{x \in \R^3:\;N_k(x)\geq c_kL_k^{1/2},\text{ for some }k \in \{1,2,3\}\big\}} (N_1(x)N_2(x)N_3(x))^{1/2}\leq c \cdot (L_1L_2L_3)^{1/2}, \end{displaymath}
where $c$ is a constant depending on $c_1$, $c_2$ and $c_3$, but independent of $\mathfrak{L}_1$, $\mathfrak{L}_2$ and $\mathfrak{L}_3$.
\end{proposition}

\textbf{Remark.} Note that, in the statement of Proposition \ref{pointslarge}, it is crucial that the points $x \in \R^3$ contributing to the sum are such that $N_k(x)$ is large; it is easy to see that, with the notation of Proposition \ref{pointslarge}, it does not hold in general that
\begin{equation} \label{eq:lattice} \sum_{x \in \R^3}(N_1(x)N_2(x)N_3(x))^{1/2}\leq c \cdot (L_1L_2L_3)^{1/2}.
\end{equation}
An example is the case where $\mathfrak{L}_1$, $\mathfrak{L}_2$ and $\mathfrak{L}_3$ are the sets of lines lying on the plane $z=0$ in $\R^3$, such that $\mathfrak{L}_1$ consists of $L$ lines, each passing through a point of the form $(i,0,0)$, for $i \in \{1,...,L\}$, and each parallel to the $y$--axis, $\mathfrak{L}_2$ consists of $L$ lines, each passing through a point of the form $(0,j,0)$, for $j \in \{1,...,L\}$, and each parallel to the $x$--axis, and $\mathfrak{L}_3$ consists of $\sim L$ lines, each parallel to the line $y=x$ on the plane $z=0$, such that each of the $L^2$ points of the set $\{(i,j,0):(i,j) \in \{1,...,L\}^2\}$ is contained in a line of $\mathfrak{L}_3$. Then, \eqref{eq:lattice} becomes
$$L^2 \lesssim (L \cdot L\cdot L)^{1/2},
$$
which does not hold for large $L$.

Let us also emphasise that Proposition \ref{pointslarge} does not hold in general in the case where, for some $i=1,2,3$, the collection $\mathfrak{L}_i$ contains more than one copy of the same line, not even if only multijoints of the collections of $\mathfrak{L}_1$, $\mathfrak{L}_2$ and $\mathfrak{L}_3$ contribute to the sum in \eqref{eq:lattice}. 

Indeed, if, in the above example, for all $i=1,2,3$ and all $l \in \mathfrak{L}_i$, we add another $k-1$ copies of the line $l$ in $\mathfrak{L}_i$, for some $k \geq L$, and, moreover, we assume that $\mathfrak{L}_3$ also contains one line through each point of the set $\{(i,j,0):(i,j) \in \{1,...,L\}^2\}$ that is perpendicular to the plane $z=0$, then we end up with collections $\mathfrak{L}_1$, $\mathfrak{L}_2$ and $\mathfrak{L}_3$ of $kL$, $kL$ and $kL+L^2 \sim kL$, respectively, lines in $\R^3$, such that, for all $x \in \{(i,j,0):(i,j) \in \{1,...,L\}^2\}$, $N_1(x)=k$, $N_2(x)=k$ and $N_3(x)=k+1\sim k$, while $N_1(x)N_2(x)N_3(x)=0$ for all $x \in \R^3 \setminus \{(i,j,0):(i,j) \in \{1,...,L\}^2\}$. Therefore, \eqref{eq:lattice} becomes
$$L^2\cdot (k \cdot k \cdot k)^{1/2} \lesssim (kL \cdot kL \cdot kL)^{1/2},
$$i.e.
$$L^2k^{3/2} \lesssim L^{3/2}k^{3/2},
$$
which does not hold for large $L$.

The fact that Proposition \ref{pointslarge} does not hold in general in the case where, for some $i=1,2,3$, the collection $\mathfrak{L}_i$ contains more than one copy of the same line, is mirrored in the proof we will provide by the fact that our arguments involve use of the Szemer\'edi-Trotter theorem, which is not scale invariant.

$\;\;\;\;\;\;\;\;\;\;\;\;\;\;\;\;\;\;\;\;\;\;\;\;\;\;\;\;\;\;\;\;\;\;\;\;\;\;\;\;\;\;\;\;\;\;\;\;\;\;\;\;\;\;\;\;\;\;\;\;\;\;\;\;\;\;\;\;\;\;\;\;\;\;\;\;\;\;\;\;\;\;\;\;\;\;\;\;\;\;\;\;\;\;\;\;\;\;\;\;\;\;\;\;\;\;\;\;\;\;\;\;\;\;\;\;\;\;\;\;\;\;\;\;\;\;\;\;\;\;\;\;\;\blacksquare$

In fact, the reason we are interested in Proposition \ref{pointslarge} is that, applying it for a set of multijoints formed by three collections of lines in $\R^3$, we immediately establish the following result.

\begin{corollary} \label{multlarge} Let $\mathfrak{L}_1$, $\mathfrak{L}_2$, $\mathfrak{L}_3$ be finite collections of $L_1$, $L_2$ and $L_3$, respectively, lines in $\R^3$. Also, for each $k \in \{1,2,3\}$, let $c_k$ be a constant such that $c_k \cdot L_k^{1/2}\gneq 1$. If $J$ denotes the set of multijoints formed by the collections $\mathfrak{L}_1$, $\mathfrak{L}_2$ and $\mathfrak{L}_3$, then

\begin{equation} \label{eq:big}\sum_{\big\{x \in J:\;N_k(x)\geq c_k L_k^{1/2},\text{ for some }k \in \{1,2,3\}\big\}} (N_1(x)N_2(x)N_3(x))^{1/2}\leq c \cdot  (L_1L_2L_3)^{1/2},\end{equation}
where $c$ is a constant depending on $c_1$, $c_2$ and $c_3$, but independent of $\mathfrak{L}_1$, $\mathfrak{L}_2$ and $\mathfrak{L}_3$.
\end{corollary}

Finally, in the third step we prove Theorem \ref{theoremmult1}, using Proposition \ref{multsimple} and Corollary \ref{multlarge}. 

Let us now proceed with the first step.

\textbf{Proposition \ref{multsimple}.} \textit{Let $\mathfrak{L}_1$, $\mathfrak{L}_2$, $\mathfrak{L}_3$ be finite collections of $L_1$, $L_2$, and $L_3$, respectively, lines in $\R^3$. For all $(N_1,N_2,N_3) \in \R_+^3$, let $J'_{N_1,N_2,N_3}$ be the set of those multijoints formed by $\mathfrak{L}_1$, $\mathfrak{L}_2$ and $\mathfrak{L}_3$, with the property that, if $x \in J'_{N_1,N_2,N_3}$, then there exist collections $\mathfrak{L}_1(x) \subseteq \mathfrak{L}_1$, $\mathfrak{L}_2(x) \subseteq \mathfrak{L}_2$ and $\mathfrak{L}_3(x) \subseteq \mathfrak{L}_3$ of lines passing through $x$, such that $|\mathfrak{L}_1(x)|\geq N_1$, $|\mathfrak{L}_2(x)|\geq N_2$ and $|\mathfrak{L}_3(x)|\geq N_3$, and, if $l_1 \in \mathfrak{L}_1(x)$, $l_2 \in \mathfrak{L}_2(x)$ and $l_3 \in \mathfrak{L}_3(x)$, then the directions of the lines $l_1$, $l_2$ and $l_3$ span $\R^3$. Then,}
\begin{displaymath} |J'_{N_1,N_2,N_3}|\leq c \cdot \frac{(L_1L_2L_3)^{1/2}}{(N_1N_2N_3)^{1/2}}, \; \forall\;(N_1,N_2,N_3) \in \R_{+}^3, \end{displaymath}
\textit{where $c$ is a constant independent of $\mathfrak{L}_1$, $\mathfrak{L}_2$ and $\mathfrak{L}_3$.}

\textbf{Remark.} The proof of Proposition \ref{multsimple} is based on counting incidences between multijoints $x \in J'_{N_1,N_2,N_3}$ and the lines of $\mathfrak{L}_i(x)$ passing through $x$, for $i=1,2,3$, rather than incidences between $x$ and all the lines of $\mathfrak{L}_i$ passing through $x$, for $i=1,2,3$ (which may not have the nice property that, whenever three lines, one of each collection, meet at $x$, they form a multijoint there). Thus, the proof of the Proposition is simpler to state in the case where all the lines of $\mathfrak{L}_1$, $\mathfrak{L}_2$ and $\mathfrak{L}_3$ have the property that, whenever three lines, one of each collection, meet at a point, they form a multijoint there. Indeed, in that case, for each multijoint $x$ formed by $\mathfrak{L}_1$, $\mathfrak{L}_2$ and $\mathfrak{L}_3$, $\mathfrak{L}_i(x)$ can be considered to be the whole set of lines of $\mathfrak{L}_i$ passing through $x$, for all $i=1,2,3$; therefore, in the proof we may consider incidences between $x$ and all the lines of $\mathfrak{L}_i$ passing through $x$, for all $i=1,2,3$. Proposition \ref{multsimple}, in its generality, follows from the observation that we can restrict our attention to incidences between $x\in J'_{N_1,N_2,N_3}$ and the lines of $\mathfrak{L}_i(x)\subseteq \mathfrak{L}_i$, $i=1,2,3$, which, by their definition, have nice transversality properties.

$\;\;\;\;\;\;\;\;\;\;\;\;\;\;\;\;\;\;\;\;\;\;\;\;\;\;\;\;\;\;\;\;\;\;\;\;\;\;\;\;\;\;\;\;\;\;\;\;\;\;\;\;\;\;\;\;\;\;\;\;\;\;\;\;\;\;\;\;\;\;\;\;\;\;\;\;\;\;\;\;\;\;\;\;\;\;\;\;\;\;\;\;\;\;\;\;\;\;\;\;\;\;\;\;\;\;\;\;\;\;\;\;\;\;\;\;\;\;\;\;\;\;\;\;\;\;\;\;\;\;\;\;\;\blacksquare$

\textit{Proof of Proposition \ref{multsimple}.} The proof will be achieved by induction on $L_1$, $L_2$ and $L_3$. Indeed, fix $(M_1,M_2,M_3)$ $\in {\N^*}^3$. For a constant $c \geq 1$ that will be specified later:

(i) It holds that \begin{displaymath} |J'_{N_1,N_2,N_3}|  \leq c \cdot \frac{(L_1L_2L_3)^{1/2}}{(N_1N_2N_3)^{1/2}}, \; \forall\;(N_1,N_2,N_3) \in \R_{+}^3,\end{displaymath} for any collections $\mathfrak{L}_1$, $\mathfrak{L}_2$, $\mathfrak{L}_3$ of $L_1$, $L_2$ and $L_3$, respectively, lines in $\R^3$, such that $L_1=L_2=L_3=1$. This is obvious, in fact, for any $c\geq 1$, as in this case $|J'_{N_1,N_2,N_3}|=0$ for all $(N_1,N_2,N_3)$ in $\R^3_+$ such that $N_i \gneq 1$ for some $i \in \{1,2,3\}$, while, for $(N_1,N_2,N_3)$ in $\R^3_+$ such that $N_i \leq 1$ for all $i \in \{1,2,3\}$, $|J'_{N_1,N_2,N_3}|$ is equal to at most 1.

(ii) Suppose that \begin{displaymath} |J'_{N_1,N_2,N_3}|  \leq c \cdot \frac{(L_1L_2L_3)^{1/2}}{(N_1N_2N_3)^{1/2}}, \; \forall\;(N_1,N_2,N_3) \in \R_{+}^3,\end{displaymath} for any collections $\mathfrak{L}_1$, $\mathfrak{L}_2$, $\mathfrak{L}_3$ of $L_1$, $L_2$ and $L_3$, respectively, lines in $\R^3$, such that $L_1\lneq M_1$, $L_2 \lneq M_2$ and $L_3 \lneq M_3$.

(iii) We will prove that  \begin{displaymath} |J'_{N_1,N_2,N_3}|  \leq c \cdot \frac{(L_1L_2L_3)^{1/2}}{(N_1N_2N_3)^{1/2}}, \; \forall\;(N_1,N_2,N_3) \in \R_{+}^3,\end{displaymath} for any collections $\mathfrak{L}_1$, $\mathfrak{L}_2$, $\mathfrak{L}_3$ of $L_1$, $L_2$ and $L_3$, respectively, lines in $\R^3$, such that $L_j= M_j$ for some $j \in \{1,2,3\}$ and $L_i \lneq M_i$, $L_k \lneq M_k$ for $\{i,k\}=\{1,2,3\}\setminus \{j\}$.

Indeed, fix such collections $\mathfrak{L}_1$, $\mathfrak{L}_2$ and $\mathfrak{L}_3$ of lines and let $(N_1,N_2,N_3) \in \R_{+}^3$.

For simplicity, let $$\mathfrak{G}:= J'_{N_1,N_2,N_3}$$and$$S:=|J'_{N_1,N_2,N_3}|.$$

We assume that \begin{displaymath} \frac{L_1}{\lceil N_1 \rceil} \leq \frac{L_2}{\lceil N_2 \rceil} \leq \frac{L_3}{\lceil N_3 \rceil}. \end{displaymath}

By the definition of $\mathfrak{G}$, if $x \in \mathfrak{G}$, then there exist at least $\lceil N_1 \rceil$ lines of $\mathfrak{L}_1$ and at least $\lceil N_2 \rceil$ lines of $\mathfrak{L}_2$ passing through $x$. Thus, the quantity $S\lceil N_1 \rceil \lceil N_2 \rceil$ is equal to at most the number of pairs of the form $(l_1,l_2)$, where $l_1 \in \mathfrak{L}_1$, $l_2 \in \mathfrak{L}_2$ and the lines $l_1$ and $l_2$ pass through the same point of $\mathfrak{G}$. Therefore, $S\lceil N_1 \rceil \lceil N_2 \rceil$ is equal to at most the number of all the pairs of the form $(l_1,l_2)$, where $l_1 \in \mathfrak{L}_1$ and $l_2 \in \mathfrak{L}_2$, i.e. to at most $L_1L_2$. So,
\begin{displaymath}S\lceil N_1\rceil \lceil N_2 \rceil \leq L_1L_2,\end{displaymath}
and therefore
\begin{displaymath} \frac{L_1L_2}{S\lceil N_1\rceil\lceil N_2\rceil}\geq 1. \end{displaymath}

Thus, $d:=A\frac{L_1L_2}{S\lceil N_1\rceil\lceil N_2\rceil}$ is a quantity $>1$ for $A>1$. We therefore assume that $A>1$, and we will specify its value later. Now, applying the Guth-Katz polynomial method for this $d>1$ and the finite set of points $\mathfrak{G}$, we deduce that there exists a non-zero polynomial $p\in \R[x,y,z]$, of degree $\leq d$, whose zero set $Z$ decomposes $\R^3$ in $\sim d^3$ cells, each of which contains $\lesssim Sd^{-3}$ points of $\mathfrak{G}$. We can assume that this polynomial is square-free, as eliminating the squares of $p$ does not inflict any change on its zero set. 

Let us assume that there are $\geq 10^{-8}S$ points of $\mathfrak{G}$ in the union of the interiors of the cells; by choosing $A$ to be a sufficiently large constant, we will be led to a contradiction.

Indeed, there are $\gtrsim S$ points of $\mathfrak{G}$ in the union of the interiors of the cells. However, there exist $\sim d^3$ cells in total, each containing $\lesssim Sd^{-3}$ points of $\mathfrak{G}$. Therefore, there exist $\gtrsim d^3$ cells, with $\gtrsim Sd^{-3}$ points of $\mathfrak{G}$ in the interior of each. We call the cells with this property ``full cells". 

For every full cell,  let $\mathfrak{G}_{cell}$ be the set of points of $\mathfrak{G}$ in the interior of the cell, $\mathfrak{L}_{1,cell}$ and $\mathfrak{L}_{2,cell}$ the sets of lines of $\mathfrak{L}_1$ and $\mathfrak{L}_2$, respectively, containing at least one point of $\mathfrak{G}_{cell}$, $S_{cell}:=|\mathfrak{G}_{cell}|$, $L_{1,cell}:=|\mathfrak{L}_{1,cell}|$ and $L_{2,cell}:= |\mathfrak{L}_{2,cell}|$. Now,
\begin{displaymath}S_{cell}\lceil N_1\rceil \lceil N_2 \rceil \lesssim L_{1,cell}L_{2,cell}, \end{displaymath}
as the quantity $S_{cell}\lceil N_1\rceil \lceil N_2 \rceil$ is equal to at most the number of pairs of the form $(l_1,l_2)$, where $l_1 \in \mathfrak{L}_{1,cell}$, $l_2 \in \mathfrak{L}_{2,cell}$ and the lines $l_1$ and $l_2$ pass through the same point of $\mathfrak{G}_{cell}$. Thus, it is equal to at most $L_{1,cell}L_{2,cell}$, which is the number of all the pairs of the form $(l_1,l_2)$, where $l_1 \in \mathfrak{L}_{1,cell}$ and $l_2 \in \mathfrak{L}_{2,cell}$.

Therefore,
\begin{displaymath}(L_{1,cell}L_{2,cell})^{1/2} \gtrsim S_{cell}^{1/2} (\lceil N_1 \rceil \lceil N_2 \rceil )^{1/2} \gtrsim \frac{S^{1/2}}{d^{3/2}}(\lceil N_1\rceil\lceil N_2\rceil)^{1/2}. \end{displaymath}
But, for $i=1,2$, each of the lines of $\mathfrak{L}_{i,cell}$ intersects the boundary of the cell at at least one point $x$, with the property that the induced topology from $\R^3$ to the intersection of the line with the closure of the cell contains an open neighbourhood of $x$; therefore, there are $ \gtrsim L_{i,cell}$ incidences of this form between $\mathfrak{L}_{i,cell}$ and the boundary of the cell. Also, the union of the boundaries of all the cells is the zero set $Z$ of $p$, and if $x$ is a point of $Z$ which belongs to a line intersecting the interior of a cell, such that the induced topology from $\R^3$ to the intersection of the line with the closure of the cell contains an open neighbourhood of $x$,  then there exists at most one other cell whose interior is also intersected by the line and whose boundary contains $x$, such that the induced topology from $\R^3$ to the intersection of the line with the closure of that cell contains an open neighbourhood of $x$. So, if $I_i$ is the number of incidences between $Z$ and the lines of $\mathfrak{L}_{i}$ not lying in $Z$, $I_{i,cell}$ is the number of incidences between $\mathfrak{L}_{i,cell}$ and the boundary of the cell, and $\mathcal{C}$ is the set of all the full cells (which, in our case, has cardinality $\gtrsim d^3$), then the above imply that
\begin{displaymath} I_i \gtrsim \sum_{cell \in \mathcal{C}} I_{i,cell} \gtrsim \sum_{cell \in \mathcal{C}}L_{i,cell},\text{ for }i=1,2.\end{displaymath}
On the other hand, if a line does not lie in the zero set $Z$ of $p$, then it intersects $Z$ in $\leq d$ points. Thus,
\begin{displaymath} I_i \leq L_{i} \cdot d,\text{ for }i=1,2. \end{displaymath}
Therefore, 
\begin{displaymath} \sum_{cell \in \mathcal{C}} \frac{S^{1/2}}{d^{3/2}}(\lceil N_1\rceil\lceil N_2 \rceil)^{1/2} \lesssim \sum_{cell \in \mathcal{C}} (L_{1,cell}L_{2,cell})^{1/2} \lesssim \end{displaymath}
\begin{displaymath} \lesssim\Bigg(\sum_{cell \in\mathcal{C}} L_{1,cell}\Bigg)^{1/2} \Bigg(\sum_{cell \in\mathcal{C}}L_{2,cell}\Bigg)^{1/2}\lesssim  I_1^{1/2}I_2^{1/2}\lesssim (L_1d)^{1/2}(L_2d)^{1/2} \sim (L_1L_2)^{1/2}d.\end{displaymath}

But the full cells number $\gtrsim d^3$. Thus, 
\begin{displaymath} S^{1/2}(\lceil N_1\rceil\lceil N_2\rceil)^{1/2}d^{3/2} \lesssim (L_1L_2)^{1/2}d,\end{displaymath}
from which we obtain
\begin{displaymath} d \lesssim \frac{L_1L_2}{S\lceil N_1 \rceil\lceil N_2\rceil}, \end{displaymath}
which in turn gives $A\lesssim 1$. In other words, there exists some explicit constant $C$, such that $A \leq C$. By fixing $A$ to be a number larger than $C$ (and of course larger than 1, to have that $d>1$), we are led to a contradiction.

Therefore, there are $< 10^{-8}S$ points of $\mathfrak{G}$ in the union of the interiors of the cells. Thus, if $\mathfrak{G}_1$ denotes the set of points in $\mathfrak{G}$ which lie in $Z$, it holds that $|\mathfrak{G}_1|>(1-10^{-8})S$. 

Now, by the definition of $\mathfrak{G}$, for each $x \in \mathfrak{G}$ we can fix collections $\mathfrak{L}_1(x) \subseteq \mathfrak{L}_1$, $\mathfrak{L}_2(x) \subseteq \mathfrak{L}_2$ and $\mathfrak{L}_3(x) \subseteq \mathfrak{L}_3$ of lines passing through $x$, such that $|\mathfrak{L}_1(x)|= \lceil N_1 \rceil$, $|\mathfrak{L}_2(x)|= \lceil N_2 \rceil$ and $|\mathfrak{L}_3(x)|= \lceil N_3 \rceil$, and, if $l_1 \in \mathfrak{L}_1(x)$, $l_2 \in \mathfrak{L}_2(x)$ and $l_3 \in \mathfrak{L}_3(x)$, then the directions of the lines $l_1$, $l_2$ and $l_3$ span $\R^3$.

Therefore, for all $j \in \{1,2,3\}$, we can define $\mathfrak{L}_j':=\bigg\{l \in \mathfrak{L}_j:\Big |\Big\{x \in \mathfrak{G}_1: l \in \mathfrak{L}_j(x)\Big |\geq \frac{1}{100} \frac{S\lceil N_j\rceil}{L_j}\Big\}\bigg\}$. In other words, for all $j \in \{1,2,3\}$, $\mathfrak{L}_j'$ is the set of lines in $\mathfrak{L}_j$, each of which contains at least $\frac{1}{100} \frac{S\lceil N_j \rceil }{L_j}$ points $x \in \mathfrak{G}_1$ with the property that the line belongs to $\mathfrak{L}_j(x)$.

Moreover, for all $j \in \{1,2,3\}$, for any subset $\mathcal{G}$ of $\mathfrak{G}$ and any subset $\mathcal{L}$ of $\mathfrak{L}_j$, we denote by $I^{(j)}_{\mathcal{G},\mathcal{L}}$ the number of pairs of the form $(x, l)$, where $x \in \mathcal{G}$ and $l \in \mathfrak{L}_j(x) \cap\mathcal{L}$; note that the set of these pairs is a subset of the set of incidences between $\mathcal{G}$ and $\mathcal{L}$.

We now analyse the situation.

Let $j\in \{1,2,3\}$. Each point $x \in \mathfrak{G}_1$ intersects $\lceil N_j \rceil$ lines of $\mathfrak{L}_j(x)$, which is a subset of $\mathfrak{L}_j$. Thus,  \begin{displaymath} I^{(j)}_{\mathfrak{G}_1,\mathfrak{L}_j} > (1-10^{-8})S\lceil N_j \rceil.\end{displaymath}
On the other hand, each line $l \in \mathfrak{L}_j \setminus \mathfrak{L}_j'$ contains fewer than $\frac{1}{100} \frac{S\lceil N_j\rceil}{L_j}$ points $x \in \mathfrak{G}_1$ with the property that $l \in \mathfrak{L}_j(x)$, so
\begin{displaymath}I^{(j)}_{\mathfrak{G}_1,\mathfrak{L}_j\setminus \mathfrak{L}'_j} \leq  |\mathfrak{L}_j\setminus \mathfrak{L}'_j| \cdot \frac{S\lceil N_j \rceil}{100L_j} \leq \frac{1}{100}S\lceil N_j\rceil.\end{displaymath}
Therefore, since $ I^{(j)}_{\mathfrak{G}_1, \mathfrak{L}_j}=I^{(j)}_{\mathfrak{G}_1,\mathfrak{L}_j \setminus \mathfrak{L}'_j}+I^{(j)}_{\mathfrak{G}_1,\mathfrak{L}'_j}$, it follows that 
\begin{displaymath}I^{(j)}_{\mathfrak{G}_1,\mathfrak{L}'_j}> (1-10^{-8}-10^{-2})S\lceil N_j \rceil, \end{displaymath}
for all $j \in \{1,2,3\}$.

Now, for all $j \in \{1,2,3\}$, we define $\mathfrak{G}'_j:=\Big\{x \in \mathfrak{G}_1: |\mathfrak{L}_j(x) \cap \mathfrak{L}_j'|\geq \frac{10^{-8}}{1+10^{-8}}(1-10^{-8}-10^{-2})\lceil N_j \rceil\Big\}$. In other words, for all $j \in \{1,2,3\}$, $x \in \mathfrak{G}_j'$ if and only if $x \in \mathfrak{G}_1$ and $x$ intersects at least $\frac{10^{-8}}{1+10^{-8}}(1-10^{-8}-10^{-2})\lceil N_j \rceil$ lines of $\mathfrak{L}_j(x) \cap \mathfrak{L}'_j$.

Let $j \in \{1,2,3\}$. Since each point $x \in \mathfrak{G}_1 \setminus \mathfrak{G}_j'$ intersects fewer than $\frac{10^{-8}}{1+10^{-8}}(1-10^{-8}-10^{-2})\lceil N_j \rceil$ lines of $\mathfrak{L}_j(x) \cap \mathfrak{L}'_j$, it follows that \begin{displaymath}I^{(j)}_{\mathfrak{G}_1 \setminus \mathfrak{G}'_j, \mathfrak{L}'_j} < |\mathfrak{G}_1 \setminus \mathfrak{G}'_j|\frac{10^{-8}}{1+10^{-8}}(1-10^{-8}-10^{-2})\lceil N_j \rceil\leq  \frac{10^{-8}}{1+10^{-8}}(1-10^{-8}-10^{-2})S\lceil N_j \rceil.\end{displaymath}
Therefore, since $I^{(j)}_{\mathfrak{G}_1, \mathfrak{L}'_j}=I^{(j)}_{\mathfrak{G}_1 \setminus \mathfrak{G}'_j, \mathfrak{L}'_j}+I^{(j)}_{\mathfrak{G}'_j, \mathfrak{L}'_j}$, we obtain
\begin{displaymath}I^{(j)}_{\mathfrak{G}'_j, \mathfrak{L}'_j} > \frac{1-10^{-8}-10^{-2}}{1+10^{-8}}S\lceil N_j \rceil. \end{displaymath}
At the same time, however, $|\mathfrak{L}_j(x)| =\lceil N_j \rceil$ for all $x \in \mathfrak{G}_j'$, and thus  $I^{(j)}_{\mathfrak{G}'_j, \mathfrak{L}'_j} \leq |\mathfrak{G}'_j|\lceil N_j\rceil$. Therefore, $ \frac{1-10^{-8}-10^{-2}}{1+10^{-8}}S\lceil N_j \rceil< |\mathfrak{G}'_j|\lceil N_j\rceil$, which implies that
\begin{displaymath} |\mathfrak{G}'_j|\geq \frac{1-10^{-8}-10^{-2}}{1+10^{-8}}S, \end{displaymath}
for all $j \in \{1,2,3\}$. In other words, for all $j\in \{1,2,3\}$, there exist at least $\frac{1-10^{-8}-10^{-2}}{1+10^{-8}}S$ points $x \in \mathfrak{G}_1$ such that $x$ intersects at least $\frac{10^{-8}}{1+10^{-8}}(1-10^{-8}-10^{-2})\lceil N_j\rceil$ lines of $\mathfrak{L}_j(x) \cap \mathfrak{L}'_j$.

But \begin{displaymath} |\mathfrak{G}_1'\cup \mathfrak{G}_2' \cup \mathfrak{G}_3'|=|\mathfrak{G}_1'|+|\mathfrak{G}_2'|+|\mathfrak{G}_3'|-|\mathfrak{G}_1'\cap\mathfrak{G}_2'|-|\mathfrak{G}_2'\cap\mathfrak{G}_3'|-|\mathfrak{G}_1'\cap\mathfrak{G}_3'|+|\mathfrak{G}_1'\cap\mathfrak{G}_2'\cap\mathfrak{G}_3'|,\end{displaymath} and thus 
$$|\mathfrak{G}_1'\cap\mathfrak{G}_2'\cap\mathfrak{G}_3'|=|\mathfrak{G}_1'\cup \mathfrak{G}_2' \cup \mathfrak{G}_3'|-(|\mathfrak{G}_1'|+|\mathfrak{G}_2'|+|\mathfrak{G}_3'|)+(|\mathfrak{G}_1'\cap\mathfrak{G}_2'|+|\mathfrak{G}_2'\cap\mathfrak{G}_3'|+|\mathfrak{G}_1'\cap\mathfrak{G}_3'|)\geq
$$
$$\geq |\mathfrak{G}_1'|-(|\mathfrak{G}_1'|+|\mathfrak{G}_2'|+|\mathfrak{G}_3'|)+$$
$$+\Big((|\mathfrak{G}_1'|+|\mathfrak{G}_2'|-|\mathfrak{G}_1'\cup\mathfrak{G}_2'|)+(|\mathfrak{G}_2'|+|\mathfrak{G}_3'|-|\mathfrak{G}_2'\cup\mathfrak{G}_3'|)+(|\mathfrak{G}_1'|+|\mathfrak{G}_3'|-|\mathfrak{G}_1'\cup\mathfrak{G}_3'|)\Big)\geq
$$
$$\geq2|\mathfrak{G}_1'|+|\mathfrak{G}_2'|+|\mathfrak{G}_3'|-|\mathfrak{G}_1'\cup\mathfrak{G}_2'|-|\mathfrak{G}_2'\cup\mathfrak{G}_3'|-|\mathfrak{G}_1'\cup\mathfrak{G}_3'|\geq
$$
\begin{displaymath}\geq4\cdot\frac{1-10^{-8}-10^{-2}}{1+10^{-8}}S-3S=\end{displaymath} 
\begin{displaymath}= \frac{4(1-10^{-8}-10^{-2})-3(1+10^{-8})}{1+10^{-8}}S=\frac{1-7 \cdot 10^{-8}-4 \cdot 10^{-2}}{1+10^{-8}}S\geq \frac{1-8\cdot 10^{-2}}{1+10^{-8}}S;\end{displaymath}
in other words, there exist at least $\frac{1-8\cdot 10^{-2}}{1+10^{-8}}S$ points $x \in \mathfrak{G}_1$ intersecting at least $\frac{10^{-8}}{1+10^{-8}}(1-10^{-8}-10^{-2})\lceil N_j\rceil$ lines of $\mathfrak{L}_j(x) \cap \mathfrak{L}'_j$, simultaneously for all $j\in\{1,2,3\}$.

\textbf{Case 1:} Suppose that, for some $j\in \{1,2,3\}$, it holds that $\frac{1}{10^{100}} \frac{S\lceil N_j\rceil}{L_j}\leq d$. Then, $\frac{S\lceil N_j\rceil}{L_j} \lesssim \frac{L_1L_2}{S\lceil N_1\rceil\lceil N_2\rceil}$, which implies that
\begin{displaymath} S\lesssim \Bigg(\frac{L_1L_2}{\lceil N_1\rceil \lceil N_2\rceil}\Bigg)^{1/2}\Bigg(\frac{L_j}{\lceil N_j\rceil}\Bigg)^{1/2} \lesssim \frac{(L_1L_2L_3)^{1/2}}{(\lceil N_1\rceil \lceil N_2\rceil \lceil N_3\rceil )^{1/2}}\lesssim \frac{(L_1L_2L_3)^{1/2}}{(N_1N_2N_3)^{1/2}}. \end{displaymath}

\textbf{Case 2:} Suppose that $\frac{1}{10^{100}} \frac{S\lceil N_j\rceil}{L_j}> d$, for all $j=1,2,3$. Then, each line in $\mathfrak{L}_1'$, $\mathfrak{L}_2'$ and $\mathfrak{L}_3'$ lies in $Z$, therefore each point in $\mathfrak{G}_1'\cap \mathfrak{G}_2' \cap \mathfrak{G}_3'$ is a critical point of $Z$. 

Now, for all $j \in \{1,2,3\}$, we define $\mathfrak{L}_{j,1}:=\big\{l \in \mathfrak{L}_j: |\{x \in \mathfrak{G}_1' \cap \mathfrak{G}_2' \cap \mathfrak{G}_3': l \in \mathfrak{L}_j(x)\}|\geq\frac{1}{10^{100}}S\lceil N_{j}\rceil L_{j}^{-1}\big\}$. In other words, for all $j \in \{1,2,3\}$, $\mathfrak{L}_{j,1}$ is the set of lines in $\mathfrak{L}_{j}$, each containing at least $\frac{1}{10^{100}}S\lceil N_{j}\rceil L_{j}^{-1}$ points $x \in \mathfrak{G}_1' \cap \mathfrak{G}_2' \cap \mathfrak{G}_3'$ with the property that the line belongs to $\mathfrak{L}_j(x)$. 

$\bullet$ Suppose that, for some $j \in \{1,2,3\}$, $|\mathfrak{L}_{j,1}|\geq \frac{L_j}{10^{1000}}$. Each line in $\mathfrak{L}_{j,1}$ contains more than $d$ critical points of $Z$, it is therefore a critical line. Thus, \begin{displaymath}\frac{L_j}{10^{1000}} \leq d^2, \end{displaymath}so \begin{displaymath} L_j \lesssim \frac{(L_1L_2)^2}{S^2(\lceil N_1\rceil \lceil N_2\rceil)^{2}}, \end{displaymath} from which it follows that 
\begin{displaymath}S\lesssim \frac{L_1L_2}{\lceil N_1\rceil \lceil N_2\rceil}\frac{1}{L_j^{1/2}} \lesssim \frac{(L_1L_2L_3)^{1/2}}{(\lceil N_1\rceil \lceil N_2\rceil \lceil N_3\rceil )^{1/2}}\lesssim \frac{(L_1L_2L_3)^{1/2}}{(N_1N_2N_3)^{1/2}}.\end{displaymath}
We are now ready to define the constant $c$ appearing in our induction process. Indeed, there exists some constant $c' \geq 1$, independent of $\mathfrak{L}_1$, $\mathfrak{L}_2$, $\mathfrak{L}_3$ and $N_1$, $N_2$ and $N_3$, such that \begin{displaymath} S \leq c' \cdot \frac{(L_1L_2L_3)^{1/2}}{(N_1N_2N_3)^{1/2}} \end{displaymath} in all the cases dealt with so far. Let $c$ be such a constant $c'$.

$\bullet$ Suppose that, for all $i \in \{1,2,3\}$, $|\mathfrak{L}_{i,1}|< \frac{L_i}{10^{1000}}$. Then, it holds that $|\mathfrak{L}_{j,1}|< \frac{L_j}{10^{1000}}$ in particular for that $j \in \{1,2,3\}$ such that $L_j=M_j$; we now fix that $j \in \{1,2,3\}$.

Independently of the fact that $|\mathfrak{L}_{j,1}|< \frac{L_j}{10^{1000}}$, it holds that
\begin{displaymath} I^{(j)}_{\mathfrak{G}_1'\cap\mathfrak{G}_2'\cap\mathfrak{G}_3', \mathfrak{L}_j'} \geq |\mathfrak{G}_1'\cap\mathfrak{G}_2'\cap\mathfrak{G}_3'| \cdot \frac{10^{-8}}{1+10^{-8}}(1-10^{-8}-10^{-2})\lceil N_j \rceil\geq \end{displaymath}
\begin{displaymath} \geq \frac{1-8\cdot 10^{-2}}{1+10^{-8}}S \cdot \frac{10^{-8}}{1+10^{-8}}(1-10^{-8}-10^{-2})\lceil N_j\rceil \geq 10^{-10}S\lceil N_j\rceil,\end{displaymath}
since each point $x \in \mathfrak{G}_1'\cap\mathfrak{G}_2'\cap\mathfrak{G}_3'$ intersects at least $\frac{10^{-8}}{1+10^{-8}}(1-10^{-8}-10^{-2})\lceil N_j \rceil $ lines of $\mathfrak{L}_j(x) \cap \mathfrak{L}'_j$.

In addition, each line $l \in \mathfrak{L}_j' \setminus \mathfrak{L}_{j,1}$ contains fewer than $\frac{1}{10^{100}}\frac{S\lceil N_j\rceil}{L_j}$ points $x \in \mathfrak{G}_1'\cap\mathfrak{G}_2'\cap\mathfrak{G}_3'$ with the property that $l \in \mathfrak{L}_j(x)$. Thus,
\begin{displaymath} I^{(j)}_{\mathfrak{G}_1'\cap\mathfrak{G}_2'\cap\mathfrak{G}_3', \mathfrak{L}_j'\setminus \mathfrak{L}_{j,1}} <|\mathfrak{L}_j' \setminus \mathfrak{L}_{j,1}| \cdot \frac{S\lceil N_j\rceil }{10^{100}L_j} \leq L_j \cdot \frac{S\lceil N_j\rceil}{10^{100}L_j}=10^{-100}S\lceil N_j\rceil.\end{displaymath} Therefore, since $I^{(j)}_{\mathfrak{G}_1'\cap\mathfrak{G}_2'\cap\mathfrak{G}_3', \mathfrak{L}_j'}=I^{(j)}_{\mathfrak{G}_1'\cap\mathfrak{G}_2'\cap\mathfrak{G}_3', \mathfrak{L}_j'\setminus \mathfrak{L}_{j,1}}+I^{(j)}_{\mathfrak{G}_1'\cap\mathfrak{G}_2'\cap\mathfrak{G}_3',  \mathfrak{L}_{j,1}}$, we obtain
\begin{displaymath} I^{(j)}_{\mathfrak{G}_1'\cap\mathfrak{G}_2'\cap\mathfrak{G}_3',  \mathfrak{L}_{j,1}} >10^{-11}S\lceil N_j\rceil . \end{displaymath} 
Now, again for that $j \in \{1,2,3\}$ such that $L_j=M_j$, we define $\mathfrak{G}':=\{x \in \mathfrak{G}_1'\cap\mathfrak{G}_2'\cap\mathfrak{G}_3': |\mathfrak{L}_j(x) \cap \mathfrak{L}_{j,1}| \geq 10^{-12}\lceil N_j\rceil\}$. In other words, for that particular $j \in \{1,2,3\}$, $x \in \mathfrak{G}'$ if and only if $x \in \mathfrak{G}_1'\cap\mathfrak{G}_2'\cap\mathfrak{G}_3'$ and $x$ intersects at least $10^{-12}\lceil N_j\rceil$ lines of $\mathfrak{L}_j(x) \cap \mathfrak{L}_{j,1}$. 

Since each point $x \in (\mathfrak{G}_1'\cap\mathfrak{G}_2'\cap\mathfrak{G}_3')\setminus \mathfrak{G}'$ intersects fewer than $10^{-12}\lceil N_j\rceil$ lines of $\mathfrak{L}_j(x) \cap \mathfrak{L}_{j,1}$, it holds that
\begin{displaymath} I^{(j)}_{(\mathfrak{G}_1'\cap\mathfrak{G}_2'\cap\mathfrak{G}_3')\setminus \mathfrak{G}',  \mathfrak{L}_{j,1}} <10^{-12}S\lceil N_j\rceil, \end{displaymath} and thus, as $I^{(j)}_{\mathfrak{G}_1'\cap\mathfrak{G}_2'\cap\mathfrak{G}_3',  \mathfrak{L}_{j,1}}=I^{(j)}_{(\mathfrak{G}_1'\cap\mathfrak{G}_2'\cap\mathfrak{G}_3')\setminus \mathfrak{G}',  \mathfrak{L}_{j,1}}+I^{(j)}_{\mathfrak{G}', \mathfrak{L}_{j,1}}$, we obtain
\begin{displaymath} I^{(j)}_{\mathfrak{G}', \mathfrak{L}_{j,1}}> (10^{-11}-10^{-12})S\lceil N_j\rceil = 9 \cdot 10^{-12}S\lceil N_j\rceil. \end{displaymath} At the same time, however, $|\mathfrak{L}_j(x)| =\lceil N_j\rceil$ for all $x \in \mathfrak{G}'$. Therefore, \begin{displaymath} I^{(j)}_{\mathfrak{G}', \mathfrak{L}_{j,1}} \leq |\mathfrak{G}'|\lceil N_j\rceil. \end{displaymath} Thus, the above imply that
\begin{displaymath} |\mathfrak{G}'| > 9\cdot 10^{-12}S. \end{displaymath}

But if $\{i,k\}=\{1,2,3\}\setminus \{j\}$, then each point $x \in \mathfrak{G}'$ is a multijoint for the collections $\mathfrak{L}_{j,1}$, $\mathfrak{L}_i'$ and $\mathfrak{L}_k'$ of lines, that lies in $\geq 10^{-12}\lceil N_j\rceil$ lines of $\mathfrak{L}_j(x) \cap \mathfrak{L}_{j,1}$, in $\geq  \frac{1-10^{-8}-10^{-2}}{1+10^{8}}\lceil N_i \rceil $ lines of $\mathfrak{L}_i(x)\cap \mathfrak{L}_i'$ and in $\geq  \frac{1-10^{-8}-10^{-2}}{1+10^{8}}\lceil N_k\rceil $ lines of $\mathfrak{L}_k(x) \cap \mathfrak{L}_k'$. Now, for all $x \in \mathfrak{G}'$, if $l_j \in \mathfrak{L}_j(x)$, $l_i \in \mathfrak{L}_i(x)$ and $l_k \in \mathfrak{L}_k(x)$, then the directions of the lines $l_i$, $l_j$ and $l_k$ span $\R^3$. Therefore, since $|\mathfrak{L}_{j,1}| < \frac{{L}_j}{10^{1000}}\lneq M_j$, our induction hypothesis implies that
\begin{displaymath} 9\cdot 10^{-12}S< |\mathfrak{G}'| \leq c\cdot \frac{(|\mathfrak{L}_{j,1}| \cdot |\mathfrak{L}_i'| \cdot |\mathfrak{L}_k'|)^{1/2}}{(10^{-12}\lceil N_j\rceil )^{1/2} \Big(\frac{1-10^{-8}-10^{-2}}{1+10^{8}}\lceil N_i\rceil \Big)^{1/2}\Big(\frac{1-10^{-8}-10^{-2}}{1+10^{8}}\lceil N_k\rceil \Big)^{1/2}},\end{displaymath} 
where $c$ is the explicit constant defined earlier, and which appears in the induction process.

Therefore, $$|\mathfrak{G}'| \leq c\cdot \frac{1}{9\cdot 10^{-12} (10^{-12})^{1/2}\Big(\frac{1-10^{-8}-10^{-2}}{1+10^{8}}\Big)^{1/2}\Big(\frac{1-10^{-8}-10^{-2}}{1+10^{8}}\Big)^{1/2}}
 \frac{(|\mathfrak{L}_{j,1}| \cdot |\mathfrak{L}_i'| \cdot |\mathfrak{L}_k'|)^{1/2}}{(\lceil N_j\rceil \lceil N_i\rceil \lceil N_k\rceil )^{1/2}}\leq$$
$$\leq c\cdot \frac{1}{9\cdot 10^{-12} (10^{-12})^{1/2}\Big(\frac{1-10^{-8}-10^{-2}}{1+10^{8}}\Big)^{1/2}\Big(\frac{1-10^{-8}-10^{-2}}{1+10^{8}}\Big)^{1/2}}
 \frac{(|\mathfrak{L}_{j,1}| \cdot |\mathfrak{L}_i'| \cdot |\mathfrak{L}_k'|)^{1/2}}{( N_j N_i N_k)^{1/2}}.
$$
However, 
$$\frac{1}{9\cdot 10^{-12} (10^{-12})^{1/2}\Big(\frac{1-10^{-8}-10^{-2}}{1+10^{8}}\Big)^{1/2}\Big(\frac{1-10^{-8}-10^{-2}}{1+10^{8}}\Big)^{1/2}}<
$$
$$<\frac{1}{10^{-18} \cdot \frac{1-10^{-8}-10^{-2}}{1+10^{8}}}<\frac{1}{10^{-18} \cdot \frac{1/2}{2 \cdot 10^{8}}}=4 \cdot \frac{10^8}{10^{-18}}<10^{27},
$$
so, since \begin{displaymath} |\mathfrak{L}_{j,1}|< \frac{L_j}{10^{1000}}, \text{ }|\mathfrak{L}_i'|\leq L_i\text{ and }|\mathfrak{L}_k'| \leq L_k, \end{displaymath} it follows that
\begin{displaymath} S \leq c \cdot \frac{(L_1L_2L_3)^{1/2}}{(N_1N_2N_3)^{1/2}}, \end{displaymath}
for the same constant $c$ as in the first two steps of the induction process.

Thus, Proposition \ref{multsimple} is proved.

\qed

Theorem \ref{theoremmult2} is an immediate corollary of Proposition \ref{multsimple} for $(N_1,N_2,N_3)=(1,1,1)$.

We now continue with establishing Proposition \ref{pointslarge} (from which, as we have already explained, Corollary \ref{multlarge} easily follows).

\textbf{Proposition \ref{pointslarge}.} \textit{Let $\mathfrak{L}_1$, $\mathfrak{L}_2$, $\mathfrak{L}_3$ be finite collections of $L_1$, $L_2$ and $L_3$, respectively, lines in $\R^3$. For all $x \in \R^3$ and $i=1,2,3$, we denote by $N_i(x)$ the number of lines of $\mathfrak{L}_i$ passing through $x$. Also, for each $k \in \{1,2,3\}$, let $c_k$ be a constant such that $c_k \cdot L_k^{1/2}\gneq 1$. Then,}

\begin{displaymath} \sum_{\big\{x \in \R^3:N_k(x)\geq c_kL_k^{1/2},\text{ for some }k \in \{1,2,3\}\big\}} (N_1(x)N_2(x)N_3(x))^{1/2}\leq c\cdot (L_1L_2L_3)^{1/2}, \end{displaymath}

\textit{where $c$ is a constant depending on $c_1$, $c_2$ and $c_3$, but independent of $\mathfrak{L}_1$, $\mathfrak{L}_2$ and $\mathfrak{L}_3$.}

In fact, the proof will immediately follow from the following three claims.

\begin{claim} \label{p1} Let $n \in \N$, $n \geq 2$. Let $\mathfrak{L}_1$, $\mathfrak{L}_2$ be finite collections of $L_1$ and $L_2$, respectively, lines in $\R^n$. If, for all $x \in \R^3$ and $i=1,2$, $N_i(x)$ denotes the number of lines of $\mathfrak{L}_i$ passing through $x$, then
\begin{equation}\label{eq:2dim} \sum_{x \in \R^n}N_1(x)N_2(x) \leq L_1 L_2. \end{equation} 
\end{claim}

\begin{proof}The left hand side of \eqref{eq:2dim} is equal to the number of pairs of the form $(l_1,l_2)$, where $l_1 \in \mathfrak{L}_1$, $l_2 \in \mathfrak{L}_2$ and the lines $l_1$, $l_2$ meet at a point of $\R^n$, so it is equal to at most the number of all pairs of the form $(l_1,l_2)$ where $l_1 \in \mathfrak{L}_1$ and $l_2 \in \mathfrak{L}_2$, which is equal to $L_1L_2$.

\end{proof}

\begin{claim} \label{p2} Let $n \in \N$, $n \geq 2$. Let $\mathfrak{L}$ be a finite collection of $L$ lines in $\R^n$, and $c$ a constant such that $c \cdot L^{1/2}\gneq 1$.  For all $x \in \R^n$, we denote by $N(x)$ the number of lines of $\mathfrak{L}$ passing through $x$. Then,
\begin{displaymath} \sum_{\{x \in \R^n:\; N(x)\geq c L^{1/2}\}}N(x) \leq C \cdot L, \end{displaymath}
where $C$ is a constant depending only on $c$.
\end{claim}

\begin{proof}Let $\mathcal{P}$ be the set of points $x \in \R^n$, such that $N(x)\geq c\cdot  L^{1/2}$. Since $c \cdot L^{1/2}\gneq 1$, the number $N(x)$ of lines of $\mathfrak{L}$ passing through any $x \in \mathcal{P}$ is equal to at least 2. Therefore, from the Szemer\'{e}di-Trotter theorem, $|\mathcal{P}| \lesssim_{c} \frac{L^2}{(L^{1/2})^3}+ \frac{L}{L^{1/2}} \sim_c L^{1/2}$, and thus, again from the Szemer\'{e}di-Trotter theorem, the number of incidences between $\mathcal{P}$ and $\mathfrak{L}$, i.e. the quantity $\sum_{\{x \in \R^n:\; N(x)\geq c L^{1/2}\}}N(x)$, is $\lesssim_c L^{2/3}|\mathcal{P}|^{2/3} + L\lesssim L^{2/3}L^{(1/2)(2/3)}\sim_c L$.

\end{proof}

\begin{claim}\label{p3} Let $\mathcal{P}$ be a finite collection of points in $\R^3$ and $\mathfrak{L}_1$, $\mathfrak{L}_2$, $\mathfrak{L}_3$ finite collections of $L_1$, $L_2$ and $L_3$, respectively, lines in $\R^3$. For all $x \in \mathcal{P}$ and $i=1,2,3$, we denote by $N_i(x)$ the number of lines of $\mathfrak{L}_i$ passing through $x$. Suppose that
\begin{displaymath}\sum_{x \in \mathcal{P}}N_k(x) \leq c\cdot L_k,
\end{displaymath}for some $k \in \{1,2,3\}$, where $c$ is an absolute constant.
Then,
\begin{displaymath} \sum_{x \in \mathcal{P}} (N_1(x)N_2(x)N_3(x))^{1/2}\leq c'\cdot (L_1L_2L_3)^{1/2}, \end{displaymath}
where $c'$ is a constant depending only on $c$.

\end{claim}

\begin{proof} Let $\{i,j\}=\{1,2,3\} \setminus \{k\}$. Then,

\begin{displaymath}\sum_{x \in \mathcal{P}} (N_1(x)N_2(x)N_3(x))^{1/2}
= \sum_{x \in \mathcal{P}} (N_i(x)N_j(x))^{1/2}N_k(x)^{1/2} \leq\end{displaymath}
\begin{displaymath}\leq \Bigg(\sum_{x \in \mathcal{P}} N_i(x)N_j(x)\Bigg)^{1/2} \Bigg(\sum_{x \in \mathcal{P}}  N_k(x)\Bigg)^{1/2} \lesssim_c \end{displaymath}
\begin{displaymath} \lesssim (L_iL_j)^{1/2} L_k^{1/2} \sim (L_1L_2L_3)^{1/2}.\end{displaymath}
Note that the first inequality is the Cauchy-Schwarz inequality, while the second one follows from Claim \ref{p1}.

\end{proof}

Proposition \ref{pointslarge} immediately follows.

\textit{Proof of Proposition \ref{pointslarge}.} Fix $k \in \{1,2,3\}$ and let $\{i,j\}=\{1,2,3\} \setminus \{k\}$. Since $c_k \cdot L_k^{1/2} \gneq 1$, it follows from Claim \ref{p2} that
\begin{displaymath} \sum_{\big\{x \in \R^3:\; N_k(x)\geq c_kL_k^{1/2}\big\}} N_k(x) \lesssim_{c_k} L_k.
\end{displaymath}
Therefore, we have by Claim \ref{p3} that
\begin{displaymath}\sum_{\{x \in \R^3:\; N_k(x)\geq c_k L_k^{1/2}\}}(N_1(x)N_2(x)N_3(x))^{1/2}
\lesssim_{c_k} (L_1L_2L_3)^{1/2}.\end{displaymath}
Since $k \in \{1,2,3\}$ was arbitrary, Proposition \ref{pointslarge} is proved.

\qed

We can now move on to the third step of the proof of Theorem \ref{theoremmult1}.

\textbf{Theorem \ref{theoremmult1}.} \textit{Let }$\mathfrak{L}_1$, $\mathfrak{L}_2$, $\mathfrak{L}_3$\textit{ be finite collections of $L_1$, $L_2$ and $L_3$, respectively, lines in }$\R^3$\textit{, such that, whenever a line of }$\mathfrak{L}_1$\textit{, a line of }$\mathfrak{L}_2$\textit{ and a line of }$\mathfrak{L}_3$\textit{ meet at a point, they form a joint there. Let }$J$\textit{ be the set of multijoints formed by the collections }$\mathfrak{L}_1$, $\mathfrak{L}_2$\textit{ and }$\mathfrak{L}_3$\textit{. Then,}
\begin{displaymath}\sum_{\{x \in J:\; N_m(x)> 10^{12}\}}(N_1(x)N_2(x)N_3(x))^{1/2} \leq c \cdot (L_1L_2L_3)^{1/2}, \end{displaymath}
\textit{where }$m \in \{1,2,3\}$\textit{ is such that }$L_m=\min\{L_1,L_2,L_3\}$\textit{, and }$c$\textit{ is a constant independent of }$\mathfrak{L}_1$, $\mathfrak{L}_2$\textit{ and }$\mathfrak{L}_3$. 

We have already mentioned that the constant $10^{12}$, which we demand as a lower bound on $N_m(x)$ for the multijoints $x \in J$ contributing to the sum in Theorem \ref{theoremmult1} above, is the smallest constant with those properties arising from our calculations, even though we expect that, in fact, the same results should hold with all $x \in J$ contributing to the sum. Although the reasons why we pick the constant $10^{12}$ will be apparent from our proof, we would like to take a moment now, to at least explain why we demand any lower bound on $N_m(x)$, for all $x \in J$ that contribute to the sum. 

The reason, essentially, is that, if we know that only one line from a finite set of lines is passing through each point of a finite set of points, then we cannot establish an upper bound on the number of points, depending only on the number of the lines. Because of this difficulty arising in the case of $x \in J$ such that $N_m(x)=1$, we can then only establish upper bounds on the cardinality of the set $\{x \in J:N_m(x)=1\}$ that also depend on the cardinalities of the potentially larger collections of lines in the set $\{\mathfrak{L}_1, \mathfrak{L}_2, \mathfrak{L}_3\}\setminus \{\mathfrak{L}_m\}$, thus only being able to obtain weaker results than the ones we hope. 

Let us now be more precise. It is obvious that Theorem \ref{theoremmult1} holds when $L_i=1$, for some $i \in \{1,2,3\}$ (due to Claim \ref{p3}). We can therefore assume that $L_i \geq 2$, for all $i \in \{1,2,3\}$. Now, thanks to Corollary \ref{multlarge}, it suffices to consider only those points $x \in J$, such that $N_i(x) \leq L_i^{1/2}$. 

By the Szemer\'edi-Trotter theorem, we know that, in $\R^3$, for any set of $S$ points and any finite collection of $L$ lines such that at least $k$ of the lines are passing through each of the points, with $k\geq 2$, we have that
\begin{displaymath} S \lesssim \frac{L^2}{k^3}+\frac{L}{k}. \end{displaymath}
If, in addition, $k \leq L^{1/2}$, then $\frac{L}{k} \leq \frac{L^2}{k^3}$, and thus $S \lesssim \frac{L^2}{k^3}$. Therefore, there exists a positive constant $c$, such that, for any set of $S$ points and any finite collection of $L$ lines such that at least $k$ of the lines are passing through each of the points, with $2 \leq k \leq L^{1/2}$, we have that
\begin{displaymath}S \leq c \cdot  \frac{L^2}{k^3}.\end{displaymath}
Now, for all $(N_1,N_2,N_3) \in \R_{\geq 1}^3$, we define $J_{N_1,N_2,N_3}:=\{x \in J: x$ lies in at least $N_1$ and fewer than $(1+10^{-8})N_1$ lines of $\mathfrak{L}_1$, in at least $N_2$ and fewer than $(1+10^{-8})N_2$ lines of $\mathfrak{L}_2$ and in at least $N_3$ and fewer than $(1+10^{-8})N_3$ lines of $\mathfrak{L}_3\}$.

In particular, fix $N_1$, $N_2$, $N_3 \in \R_{\geq 1}$, such that $N_i \leq L_i^{1/2}$ for all $i=1,2,3$, and let us assume that $\{i,j,k\}=\{1,2,3\}$, such that \begin{displaymath}\frac{L_i^2}{N_i^3} \leq \frac{L_j^2}{N_j^3} \leq \frac{L_k^2}{N_k^3}. \end{displaymath}

Now, suppose that $|J_{N_1,N_2,N_3}| \geq 4c \cdot  \frac{L_i^2}{N_i^3}$. By the above, this means that $N_i= 1$; it thus also follows that $L_i =\min\{L_1,L_2,L_3\}$ and that $\frac{L_i}{N_i}=\min\{\frac{L_1}{N_1}, \frac{L_2}{N_2}, \frac{L_3}{N_3}\}$. In particular, in this case we can only ensure that $|J_{N_1,N_2,N_3}| \lesssim \frac{L_iL_j}{N_iN_j}\;(=\frac{L_iL_j}{N_j})$, a quantity that is larger, up to multiplication by constants, than $\frac{L_i^2}{N_i^3}\;(=L_i^2)$, a fact that leads us to weaker results than the ones we expect. If, on the other hand, we know that $|J_{N_1,N_2,N_3}| < 4c \cdot  \frac{L_i^2}{N_i^3}\;\Big(\leq 4c \cdot \frac{L_j^2}{N_j^3} \leq 4c \cdot \frac{L_k^2}{N_k^3}\Big)$, then we manage to derive strong results.

We thus assume from now on that $N_m >1$, where $m\in \{1,2,3\}$ is such that $L_m=\min\{L_1,L_2,L_3\}$, and continue our analysis having ensured that all the inequalities $|J_{N_1,N_2,N_3}|$ $ \lesssim \frac{L_1^2}{N_1^3}$, $|J_{N_1,N_2,N_3}| \lesssim \frac{L_2^2}{N_2^3}$ and $|J_{N_1,N_2,N_3}| \lesssim \frac{L_3^2}{N_3^3}$ hold.

(In fact, for the proof of Theorem \ref{theoremmult1} that we are providing, we will need at some point a small multiple of $N_m$ to be larger than 1, for different reasons than the ones we describe above, and  that is why we consider $x \in J$ such that $N_m(x)$ is actually larger than $10^{12}$, instead of just 2.)

\textbf{Remark.} As we have mentioned in Chapter \ref{1}, even though we prove that \begin{displaymath}\sum_{\{x \in J:\; N_m(x)> 10^{12}\}}(N_1(x)N_2(x)N_3(x))^{1/2} \lesssim (L_1L_2L_3)^{1/2}, \end{displaymath}for the set $J$ of multijoints formed by collections $\mathfrak{L}_1$, $\mathfrak{L}_2$ and $\mathfrak{L}_3$ of $L_1$, $L_2$ and $L_3$, respectively, lines in $\R^3$ (under the particular assumption on the transversality properties of the collections $\mathfrak{L}_1$, $\mathfrak{L}_2$ and $\mathfrak{L}_3$ in the statement of Theorem \ref{theoremmult1}), it will be obvious from the proof of Theorem \ref{theoremmult1} that, in order to prove (under the same transversality assumptions) that \begin{displaymath}\sum_{x \in J}(N_1(x)N_2(x)N_3(x))^{1/2} \lesssim (L_1L_2L_3)^{1/2}, \end{displaymath}it suffices to show that
\begin{displaymath}\sum_{\{x \in J:\; N_m(x)=1\}}(N_1(x)N_2(x)N_3(x))^{1/2} \lesssim (L_1L_2L_3)^{1/2}. \end{displaymath} 
$\;\;\;\;\;\;\;\;\;\;\;\;\;\;\;\;\;\;\;\;\;\;\;\;\;\;\;\;\;\;\;\;\;\;\;\;\;\;\;\;\;\;\;\;\;\;\;\;\;\;\;\;\;\;\;\;\;\;\;\;\;\;\;\;\;\;\;\;\;\;\;\;\;\;\;\;\;\;\;\;\;\;\;\;\;\;\;\;\;\;\;\;\;\;\;\;\;\;\;\;\;\;\;\;\;\;\;\;\;\;\;\;\;\;\;\;\;\;\;\;\;\;\;\;\;\;\;\;\;\;\;\;\;\blacksquare$

We can now continue with the proof of Theorem \ref{theoremmult1}.

\textit{Proof of Theorem \ref{theoremmult1}.} Let $\mathfrak{L}_1$, $\mathfrak{L}_2$ and $\mathfrak{L}_3$ be collections of $L_1$, $L_2$ and $L_3$, respectively, lines in $\R^3$, such that, whenever a line of $\mathfrak{L}_1$, a line of $\mathfrak{L}_2$ and a line of $\mathfrak{L}_3$ meet, they form a joint. 

Let $m \in \{1,2,3\}$ be such that $L_m=\min\{L_1,L_2,L_3\}$. As we have already explained, it suffices to assume that $L_i\geq 2$ for all $i=1,2,3$, and show that  \begin{displaymath}\sum_{\{x \in J:\; N_m(x)> 10^{12}\text{ and }N_i(x) \leq L_i^{1/2} \;\forall\; i=1,2,3\}}(N_1(x)N_2(x)N_3(x))^{1/2} \lesssim (L_1L_2L_3)^{1/2}, \end{displaymath}or, equivalently, that \begin{displaymath}\sum_{(N_1,N_2,N_3) \in \mathcal{M}}|J_{N_1,N_2,N_3}|(N_1N_2N_3)^{1/2} \lesssim (L_1L_2L_3)^{1/2}, \end{displaymath}where $\mathcal{M}:=\{(N_1,N_2,N_3) \in \R_{\geq 1}^3:N_i \leq L_i^{1/2}$ for all $i=1,2,3$, $N_m> 10^{12}$, $N_1=(1+10^{-8})^{k_1}, N_2=(1+10^{-8})^{k_2}$ and $N_3=(1+10^{-8})^{k_3}$ for some $k_1$, $k_2$, $k_3 \in \mathbb{Z}_{\geq 0}$, and $J_{N_1,N_2,N_3} \neq \emptyset\}$.

Even though this will not seem natural at the moment, let us mention that the above will be achieved by showing that, for all $\{i_1,i_2,i_3,i_0\}=\{1,2,3\}$, if 

$\mathcal{M}_{i_1,i_2,i_3,i_0}:=\Bigg\{(N_1,N_2,N_3) \in \mathcal{M}:\frac{L_{i_1}}{N_{i_1}}\leq \frac{L_{i_2}}{N_{i_2}}\leq \frac{L_{i_3}}{N_{i_3}}$ and $\frac{L_{i_0}^2}{N_{i_0}^{8/3}}=\min\bigg\{\frac{L_{i_1}^2}{N_{i_1}^{8/3}},\frac{L_{i_2}^2}{N_{i_2}^{8/3}},\frac{L_{i_3}^2}{N_{i_3}^{8/3}}\bigg\}\Bigg\}$, then

\begin{displaymath} |J_{N_1,N_2,N_3}|\lesssim \frac{(L_1L_2L_3)^{1/2}}{(N_1N_2N_3)^{1/2+1/10^4}}, \; \forall\;(N_1,N_2,N_3) \in \mathcal{B}_{i_1,i_2,i_3,i_0}, \end{displaymath}

where $\mathcal{B}_{i_1,i_2,i_3,i_0}:=\bigg\{(N_1,N_2,N_3)\in \mathcal{M}_{i_1,i_2,i_3,i_0}:i_0=i_3$ or $N_{i_1} \gtrsim N_{i_3}^{1/1000}$ or $N_{i_2} \gtrsim N_{i_3}^{1/1000}$ or $|J_{N_1,N_2,N_3}| \lesssim L_{i_0}^{3/2}$ or $|J_{N_1,N_2,N_3}| \lesssim L_{i_0}L_{j}^{1/2}$ or $\frac{L_j^2}{N_j^{8/3}}\gtrsim \frac{L_{i_3}^2}{N_{i_3}^{7/3}}$, for $j=i_{\lambda}$, where $\lambda$ is the minimal element of $\{1,2,3\}$ such that $i_\lambda \neq i_0\bigg\}$ (this will be the easier case), as well as that 

\begin{displaymath} \sum_{(N_1,N_2,N_3) \in \mathcal{C}_{i_1,i_2,i_3,i_0}}|J_{N_1,N_2,N_3}|(N_1N_2N_3)^{1/2}\lesssim (L_1L_2L_3)^{1/2},\end{displaymath}

where $\mathcal{C}_{i_1,i_2,i_3,i_0}:=\mathcal{M}_{i_1,i_2,i_3,i_0}\setminus \mathcal{B}_{i_1,i_2,i_3,i_0}\;\bigg(=\bigg\{(N_1,N_2,N_3)\in \mathcal{M}_{i_1,i_2,i_3,i_0}: i_0\neq i_3$, $N_{i_1} \lesssim N_{i_3}^{1/1000}$, $N_{i_2} \lesssim N_{i_3}^{1/1000}$, $|J_{N_1,N_2,N_3}| \gtrsim L_{i_0}^{3/2}$, $|J_{N_1,N_2,N_3}| \gtrsim L_{i_0}L_{j}^{1/2}$, $\frac{L_j^2}{N_j^{8/3}}\lesssim \frac{L_{i_3}^2}{N_{i_3}^{7/3}}$ and $\frac{L_j^2}{N_j^{8/3}}\lesssim \frac{L_{i_3}^2}{N_{i_3}^{7/3}}$, for $j=i_{\lambda}$, where $\lambda$ is the minimal element of $\{1,2,3\}$ such that $i_\lambda \neq i_0\bigg\}\bigg)$ (this will be the harder case).

The sets defined above will naturally arise in the proof.

Now, without loss of generality, we assume that $(N_1,N_2,N_3) \in \mathcal{M}_{1,2,3,i_0}$, for some $i_0 \in \{1,2,3\}$. Then, \begin{displaymath}\frac{L_1}{N_1} \leq \frac{L_2}{N_2} \leq \frac{L_3}{N_3},\end{displaymath}
while also
\begin{displaymath} \frac{L_{i_0}^2}{N_{i_0}^{8/3}}= \min \Bigg\{\frac{L_1^2}{N_1^{8/3}},\frac{L_2^2}{N_2^{8/3}},\frac{L_3^2}{N_3^{8/3}}\Bigg\}. \end{displaymath}
For simplicity, let \begin{displaymath}\mathfrak{G}:= J_{N_1,N_2,N_3}\end{displaymath} and \begin{displaymath} S:=|J_{N_1,N_2,N_3}|.\end{displaymath}
As $N_m >1$, it holds that \begin{displaymath}S \lesssim \min \Bigg\{\frac{L_1^2}{N_1^{3}},\frac{L_2^2}{N_2^{3}},\frac{L_3^2}{N_3^{3}}\Bigg\}\lesssim\frac{L_{i_0}^2}{N_{i_0}^{3}}, \end{displaymath}
and thus the quantity $d:=AL_{i_0}^2S^{-1}N_{i_0}^{-3}$ is larger than 1 for some sufficiently large constant $A$. We therefore assume that $A$ is large enough for this to hold, and we will specify its value later. Now, applying the Guth-Katz polynomial method for this $d>1$ and the finite set of points $\mathfrak{G}$, we deduce that there exists a non-zero polynomial $p\in \R[x,y,z]$, of degree $\leq d$, whose zero set $Z$ decomposes $\R^3$ in $\sim d^3$ cells, each of which contains $\lesssim Sd^{-3}$ points of $\mathfrak{G}$. We can assume that this polynomial is square-free, as eliminating the squares of $p$ does not inflict any change on its zero set. 

We clarify here that it is only for technical reasons that we are not defining $i_0$ more naturally as an element of $\{1,2,3\}$ for which 
\begin{displaymath} \frac{L_{i_0}^2}{N_{i_0}^{3}}= \min \Bigg\{\frac{L_1^2}{N_1^{3}},\frac{L_2^2}{N_2^{3}},\frac{L_3^2}{N_3^{3}}\Bigg\}.\end{displaymath}

\textbf{Cellular case:} Suppose that there are $\geq 10^{-8}S$ points of $\mathfrak{G}$ in the union of the interiors of the cells. 

However, we also know that there exist $\sim d^3$ cells in total, each containing $\lesssim Sd^{-3}$ points of $\mathfrak{G}$. Therefore, there exist $\gtrsim d^3$ cells, with $\gtrsim Sd^{-3}$ points of $\mathfrak{G}$ in the interior of each. We call the cells with this property ``full cells". Now:

\textbf{Subcase 1:} If the interior of some full cell contains $< N_{i_0}$ points of $\mathfrak{G}$, then $Sd^{-3} \lesssim N_{i_0}$, so 
\begin{displaymath}S \lesssim \Bigg(\frac{L_{i_0}^2}{SN_{i_0}^3}\Bigg)^3N_{i_0},
\end{displaymath}from which we obtain
\begin{displaymath} S \lesssim\Big(\frac{L_{i_0}^{2}}{N_{i_0}^{3}}\Big)^{3/4} N_{i_0}^{1/4} \sim \frac{L_{i_0}^{3/2}}{N_{i_0}^{2}} \sim \bigg(\frac{L_{i_0}^{2}}{N_{i_0}^{8/3}}\bigg)^{3/4} \lesssim \end{displaymath} \begin{displaymath} \lesssim \bigg(\frac{L_1^{2}}{N_1^{8/3}}\bigg)^{1/4} \bigg(\frac{L_2^{2}}{N_2^{8/3}}\bigg)^{1/4} \bigg(\frac{L_3^{2}}{N_3^{8/3}}\bigg)^{1/4} \sim \frac{(L_1L_2L_3)^{1/2}}{(N_1N_2N_3)^{2/3}}. \end{displaymath}

\textbf{Subcase 2:} If the interior of each full cell contains $\geq N_{i_0}$ points of $\mathfrak{G}$, then we will be led to a contradiction by choosing $A$ so large, that there will be too many intersections between the zero set $Z$ of $p$ and the lines of $\mathfrak{L}_{i_0}$ which do not lie in $Z$. Indeed:

Let $\mathfrak{L}_Z$ be the set of lines of $\mathfrak{L}_{i_0}$ which are lying in $Z$. Consider a full cell and let $S_{cell}$ be the number of points of $\mathfrak{G}$ in the interior of the cell, $\mathfrak{L}_{cell}$ the set of lines of $\mathfrak{L}_{i_0}$ that intersect the interior of the cell and $L_{cell}$ the number of these lines. Obviously, $\mathfrak{L}_{cell} \subset \mathfrak{L}_{i_0} \setminus \mathfrak{L}_Z$.

Now, each point of $\mathfrak{G}$ has at least $N_{i_0}$ lines of $\mathfrak{L}_{i_0}$ passing through it, therefore each point of $\mathfrak{G}$ lying in the interior of the cell has at least $N_{i_0}$ lines of $\mathfrak{L}_{cell}$ passing through it. Thus, since $S_{cell} \geq N_{i_0}$, it follows that $ L_{cell} \geq N_{i_0} + (N_{i_0}-1) + (N_{i_0}-2) +...+1 \gtrsim N_{i_0}^2$, so \begin{displaymath} L^2_{cell}N_{i_0}^{-3} \gtrsim L_{cell} N_{i_0}^{-1}.\end{displaymath} 
However, $N_{i_0}>1$; indeed, if $N_{i_0}$ was equal to 1, then it would follow that $L_{i_0}^2=\frac{L_{i_0}^2}{N_{i_0}^{8/3}}\leq \frac{L_{m}^2}{N_{m}^{8/3}}\lneq L_m^2$ (since $N_m \gneq 1$), and thus $L_{i_0}$ would be strictly smaller than $L_m$, which is not true. So, by the Szemer\'{e}di-Trotter theorem,
\begin{displaymath} S_{cell} \lesssim L^2_{cell}N_{i_0}^{-3} + L_{cell} N_{i_0}^{-1}.\end{displaymath}
Therefore, $S_{cell} \lesssim L^2_{cell}N_{i_0}^{-3}$, thus, since we are working in a full cell, $Sd^{-3} \lesssim L^2_{cell}N_{i_0}^{-3}$, and rearranging we see that \begin{displaymath} L_{cell} \gtrsim S^{1/2}d^{-3/2}N_{i_0}^{3/2}.\end{displaymath}
But each of the lines of $\mathfrak{L}_{cell}$ intersects the boundary of the cell at at least one point $x$, with the property that the induced topology from $\R^3$ to the intersection of the line with the closure of the cell contains an open neighbourhood of $x$; therefore, there are $ \gtrsim S^{1/2}d^{-3/2}N_{i_0}^{3/2}$ incidences of this form between $\mathfrak{L}_{cell}$ and the boundary of the cell. Also, the union of the boundaries of all the cells is the zero set $Z$ of $p$, and if $x$ is a point of $Z$ which belongs to a line intersecting the interior of a cell, such that the induced topology from $\R^3$ to the intersection of the line with the closure of the cell contains an open neighbourhood of $x$,  then there exists at most one other cell whose interior is also intersected by the line and whose boundary contains $x$, such that the induced topology from $\R^3$ to the intersection of the line with the closure of that cell contains an open neighbourhood of $x$. So, if $I$ is the number of incidences between $Z$ and $\mathfrak{L}_{i_0} \setminus \mathfrak{L}_Z$, $I_{cell}$ is the number of incidences between $\mathfrak{L}_{cell}$ and the boundary of the cell, and $\mathcal{C}$ is the set of all the full cells (which, in our case, has cardinality $\gtrsim d^3$), then the above imply that
\begin{displaymath} I \gtrsim \sum_{cell \in \mathcal{C}} I_{cell} \gtrsim (S^{1/2}d^{-3/2}N_{i_0}^{3/2})\cdot d^3 = S^{1/2}d^{3/2}N_{i_0}^{3/2}.\end{displaymath}
On the other hand, if a line does not lie in the zero set $Z$ of $p$, then it intersects $Z$ in $\leq d$ points. Thus,
\begin{displaymath} I \leq L_{i_0} \cdot d \end{displaymath}
This means that \begin{displaymath} S^{1/2}d^{3/2}N_{i_0}^{3/2} \lesssim L_{i_0} \cdot d, \end{displaymath}
which in turn gives $A\lesssim 1$. In other words, there exists some constant $C$, independent of $\mathfrak{L}_1$, $\mathfrak{L}_2$, $\mathfrak{L}_3$ and $N_1$, $N_2$, $N_3$, such that $A \leq C$. By fixing $A$ to be a constant larger than $C$ (and of course large enough to have that $d>1$), we are led to a contradiction.

\textbf{Algebraic case:} Suppose that there are $< 10^{-8}S$ points of $\mathfrak{G}$ in the union of the interiors of the cells. We denote by $\mathfrak{G}_1$ the set of points of $\mathfrak{G}$ which lie in $Z$; it holds that $|\mathfrak{G}_1|>(1-10^{-8})S$. 

In addition, for all $j \in \{1,2,3\}$, let $\mathfrak{L}_j'$ be the set of lines in $\mathfrak{L}_j$, such that each contains $\geq \frac{1}{100} \frac{SN_j}{L_j}$ points of $\mathfrak{G}_1$. We now analyse the situation.

Let $j\in \{1,2,3\}$. Each point of $\mathfrak{G}_1$ intersects at least $N_j$ lines of $\mathfrak{L}_j$. Thus,  \begin{displaymath} I_{\mathfrak{G}_1,\mathfrak{L}_j} > (1-10^{-8})SN_j.\end{displaymath}
On the other hand, each line in $ \mathfrak{L}_j\setminus \mathfrak{L}'_j$ contains fewer than $\frac{1}{100} \frac{SN_j}{L_j}$ points of $\mathfrak{G}_1$, and thus
\begin{displaymath}I_{\mathfrak{G}_1,\mathfrak{L}_j\setminus \mathfrak{L}'_j} <  |\mathfrak{L}_j\setminus \mathfrak{L}'_j| \cdot \frac{SN_j}{100L_j} \leq \frac{1}{100}SN_j.\end{displaymath}
Therefore, since $ I_{\mathfrak{G}_1, \mathfrak{L}_j}=I_{\mathfrak{G}_1,\mathfrak{L}_j \setminus \mathfrak{L}'_j}+I_{\mathfrak{G}_1,\mathfrak{L}'_j}$, it follows that 
\begin{displaymath}I_{\mathfrak{G}_1,\mathfrak{L}'_j}> (1-10^{-8}-10^{-2})SN_j, \end{displaymath}
for all $j \in \{1,2,3\}$.

Now, for all $j \in \{1,2,3\}$, let $\mathfrak{G}'_j$ be the set of points of $\mathfrak{G}_1$ each of which intersects $\geq \frac{10^{-8}}{1+10^{-8}}(1-10^{-8}-10^{-2})N_j$ lines of $\mathfrak{L}'_j$.

Let $j \in \{1,2,3\}$. By the definition of $\mathfrak{G}_j'$, each point of $\mathfrak{G}_1 \setminus \mathfrak{G}_j'$ intersects fewer than $\frac{10^{-8}}{1+10^{-8}}(1-10^{-8}-10^{-2})N_j$ lines of $\mathfrak{L}'_j$, and therefore \begin{displaymath}I_{\mathfrak{G}_1 \setminus \mathfrak{G}'_j, \mathfrak{L}'_j} < |\mathfrak{G}_1 \setminus \mathfrak{G}'_j| \frac{10^{-8}}{1+10^{-8}}(1-10^{-8}-10^{-2})N_j \leq  \frac{10^{-8}}{1+10^{-8}}(1-10^{-8}-10^{-2})SN_j.\end{displaymath}
Thus, since $I_{\mathfrak{G}_1, \mathfrak{L}'_j}=I_{\mathfrak{G}_1 \setminus \mathfrak{G}'_j, \mathfrak{L}'_j}+I_{\mathfrak{G}'_j, \mathfrak{L}'_j}$, we obtain
\begin{displaymath}I_{\mathfrak{G}'_j, \mathfrak{L}'_j} > \frac{1-10^{-8}-10^{-2}}{1+10^{-8}}SN_j. \end{displaymath}
And $I_{\mathfrak{G}'_j, \mathfrak{L}'_j} \leq |\mathfrak{G}'_j|\cdot (1+10^{-8})N_j$, since each point of $\mathfrak{G}$ intersects fewer than $(1+10^{-8})N_{j}$ lines of $\mathfrak{L}_{j}$. Therefore, $ \frac{1-10^{-8}-10^{-2}}{1+10^{-8}}SN_j< |\mathfrak{G}'_j|\cdot (1+10^{-8})N_j$, and thus
\begin{displaymath} |\mathfrak{G}'_j|\geq \frac{1-10^{-8}-10^{-2}}{(1+10^{-8})^2}S; \end{displaymath}
in other words, for all $j\in \{1,2,3\}$, there exist at least $\frac{1-10^{-8}-10^{-2}}{(1+10^{-8})^2}S$ points of $\mathfrak{G}_1$, each intersecting at least $\frac{10^{-8}}{1+10^{-8}}(1-10^{-8}-10^{-2})N_j$ lines of $\mathfrak{L}'_j$.

But \begin{displaymath} |\mathfrak{G}_1'\cup \mathfrak{G}_2' \cup \mathfrak{G}_3'|=|\mathfrak{G}_1'|+|\mathfrak{G}_2'|+|\mathfrak{G}_3'|-|\mathfrak{G}_1'\cap\mathfrak{G}_2'|-|\mathfrak{G}_2'\cap\mathfrak{G}_3'|-|\mathfrak{G}_1'\cap\mathfrak{G}_3'|+|\mathfrak{G}_1'\cap\mathfrak{G}_2'\cap\mathfrak{G}_3'|,\end{displaymath} thus 
\begin{displaymath}|\mathfrak{G}_1'\cap\mathfrak{G}_2'\cap\mathfrak{G}_3'|\geq\frac{1-10^{-8}-10^{-2}}{(1+10^{-8})^2}S-3S+3 \frac{2(1-10^{-8}-10^{-2})-(1+10^{-8})^2}{(1+10^{-8})^2}S=\end{displaymath} 
\begin{displaymath}= \frac{7(1-10^{-8}-10^{-2})-6(1+10^{-8})^2}{(1+10^{-8})^2}S=\frac{1-19 \cdot 10^{-8}-7 \cdot 10^{-2}-6 \cdot 10^{-16}}{(1+10^{-8})^2}S;\end{displaymath}
in other words, there exist at least $\frac{1-8\cdot 10^{-2}}{(1+10^{-8})^2}S$ points of $\mathfrak{G}_1$, each intersecting at least $\frac{10^{-8}}{1+10^{-8}}(1-10^{-8}-10^{-2})N_j$ lines of $\mathfrak{L}'_j$, simultaneously for all $j\in\{1,2,3\}$.

Now, let $\mathfrak{L}_{i_0,1}$ be the set of lines in $\mathfrak{L}_{i_0}$, each containing $\geq\frac{1}{10^{100}}\frac{SN_{i_0}}{L_{i_0}}$ points of $\mathfrak{G}_1' \cap \mathfrak{G}_2' \cap \mathfrak{G}_3'$. 

Since each point of $\mathfrak{G}_1'\cap\mathfrak{G}_2'\cap\mathfrak{G}_3'$ intersects at least $\frac{10^{-8}}{1+10^{-8}}(1-10^{-8}-10^{-2})N_{i_0}$ lines of $\mathfrak{L}'_{i_0}$, it follows that
\begin{displaymath} I_{\mathfrak{G}_1'\cap\mathfrak{G}_2'\cap\mathfrak{G}_3', \mathfrak{L}_{i_0}'}\geq |\mathfrak{G}_1'\cap\mathfrak{G}_2'\cap\mathfrak{G}_3'| \frac{10^{-8}}{1+10^{-8}}(1-10^{-8}-10^{-2})N_{i_0}\geq 
\end{displaymath}
\begin{displaymath} \geq \frac{1-8\cdot 10^{-2}}{(1+10^{-8})^2}S \frac{10^{-8}}{1+10^{-8}}(1-10^{-8}-10^{-2})N_j\geq 10^{-10}SN_{i_0}.\end{displaymath}
On the other hand, each line in $\mathfrak{L}_{i_0}'\setminus \mathfrak{L}_{i_0,1}$ contains fewer than $\frac{1}{10^{100}}\frac{SN_{i_0}}{L_{i_0}}$ points of $\mathfrak{G}_1' \cap \mathfrak{G}_2' \cap \mathfrak{G}_3'$, so
\begin{displaymath} I_{\mathfrak{G}_1'\cap\mathfrak{G}_2'\cap\mathfrak{G}_3', \mathfrak{L}_{i_0}'\setminus \mathfrak{L}_{i_0,1}} < L_{i_0} \cdot \frac{SN_{i_0}}{10^{100}L_{i_0}}=10^{-100}SN_{i_0}.\end{displaymath} Therefore, since $I_{\mathfrak{G}_1'\cap\mathfrak{G}_2'\cap\mathfrak{G}_3', \mathfrak{L}_{i_0}'}=I_{\mathfrak{G}_1'\cap\mathfrak{G}_2'\cap\mathfrak{G}_3', \mathfrak{L}_{i_0}'\setminus \mathfrak{L}_{i_0,1}}+I_{\mathfrak{G}_1'\cap\mathfrak{G}_2'\cap\mathfrak{G}_3',  \mathfrak{L}_{i_0,1}}$, we obtain
\begin{displaymath} I_{\mathfrak{G}_1'\cap\mathfrak{G}_2'\cap\mathfrak{G}_3',  \mathfrak{L}_{i_0,1}} >10^{-11}SN_j. \end{displaymath} Now, let $\mathfrak{G}'$ be the set of points in $\mathfrak{G}_1'\cap\mathfrak{G}_2'\cap\mathfrak{G}_3'$, each of which lies in $\geq10^{-12}N_{i_0}$ lines of $\mathfrak{L}_{i_0,1}$. By the definition of $\mathfrak{G}'$, it holds that
\begin{displaymath} I_{(\mathfrak{G}_1'\cap\mathfrak{G}_2'\cap\mathfrak{G}_3')\setminus \mathfrak{G}',  \mathfrak{L}_{i_0,1}} <10^{-12}SN_{i_0}, \end{displaymath} and thus, as $I_{\mathfrak{G}_1'\cap\mathfrak{G}_2'\cap\mathfrak{G}_3',  \mathfrak{L}_{i_0,1}}=I_{(\mathfrak{G}_1'\cap\mathfrak{G}_2'\cap\mathfrak{G}_3')\setminus \mathfrak{G}',  \mathfrak{L}_{i_0,1}}+I_{\mathfrak{G}', \mathfrak{L}_{i_0,1}}$, it follows that
\begin{displaymath} I_{\mathfrak{G}', \mathfrak{L}_{i_0,1}}> (10^{-11}-10^{-12})SN_{i_0}= 9 \cdot 10^{-12}SN_{i_0}. \end{displaymath} At the same time, however, \begin{displaymath} I_{\mathfrak{G}', \mathfrak{L}_{i_0,1}} < |\mathfrak{G}'|\cdot (1+10^{-8})N_{i_0}, \end{displaymath}
since each point of $\mathfrak{G}$ intersects fewer than $(1+10^{-8})N_{i_0}$ lines of $\mathfrak{L}_{i_0}$. Therefore,
\begin{displaymath} |\mathfrak{G}'| > \frac{9\cdot 10^{-12}}{1+10^{-8}}S. \end{displaymath}

From now on, we fix $\{j,k\}=\{1,2,3\}\setminus \{i_0\}$, such that $j<k$. We have thus so far reached the conclusion that there exist $\gtrsim S$ points of $\mathfrak{G}_1$, each intersecting at least $10^{-12}N_{i_0}$ lines of $\mathfrak{L}_{i_0,1}$, at least $\frac{10^{-8}}{1+10^{-8}}(1-10^{-8}-10^{-2})N_j$ lines of $\mathfrak{L}'_j$, and at least $\frac{10^{-8}}{1+10^{-8}}(1-10^{-8}-10^{-2})N_k$ lines of $\mathfrak{L}'_k$.

Suppose that $\frac{1}{10^{100}}\frac{SN_{3}}{L_{3}} > d$.

Then $\frac{1}{10^{100}}\frac{SN_{1}}{L_{1}} > d$ and $\frac{1}{10^{100}}\frac{SN_{2}}{L_{2}} > d$ as well. So, all the lines of $\mathfrak{L}_{i_0,1}$, $\mathfrak{L}_j'$ and $\mathfrak{L}_k'$ lie in $Z$. Therefore, each point of $\mathfrak{G}'$ is a critical point of $Z$, and thus each line of $\mathfrak{L}_{i_0,1}$ is a critical line, while all the critical lines number $\lesssim d^2$. 

On the other hand, each point of $\mathfrak{G}'$ is a multijoint for the collections $\mathfrak{L}_{i_0,1}$, $\mathfrak{L}_j$ and $\mathfrak{L}_k$ of lines, each lying in $\gtrsim N_{i_0}$ lines of $\mathfrak{L}_{i_0,1}$, in $\gtrsim N_j$ lines of $\mathfrak{L}_j$ and in $\gtrsim N_k$ lines of $\mathfrak{L}_k$. Therefore, due to the fact that, whenever a line of $\mathfrak{L}_{i_0,1}$, a line of $\mathfrak{L}_j$ and a line of $\mathfrak{L}_k$ meet at a point, they form a joint there, it follows by Proposition \ref{multsimple} that
\begin{displaymath} |\mathfrak{G}'| \lesssim \frac{(|\mathfrak{L}_{i_0,1}||\mathfrak{L}_j||\mathfrak{L}_k|)^{1/2}}{(N_{i_0}N_jN_k)^{1/2}} \lesssim \frac{(d^2L_jL_k)^{1/2}}{(N_{i_0}N_jN_k)^{1/2}} \sim d\frac{(L_jL_k)^{1/2}}{(N_{i_0}N_jN_k)^{1/2}}\sim \frac{L_{i_0}^2}{SN_{i_0}^3}\frac{(L_jL_k)^{1/2}}{(N_{i_0}N_jN_k)^{1/2}}\sim \end{displaymath}
\begin{displaymath}\sim \frac{L_{i_0}^2}{SN_{i_0}^{7/2}}\frac{(L_jL_k)^{1/2}}{(N_jN_k)^{1/2}}.
\end{displaymath}
Thus, since $|\mathfrak{G}'| \gtrsim S$, we have that 
\begin{displaymath} S \lesssim \frac{L_{i_0}}{N_{i_0}^{7/4}} \frac{L_j^{1/4}}{N_j^{1/4}}\frac{L_k^{1/4}}{N_k^{1/4}}\lesssim\frac{L_{i_0}}{N_{i_0}^{4/3}} \frac{L_j^{1/4}}{N_j^{1/4}}\frac{L_k^{1/4}}{N_k^{1/4}}\sim \Bigg(\frac{L_{i_0}^2}{N_{i_0}^{8/3}}\Bigg)^{1/2} \frac{L_j^{1/4}}{N_j^{1/4}}\frac{L_k^{1/4}}{N_k^{1/4}}\lesssim\end{displaymath}
\begin{displaymath}\lesssim \Bigg(\frac{L_{i_0}^2}{N_{i_0}^{8/3}}\Bigg)^{\frac{1}{2}\cdot \frac{1}{2}}\Bigg(\frac{L_{j}^2}{N_{j}^{8/3}}\Bigg)^{\frac{1}{4}\cdot \frac{1}{2}}\Bigg(\frac{L_{k}^2}{N_{k}^{8/3}}\Bigg)^{\frac{1}{4}\cdot \frac{1}{2}} \frac{L_j^{1/4}}{N_j^{1/4}}\frac{L_k^{1/4}}{N_k^{1/4}} \sim \frac{L_{i_0}^{1/2}}{N_{i_0}^{2/3}}\frac{L_j^{1/4}}{N_j^{1/3}}\frac{L_k^{1/4}}{N_k^{1/3}}\frac{L_j^{1/4}}{N_j^{1/4}}\frac{L_k^{1/4}}{N_k^{1/4}} \sim \end{displaymath}
\begin{displaymath} \sim \frac{L_{i_0}^{1/2}}{N_{i_0}^{2/3}}\frac{L_j^{1/2}}{N_j^{7/12}}\frac{L_k^{1/2}}{N_k^{7/12}}\;\Bigg(\lesssim \frac{(L_1L_2L_3)^{1/2}}{(N_1N_2N_3)^{1/2+1/10^4}}\Bigg).\end{displaymath}

Therefore, we can assume from now on that $\frac{1}{10^{100}}\frac{SN_{3}}{L_{3}} \leq d$.

\textbf{Case A:} Suppose that $i_0=3$. Then, $d \sim \frac{L_3^2}{SN_3^3}$, and thus
\begin{displaymath}\frac{SN_3}{L_3}\lesssim \frac{L_3^2}{SN_3^3}, \end{displaymath}which, since $i_0=3$, implies that
\begin{displaymath} S \lesssim \frac{L_{i_0}^{3/2}}{N_{i_0}^{2}}\sim \Bigg(\frac{L_{i_0}^2}{N_{i_0}^{8/3}}\Bigg)^{3/4}\lesssim \Bigg(\frac{L_1^2}{N_1^{8/3}}\Bigg)^{\frac{1}{3}\cdot \frac{3}{4}}\Bigg(\frac{L_2^2}{N_2^{8/3}}\Bigg)^{\frac{1}{3}\cdot \frac{3}{4}}\Bigg(\frac{L_3^2}{N_3^{8/3}}\Bigg)^{\frac{1}{3}\cdot \frac{3}{4}}\sim\end{displaymath}
\begin{displaymath} \sim\frac{(L_1L_2L_3)^{1/2}}{(N_1N_2N_3)^{2/3}} \;\Bigg(\lesssim \frac{(L_1L_2L_3)^{1/2}}{(N_1N_2N_3)^{1/2+1/10^4}}\Bigg). \end{displaymath}

\textbf{Case B:} Suppose that $N_{1} \gtrsim N_3^{1/1000}$ or $N_{2} \gtrsim N_3^{1/1000}$. Then
\begin{displaymath}S \lesssim \frac{L_3}{N_3}\cdot d \sim \frac{L_3}{N_3}\cdot\frac{L_{i_0}^2}{SN_{i_0}^3},
\end{displaymath}
from which we obtain
\begin{displaymath} S \lesssim \frac{L_{i_0}}{N_{i_0}^{3/2}} \frac{L_3^{1/2}}{N_3^{1/2}} \lesssim\frac{L_{i_0}}{N_{i_0}^{4/3}} \frac{L_3^{1/2}}{N_3^{1/2}} \sim \Bigg(\frac{L_{i_0}^2}{N_{i_0}^{8/3}}\Bigg)^{1/2} \frac{L_3^{1/2}}{N_3^{1/2}} \lesssim \Bigg(\frac{L_{1}^2}{N_{1}^{8/3}}\Bigg)^{\frac{1}{2}\cdot\frac{1}{2}}\Bigg(\frac{L_{2}^2}{N_{2}^{8/3}}\Bigg)^{\frac{1}{2}\cdot\frac{1}{2}} \frac{L_3^{1/2}}{N_3^{1/2}}\sim\end{displaymath}
\begin{displaymath}\sim  \frac{L_{1}^{1/2}}{N_{1}^{2/3}}\frac{L_2^{1/2}}{N_2^{2/3}}\frac{L_3^{1/2}}{N_3^{1/2}} \lesssim \frac{(L_1L_2L_3)^{1/2}}{(N_1N_2N_3)^{1/2+1/10^4}}. \end{displaymath}

\textbf{Case C:} Suppose that $i_0\neq 3$, $N_{1} \lesssim N_3^{1/1000}$ and $N_{2} \lesssim N_3^{1/1000}$, while also $S \lesssim L_{i_0}^{3/2}$ or $S \lesssim L_{i_0}L_j^{1/2}$. 

Under these assumptions, $\{i_0,j\}=\{1,2\}$, therefore
\begin{displaymath} S \lesssim L_{i_0}^{3/2}+L_{i_0}L_j^{1/2}\sim \frac{L_{i_0}^{3/2}}{N_{i_0}^2}N_{i_0}^2 +\frac{L_{i_0}}{N_{i_0}^{4/3}}L_j^{1/2}N_{i_0}^{4/3}\sim 
\end{displaymath}
\begin{displaymath}\sim \Bigg(\frac{L_{i_0}^{2}}{N_{i_0}^{8/3}}\Bigg)^{3/4}N_{i_0}^2 +\Bigg(\frac{L_{i_0}^2}{N_{i_0}^{8/3}}\Bigg)^{1/2}L_j^{1/2}N_{i_0}^{4/3} \lesssim 
\end{displaymath}
\begin{displaymath} \lesssim \Bigg(\frac{L_{1}^{2}}{N_{1}^{8/3}}\Bigg)^{\frac{1}{3} \cdot \frac{3}{4}}\Bigg(\frac{L_{2}^{2}}{N_{2}^{8/3}}\Bigg)^{\frac{1}{3} \cdot \frac{3}{4}}\Bigg(\frac{L_{3}^{2}}{N_{3}^{8/3}}\Bigg)^{\frac{1}{3} \cdot \frac{3}{4}}N_{i_0}^2 +
\end{displaymath}
\begin{displaymath} + \Bigg(\frac{L_{i_0}^2}{N_{i_0}^{8/3}}\Bigg)^{\frac{1}{2}\cdot \frac{1}{2}}\Bigg(\frac{L_{3}^2}{N_{3}^{8/3}}\Bigg)^{\frac{1}{2}\cdot \frac{1}{2}}L_j^{1/2}N_{i_0}^{4/3}\sim \end{displaymath} 
\begin{displaymath} \sim \frac{L_{1}^{1/2}}{N_{1}^{2/3}}\frac{L_2^{1/2}}{N_{2}^{2/3}}  \frac{L_{3}^{1/2}}{N_{3}^{2/3}}N_{i_0}^2+\frac{L_{i_0}^{1/2}}{N_{i_0}^{2/3}}L_j^{1/2}  \frac{L_{3}^{1/2}}{N_{3}^{2/3}}N_{i_0}^{4/3}\lesssim \end{displaymath} \begin{displaymath} \lesssim \frac{L_{1}^{1/2}}{N_{1}^{2/3}}\frac{L_2^{1/2}}{N_{2}^{2/3}}  \frac{L_{3}^{1/2}}{N_{3}^{2/3}}N_3^{2/1000}+\frac{L_{i_0}^{1/2}}{N_{i_0}^{2/3}}L_j^{1/2}  \frac{L_{3}^{1/2}}{N_{3}^{2/3}}N_{3}^{(4/3)(1/1000)} \lesssim
\end{displaymath}
\begin{displaymath} \lesssim \frac{(L_1L_2L_3)^{1/2}}{(N_1N_2N_3)^{1/2+1/10^4}}. \end{displaymath}

\textbf{Case D:} Suppose that $i_0\neq 3$ (and thus $\{i_0,j\}=\{1,2\}$), $N_{1} \lesssim N_3^{1/1000}$, $N_{2} \lesssim N_3^{1/1000}$, $S \gtrsim L_{i_0}^{3/2}$ and $S \gtrsim L_{i_0}L_j^{1/2}$.

\textbf{Subcase D.1:} If $\frac{L_3^2}{N_3^{7/3}} \lesssim \frac{L_j^2}{N_j^{8/3}}$, then, since
\begin{displaymath} \frac{SN_{3}}{L_{3}} \lesssim d \sim \frac{L_{i_0}^2}{SN_{i_0}^3},\end{displaymath}
we obtain
\begin{displaymath}S \lesssim \frac{L_3^{1/2}}{N_3^{1/2}}\frac{L_{i_0}}{N_{i_0}^{3/2}} \sim \frac{L_3^{1/4}}{N_3^{7/24}} \frac{L_3^{1/4}}{N_3^{5/24}} \frac{L_{i_0}}{N_{i_0}^{3/2}} \sim \Bigg(\frac{L_3^2}{N_3^{7/3}}\Bigg)^{1/8}\frac{L_3^{1/4}}{N_3^{5/24}}\frac{L_{i_0}}{N_{i_0}^{3/2}} \lesssim \end{displaymath}
\begin{displaymath} \lesssim \Bigg(\frac{L_j^2}{N_j^{8/3}}\Bigg)^{1/8}\frac{L_3^{1/4}}{N_3^{5/24}}\frac{L_{i_0}}{N_{i_0}^{4/3}}\sim
\frac{L_j^{1/4}}{N_j^{1/3}}\frac{L_3^{1/4}}{N_3^{5/24}}\Bigg(\frac{L_{i_0}^2}{N_{i_0}^{8/3}}\Bigg)^{1/2}\lesssim 
\end{displaymath}
\begin{displaymath}\lesssim \frac{L_j^{1/4}}{N_j^{1/3}}\frac{L_3^{1/4}}{N_3^{5/24}}\Bigg(\frac{L_{i_0}^2}{N_{i_0}^{8/3}}\Bigg)^{\frac{1}{2}\cdot \frac{1}{2}}\Bigg(\frac{L_{j}^2}{N_{j}^{8/3}}\Bigg)^{\frac{1}{4}\cdot \frac{1}{2}}\Bigg(\frac{L_{3}^2}{N_{3}^{8/3}}\Bigg)^{\frac{1}{4}\cdot \frac{1}{2}} \sim 
\end{displaymath} 
\begin{displaymath} \sim \frac{L_j^{1/4}}{N_j^{1/3}}\frac{L_3^{1/4}}{N_3^{5/24}}\frac{L_{i_0}^{1/2}}{N_{i_0}^{2/3}}\frac{L_{j}^{1/4}}{N_{j}^{1/3}}\frac{L_{3}^{1/4}}{N_{3}^{1/3}} \sim \frac{(L_{i_0}L_jL_3)^{1/2}}{N_{i_0}^{2/3}N_j^{2/3}N_3^{13/24}} \lesssim
\end{displaymath}
\begin{displaymath}\lesssim\frac{(L_1L_2L_3)^{1/2}}{(N_1N_2N_3)^{1/2+1/24}}\;\Bigg(\lesssim\frac{(L_1L_2L_3)^{1/2}}{(N_1N_2N_3)^{1/2+1/10^4}}\Bigg). \end{displaymath}

\textbf{Remark:} Note that we have already proved that
\begin{displaymath} |J_{N_1,N_2,N_3}| \lesssim \frac{(L_1L_2L_3)^{1/2}}{(N_1N_2N_3)^{1/2+1/10^4}} \end{displaymath}
for all $(N_1,N_2,N_3) \in \mathcal{B}_{1,2,3,i_0}$, and that the $(N_1,N_2,N_3)$ that correspond to the remaining cases all belong to $\mathcal{C}_{1,2,3,i_0}$. 

$\;\;\;\;\;\;\;\;\;\;\;\;\;\;\;\;\;\;\;\;\;\;\;\;\;\;\;\;\;\;\;\;\;\;\;\;\;\;\;\;\;\;\;\;\;\;\;\;\;\;\;\;\;\;\;\;\;\;\;\;\;\;\;\;\;\;\;\;\;\;\;\;\;\;\;\;\;\;\;\;\;\;\;\;\;\;\;\;\;\;\;\;\;\;\;\;\;\;\;\;\;\;\;\;\;\;\;\;\;\;\;\;\;\;\;\;\;\;\;\;\;\;\;\;\;\;\;\;\;\;\;\;\;\blacksquare$

\textbf{Subcase D.2:} Suppose that $\frac{L_j^2}{N_j^{8/3}}\lesssim \frac{L_3^2}{N_3^{7/3}}$.

We may assume that $\frac{1}{2}\cdot \frac{1}{10^{100}}\frac{SN_{i_0}}{L_{i_0}} > 3d$ and  $\frac{1}{2}\cdot \frac{1}{10^{100}}\frac{SN_{j}}{L_{j}} > 3d$, as the inequalities $\frac{1}{2}\cdot \frac{1}{10^{100}}\frac{SN_{i_0}}{L_{i_0}} \leq 3d$ and  $\frac{1}{2} \cdot \frac{1}{10^{100}}\frac{SN_{j}}{L_{j}} \leq 3d$ would imply that $S \lesssim \frac{L_{i_0}^{3/2}}{N_{i_0}^2}\lesssim L_{i_0}^{3/2}$ and $S \lesssim \frac{L_{i_0}}{N_{i_0}^{3/2}}\frac{L_j^{1/2}}{N_j^{1/2}}\lesssim L_{i_0}L_j^{1/2}$, respectively, something that we can assume is false in Case D. 

In particular, this implies that all the lines in $\mathfrak{L}_{i_0,1}$ and $\mathfrak{L}_j'$ lie in $Z$.

If, in addition, we assume that, for each point of $\mathfrak{G}'$, there exist at least three lines of $\mathfrak{L}_{i_0,1} \cup \mathfrak{L}_j'$ passing through it, then it follows that each point of $\mathfrak{G}'$ is either critical or flat, and eventually that each line in $\mathfrak{L}_{i_0,1}$ is either critical or flat (since each line in $\mathfrak{L}_{i_0,1}$ contains at least $\frac{1}{10^{100}}\frac{SN_{i_0}}{L_{i_0}}$ points of $\mathfrak{G}'$, of which either at least half are critical or at least half are flat). In whatever follows we accept that this is true; in other words, that, for each point of $\mathfrak{G}'$, there exist at least three lines of $\mathfrak{L}_{i_0,1} \cup \mathfrak{L}_j'$ passing through it. 

\textbf{Remark.} In fact, the above certainly holds if either the quantity $10^{-12}N_{i_0}$, which is a lower bound on the number of lines of $\mathfrak{L}_{i_0,1}$ passing through each point of $\mathfrak{G}'$, or the quantity $\frac{1-10^{-8}-10^{-2}}{1+10^{8}}N_j$, which is a lower bound on the number of lines of $\mathfrak{L}_{j}'$ passing through each point of $\mathfrak{G}'$, is strictly larger than 1.

Moreover, it certainly holds that at least one of the quantities $10^{-12}N_{i_0}$ and $\frac{1-10^{-8}-10^{-2}}{1+10^{8}}\cdot N_j$ is strictly larger than 1 under the assumption that $10^{-12}N_m>1$. 

The reason for this is that $m\in \{i_0,j\}=\{1,2\}$. Indeed, $L_{i_0}$ is equal to $\frac{L_{i_0}}{N_{i_0}^{4/3}}N_{i_0}^{4/3}$, a quantity that can be assumed to be strictly smaller than $L_3$, since $\frac{L_{i_0}}{N_{i_0}^{4/3}}=\Bigg(\frac{L_{i_0}^2}{N_{i_0}^{8/3}}\Bigg)^{1/2} \leq \Bigg(\frac{L_{3}^2}{N_{3}^{8/3}}\Bigg)^{1/2}= \frac{L_3}{N_3^{4/3}}$ and $N_{i_0}^{\frac{4}{3}} \lesssim N_3^{\frac{4}{3}\cdot \frac{1}{1000}}$, where the explicit constant now hiding behind the $\lesssim$ symbol has so far not been constrained. This means that $L_3$ cannot be equal to $L_m=\min\{L_1,L_2,L_3\}$, and thus $m\in \{i_0,j\}=\{1,2\}$. 

Consequently, under the assumption that
\begin{displaymath}N_m > 10^{12},
\end{displaymath}it holds that, for each point of $\mathfrak{G}'$, there exist at least three lines of $\mathfrak{L}_{i_0,1} \cup \mathfrak{L}_j'$ passing through it.

In particular, this is the reason why, in the statement of Theorem \ref{theoremmult1}, we consider only the multijoints $x \in J$ for which $N_m(x)>10^{12}$; it is a convenient way to ensure that, for each point of $\mathfrak{G}'$, there exist at least three lines of $\mathfrak{L}_{i_0,1} \cup \mathfrak{L}_j'$ passing through it, a fact which allows us to continue our analysis. However, in reality, the only case we cannot tackle here is the one where there exists exactly one line of $\mathfrak{L}_{i_0,1}$ and exactly one line of $\mathfrak{L}_j'$ passing through each point of $\mathfrak{G}'$, which, since $m \in \{i_0,j\}$, falls under the case where exactly one line of $\mathfrak{L}_m$ passes through each multijoint in $J$, and which, in turn, we have already excluded from our analysis.

That is the reason why, after having proved the statement of Theorem \ref{theoremmult1}, it only suffices to show that
\begin{displaymath} \sum_{\{x \in J: \; N_m(x)=1\}}(N_1(x)N_2(x)N_3(x))^{1/2} \lesssim (L_1L_2L_3)^{1/2}
\end{displaymath}in order to deduce that
\begin{displaymath} \sum_{x \in J}(N_1(x)N_2(x)N_3(x))^{1/2} \lesssim (L_1L_2L_3)^{1/2}
\end{displaymath} under the assumptions of Theorem \ref{theoremmult1}.

$\;\;\;\;\;\;\;\;\;\;\;\;\;\;\;\;\;\;\;\;\;\;\;\;\;\;\;\;\;\;\;\;\;\;\;\;\;\;\;\;\;\;\;\;\;\;\;\;\;\;\;\;\;\;\;\;\;\;\;\;\;\;\;\;\;\;\;\;\;\;\;\;\;\;\;\;\;\;\;\;\;\;\;\;\;\;\;\;\;\;\;\;\;\;\;\;\;\;\;\;\;\;\;\;\;\;\;\;\;\;\;\;\;\;\;\;\;\;\;\;\;\;\;\;\;\;\;\;\;\;\;\;\;\blacksquare$

As we have already mentioned, we are in the case where $(N_1,N_2,N_3)\in \mathcal{C}_{1,2,3,i_0}$.

Each line of $\mathfrak{L}_{i_0,1}$ is either a critical or a flat line of $Z$. Let $\mathfrak{L}_{crit}$ be the set of lines in $\mathfrak{L}_{i_0,1}$ that are critical lines of $Z$, $\mathfrak{L}_{flat,1}$ the set of lines in $\mathfrak{L}_{i_0,1}$ that are flat lines of $Z$ not lying in the planes of $Z$, and $\mathfrak{L}_{flat,2}$ the set of lines in $\mathfrak{L}_{i_0,1}$ that are flat lines of $Z$ lying in the planes of $Z$. Since $|\mathfrak{G'}| \gtrsim S$, it follows that either $\gtrsim S$ points of $\mathfrak{G}'$ lie in $\mathfrak{L}_{crit} \cup \mathfrak{L}_{flat,1}$ or $\gtrsim S$ points of $\mathfrak{G}'$ lie in $\mathfrak{L}_{flat,2}$.

\textbf{Subcase D.2.1:} Suppose that $\gtrsim S$ points of $\mathfrak{G}'$ lie in $\mathfrak{L}_{crit} \cup \mathfrak{L}_{flat,1}$.

Then, the sets of lines $\mathfrak{L}_{crit} \cup \mathfrak{L}_{flat,1}$, $\mathfrak{L}_j$ and $\mathfrak{L}_3$ form $\gtrsim S$ multijoints, with the property that $\geq 1$ line of $\mathfrak{L}_{crit} \cup \mathfrak{L}_{flat,1}$, $\gtrsim N_j$ lines of $\mathfrak{L}_j$ and $\gtrsim N_3$ lines of $\mathfrak{L}_3$ are passing through each. Therefore, since, whenever a line of $\mathfrak{L}_{crit} \cup \mathfrak{L}_{flat,1}$, a line of $\mathfrak{L}_j$ and a line of $\mathfrak{L}_3$ meet at a point, they form a joint there, it follows by Proposition \ref{multsimple} that

\begin{displaymath} S(N_jN_3)^{1/2} \lesssim (|\mathfrak{L}_{crit} \cup \mathfrak{L}_{flat,1}| L_jL_3)^{1/2}, \end{displaymath}
and since
\begin{displaymath}|\mathfrak{L}_{crit} \cup \mathfrak{L}_{flat,1}| \lesssim d^2, \end{displaymath}
it holds that
\begin{displaymath} S \lesssim \frac{(d^2L_jL_3)^{1/2}}{(N_jN_3)^{1/2}} \sim d\frac{(L_jL_3)^{1/2}}{(N_jN_3)^{1/2}}\sim \frac{L_{i_0}^2}{SN_{i_0}^3}\frac{(L_jL_3)^{1/2}}{(N_jN_3)^{1/2}},\end{displaymath}
from which we obtain
\begin{displaymath} S \lesssim \frac{L_{i_0}}{N_{i_0}^{3/2}} \frac{L_j^{1/4}}{N_j^{1/4}}\frac{L_3^{1/4}}{N_3^{1/4}}\lesssim\frac{L_{i_0}}{N_{i_0}^{4/3}} \frac{L_j^{1/4}}{N_j^{1/4}}\frac{L_3^{1/4}}{N_3^{1/4}}\sim \Bigg(\frac{L_{i_0}^2}{N_{i_0}^{8/3}}\Bigg)^{1/2} \frac{L_j^{1/4}}{N_j^{1/4}}\frac{L_3^{1/4}}{N_3^{1/4}}\lesssim\end{displaymath}
\begin{displaymath}\lesssim \Bigg(\frac{L_{i_0}^2}{N_{i_0}^{8/3}}\Bigg)^{\frac{1}{2}\cdot \frac{1}{2}}\Bigg(\frac{L_{j}^2}{N_{j}^{8/3}}\Bigg)^{\frac{1}{4}\cdot \frac{1}{2}}\Bigg(\frac{L_{3}^2}{N_{3}^{8/3}}\Bigg)^{\frac{1}{4}\cdot \frac{1}{2}} \frac{L_j^{1/4}}{N_j^{1/4}}\frac{L_3^{1/4}}{N_3^{1/4}} \sim \frac{L_{i_0}^{1/2}}{N_{i_0}^{2/3}}\frac{L_j^{1/4}}{N_j^{1/3}}\frac{L_3^{1/4}}{N_3^{1/3}}\frac{L_j^{1/4}}{N_j^{1/4}}\frac{L_3^{1/4}}{N_3^{1/4}} \sim \end{displaymath}
\begin{displaymath} \sim \frac{L_{i_0}^{1/2}}{N_{i_0}^{2/3}}\frac{L_j^{1/2}}{N_j^{7/12}}\frac{L_3^{1/2}}{N_3^{7/12}}\;\Bigg(\lesssim \frac{(L_1L_2L_3)^{1/2}}{(N_1N_2N_3)^{1/2+1/10^4}}\Bigg).\end{displaymath}
\textbf{Subcase D.2.2:} Suppose that $\lesssim S$ points of $\mathfrak{G}'$ lie in the lines of $\mathfrak{L}_{crit} \cup \mathfrak{L}_{flat,1}$. Then, $\gtrsim S$ points of $\mathfrak{G}'$ lie in the lines of $\mathfrak{L}_{flat,2}$. 

If $\gtrsim S$ of these points are critical, then one of the lines of $\mathfrak{L}_{flat,2}\;(\subseteq \mathfrak{L}_{i_0})$ contains $\geq \frac{S}{L_{i_0}}$ critical points. However, the line is flat, and thus it contains at most $d$ critical points. Therefore,
\begin{displaymath} \frac{S}{L_{i_0}} \leq d \sim \frac{L_{i_0}^2}{SN_{i_0}^3}, \end{displaymath}
from which we obtain
\begin{displaymath} S \lesssim L_{i_0}^{3/2}, \end{displaymath}

which we may assume to be a contradiction since D.2.2 is a subcase of Case D. 

Therefore, $\gtrsim S$ of these points are flat. Let $\mathfrak{G}_{flat}$ be the set of these points, $\mathfrak{L}_{i_0, flat,2}$ the set of lines in $\mathfrak{L}_{flat,2}$ each of which contains at least one point of $\mathfrak{G}_{flat}$, and $\mathfrak{L}_{j,flat}'$ the set of lines in $\mathfrak{L}_{j}'$ each of which contains at least one point of $\mathfrak{G}_{flat}$. Note that, since each line of $\mathfrak{L}_{i_0,1}$ and $\mathfrak{L}_{j}'$ lies in $Z$ and each point of $\mathfrak{G}_{flat}$ is a regular point of $Z$, it follows that all the lines of $\mathfrak{L}_{i_0,1}\cup\mathfrak{L}_{j}'$ passing through a point of $\mathfrak{G}_{flat}$, i.e. all the lines of $\mathfrak{L}_{i_0, flat,2} \cup \mathfrak{L}_{j,flat}'$ passing through it, lie on the (unique) plane in $Z$ containing the point.

Now, for every $(N_1',N_2',N_3') \in \mathcal{C}_{1,2,3,i_0}$, we denote by $d_{N_1',N_2',N_3'}$ the degree of the polynomial with the zero set of which we achieve the cell decomposition in the case of the triple $(N_1',N_2',N_3')$, by $Z_{d_{N_1',N_2',N_3'}}$ the zero set of that polynomial, and by $\Pi_{d_{N_1',N_2',N_3'}}$ the set of planes contained in $Z_{d_{N_1',N_2',N_3'}}$. In addition, let $\mathfrak{L}_{i_0,{d_{N_1',N_2',N_3'}}}$, $\mathfrak{L}_{i_0,1,{d_{N_1',N_2',N_3'}}}$, $\mathfrak{L}_{i_0,flat,2,d_{N_1',N_2',N_3'}}$, $\mathfrak{L}_{j,d_{N_1',N_2',N_3'}}'$, $\mathfrak{L}_{j,flat,d_{N_1',N_2',N_3'}}'$ and $\mathfrak{G}_{flat, d_{N_1',N_2',N_3'}}$ be the sets $\mathfrak{L}_{i_0}$, $\mathfrak{L}_{i_0,1}$, $\mathfrak{L}_{i_0,flat,2}$,  $\mathfrak{L}_{j}'$, $\mathfrak{L}_{j,flat}'$ and $\mathfrak{G}_{flat}$, respectively, corresponding to the triple $(N_1',N_2',N_3')$. 

We now define $\mathcal{D}_{1,2,3,i_0}$ to be the set of $(N_1',N_2',N_3') \in \mathcal{C}_{1,2,3,i_0}$, such that all the lines in $\mathfrak{L}_{i_0,1,{d_{N_1',N_2',N_3'}}}$ and $\mathfrak{L}_{j,d_{N_1',N_2',N_3'}}'$ are contained in $Z_{d_{N_1',N_2',N_3'}}$. 

\textbf{Remark.} It holds that $(N_1,N_2,N_3) \in \mathcal{D}_{1,2,3,i_0}$. Therefore, we have already shown that, for all $(N_1,N_2,N_3) \in \mathcal{M} \setminus \mathcal{D}_{1,2,3,i_0}$,
\begin{displaymath}|J_{N_1,N_2,N_3}| \lesssim \frac{(L_1L_2L_3)^{1/2}}{(N_1N_2N_3)^{1/2+1/10^4}}.
\end{displaymath}
$\;\;\;\;\;\;\;\;\;\;\;\;\;\;\;\;\;\;\;\;\;\;\;\;\;\;\;\;\;\;\;\;\;\;\;\;\;\;\;\;\;\;\;\;\;\;\;\;\;\;\;\;\;\;\;\;\;\;\;\;\;\;\;\;\;\;\;\;\;\;\;\;\;\;\;\;\;\;\;\;\;\;\;\;\;\;\;\;\;\;\;\;\;\;\;\;\;\;\;\;\;\;\;\;\;\;\;\;\;\;\;\;\;\;\;\;\;\;\;\;\;\;\;\;\;\;\;\;\;\;\;\;\;\blacksquare$

As we have already explained in the case of the triple $(N_1,N_2,N_3)$, if $(N_1',N_2',N_3') \in \mathcal{D}_{1,2,3,i_0}$ and $x$ is a regular point of $Z_{flat, d_{N_1',N_2',N_3'}}$ lying in $\Pi_{d_{N_1',N_2',N_3'}}$, then, since all the lines in $\mathfrak{L}_{{i_0,1},{d_{N_1',N_2',N_3'}}}$ and $\mathfrak{L}_{j,d_{N_1',N_2',N_3'}}'$ are lying in $Z_{d_{N_1',N_2',N_3'}}$, it follows that all the lines in $\mathfrak{L}_{i_0,flat,2,{d_{N_1',N_2',N_3'}}}$ and $\mathfrak{L}'_{j,flat,d_{N_1',N_2',N_3'}}$ passing through $x$ lie on the (unique) plane in $\Pi_{d_{N_1',N_2',N_3'}}$ that $x$ lies on.

Finally, for $(N_1',N_2',N_3') \in \mathcal{D}_{1,2,3,i_0}$, we define $\mathcal{L}_{d_{N_1',N_2',N_3'}}$ as the set consisting of those lines in $(\mathfrak{L}_{i_0,flat,2} \cup \mathfrak{L}_{j,flat}')\cap \bigg(\mathfrak{L}_{i_0,flat,2,d_{N_1',N_2',N_3'}} \cup \mathfrak{L}_{j,flat, d_{N_1',N_2',N_3'}}'\bigg)$, each of which is equal to the intersection of a plane in $\Pi_{d_{N_1,N_2,N_3}}$ with a plane in $\Pi_{d_{N_1',N_2',N_3'}}$.

\textbf{Subcase D.2.2.i:} Suppose that $\gtrsim S$ points of $\mathfrak{G}_{flat}\;\Big(=\mathfrak{G}_{flat,d_{N_1,N_2,N_3}}\Big)$ lie in the lines that belong to the union of the sets $\mathcal{L}_{d_{N_1',N_2',N_3'}}$ over all $(N_1',N_2',N_3') \in \mathcal{D}_{1,2,3,i_0}$ that are different from $(N_1,N_2,N_3)$.

Then, there exists a triple $(N_1',N_2',N_3')$ in $\mathcal{D}_{1,2,3,i_0}$, different from $(N_1,N_2,N_3)$, such that $\gtrsim \frac{S}{(N_1'N_2'N_3')^{1/1000}}$ points of $\mathfrak{G}_{flat}$ lie in the lines of $\mathcal{L}_{d_{N_1',N_2',N_3'}}$. Fix this particular triple $(N_1',N_2',N_3')$, and  let $\mathfrak{G}_{flat}'$ be the set of points of $\mathfrak{G}_{flat}$ lying in the lines of $\mathcal{L}_{d_{N_1',N_2',N_3'}}$.

$\bullet$ Suppose that, for $\gtrsim \frac{S}{(N_1'N_2'N_3')^{1/1000}}$ of the points of $\mathfrak{G}_{flat}'$, each of them has the property that either all the lines of $\mathfrak{L}_{i_0,1}$ passing through it do not lie in $Z_{d_{N_1',N_2',N_3'}}$ or all the lines of $\mathfrak{L}_{j}'$ passing through it do not lie in $Z_{d_{N_1',N_2',N_3'}}$. 

Then, since $\gtrsim N_{i_0}$ lines of $\mathfrak{L}_{i_0,1}$ and $\gtrsim N_{j}$ lines of $\mathfrak{L}_j'$ are passing through each point of $\mathfrak{G}_{flat}'$, it follows that
\begin{displaymath} \frac{S}{(N_1'N_2'N_3')^{1/1000}} N_{i_0} \lesssim L_{i_0} d_{N_1',N_2',N_3'} \end{displaymath} 
or
\begin{displaymath} \frac{S}{(N_1'N_2'N_3')^{1/1000}} N_{j} \lesssim L_{j} d_{N_1',N_2',N_3'}. \end{displaymath}
Now, since $(N_1',N_2',N_3')$ belongs to $\mathcal{D}_{1,2,3,i_0}$, and thus to $\mathcal{C}_{1,2,3,i_0}$, we have that \begin{displaymath} \frac{L_{i_0}^2}{N_{i_0}'^{8/3}} = \min\Bigg\{\frac{L_1^2}{N_1'^{8/3}},\frac{L_2^2}{N_2'^{8/3}},\frac{L_3^2}{N_3'^{8/3}}\Bigg\}, \end{displaymath} and therefore
\begin{displaymath}d_{N_1',N_2',N_3'}=\frac{L_{i_0}^2}{|J_{N_1',N_2',N_3'}|N_{i_0}'^{3}},\end{displaymath}
while also
\begin{displaymath}|J_{N_1',N_2',N_3'}|\gtrsim L_{i_0}^{3/2},
\end{displaymath}
which give
\begin{displaymath}d_{N_1',N_2',N_3'} \lesssim \frac{L_{i_0}^{1/2}}{N_{i_0}'^3} \lesssim L_{i_0}^{1/2}.
\end{displaymath}
It thus follows that
\begin{displaymath} \frac{S}{(N_1'N_2'N_3')^{1/1000}} \lesssim L_{i_0} d_{N_1',N_2',N_3'} + L_{j} d_{N_1',N_2',N_3'} \lesssim  \end{displaymath}
\begin{displaymath} \lesssim  L_{i_0}^{3/2}+ L_jL_{i_0}^{1/2}.
\end{displaymath}
Our aim is to show that
\begin{displaymath} S \lesssim \frac{(L_1L_2L_3)^{1/2}}{(N_1N_2N_3)^{1/2+1/10^4}}, \end{displaymath}
and thus we will bound both $L_{i_0}^{3/2}$ and $ L_jL_{i_0}^{1/2}$ from above, up to multiplication by constants, by the quantity $\frac{(L_1L_2L_3)^{1/2}}{(N_1N_2N_3)^{1/2+1/10^4}}\cdot \frac{1}{(N_1'N_2'N_3')^{1/1000}}$.

More particularly, we will bound the quantity $L_{i_0}^{3/2}$ from above using that, since both $(N_1',N_2',N_3')$ and $(N_1,N_2,N_3)$ belong to $\mathcal{C}_{1,2,3,i_0}$,
\begin{displaymath}N_1' \lesssim N_3'^{1/1000}\text{, }N_2' \lesssim N_3'^{1/1000}
\end{displaymath}
and
\begin{displaymath}N_1 \lesssim N_3^{1/1000}\text{, }N_2\lesssim N_3^{1/1000}.\end{displaymath}
Moreover, we will bound the quantity $L_jL_{i_0}^{1/2}$ from above in a similar way, but also using that, again due the fact that $(N_1',N_2',N_3')$ and $(N_1,N_2,N_3)$ belong to $\mathcal{C}_{1,2,3,i_0}$,
\begin{displaymath}
 \frac{L_j^2}{N_j'^{8/3}} \lesssim \frac{L_3^2}{N_3'^{7/3}}\end{displaymath}
and
\begin{displaymath}\frac{L_j^2}{N_j^{8/3}} \lesssim \frac{L_3^2}{N_3^{7/3}}.
\end{displaymath}

Indeed,
\begin{displaymath} L_{i_0}^{3/2} \sim L_{i_0}^{3/2(1-1/100)}L_{i_0}^{(3/2)(1/100)} \sim\end{displaymath}
\begin{displaymath} \sim \Bigg(\frac{L_{i_0}^2}{N_{i_0}^{8/3}}\Bigg)^{3/4(1-1/100)}\Bigg(\frac{L_{i_0}^2}{N_{i_0}'^{8/3}}\Bigg)^{(3/4)(1/100)} N_{i_0}^{(8/3)(3/4)(1-1/100)}N_{i_0}'^{(8/3)(3/4)(1/100)} \lesssim
\end{displaymath}

\begin{displaymath}
\lesssim
\Bigg(\frac{L_{1}^2}{N_{1}^{8/3}}\Bigg)^{(1/3)(3/4)(1-1/100)}\Bigg(\frac{L_{2}^2}{N_{2}^{8/3}}\Bigg)^{(1/3)(3/4)(1-1/100)}\Bigg(\frac{L_{3}^2}{N_{3}^{8/3}}\Bigg)^{(1/3)(3/4)(1-1/100)}\cdot
\end{displaymath}
\begin{displaymath}
\cdot \Bigg(\frac{L_{1}^2}{N_{1}'^{8/3}}\Bigg)^{(1/3)(3/4)(1/100)}\Bigg(\frac{L_{2}^2}{N_{2}'^{8/3}}\Bigg)^{(1/3)(3/4)(1/100)}\Bigg(\frac{L_{3}^2}{N_{3}'^{8/3}}\Bigg)^{(1/3)(3/4)(1/100)}\cdot
\end{displaymath}
\begin{displaymath}
\cdot N_{i_0}^{(8/3)(3/4)(1-1/100)}N_{i_0}'^{(8/3)(3/4)(1/100)}
\sim
\end{displaymath}
\begin{displaymath}
\sim \frac{(L_1L_2L_3)^{1/2}}{(N_1N_2N_3)^{2/3(1-1/100)}} \cdot \frac{1}{(N_1'N_2'N_3')^{(2/3)(1/100)}}N_{i_0}^{2(1-1/100)}N_{i_0}'^{(1/50)}
\lesssim
\end{displaymath}
\begin{displaymath}
\lesssim \frac{(L_1L_2L_3)^{1/2}}{(N_1N_2N_3)^{2/3(1-1/100)}} \cdot \frac{1}{(N_1'N_2'N_3')^{(2/3)(1/100)}}N_{3}^{(1/1000)2(1-1/100)}N_{3}'^{(1/1000)(1/50)} \lesssim
\end{displaymath}
\begin{displaymath}
\lesssim \frac{(L_1L_2L_3)^{1/2}}{(N_1N_2N_3)^{1/2+1/10^4}} \cdot \frac{1}{(N_1'N_2'N_3')^{1/1000}}.
\end{displaymath}

In addition,
\begin{displaymath}
 L_{i_0}^{1/2}L_j\sim L_{i_0}^{1/2} L_j^{1-1/100}L_j^{1/100} \sim
\end{displaymath}
\begin{displaymath}
\sim L_{i_0}^{1/2} \Bigg(\frac{L_j^2}{N_j^{8/3}}\Bigg)^{1/2(1-1/100)}\Bigg(\frac{L_j^2}{N_j'^{8/3}}\Bigg)^{(1/2)(1/100)}N_j^{(8/3)(1/2)(1-1/100)}N_j'^{(8/3)(1/2)(1/100)}
\sim
\end{displaymath}
\begin{displaymath}
\sim
L_{i_0}^{1/2} \Bigg(\frac{L_j^2}{N_j^{8/3}}\Bigg)^{(1/4)(1-1/100)}\Bigg(\frac{L_j^2}{N_j^{8/3}}\Bigg)^{(1/4)(1-1/100)}\cdot
\end{displaymath}
\begin{displaymath}
\cdot \Bigg(\frac{L_j^2}{N_j'^{8/3}}\Bigg)^{(1/4)(1/100)}\Bigg(\frac{L_j^2}{N_j'^{8/3}}\Bigg)^{(1/4)(1/100)}N_j^{(4/3)(1-1/100)}N_j'^{(4/3)(1/100)}\lesssim
\end{displaymath}
\begin{displaymath}
\lesssim
L_{i_0}^{1/2} \Bigg(\frac{L_j^2}{N_j^{8/3}}\Bigg)^{(1/4)(1-1/100)}\Bigg(\frac{L_3^2}{N_3^{7/3}}\Bigg)^{(1/4)(1-1/100)}\cdot
\end{displaymath}
\begin{displaymath}
\cdot \Bigg(\frac{L_j^2}{N_j'^{8/3}}\Bigg)^{(1/4)(1/100)}\Bigg(\frac{L_3^2}{N_3'^{7/3}}\Bigg)^{(1/4)(1/100)}N_3^{(1/1000)(4/3)(1-1/100)}N_3'^{(1/1000)(4/3)(1/100)}\sim
\end{displaymath}
\begin{displaymath}
\sim \frac{(L_1L_2L_3)^{1/2}}{N_j^{(2/3)(1-1/100)}N_3^{(7/12)(1-1/100)}} \frac{1}{N_j'^{(2/3)(1/100)}N_3'^{(7/12)(1/100)}}\cdot
\end{displaymath}
\begin{displaymath}
\cdot N_3^{(1/1000)(4/3)(1-1/100)}N_3'^{(1/1000)(4/3)(1/100)} \lesssim
\end{displaymath}
\begin{displaymath}
\lesssim \frac{(L_1L_2L_3)^{1/2}}{(N_1N_2N_3)^{1/2+1/10^4}} \cdot \frac{1}{(N_1'N_2'N_3')^{1/1000}}.
\end{displaymath}

Therefore,
\begin{displaymath}\frac{S}{(N_1'N_2'N_3')^{1/1000}} \lesssim \frac{(L_1L_2L_3)^{1/2}}{(N_1N_2N_3)^{1/2+1/10^4}}\cdot \frac{1}{(N_1'N_2'N_3')^{1/1000}},\end{displaymath}
from which it follows that
\begin{displaymath} S \lesssim \frac{(L_1L_2L_3)^{1/2}}{(N_1N_2N_3)^{1/2+1/10^4}}.\end{displaymath}

$\bullet$ Suppose that, for $\gtrsim \frac{S}{(N_1'N_2'N_3')^{1/1000}}$ of the points of $\mathfrak{G}_{flat}'$, at least one of the lines of $\mathfrak{L}_{i_0,1}$ and at least one of the lines of $\mathfrak{L}_{j}'$ through each lies in $Z_{d_{N_1',N_2',N_3'}}$. However, these points lie in lines of the set $\mathcal{L}_{d_{N_1',N_2',N_3'}}$, which is a subset of $\mathfrak{L}_{i_0} \cup \mathfrak{L}_j$. Therefore, there exists a line $l$ in $\mathcal{L}_{d_{N_1',N_2',N_3'}}$, containing $\gtrsim \frac{S}{(N_1'N_2'N_3')^{1/1000}|\mathfrak{L}_{i_0}\cup \mathfrak{L}_j|}\gtrsim \frac{S}{(N_1'N_2'N_3')^{1/1000}\max\{L_{i_0},L_j\}}$ of these points. 

This implies that the lines of $\mathfrak{L}_{i_0,1} \cup \mathfrak{L}_{j}'$ lying in $Z_{d_{N_1',N_2',N_3'}}$ and passing through these points are $\gtrsim \frac{S}{(N_1'N_2'N_3')^{1/1000}max\{L_{i_0},L_j\}}$. Indeed, if $l \in \mathfrak{L}_j'$, then there exists a different line of $\mathfrak{L}_{i_0,1}$ passing through each of the $\gtrsim \frac{S}{(N_1'N_2'N_3')^{1/1000}\max\{L_{i_0},L_j\}}$ points of $\mathfrak{G}_{flat}$ in question. If $l \in \mathfrak{L}_{i_0,1}$, then there exists a different line of $\mathfrak{L}_{j}'$ passing through each of the $\gtrsim \frac{S}{(N_1'N_2'N_3')^{1/1000}\max\{L_{i_0},L_j\}}$ points of $\mathfrak{G}_{flat}$ in question.

On the other hand, the lines of $\mathfrak{L}_{i_0,1} \cup \mathfrak{L}_{j}'$ lying in $Z_{d_{N_1',N_2',N_3'}}$ and passing through the points of $l \cap \mathfrak{G}_{flat}$ are $\leq d_{N_1',N_2',N_3'}$.

Indeed, since $l$ belongs to $\mathcal{L}_{d_{N_1',N_2',N_3'}}$, it follows that $l$ is a flat line of both $Z_{d_{N_1',N_2',N_3'}}$ and $Z_{d_{N_1,N_2,N_3}}$, equal to the intersection of a plane $\Pi_1$ in $\Pi_{d_{N_1',N_2',N_3'}}$ with a plane $\Pi_2$ in $\Pi_{d_{N_1,N_2,N_3}}$, which are such that $\Pi_1$ contains all the lines of $\mathfrak{L}_{i_0,1,d_{N_1',N_2',N_3'}}$ and $\mathfrak{L}_{j,d_{N_1',N_2',N_3'}}'$ passing through the regular points of $Z_{d_{N_1',N_2',N_3'}}$ in $l$, while $\Pi_2$ contains all the lines of $\mathfrak{L}_{i_0,1}$ and $\mathfrak{L}_{j}'$ passing through the regular points of $Z_{d_{N_1,N_2,N_3}}$ in $l$. Therefore, the number of lines of $\mathfrak{L}_{i_0,1} \cup \mathfrak{L}_{j}'$ lying in $Z_{d_{N_1',N_2',N_3'}}$ and passing through the points of $l \cap \mathfrak{G}_{flat}$ is equal to at most the number of lines of $\mathfrak{L}_{i_0,1} \cup \mathfrak{L}_{j}'$ lying in $\Pi_2$. Moreover, since $l$ is a flat line of $Z_{d_{N_1',N_2',N_3'}}$, there exists a point of $l$ that is a regular point of $Z_{d_{N_1',N_2',N_3'}}$, and since that point already lies in the plane $\Pi_1 \subset Z_{d_{N_1',N_2',N_3'}}$, the plane $\Pi_2$ does not lie in $Z_{d_{N_1',N_2',N_3'}}$ (otherwise the point would not be a regular point of $Z_{d_{N_1',N_2',N_3'}}$). Consequently, the number of lines of $\mathfrak{L}_{i_0,1}\cup \mathfrak{L}_{j}'$ lying in $\Pi_2\cap Z_{d_{N_1',N_2',N_3'}}$ is equal to at most $d_{N_1',N_2',N_3'}$, and therefore so is the number of lines of $\mathfrak{L}_{i_0,1} \cup \mathfrak{L}_{j}'$ lying in $Z_{d_{N_1',N_2',N_3'}}$ and passing through the points of $l \cap \mathfrak{G}_{flat}$.

It follows from the above that
\begin{displaymath} \frac{S}{(N_1'N_2'N_3')^{1/1000}\max\{L_{i_0},L_j\}} \lesssim d_{N_1',N_2',N_3'}, \end{displaymath}
which implies that\begin{displaymath}\frac{S}{(N_1'N_2'N_3')^{1/1000}} \lesssim L_{i_0}d_{N_1',N_2',N_3'}+L_jd_{N_1',N_2',N_3'}.\end{displaymath}
Now, following the same procedure as in the case above, we obtain
\begin{displaymath}S \lesssim \frac{(L_1L_2L_3)^{1/2}}{(N_1N_2N_3)^{1/2+1/10^4}}. \end{displaymath}

\textbf{Case D.2.2.ii:} Suppose that $\big|\mathfrak{G}''_{flat,d_{N_1,N_2,N_3}}\big|\gtrsim |J_{N_1,N_2,N_3}|$, where, for all $(N_1',N_2',N_3') \in \mathcal{D}_{1,2,3,i_0}$, $\mathfrak{G}''_{flat, d_{N_1',N_2',N_3'}}$ is the set of points in $\mathfrak{G}_{flat, d_{N_1',N_2',N_3'}}$ that lie outside the lines in the union of $\mathcal{L}_{d_{N_1'',N_2'',N_3''}}$, over all $(N_1'',N_2'',N_3'') \in \mathcal{D}_{1,2,3,i_0}$ that are different from $(N_1',N_2',N_3')$ (this is the final case).

In fact, let $\mathcal{E}_{1,2,3,i_0}$ be the set of all $(N_1',N_2',N_3') \in \mathcal{D}_{1,2,3,i_0}$ with the property that $\big|\mathfrak{G}''_{flat,d_{N_1',N_2',N_3'}}\big|\gtrsim |J_{N_1',N_2',N_3'}|$; in particular, $(N_1,N_2,N_3) \in \mathcal{E}_{1,2,3,i_0}$.

In this final case, we are not able to bound the quantity $|J_{N_1,N_2,N_3}|$ from above, up to multiplication by constants, by a quantity of the form $\frac{(L_1L_2L_3)^{1/2}}{(N_1N_2N_3)^{1/2+\epsilon}}$, for some absolute constant $\epsilon >0$, independent of $\mathfrak{L}_1$, $\mathfrak{L}_2$, $\mathfrak{L}_3$, $N_1$, $N_2$ and $N_3$. However, we will show that
\begin{equation} \label{eq:finally...} \sum_{(N_1,N_2,N_3) \in \mathcal{E}_{1,2,3,i_0}}|J_{N_1,N_2,N_3}|(N_1N_2N_3)^{1/2}\lesssim (L_1L_2L_3)^{1/2}.
\end{equation}
The proof of \eqref{eq:finally...} will complete the proof of Theorem \ref{theoremmult1}, since we have already shown that
\begin{displaymath}|J_{N_1',N_2',N_3'}|\lesssim \frac{(L_1L_2L_3)^{1/2}}{(N_1'N_2'N_3')^{1/2+1/10^4}},
\end{displaymath}
for all $(N_1',N_2',N_3') \in \mathcal{M}\setminus \mathcal{E}_{1,2,3,i_0}$.

Let us now prove \eqref{eq:finally...}. 

Fix $(N_1',N_2',N_3') \in \mathcal{E}_{1,2,3,i_0}$. Let $\mathfrak{L}'_{i_0,flat,2,d_{N_1',N_2',N_3'}}$ be the set of lines of $\mathfrak{L}_{i_0,1,d_{N_1',N_2',N_3'}}$ each containing a point of $\mathfrak{G}''_{flat, d_{N'_1,N_2',N_3'}}$, and $\mathfrak{L}''_{j,flat,d_{N_1',N_2',N_3'}}$ the set of lines of $\mathfrak{L}_{j,d_{N'_1,N_2',N_3'}}'$ each containing a point of $\mathfrak{G}''_{flat,d_{N_1',N_2',N_3'}}$.

Therefore, each point of $\mathfrak{G}''_{flat, d_{N'_1,N_2',N_3'}}$ is a mutlijoint for the collections $\mathfrak{L}'_{i_0,flat,2,d_{N_1',N_2',N_3'}}$, $\mathfrak{L}''_{j,flat,d_{N_1',N_2',N_3'}}$ and $\mathfrak{L}_3$ of lines, each lying in $\sim N_{i_0}'$ lines in $\mathfrak{L}'_{i_0,flat,2,d_{N_1',N_2',N_3'}}$, $\sim N_j'$ lines of $\mathfrak{L}''_{j,flat,d_{N_1',N_2',N_3'}}$ and $\sim N_3'$ lines of $\mathfrak{L}_3$.

Consequently, since it also holds that $\big|\mathfrak{G}''_{flat,d_{N_1',N_2',N_3'}}\big|\gtrsim |J_{N_1',N_2',N_3'}|$, in order to prove \eqref{eq:finally...} it suffices to show that
\begin{equation}\label{eq:lastone?}\sum_{x \in \cup_{(N_1'',N_2'',N_3'') \in \mathcal{E}_{1,2,3,i_0}}\mathfrak{G}''_{flat,d_{N_1'',N_2'',N_3''}}}(M_{i_0}(x)M_j(x)M_3(x))^{1/2}\lesssim (L_1L_2 L_3)^{1/2},
\end{equation}
where, for all $(N_1'',N_2'',N_3'') \in \mathcal{E}_{1,2,3,i_0}$ and $x \in \mathfrak{G}''_{flat,d_{N_1'',N_2'',N_3''}}$, $M_{i_0}(x)$ is the number of lines of $\mathfrak{L}'_{i_0,flat,2,d_{N_1'',N_2'',N_3''}}$ passing through $x$, $M_{j}(x)$ is the number of lines of $\mathfrak{L}''_{j,flat,d_{N_1'',N_2'',N_3'}}$ passing through $x$, and $M_3(x)$ is the number of lines of $\mathfrak{L}_3$ passing through $x$.

Now, let us again fix $(N_1',N_2',N_3') \in \mathcal{E}_{1,2,3,i_0}$. We know that, if $x \in \mathfrak{G}''_{flat,d_{N_1',N_2',N_3'}}$, then the lines of $\mathfrak{L}'_{i_0,flat,2,d_{N_1',N_2',N_3'}} \cup \mathfrak{L}''_{j,flat,d_{N_1',N_2',N_3'}}$ through each point of $\mathfrak{G}''_{flat,d_{N_1',N_2',N_3'}}$ have to lie in the unique plane of $Z_{d_{N_1',N_2',N_3'}}$ containing $x$, as $x$ is flat point of $Z_{d_{N_1',N_2',N_3'}}$ and the lines of $\mathfrak{L}_{i_0,flat,2,d_{N_1',N_2',N_3'}} \cup \mathfrak{L}''_{j,flat,d_{N_1',N_2',N_3'}}$ ($\subseteq \mathfrak{L}_{i_0,1}\cup \mathfrak{L}_{j}')$ all lie in $Z_{d_{N_1',N_2',N_3'}}$.

Let us emphasise that, if a line of $\mathfrak{L}'_{i_0,flat,2,d_{N_1',N_2',N_3'}} \cup \mathfrak{L}''_{j,flat,d_{N_1',N_2',N_3'}}$ belongs to one of the planes of $Z_{d_{N_1',N_2',N_3'}}$, then it does not belong to any other plane of $Z_{d_{N_1',N_2',N_3'}}$; the reason is that the line contains at least one point of $Z_{d_{N_1',N_2',N_3'}}$ that is flat, and therefore cannot belong to two intersecting planes of $Z_{d_{N_1',N_2',N_3'}}$.

On the other hand, $\mathfrak{G}''_{flat, d_{N_1',N_2',N_3'}}$ is, by definition, the set of points in $\mathfrak{G}_{flat, d_{N_1',N_2',N_3'}}$ that lie outside the lines that belong to the union of the sets $\mathcal{L}_{d_{N_1'',N_2'',N_3''}}$, over all $(N_1'',N_2'',N_3'') \in \mathcal{D}_{1,2,3,i_0}$ that are different from $(N_1',N_2',N_3')$. Therefore, no line in $\mathfrak{L}_{i_0,flat,2,d_{N_1',N_2',N_3'}} \cup \mathfrak{L}''_{j,flat,d_{N_1',N_2',N_3'}}$ lies in the union of the sets $\mathcal{L}_{d_{N_1'',N_2'',N_3''}}$, over all $(N_1'',N_2'',N_3'') \in \mathcal{D}_{1,2,3,i_0}$ that are different from $(N_1',N_2',N_3')$. This means that, if $l \in \mathfrak{L}_{i_0,flat,2,d_{N_1',N_2',N_3'}} \cup \mathfrak{L}''_{j,flat,d_{N_1',N_2',N_3'}}$, then $l \notin \mathfrak{L}_{i_0,flat,2,d_{N_1'',N_2'',N_3''}} \cup \mathfrak{L}''_{j,flat,d_{N_1'',N_2'',N_3''}}$, for all $(N_1'',N_2'',N_3'') \in \mathcal{E}_{1,2,3,i_0}$ that are different from $(N_1',N_2',N_3')$.

From the above, it is clear that, in order to show \eqref{eq:lastone?} and, consequently, complete the proof of Theorem \ref{theoremmult1}, it suffices to prove the following.

\begin{lemma} Suppose that $\Pi_1$, $\Pi_2$,..., $\Pi_{\rho}$ are planes in $\R^3$, while $\mathfrak{L}_1$, $\mathfrak{L}_2$ and $\mathfrak{L}_3$ are finite collections of $L_1$, $L_2$ and $L_3$, respectively, lines in $\R^3$, such that each line in $\mathfrak{L}_1\cup \mathfrak{L}_2$ lies in $\Pi_1\cup\Pi_2\cup...\cup\Pi_{\rho}$. Moreover, suppose that, if, for each $i \in \{1,...,\rho\}$, $l_{1,i}$ is the set of lines in $\mathfrak{L}_1$ that lie on $\Pi_{i}$, while $l_{2,i}$ is the set of lines in $\mathfrak{L}_2$ that lie on $\Pi_{i}$, then $l_{1,i} \cap l_{1,j} = \emptyset$ and $l_{2,i} \cap l_{2,j} = \emptyset$, for all $i \neq j$, $i,j\in \{1,...,\rho\}$. Let $J$ be the set of multijoints formed by the collections $\mathfrak{L}_1$, $\mathfrak{L}_2$ and $\mathfrak{L}_3$ of lines. Then,
\begin{equation} \label{eq:lemma} \sum_{x \in J}(N_1(x)N_2(x)N_3(x))^{1/2} \lesssim (L_1L_2L_3)^{1/2}. \end{equation}
\end{lemma}

\begin{proof} This statement is true if the collection of planes consists of only one plane $\Pi$. Indeed, in this particular case, $\sum_{x \in J}N_3(x)\lesssim L_3$, since no line of $\mathfrak{L}_3$ lies on the plane, and we can thus deduce \eqref{eq:lemma} by Claim \ref{p3}.

Now, using this fact, \eqref{eq:lemma} follows in the general case as such:
\begin{displaymath} \sum_{x \in J}(N_1(x)N_2(x)N_3(x))^{1/2}= \end{displaymath}
\begin{displaymath} =\sum_{i=1}^{\rho} \sum_{x \in \Pi_{i}\cap J}(N_1(x)N_2(x)N_3(x))^{1/2} \lesssim \sum_{i=1}^{\rho} (|l_{1,i}||l_{2,i}|L_3)^{1/2} \sim \end{displaymath} 
\begin{displaymath} \sim \Bigg(\sum_{i=1}^{\rho} (|l_{1,i}|^{1/2}|l_{2,i}|^{1/2}\Bigg) L_3^{1/2} \lesssim  \Bigg(\sum_{i=1}^{\rho} |l_{1,i}|\Bigg)^{1/2}\Bigg(\sum_{i=1}^{\rho} |l_{2,i}|\Bigg)^{1/2} L_3^{1/2} \lesssim \end{displaymath}
\begin{displaymath} \lesssim (L_1L_2L_3)^{1/2}. \end{displaymath}
Note that the first inequality is due to the fact that \eqref{eq:lemma} holds on each plane $\Pi_i$, $i=1,...,\rho$, while the second inequality is the Cauchy-Schwarz inequality.

Thus, the proof of Theorem \ref{theoremmult1} is complete.

\end{proof}

\newpage
\thispagestyle{plain}
\cleardoublepage

\chapter{From lines to curves} \label{5}

In Chapter \ref{6} we will extend the definition of a joint from a point of intersection of lines with particular transversality properties to a point of intersection of more general curves with particular transversality properties, and we will, in fact, extend the statement of Theorem \ref{1.1} to that case, under certain assumptions on the properties of the curves. To that end, we need to extend certain computational results for sets of lines to sets of more general curves, by recalling and further analysing some facts from algebraic geometry, and this is what this chapter is dedicated to.

If $\mathbb{K}$ is a field, then any set of the form \begin{displaymath} \{x \in \mathbb{K}^n: p_i(x)=0, \; \forall \; i=1,...,k\}, \end{displaymath} where $k \in \N$ and $p_i \in \mathbb{K}[x_1,...,x_n]$ for all $i =1,...,k$, is called an \textit{algebraic set} or an \textit{affine variety} or simply a \textit{variety} in $\mathbb{K}^n$, and is denoted by $V(p_1,...,p_k)$. A variety $V$ in $\mathbb{K}^n$ is \textit{irreducible} if it cannot be expressed as the union of two non-empty varieties in $\mathbb{K}^n$ which are strict subsets of $V$.

Now, if $V$ is a variety in $\mathbb{K}^n$, the set \begin{displaymath} I(V):= \{ p \in \mathbb{K}[x_1,...,x_n]: p(x)=0, \; \forall \; x \in V\} \end{displaymath}
is an ideal in $\mathbb{K}[x_1,...,x_n]$. If, in particular, $V$ is irreducible, then $I(V)$ is a prime ideal of $\mathbb{K}[x_1,...,x_n]$, and the transcendence degree of the ring $\mathbb{K}[x_1,...,x_n]/I(V)$ over $\mathbb{K}$ is the \textit{dimension} of the irreducible variety $V$. The \textit{dimension of an algebraic set} is the maximal dimension of all the irreducible varieties contained in the set. If an algebraic set has dimension 1 it is called an \textit{algebraic curve}, while if it has dimension $n-1$ it is called an \textit{algebraic hypersurface}.

Now, if $\gamma$ is an algebraic curve in $\C^n$, a generic hyperplane of $\C^n$ intersects the curve in a specific number of points (counted with appropriate multiplicities), which is called the \textit{degree} of the curve.

A consequence of B\'ezout's theorem (see, for example, \cite[Theorem 12.3]{MR732620} or \cite[Chapter 3, \S3]{MR2122859}) is the following.

\begin{theorem}\emph{\textbf{(B\'ezout)}}\label{4.1.2} Let $\gamma$ be an irreducible algebraic curve in $\C^n$ of degree $b$, and $p \in \C[x_1,...,x_n]$. If $\gamma$ is not contained in the zero set of $p$, it intersects the zero set of $p$ at most $b \cdot \deg p$ times. \end{theorem}

Now, if $\mathbb{K}$ is a field, an order $\prec$ on the set of monomials in $\mathbb{K}[x_1,...,x_n]$ is called a \textit{term order}, if it is a total order on the monomials of $\mathbb{K}[x_1,...,x_n]$, such that it is multiplicative (i.e. it is preserved by multiplication by the same monomial) and the constant monomial is the $\prec$-smallest monomial. Then, if $I$ is an ideal in $\mathbb{K}[x_1,...,x_n]$, we define $in_{\prec}(I)$ as the ideal of $\mathbb{K}[x_1,...,x_n]$ generated by the $\prec$-initial terms, i.e. the $\prec$-largest monomial terms, of all the polynomials in $I$. 

Let $V$ be a variety in $\mathbb{K}[x_1,...,x_n]$ and $\prec$ a term order on the set of monomials in $\mathbb{K}[x_1,...,x_n]$. Also, let $S$ be a maximal subset of the set of variables $\{x_1,...,x_n\}$, with the property that no monomial in the variables in $S$ belongs to $in_{\prec}(I(V))$. Then, it holds that the dimension of $V$ is the cardinality of $S$ (see \cite{Sturmfels_2005}). From this fact, we deduce the following.

\begin{lemma}\label{extr}An irreducible real algebraic curve $\gamma$ in $\R^n$ is contained in an irreducible complex algebraic curve in $\C^n$. \end{lemma}

\begin{proof}  We clarify that, by saying that a real algebraic curve $\gamma_1$ in $\R^n$ is contained in a complex algebraic curve $\gamma_2$ in $\C^n$, we mean that, if $x \in \gamma_1$, then the point $x$, seen as an element of $\C^n$, belongs to $\gamma_2$ as well. 

Let $\gamma$ be an irreducible real algebraic curve in $\R^n$, and $\prec$ a term order on the set of monomials in the variables $x_1$, ..., $x_n$. From the discussion above, for every $i\neq j$, $i,j \in \{1,...,n\}$, there exists a monomial in the variables $x_i$ and $x_j$ in the ideal $in_{\prec}(I(\gamma))$. 

Now, the ideal $I(\gamma)$ is finitely generated, like any ideal of $\R[x_1,...,x_n]$. Let $\{p_1,...,p_k\}$ be a finite set of generators of $I(\gamma)$, and let $I':=(p_1,...,p_k)$ be the ideal in $\C[x_1,...,x_n]$ generated by the polynomials $p_1$, ..., $p_k$, this time seen as elements of $\C[x_1,...,x_n]$. We consider the complex variety $V'=V(p_1,...,p_k)$ and the ideal $I(V')$ in $\C[x_1,...,x_n]$. Since the polynomials in $I(\gamma)$, seen as elements of $\C[x_1,...,x_n]$, are elements of $I(V')$, it holds that for every $i\neq j$, $i,j \in \{1,...,n\}$, there exists a monomial in the variables $x_i$ and $x_j$ in the ideal $in_{\prec}(I(V'))$. Therefore, the variety $V'$ has dimension 1 (it cannot have dimension 0, as it is not a finite set of points). 

Therefore, $\gamma$ is contained in a complex algebraic curve. It is finally easy to see by B\'ezout's theorem that $\gamma$ is contained in an irreducible component of that curve.

\end{proof}

Now, by B\'ezout's theorem, we can deduce the following.

\begin{corollary}\label{4.1.4} Let $\gamma _1$, $\gamma_2$ be two distinct irreducible complex algebraic curves in $\C^n$. Then, they have at most $\deg \gamma_1\cdot \deg \gamma_2$ common points. \end{corollary}

\begin{proof} Since $\gamma_2$ is an algebraic curve in $\C^n$ and $\C$ is an algebraically closed field, it follows that $\gamma_2$ is the intersection of the zero sets of  $\lesssim_{n, \deg \gamma_2}1$ irreducible polynomials in $\C[x_1,...,x_n]$, each of which has degree at most $\deg \gamma_2$ (see \cite[Theorem A.3]{MR2827010}). The zero set of at least one of these polynomials does not contain $\gamma_1$, so, by Theorem \ref{4.1.2}, $\gamma_1$ intersects it at most $\deg \gamma_1\cdot \deg \gamma_2$ times. Therefore, $\gamma_1$ intersects $\gamma_2$, which is contained in the zero set of the above-mentioned polynomial, at most $\deg \gamma_1\cdot \deg \gamma_2$ times.

\end{proof}

Corollary \ref{4.1.4} easily implies the following.

\begin{lemma}\label{4.1.1} An irreducible real algebraic curve $\gamma$ in $\R^n$ is contained in a unique irreducible complex algebraic curve in $\C^n$.
\end{lemma}

\begin{proof} Let $\gamma$ be a real algebraic curve in $\R^n$. By Lemma \ref{extr}, $\gamma$ is contained in an irreducible complex algebraic curve in $\C^n$. Suppose that there exist two irreducible complex algebraic curves $\gamma_1$ and $\gamma_2$ in $\C^n$ containing $\gamma$. Then, $\gamma_1$ and $\gamma_2$ intersect at infinitely many points, and thus, by Corollary \ref{4.1.4}, they coincide.

\end{proof}

Note that, by the above, the smallest complex algebraic curve containing a real algebraic curve is the union of the irreducible complex algebraic curves, each of which contains an irreducible component of the real algebraic curve.

In particular, the following holds.

\begin{lemma} \label{newstuff} Any real algebraic curve in $\R^n$ is the intersection of $\R^n$ with the smallest complex algebraic curve containing it. 
\end{lemma}

\begin{proof} Let $\gamma$ be a real algebraic curve in $\R^n$ and $\gamma _{\C}$ the smallest complex algebraic curve containing it. We will show that $\gamma = \R^n \cap \gamma _{\C}$. 

Let $x \in \R^n$, such that $x \notin \gamma$; then, $x \notin \gamma_{\C}$. Indeed, $\gamma$ is the intersection of the zero sets, in $\R^n$, of some polynomials $p_1$, ..., $p_k \in \R[x_1,...,x_n]$. Since $x \notin \gamma$, it follows that $x$ does not belong to the zero set of $p_i$ in $\R^n$, for some $i \in \{1,...,k\}$. However, $x \in \R^n$, so it does not belong to the zero set of $p_i$ in $\C^n$, either. 

Now, the zero set of $p_i$ in $\C^n$ is a complex algebraic set containing $\gamma$, and therefore its intersection with $\gamma_{\C}$ is a complex algebraic set containing $\gamma$; in fact, it is a complex algebraic curve, since it contains the infinite set $\gamma$ and lies inside the complex algebraic curve $\gamma _{\C}$. Therefore, the intersection of the zero set of $p_i$ in $\C^n$ and $\gamma_{\C}$ is equal to $\gamma _{\C}$, as otherwise it would be a complex algebraic curve, smaller that $\gamma_{\C}$, containing $\gamma$. This means that $\gamma_{\C}$ is contained in the zero set of $p_i$ in $\C^n$, and since $x$ does not belong to the zero set of $p_i$ in $\C^n$, it does not belong to $\gamma _{\C}$ either.

Therefore, $\gamma = \R^n \cap \gamma _{\C}$.

\end{proof}

Now, even though a generic hyperplane of $\C^n$ intersects a complex algebraic curve in $\C^n$ in a fixed number of points, this is not true in general for real algebraic curves. However, by Lemma \ref{4.1.1}, we can define the \textit{degree of an irreducible real algebraic curve} in $\R^n$ as the degree of the (unique) irreducible complex algebraic curve in $\C^n$ containing it. Furthermore, we can define the \textit{degree of a real algebraic curve} in $\R^n$ as the degree of the smallest complex algebraic curve in $\C^n$ containing it. With this definition, and due to Lemma \ref{newstuff}, the degree of a real algebraic curve in $\R^n$ is equal to the sum of the degrees of its irreducible components (Lemma \ref{newstuff} ensures that distinct irreducible components of a real algebraic curve in $\R^n$ are contained in distinct irreducible complex algebraic curves in $\C^n$).

Therefore, if, by saying that a real algebraic curve $\gamma$ in $\R^n$ \textit{crosses itself at the point }$x_0 \in \gamma$, we mean that any neighbourhood of $x_0$ in $\gamma$ is homeomorphic to at least two intersecting lines, it follows that a real algebraic curve in $\R^n$ crosses itself at a point at most as many times as its degree.

An immediate consequence of the discussion above is the following.

\begin{corollary}\label{4.1.3} Let $\gamma$ be an irreducible real algebraic curve in $\R^n$ of degree $b$, and $p \in \R[x_1,...,x_n]$. If $\gamma$ is not contained in the zero set of $p$, it intersects the zero set of $p$ at most $b \cdot \deg p$ times. \end{corollary}

We now discuss projections of real algebraic curves. This leads us to the study of semi-algebraic sets.

More particularly, a \textit{basic real semi-algebraic set} in $\R^n$ is any set of the form \begin{displaymath} \{x\in \R^n: P(x)=0 \text{ and }  Q(x)>0, \; \forall \; Q\in \mathcal{Q} \}, \end{displaymath} where $P\in \R[x_1,...,x_n]$ and $\mathcal{Q}$ is a finite family of polynomials in $\R[x_1,...,x_n]$. A \textit{real semi-algebraic set} in $\R^n$ is defined as a finite union of basic real semi-algebraic sets. Note that a real algebraic set in $\R^n$ is, in fact, a basic real semi-algebraic set in $\R^n$, since it can be expressed as the zero set of a single real $n$-variate polynomial (a real algebraic set in $\R^n$ is the intersection of the zero sets, in $\R^n$, of some polynomials $p_1$, ..., $p_k \in \R[x_1,...,x_n]$, which is equal to the zero set, in $\R^n$, of the polynomial $p_1^2+...+p_k^2 \in \R[x_1,...,x_n]$). 

What holds is the following (see \cite[Chapter 2, \S3]{MR2248869} for a proof).

\begin{theorem}\label{4.1.5} The projection of a real algebraic set of $\R^n$ on any hyperplane of $\R^n$ is a real semi-algebraic set. \end{theorem}

We further notice that any set of the form $ \{x\in \R^{n}: Q(x)>0, \; \forall \; Q \in\mathcal{Q}\}$, where $\mathcal{Q}$ is a finite subset of $\R[x_1,...,x_n]$, is open in $\R^{n}$ (with the usual topology). Therefore, a basic real semi-algebraic set in $\R^n$ that is not open in $\R^n$ is of the form $ \{x\in \R^{n}: P(x)=0 \text{ and }  Q(x)>0, \; \forall \; Q\in \mathcal{Q} \}$, where $\mathcal{Q}$ is a finite subset of $\R[x_1,...,x_{n}]$ and $P\in \R[x_1,...,x_n]$ is a non-zero polynomial. Thus, each basic real semi-algebraic set in $\R^n$ that is not open in $\R^n$ (with the usual topology) is contained in a real algebraic set of dimension at most $n-1$.

Now, if $\gamma$ is a real algebraic curve in $\R^3$, its projection on a generic plane $H \simeq \R^2$ is a finite union of basic real semi-algebraic sets which are not open in $H$, so each of them is contained in some real algebraic set of dimension at most 1. However, the projection of a curve in $\R^3$ on a generic plane is not a finite set of points. Therefore, at least one of these basic real semi-algebraic sets is an infinite set of points, contained in some real algebraic curve in $H$. From this fact, as well as a closer study of the algorithm that constitutes the proof of Theorem \ref{4.1.5} as described in \cite[Chapter 2, \S3]{MR2248869}, we can finally see that the projection of $\gamma$ on a generic plane $H$ is the union of at most $B_{\deg \gamma}$ basic real semi-algebraic sets, each of which either consists of at most $B'_{\deg \gamma}$ points or is contained in a real algebraic curve in $H$ of degree at most $B'_{\deg \gamma}$, where $B_{\deg \gamma}$, $B'_{\deg \gamma}$ are integers depending only on the degree $\deg \gamma$ of $\gamma$. Therefore, the following is true.

\begin{lemma}\label{4.1.6} Let $\gamma$ be a real algebraic curve in $\R^3$. There exists an integer $C_{\deg \gamma} \geq \deg \gamma$, depending only on the degree $\deg \gamma$ of $\gamma$, such that the projection of $\gamma$ on a generic plane is contained in a planar real algebraic curve of degree at most $C_{\deg \gamma}$. \end{lemma}

Note that this means that the \textit{Zariski closure} of the projection of a real algebraic curve $\gamma$ of $\R^3$ on a generic plane, i.e. the smallest variety containing that projection, is, in fact, a planar real algebraic curve.

Our aim now is to find an upper bound on the number of times a planar real algebraic curve $\gamma$ crosses itself, and eventually establish an upper bound on the number of times a real algebraic curve in $\R^3$ crosses itself. To that end, we proceed to show that a planar real algebraic curve $\gamma$ is the zero set of a single, square-free bivariate real polynomial, of degree $\lesssim \deg \gamma$.

\begin{lemma} \label{newstuff2} Let $\gamma$ be an irreducible planar complex algebraic curve. Then, $\gamma$ is the zero set, in $\C^2$, of a single, irreducible polynomial $p \in \C[x,y]$, of degree $\leq \deg \gamma$.
\end{lemma}

\begin{proof} Let $\gamma$ be an irreducible planar complex algebraic curve. Then, $\gamma$ is the intersection of the zero sets, in $\C^2$, of some polynomials $p_1$, ..., $p_k \in \C[x,y]$, for $k\lesssim_{\deg \gamma} 1$, of degrees $\leq \deg \gamma$ (see \cite[Theorem A.3]{MR2827010}). 

Now, for all $i=1,...,k$, the zero set of $p_i$ in $\C^2$ contains $\gamma$, and is thus an algebraic set of dimension at least 1; in fact, equal to 1, as otherwise the zero set of $p_i$ would be the whole of $\C^2$ and $p_i$ would be the zero polynomial. Therefore, the zero set of $p_i$ in $\C^2$ is a planar complex algebraic curve containing $\gamma$, for all $i=1,...,k$. Consequently, $\gamma$ is contained, in particular, in the planar complex algebraic curve that is the zero set of $p_1$ in $\C^2$, and, since $\gamma$ is irreducible, it is equal to one of the irreducible components of the zero set of $p_1$, which is the zero set of an irreducible factor of $p_1$. 

Therefore, $\gamma$ is the zero set, in $\C^2$, of a single, irreducible polynomial $p\in \C[x,y]$, of degree $\leq \deg \gamma$.

\end{proof}

We can therefore easily deduce the following.

\begin{corollary} \label{newstuff3} Let $\gamma$ be an irreducible planar real algebraic curve. Then, $\gamma$ is the zero set, in $\R^2$, of a single, irreducible polynomial $p \in \R[x,y]$, of degree $\leq 2\deg \gamma$.
\end{corollary}

\begin{proof} Let $\gamma_{\C}$ be the (unique) irreducible planar complex algebraic curve containing $\gamma$.

Now, by Lemma \ref{newstuff2}, $\gamma_{\C}$ is the zero set, in $\C^2$, of a single, irreducible polynomial $p \in \C[x,y]$, of degree $\leq \deg \gamma_{\C}\;(=\deg \gamma)$. Thus, by Lemma \ref{newstuff}, $\gamma$ is the zero set of $p$ in $\R^2$, and, since the polynomials $p$ and $\bar{p}$ have the same zero set in $\R^2$, $\gamma$ is the zero set, in $\R^2$, of the polynomial $p\bar{p} \in \R[x,y]$, which is irreducible in $\R[x,y]$, since $p$ is irreducible in $\C[x,y]$. 

Therefore, the statement of the Lemma is proved.

\end{proof}

An immediate consequence of Corollary \ref{newstuff3} is the following.

\begin{corollary}\label{newstuff4} Let $\gamma$ be a planar real algebraic curve. Then, $\gamma$ is the zero set, in $\R^2$, of a single, square-free polynomial $p \in \R[x,y]$, of degree $\leq 2\deg \gamma$.

\end{corollary}

We can now bound from above the number of times a planar real algebraic curve crosses itself.

\begin{lemma} \label{newstuff5}Let $\gamma$ be a planar real algebraic curve. Then, $\gamma$ crosses itself at most $4 (\deg \gamma)^2$ times.
\end{lemma}

\begin{proof} By Corollary \ref{newstuff4}, $\gamma$ is the zero set, in $\R^2$, of a single, square-free polynomial $p \in \R[x,y]$, of degree $\leq 2\deg \gamma$. Since $p$ is square-free, $p$ and $\nabla{p}$ do not have a common factor, so, by B\'ezout's theorem, $p$ and $\nabla{p}$ have at most $(\deg p)^2 \leq (2\deg \gamma)^2$ common roots. 

Indeed, if $p=p_1\cdots p_k$, where $p_1$, ..., $p_k \in \R[x,y]$ are irreducible polynomials, then each common root of $p$ and $\nabla{p}$ is a common root of an irreducible factor $p_i$ of $p$, for some $i \in \{1,...,k\}$, and a polynomial $g_i \in \Big\{\frac{\partial p}{\partial x}, \frac{\partial p}{\partial y}\Big\}$, which does not have $p_i$ as a factor. Therefore, the number of common roots of $p$ and $\nabla{p}$ is equal to at most $\sum_{i=1,...,k}r_i$, where, for each $i\in \{1,...,k\}$, $r_i$ is the number of common roots of $p_i$ and $g_i$. However, for all $i\in\{1,...,k\}$, the polynomials $p_i$ and $g_i \in \R[x,y]$ do not have a common factor, and thus, by B\'ezout's theorem, $r_i \leq \deg p_i \cdot \deg g_i \leq \deg p_i \cdot d$. Therefore, the number of common roots of $p$ and $\nabla{p}$ is $\leq \sum_{i=1,...,k}\deg p_i \cdot d\leq d^2$.

But if $\gamma$ crosses itself at a point $x$, then $x$ is a common root of $p$ and $\nabla{p}$, because otherwise $\gamma$ would be a manifold locally around $x$. So, $\gamma$ crosses itself at most $4(\deg \gamma)^2$ times.

\end{proof}

Lemma \ref{newstuff5} immediately gives an upper bound on the number of times a real algebraic curve in $\R^3$ crosses itself.

\begin{lemma}\label{4.1.7} Let $\gamma$ be a real algebraic curve in $\R^3$. Then, $\gamma$ crosses itself at most $4(\deg \overline{\pi(\gamma)})^2$ times, where $\overline{\pi(\gamma)}$ the smallest planar real algebraic curve containing the projection $\pi(\gamma)$ of the curve $\gamma$ on a generic plane (i.e. the curve that constitutes the Zariski closure of $\pi(\gamma)$). \end{lemma}

\begin{proof} Obviously, $\gamma$ crosses itself at most as many times as $\overline{\pi(\gamma)}$ crosses itself, thus, by Lemma  \ref{newstuff5}, at most $4(\deg \overline{\pi(\gamma)})^2$ times.

\end{proof}

We are now ready to establish an analogue of the Szemer\'edi-Trotter theorem for real algebraic curves in $\R^3$. Indeed, the following is known.

\begin{theorem}\emph{\textbf{ (Kaplan, Matou\v{s}ek, Sharir, \cite[Theorem 4.1]{KMS})}}\label{4.1.8} Let $b$, $k$, $C$ be positive constants. Also, let $P$ be a finite set of points in $\R^2$ and $\Gamma$ a finite set of planar real algebraic curves, such that

\emph{(i)} every $\gamma \in \Gamma$ has degree at most $b$, and \newline
\emph{(ii)} for every $k$ distinct points in $\R^2$, there exist at most $C$ distinct curves in $\Gamma$ passing through all of them.

Then, \begin{displaymath} I_{P,\Gamma} \lesssim_{b,k,C} |P|^{k/(2k-1)}|\Gamma|^{(2k-2)/(2k-1)}+|P|+|\Gamma|. \end{displaymath} \end{theorem}

Combining Theorem \ref{4.1.8} with Lemmas \ref{4.1.6} and \ref{4.1.7}, we deduce the following fact on point--real algebraic curve incidences in $\R^3$.

\begin{lemma}\label{4.1.9} Let $b$ be a positive constant. Also, let $\Gamma$ be a finite set of real algebraic curves in $\R^3$, each of degree at most $b$, and $P$ a finite set of points in $\R^3$. Then, there exists a natural number $D_b \geq b^2+1$, depending only on $b$, such that 

\emph{(i)} $I'_{P,\Gamma} \lesssim_b |P|^{D_b/(2D_b-1)}|\Gamma|^{(2D_b-2)/(2D_b-1)}+|P|+|\Gamma|$, where $I'_{P,\Gamma}$ denotes the number of all pairs $(p,\gamma)$ such that $p\in P$, $\gamma \in \Gamma$, $p\in \gamma$ and $p$ is not an isolated point of $\gamma$, and

\emph{(ii)} if there exist $S$ points in $\R^3$, such that each lies in at least $k$ curves of $\Gamma$ which do not have the point as an isolated point, where $k \geq 2$, then $S\lesssim_b {|\Gamma}| ^2/k^{(2D_b-1)/(D_b-1)} + |\Gamma|/k$. \end{lemma}

\begin{proof} Let $\pi :\R^3 \rightarrow H$ be the projection map of $\R^3$ on a generic plane $H \simeq \R^2$. By Lemma \ref{4.1.6} we know that, for all $\gamma \in \Gamma$, $\pi(\gamma)$ is contained in a planar real algebraic curve $\overline{\pi(\gamma)}$ of degree at most $C_b$, where $C_b \geq b$ is an integer depending only on b. Thus, if $\pi(\Gamma):=\{\pi(\gamma):\gamma \in \Gamma\}$ and $Irr(\overline{\pi(\Gamma)}):=\{$1-dimensional irreducible components of $\overline{\pi(\gamma)}: \gamma \in \Gamma\}$, we have that
\begin{displaymath} I'_{P,\Gamma} \leq I'_{\pi(P), \pi(\Gamma)} \leq I_{\pi(P), Irr(\overline{\pi(\Gamma)})} + 4{C_b}^2\cdot |Irr(\overline{\pi(\Gamma)})|,\end{displaymath} as, by Lemma \ref{newstuff5}, each curve in $Irr(\overline{\pi(\Gamma)})$ crosses itself at most $4(\deg \overline{\pi(\gamma)})^2\leq 4 C_b^2$ times. In addition, by B\'ezout's theorem, for each $C_b^2 +1$ distinct points of $\R^2$ there exists at most 1 curve in $Irr(\overline{\pi(\Gamma)})$ passing through all of them. The application, therefore, of Theorem \ref{4.1.8} for $k=D_b:=C_b^2+1$, the set $\pi(P)$ of points and the set $Irr(\overline{\pi(\Gamma)})$ of planar real algebraic curves, whose cardinality is obviously $\leq C_b\cdot|\Gamma|$, completes the proof of (i), while (ii) is an immediate corollary of (i).

\end{proof}

For the analysis that follows, we introduce the notion of the resultant of two polynomials, a useful tool for deducing whether two polynomials have a common factor (for details, see \cite[Chapter 3]{MR2122859} or \cite{Guth_Katz_2008}).

More particularly, let $f$, $g \in \C[x]$, of positive degrees $l$ and $m$, respectively, with
\begin{displaymath}f(x)=a_lx^l+a_{l-1}x^{l-1}+...+a_0
\end{displaymath}
and
\begin{displaymath}g(x)=b_mx^m+b_{m-1}x^{m-1}+...+b_0.
\end{displaymath}

We define the resultant $Res(f,g)$ of $f$ and $g$ as the determinant of the $(l+m)  \times  (l+m)$ matrix $(c_{ij})$, where $c_{ij} =a_{j-i}$ if $1 \leq i \leq m$ and $i \leq j \leq i+l$,
$c_{ij}= b_{j-i +m}$ if $m+1 \leq i \leq m+l$ and $i-m \leq j \leq i-m+l$, and $c_{ij}=0$ otherwise.

Note that the columns of the matrix $(c_{ij})$ represent the coefficients of the polynomial
$f$ multiplied by $x^j$, where $j$ runs from 0 to $m-1$, and the
coefficents of the polynomial $g$ multiplied by $x^k$,  where $k$ runs
from 0 to $l-1$. Therefore, the resultant of $f$ and $g$ is 0 if and only if this set of polynomials is linearly independent. This leads to a connection between the existence of a common factor of two polynomials and the value of their resultant. 

Indeed, let $f$, $g \in {\C}[x_1,..., x_n]$ be polynomials of positive degree in $x_1$. Viewing $f$ and $g$ as polynomials in $x_1$ with coefficients in $\C[x_2..., x_n]$, we define the resultant of $f$ and $g$ with respect to $x_1$ as the polynomial $Res(f,g;x_1)\in \C[x_2,...,x_n]$.

\begin{theorem} \label{factor} Let $f$, $g \in {\C}[x_1, ... , x_n]$ be polynomials of positive degree when viewed as polynomials in $x_1$. Then, $f$ and $g$ have a common factor of positive degree in $x_1$ if and only if $Res(f,g;x_1)$ is the zero polynomial.
\end{theorem}

Theorem \ref{factor} is \S 3.6 Proposition 1 (ii) in \cite{Cox+Others/1991/Ideals}. In fact, the following is true (see \cite{Cox+Others/1991/Ideals}).

\begin{lemma} \label{GCD} Let $f$, $g \in \C[x_1,...,x_n]$ of positive degree in $x_1$. Then, there exist $A$, $B \in \C[x_2,...,x_n][x_1]$, such that $Res(f,g;x_1)=Af+Bg$.

\end{lemma}

In particular, Lemma \ref{GCD} implies the following.

\begin{lemma} \label{rootsresultant} Let $f$, $g \in \C[x_1,x_2]$ be polynomials of positive degree in $x_1$. If $f$, $g$ both vanish at the point $(r_1,r_2)\in \C^2$, then $Res(f,g;x_1)$ vanishes at $r_2$.

\end{lemma}

On the other hand, by the definition of the resultant of two polynomials, it is easy to see the following (see \cite{Cox+Others/1991/Ideals}).

\begin{lemma} \label{emergency} Let $f$, $g \in \C[x_1,x_2]$ be polynomials of positive degree in $x_1$. Then, $Res(f,g;x_1)$ is a polynomial in $x_2$, of degree at most $\deg f \cdot \deg g$.

\end{lemma}

We are now ready to extend the proof of \cite[Corollary 2.5]{Guth_Katz_2008} to a more general setting, to deduce the following.

\begin{lemma}\label{4.1.12} Suppose that $f$, $g$ are non-constant polynomials in $ \C[x,y,z]$ which do not have a common factor. Then, the number of irreducible complex algebraic curves which are simultaneously contained in the zero set of $f$ and the zero set of $g$ in $\C^3$ is $\leq \deg f \cdot \deg g$. \end{lemma}

\begin{proof} Let $\Gamma$ be the family of irreducible complex algebraic curves in $\C^3$. Suppose that there exist $\deg f \cdot \deg g+1$ curves in $\Gamma$, simultaneously contained in the zero set of $f$ and the zero set of $g$. A generic complex plane intersects a complex algebraic curve in $\C^3$ at least once and finitely many times, while each two curves in $\Gamma$ intersect in finitely many points of $\C^3$. Therefore, we can change the coordinates, so that $f$ and $g$ have positive degree in $x$, and also so that there exists some point $p=(p_1,p_2,p_3)\in \C^3$ and some $\epsilon >0$, such that any plane in the family $\mathcal{A}:=\big\{$planes in $\C^3$, perpendicular to $(0,0,1)$ and passing through a point of the form $p+ \delta \cdot (0,0,1)$, for $\delta \in (-\epsilon, \epsilon)\}$ is transverse to all the $\deg f \cdot \deg g +1$ curves, intersecting them at points with distinct $y$ coordinates. Thus, each such plane contains at least $\deg f \cdot \deg g +1$ points of $\C^3$ with distinct $y$ coordinates, where both $f$ and $g$ vanish.

Therefore, if $\Pi \in \mathcal{A}$, then $f_{|\Pi}$, $g_{|\Pi}$ are two polynomials in $\C[x,y]$, vanishing at $\geq \deg f \cdot \deg g +1 \geq \deg f_{|\Pi} \cdot \deg g_{|\Pi}+1$ points of $\C^2$ with distinct $y$ coordinates.

At the same time, there are at most $\deg f$, i.e. finitely many, planes $\Pi$ in $\mathcal{A}$, such that $f_{|\Pi}$ does not have positive degree in $x$, and at most $\deg g$, i.e. finitely many, planes $\Pi$ in $\mathcal{A}$, such that $g_{|\Pi}$ does not have positive degree in $x$. Indeed, suppose that there are more than $\deg f$ planes in $\mathcal{A}$, such that $f_{|\Pi}$ does not have positive degree in $x$. Let $h$ be the coefficient of a positive power of $x$ in the expression of the polynomial $f$ as a polynomial of $x$. We view $h$ as a complex polynomial in $x$, $y$, $z$, of non-positive degree in $x$. By our assumption, there are more than $\deg f\geq \deg h$ planes in $\mathcal{A}$ on which $h$ vanishes, therefore $h$ vanishes on $\C^3$ (since a generic line in $\C^3$ intersects all those planes, and thus lies in the zero set of $h$). Hence, $h$ is the zero polynomial; and since $h$ was the coefficient of an arbitrary positive power of $x$ in the expression of $f$ as a polynomial of $x$, it follows that $f$ does not have positive degree in $x$, which is a contradiction. We similarly get a contradiction if we assume that there exist more than $\deg g$ planes in $\mathcal{A}$, such that $g_{|\Pi}$ does not have positive degree in $x$.

Thus, there exists an open interval $I \subset (-\epsilon, \epsilon)$, such that, if $\Pi$ is a plane in the family $\mathcal{A}':=\big\{$planes in $\C^3$, perpendicular to $(0,0,1)$ and passing through a point of the form $p+\delta \cdot (0,0,1)$, for $\delta \in I\}$, then $f_{|\Pi}$, $g_{|\Pi}$ are two polynomials in $\C[x,y]$, of positive degree in $x$, vanishing at $\geq \deg f_{|\Pi} \cdot \deg g_{|\Pi}+1$ points of $\C^2$ with distinct $y$ coordinates. Thus, by Lemma \ref{emergency}, $Res(f_{|\Pi},g_{|\Pi};x)\equiv 0$, for all $\Pi \in \mathcal{A}'$. However, $Res(f_{|\Pi}, g_{|\Pi};x)\equiv Res(f,g;x)_{|\Pi}$, for all $\Pi \in \mathcal{A}'$.

As a result, we have that, for all $\Pi \in \mathcal{A}'$, $Res(f,g;x)_{|\Pi} \equiv 0$; this means that the polynomial $Res(f,g;x) \in \C[y,z]$ vanishes for all $(y,z) \in \C^2$, such that $y \in \C$ and $z \in J$, for some subset $J$ of $\C$ of the form $\{z \in \C: z=p_3+a$, for $a \in (a_1, a_2)\}$, where $a_1$, $a_2 \in \R$. In other words, the polynomial $Res(f,g;x) \in \C[y,z]$ vanishes on a rectangle of $\C^2$. Therefore, it vanishes identically. And $Res(f,g;x) \equiv 0$ means that $f$ and $g$ have a common factor (since they both have positive degree when viewed as polynomials in $x$). We are thus led to a contradiction, which means that there exist $\leq  \deg f \cdot \deg g$ curves of $\Gamma$ simultaneously contained in the zero set of $f$ and the zero set of $g$.

\end{proof}

\begin{corollary}\label{4.1.13} Let $f$ and $g$ be non-constant polynomials in $\R[x,y,z]$. Suppose that $f$ and $g$ do not have a common factor. Then, the number of irreducible real algebraic curves which are simultaneously contained in the zero set of $f$ and the zero set of $g$ in $\R^3$ is $\leq \deg f \cdot \deg g$. \end{corollary}

\begin{proof} We see $f$ and $g$ as polynomials in $\C[x,y,z]$, and viewed as such we denote them by $f_{\C}$,  $g_{\C}$, respectively. Also, let $\Gamma$ be the family of irreducible real algebraic curves in $\R^3$. For all $\gamma \in \Gamma$, we denote by $\gamma_{\C}$ the (unique) irreducible complex algebraic curve containing $\gamma$. 

Since the polynomials $f$, $g\in \R[x,y,z]$ do not have a common factor in $\R[x,y,z]$, the polynomials $f_{\C}$, $g_{\C}\in \C[x,y,z]$ do not have a common factor in $\C[x,y,z]$. Indeed, if $h \in \C[x,y,z]$ was a common factor of $f_{\C}$, $g_{\C}$, which are polynomials with real coefficients, then $\bar{h} \in \C[x,y,z]$ would also be a common factor of $f_{\C}$, $g_{\C}$, therefore $h\bar{h} \in \R[x,y,z]$ would be a common factor of the polynomials $f$, $g \in \R[x,y,z]$.

Now, suppose that a curve $\gamma \in \Gamma$ lies in both the zero set of $f$ and the zero set of $g$. We know that, as it contains $\gamma$, the irreducible complex algebraic curve $\gamma _{\C}$ intersects the zero set of $f_{\C}$ (which contains the zero set of $f$) infinitely many times; thus, by B\'ezout's theorem, it is contained in the zero set of $f_{\C}$. Similarly, $\gamma_{\C}$ is contained in the zero set of $g_{\C}$.

Moreover, if $\gamma^{(1)}$, $\gamma^{(2)} \in\Gamma$ are such that $\gamma^{(1)}\not\equiv \gamma^{(2)}$, then $\gamma^{(1)}_{\C}\not\equiv \gamma^{(2)}_{\C}$. The reason for this is that $\gamma^{(1)}$ is the intersection of $\gamma^{(1)}_{\C}$ with $\R^3$, while $\gamma^{(2)}$ is the intersection of $\gamma^{(2)}_{\C}$ with $\R^3$.

Thus, if $>\deg f \cdot \deg g$ curves of $\Gamma$ lie simultaneously in the zero set of $f$ and the zero set of $g$, then there exist $> \deg f \cdot \deg g$ irreducible complex algebraic curves in $\C^3$, lying in both the zero set of $f_{\C}$ and the zero set of $g_{\C}$, where $f_{\C}$, $g_{\C} \in \C[x,y,z]$ do not have a common factor in $\C[x,y,z]$. By Lemma \ref{4.1.12} though, this is a contradiction. Therefore, the number of curves in $\Gamma$, i.e. of irreducible real algebraic curves in $\R^3$, which are simultaneously contained in the zero set of $f$ and the zero set of $g$, is $\leq \deg f \cdot \deg g$.

\end{proof}

\begin{definition} Let $Z$ be the zero set of a polynomial $p \in \R[x,y,z]$. A curve $\gamma$ in $\R^3$ is a \emph{critical curve} of $Z$ if each point of $\gamma$ is a critical point of $Z$.
\end{definition}

We are now able to deduce the following.

\begin{corollary}\label{4.1.14} The zero set of a polynomial $p \in \R[x,y,z]$ contains at most $(\deg p)^2$ critical irreducible real algebraic curves of $\R^3$. \end{corollary}

\begin{proof} The polynomials $p_{sf}$ and $\nabla{p_{sf}}$, the intersection of the zero sets of which is the set of critical points of the zero set of $p$, do not have a common factor, as $p_{sf}$ is square-free. Therefore, the result follows by Corollary \ref{4.1.13}.

\end{proof}

On a different subject, it is known (see \cite[Chapter 5]{MR2248869}) that each real semi-algebraic set is the finite, disjoint union of path-connected components. We observe the following.

\begin{lemma}\label{4.1.15} A real algebraic curve in $\R^n$ is the finite, disjoint union of $\lesssim_{b,n} 1$ path-connected components. \end{lemma}

\begin{proof} This is obvious by a closer study of the algorithm in \cite[Chapter 5]{MR2248869} that constitutes the proof of the fact that every real semi-algebraic set is the finite, disjoint union of path-connected components.

\end{proof}

Finally, we are interested in curves in $\R^3$ parametrised by $t \rightarrow \big(p_1(t), p_2(t),p_3(t)\big)$ for $t \in \R$, where $p_i \in \R[t]$ for $i=1,2,3$. Note that, although curves in $\C^3$ with a polynomial parametrisation are, in fact, complex algebraic curves of degree equal to the maximal degree of the polynomials realising the parametrisation (see \cite[Chapter 3, \S3]{Cox+Others/1991/Ideals}), curves in $\R^3$ with a polynomial parametrisation are not, in general, real algebraic curves, which is why we treat their case separately.

More particularly, if a curve $\gamma$ in $\R^3$ is parametrised by $t \rightarrow \big(p_1(t), p_2(t),p_3(t)\big)$ for $t \in \R$, where the $p_i \in \R[t]$, for $i=1,2,3$, are polynomials not simultaneously constant, then the complex algebraic curve $\gamma_{\C}$ parametrised by the same polynomials viewed as elements of $\C[t]$ is irreducible (it is easy to see that if it contains a complex algebraic curve, then the two curves are identical). Therefore, by B\'ezout's theorem, $\gamma_{\C}$ is the unique complex algebraic curve containing $\gamma$.

Taking advantage of this fact, we will show here that each curve in $\R^3$ with a polynomial parametrisation is contained in a real algebraic curve in $\R^3$.

To that end, we first show the following.

\begin{lemma} \label{maybelast}The intersection of a complex algebraic curve in $\C^n$ with $\R^n$ is a real algebraic set, of dimension at most 1.
\end{lemma}

\begin{proof} Let $\gamma_{\C}$ be a complex algebraic curve in $\C^n$, and $\gamma$ the intersection of $\gamma_{\C}$ with $\R^n$. We show that $\gamma$ is a real algebraic set, of dimension at most 1.

Indeed, since $\gamma_{\C}$ is a complex algebraic set in $\C^n$, there exist polynomials $p_1$, ..., $p_k \in \C[x_1,...,x_n]$, such that $\gamma_{\C}$ is the intersection of the zero sets of $p_1$, ..., $p_k$ in $\C^n$. Now, for $i=1,...,k$, the intersection of the zero set of the polynomial $p_i$ in $\C^n$ with $\R^n$ is equal to the zero set of $p_i$ in $\R^n$, which is the same as the zero set of the polynomial $p_i\bar{p_i} \in \R[x_1,...,x_n]$ in $\R^n$. Therefore, $\gamma$ is the intersection of the zero sets of the polynomials $p_1\bar{p_1}$, ..., $p_k\bar{p_k} \in \R[x_1,...,x_n]$ in $\R^n$, it is thus a real algebraic set.

Moreover, let $\prec$ be a term order on the set of monomials in the variables $x_1$, ..., $x_n$. Since $\gamma_{\C}$ is a complex algebraic curve in $\C^n$, it holds that, for every $i\neq j$, $i,j \in \{1,...,n\}$, there exists a monomial in the variables $x_i$ and $x_j$ in the ideal $in_{\prec}(I(\gamma_{\C}))$. Now, if a polynomial $p \in \C[x_1,...,x_n]$ belongs to $I(\gamma_{\C})$, i.e. vanishes on $\gamma_{\C}$, then the polynomial $p \bar{p} \in \R[x_1,...,x_n]$ vanishes on $\gamma$, and thus belongs to the ideal $I(\gamma)$ of $\R[x_1,...,x_n]$. Therefore, there exists a monomial in the variables $x_i$ and $x_j$ in the ideal $in_{\prec}(I(\gamma))$. Consequently, the algebraic set $\gamma$ has dimension at most 1. 

\end{proof}

\begin{corollary} \label{parametrisedcurves} Let $\gamma$ be a curve in $\R^3$, parametrised by $t \rightarrow \big(p_1(t), p_2(t),p_3(t)\big)$ for $t \in \R$, where the $p_i \in \R[t]$, for $i=1,2,3$, are polynomials not simultaneously constant, of degree at most $b$. Then, $\gamma$ is contained in an irreducible real algebraic curve in $\R^3$, of degree at most $b$.
\end{corollary}

\begin{proof} Let $\gamma_{\C}$ be the curve in $\C^3$, parametrised by $t \rightarrow \big(p_1(t), p_2(t),p_3(t)\big)$ for $t \in \C$. As we have already discussed, $\gamma_{\C}$ is the (unique) irreducible complex algebraic curve containing $\gamma$.

Clearly, $\gamma$ is contained in the intersection of $\gamma_{\C}$ with $\R^3$, which, by Lemma \ref{maybelast}, is a real algebraic set, of dimension at most 1. However, since $\gamma_{\C} \cap \R^3$ contains the parametrised curve $\gamma$, it has, in fact, dimension equal to 1. Therefore, $\gamma$ is contained in the real algebraic curve $\gamma_{\C}\cap\R^3$. In fact, $\gamma_{\C}\cap\R^3$ is an irreducible real algebraic curve. Indeed, if $\gamma ' \subsetneq \gamma_{\C}'\cap\R^3$ was an irreducible real algebraic curve, and $\gamma_{\C}'$ was the (unique) irreducible complex algebraic curve containing it, then $\gamma_{\C}'\cap \R^3 \supsetneq \gamma'=\gamma_{\C}'\cap \R^3$, and thus $\gamma_{\C}\cap \gamma_{\C}'$ $\subsetneq \gamma_{\C}$ would be a complex algebraic curve, which cannot hold, since $\gamma_{\C}$ is irreducible.

Moreover, since $\gamma_{\C}$ is an irreducible algebraic curve in $\C^3$, it is the smallest complex algebraic curve containing the real algebraic curve $\gamma_{\C} \cap \R^3$, and thus the degree of $\gamma_{\C}\cap\R^3$ is equal to $\deg \gamma_{\C}=\max\{\deg p_1, \deg p_2, \deg p_3\}$, and thus equal to at most $b$.

\end{proof}

Note that the Szemer\'edi-Trotter type theorem \ref{4.1.8} gives upper bounds on incidences between points and real algebraic curves in $\R^2$, of uniformly bounded degree. However, it cannot be extended to hold for a family of curves in $\R^2$ parametrised by real univariate polynomials of uniformly bounded degree, without extra hypotheses on the family of the cuvres. Indeed, half-lines in $\R^2$ are such curves, and if the Szemer\'edi-Trotter theorem held for any finite collection of half-lines in $\R^2$, then the Szemer\'edi-Trotter theorem for lines would be scale invariant, since there exist infinitely many distinct half-lines lying on the same line.

Similarly, there does not exist, in general, an upper bound on the number of critical curves parametrised by real univariate polynomials of uniformly bounded degree, contained in an algebraic hypersurface in $\R^3$.

\newpage
\thispagestyle{plain}
\cleardoublepage

\chapter{Transversality of more general curves} \label{6}

In this chapter we extend the definition of joints and multijoints to a more general setting, and show that some of the results we have so far described (in fact, the corresponding statements of Theorem \ref{1.1} and Theorem \ref{theoremmult2}) still hold.

Indeed, we consider the family $\mathcal{F}$ of all non-empty sets in $\R^3$ with the property that, if $\gamma \in \mathcal{F}$ and $x \in \gamma$, then a basic neighbourhood of $x$ in $\gamma$ is either $\{x\}$ or the finite union of parametrised curves, each homeomorphic to a semi-open line segment with one endpoint the point $x$. In addition, if there exists a parametrisation $f:[0,1)\rightarrow \R^3$ of one of these curves, with $f(0)=x$ and $f'(0)\neq 0$, then the line in $\R^3$ passing through $x$ with direction $f'(0)$ is tangent to $\gamma$ at $x$. If $\Gamma \subset \mathcal{F}$, we denote by $T_x^{\Gamma}$ the set of directions of all tangent lines at $x$ to the sets of $\Gamma$ passing through $x$ (note that $T_x^{\Gamma}$ might be empty and that there might exist many tangent lines to a set of $\Gamma$ at $x$).

Real algebraic curves in $\R^3$, as well as curves in $\R^3$ parametrised by real polynomials, belong to the family $\mathcal{F}$.

\section{Joints}

\begin{definition} Let $\Gamma$ be a collection of sets in $\mathcal{F}$.  Then a point $x$ in $\R^3$ is a joint for the collection $\Gamma$ if

\emph{(i)} $x$ belongs to at least one of the sets in $\Gamma$, and\newline
\emph{(ii)} there exist at least 3 vectors in $T_x^{\Gamma}$ spanning $\R^3$.

The multiplicity $N(x)$ of the joint $x$ is defined as the number of triples of lines in $\R^3$ passing through $x$, whose directions are linearly independent vectors in $T_x^{\Gamma}$. \end{definition}

We will show here that, under certain assumptions on the properties of the sets in a finite collection $\Gamma \subset \mathcal{F}$, the statement of Theorem \ref{1.1} still holds, i.e. \begin{displaymath} \sum_{x \in J}N(x)^{1/2} \leq c \cdot |\Gamma|^{3/2}, \end{displaymath} where $J$ is the set of joints formed by $\Gamma$. 

Indeed, thanks to the results of Chapter \ref{5}, we are now ready to formulate and prove the following extension of Theorem \ref{1.1}.

\begin{theorem}\label{4.2.1} Let $b$ be a positive constant and $\Gamma$ a finite collection of real algebraic curves in $\R^3$, of degree at most $b$. Let $J$ be the set of joints formed by $\Gamma$. Then,
\begin{displaymath} \sum_{x \in J}N(x)^{1/2} \leq c_b \cdot |\Gamma|^{3/2},\end{displaymath}
where $c_b$ is a constant depending only on $b$. \end{theorem}

The proof of Theorem \ref{4.2.1} is completely analogous to the proof of Theorem \ref{1.1}. Indeed, if $\gamma$ is a real algebraic curve in $\R^3$, of degree at most $b$, and $x\in \gamma$ is not an isolated point of $\gamma$, then $\gamma$ crosses itself at $x$ at most $b$ times, while there exists at least one tangent line to $\gamma$ at $x$; thus, there exist at least 1 and at most $b$ tangent lines to $\gamma$ at $x$. So, if $x$ is a joint of multiplicity $N$ for $\Gamma$, such that at most $k$ curves of $\Gamma$, of which $x$ is not an isolated point, are passing through $x$, then $N \leq (bk)^3$. Therefore, the following lemmas hold, whose statements and proofs are analogous to those of Lemmas \ref{3.1} and \ref{3.2}.

\begin{lemma}\label{4.2.2} Let $x$ be a joint of multiplicity $N$ for a finite collection $\Gamma$ of real algebraic curves in $\R^3$, of degree at most $b$. Suppose that $x$ lies in $\leq 2k$ of the curves in $\Gamma$ of which it is not an isolated point. If, in addition, $x$ is a joint of multiplicity $\leq N/2$ for a subcollection $\Gamma '$ of $\Gamma$, or if it is not a joint at all for the subcollection $\Gamma '$, then there exist $\geq \frac{N}{1000b^3\cdot k^2}$ curves of $\Gamma \setminus \Gamma '$, of which $x$ is not an isolated point, passing through $x$. \end{lemma}

\begin{lemma}\label{4.2.3} Let $x$ be a joint of multiplicity $N$ for a finite collection $\Gamma$ of real algebraic curves in $\R^3$, of degree at most $b$. Suppose that $x$ lies in $\leq 2k$ of the curves in $\Gamma$ of which it is not an isolated point. Then, for every plane containing $x$, there exist $\geq \frac{N}{1000b^3 \cdot k^2}$ curves in $\Gamma$, such that their tangent vectors at $x$ are well-defined and not parallel to the plane. \end{lemma}

Now, for a collection $\Gamma$ of real algebraic curves in $\R^3$, if $J$ is the set of joints formed by $\Gamma$, we define

$J_N:=\{x \in J: N \leq N(x) <2N\}$, for all $N \in \N$, and

$J_N^{k}:=\{x \in J_N$: $x$ intersects at least $k$ and fewer than $2k$ curves of $\Gamma$ of which $x$ is not an isolated point$\}$, for all $N, k \in \N$.

Then, Theorem \ref{4.2.1} easily follows from Proposition \ref{4.2.4}, the statement and a sketch of the proof of which we now present.

\begin{proposition}\label{4.2.4} Let $b \in \N$ and $\Gamma$ a finite collection of real algebraic curves in $\R^3$, of degree at most $b$. Then, \begin{displaymath} |J^{k}_{N}| \cdot N^{1/2} \leq c_b \cdot \bigg(\frac{|\Gamma|^{3/2}}{k^{1/(2D_b-2)}}+ \frac{|\Gamma|}{k}\cdot N^{1/2}\bigg), \end{displaymath}
where $D_b$ and $c_b$ are constants depending only on $b$ (and, in particular, $D_b\geq b^2+1$). \end{proposition}

\begin{proof} Each real algebraic curve in $\R^3$ of degree at most $b$ consists of $\leq b \lesssim _b 1$ irreducible components; we may therefore assume that each $\gamma \in \Gamma$ is irreducible.

Keeping in mind that a curve $\gamma \in \Gamma$ crosses itself at a point $x$ at most $b$ times, and therefore the number of tangent lines to $\gamma$ at $x$ is at most $b$, the proof is completely analogous to that of Proposition 1.2. The main differences lie at the beginning and the cellular case, we thus go on to point them out.

By Lemma \ref{4.1.6}, there exists an integer $C_b\geq b$, such that, if $\gamma$ is a real algebraic curve in $\R^3$, of degree at most $b$, then the projection of $\gamma$ on a generic plane is contained in a planar real algebraic curve, of degree $\leq C_b$.

Therefore, by Lemmas \ref{4.1.7} and \ref{4.1.9}, the integer $D_b:=C_b^2+1$ has the following properties.

(i) If $\gamma$ is a real algebraic curve in $\R^3$, of degree at most $b$, then $\gamma$ crosses itself at most $4D_b$ times.

(ii) There exists at most 1 real algebraic curve in $\R^3$, of degree at most $b$, passing through any fixed $D_b$ points in $\R^3$.

(iii) For any finite collection $\Gamma$ of real algebraic curves in $\R^3$, of degree at most $b$, it holds that $|J_N^k|\lesssim_b |\Gamma|^2/k^{(2D_b-1)/(D_b-1)} + |\Gamma|/k$.

This will be the integer $D_b$ appearing in the statement of the Proposition.

Now, the proof of the Proposition will be achieved by induction on the cardinality of $|\Gamma|$. Indeed, let $M \in \N$. For $c_b$ an explicit constant $\geq D_b$, which depends only on $b$ and will be specified later:

- For any collection $\Gamma$ of irreducible real algebraic curves in $\R^3$, of degree at most $b$, such that $|\Gamma|=1$, it holds that \begin{displaymath} |J^k_{N}|\cdot N^{1/2} \leq c_b \cdot \bigg( \frac{1^{3/2}}{k^{1/(2D_b-2)}} + \frac{1}{k}\cdot N^{1/2}\bigg), \; \forall \; N, \;k \in \N\end{displaymath} (this is obvious, in fact, for any $c_b \geq 4D_b$, as in this case $|J_{N}|=|J_N^1|\leq 4D_b$ for all $N \in \N$, since a real algebraic curve in $\R^3$, of degree at most $b$, crosses itself at most $4D_b$ times). 

- We assume that \begin{displaymath} |J_N^{k}| \cdot N^{1/2} \leq c_b \cdot \bigg(\frac{|\Gamma|^{3/2}}{k^{1/(2D_b-2)}} + \frac{|\Gamma|}{k}\cdot N^{1/2}\bigg), \; \forall \; N,\;k \in \N,\end{displaymath} for any finite collection $\Gamma $ of irreducible real algebraic curves in $\R^3$, of degree at most $b$, such that $|\Gamma| \lneq M.$

- We will now prove that  \begin{equation} |J_N^{k}| \cdot N^{1/2} \leq c_b \cdot \bigg(\frac{|\Gamma|^{3/2}}{k^{1/(2D_b-2)}} + \frac{|\Gamma|}{k}\cdot N^{1/2}\bigg),\; \forall \; N, \;k \in \N, \label{eq:final'}\end{equation}
for any collection $\Gamma$ of irreducible real algebraic curves in $\R^3$, of degree at most $b$, such that $|\Gamma|=M$.

Indeed, let $\Gamma$ be a collection of irreducible real algebraic curves in $\R^3$, of degree at most $b$, such that $|\Gamma|=M$. Fix $N$ and $k$ in $\N$, and let \begin{displaymath}\mathfrak{G}:=J_{N}^k\end{displaymath} and $$ S:=|J_{N}^k|$$ for this collection $\Gamma$.

Now, we know that $S\cdot N^{1/2} \leq c_{0,b} \cdot (|\Gamma| ^2/k^{(2D_b-1)/(D_b-1)} + |\Gamma|/k)$ for some constant $c_{0,b}$ depending only on $b$. Thus:

If $\frac{S}{2}\leq c_{0,b}\cdot  \frac{|\Gamma|}{k}$, then $S \cdot N^{1/2} \leq 2c_{0,b} \cdot \frac{|\Gamma|}{k}\cdot N^{1/2}$ (where $2c_{0,b}$ is a constant depending only on $b$). 

Otherwise, $\frac{S}{2}< c_{0,b} \cdot|\Gamma| ^2/k^{(2D_b-1)/(D_b-1)}$, so $S < 2c_{0,b} \cdot  |\Gamma|^2 k^{-(2D_b-1)/(D_b-1)}$. 

Therefore, $d:=A_b|\Gamma|^2S^{-1}k^{-(2D_b-1)/(D_b-1)}$ is a quantity $>1$ whenever $A_b\geq 2c_{0,b}$; we thus choose $A_b$ to be large enough for this to hold, and we will specify its value later. Now, applying the Guth-Katz polynomial method for this $d>1$ and the finite set of points $\mathfrak{G}$, we deduce that there exists a non-zero polynomial $p\in \R[x,y,z]$, of degree $\leq d$, whose zero set $Z$:

(i) decomposes $\R^3$ in $\sim d^3$ cells, each of which contains $\lesssim Sd^{-3}$ points of $\mathfrak{G}$, and

(ii) contains 6 distinct generic planes, each of which contains a face of a fixed cube $Q$ in $\R^3$, such that the interior of $Q$ contains $\mathfrak{G}$ (and each of the planes is generic in the sense that the plane in $\C^3$ containing it intersects the smallest complex algebraic curve in $\C^3$ containing $\gamma$, for all $\gamma \in \Gamma$);

to achieve this, we first fix a cube $Q$ in $\R^3$, with the property that its interior contains $\mathfrak{G}$ and the planes containing its faces are generic in the above sense. Then, we multiply the polynomials we end up with at each step of the Guth-Katz polynomial method with the same (appropriate) six linear polynomials, the zero set of each of which is a plane containing a different face of the cube, and stop the application of the method when we finally get a polynomial of degree at most $d$, whose zero set decomposes $\R^3$ in $\lesssim d^3$ cells (the set of the cells now consists of the non-empty intersections of the interior of the cube $Q$ with the cells that arise from the application of the Guth-Katz polynomial method, as well as the complement of the cube).

We can assume that the polynomial $p$ is square-free, as eliminating the squares of $p$ does not inflict any change on its zero set.

If there are $\geq 10^{-8}S$ points of $\mathfrak{G}$ in the union of the interiors of the cells, we are in the cellular case. Otherwise, we are in the algebraic case.

\textbf{Cellular case:} There are $\gtrsim S$ points of $\mathfrak{G}$ in the union of the interiors of the cells. However, we also know that there exist $\sim d^3$ cells in total, each containing $\lesssim Sd^{-3}$ points of $\mathfrak{G}$. Therefore, there exist $\gtrsim d^3$ cells, with $\gtrsim Sd^{-3}$ points of $\mathfrak{G}$ in the interior of each. We call the cells with this property ``full cells". Now:

$\bullet$ If the interior of some full cell contains $< k^{1/(D_b-1)}$ points of $\mathfrak{G}$, then $Sd^{-3} \lesssim k^{1/(D_b-1)}$, and since $N \lesssim b^3 k^3\lesssim_b k^3$, we have that $S \cdot N^{1/2} \lesssim_b |\Gamma|^{3/2}/k^{1/(2D_b-2)}$.

$\bullet$ If the interior of each full cell contains $\geq k^{1/(D_b-1)}$ points of $\mathfrak{G}$, then we will be led to a contradiction by choosing $A_b$ sufficiently large. Indeed:

Consider a full cell and let $\mathfrak{G}_{cell}$ be the set of points of $\mathfrak{G}$ lying in the interior of the cell, $S_{cell}$ the cardinality of $\mathfrak{G}_{cell}$ and $\Gamma_{cell}:= \{\gamma \in \Gamma:\exists \; x \in \gamma\cap \mathfrak{G}_{cell}$, such that $x$ is not an isolated point of $\gamma\}$.

Let $\mathfrak{G}_{cell}'$ be a subset of $\mathfrak{G}_{cell}$ of cardinality $k^{1/(D_b-1)}$. Since each point of $\mathfrak{G}_{cell}$ has at least $k$ curves of $\Gamma_{cell}$ passing through it, there exist at least $k^{D_b/(D_b-1)}$ incidences between $\Gamma_{cell}$ and $\mathfrak{G}_{cell}'$. On the other hand, the curves in $\Gamma_{cell}$ containing at most $D_b-1$ points of $\mathfrak{G}_{cell}'$ contribute at most $(D_b-1) \cdot|\Gamma_{cell}|$ incidences with $\mathfrak{G}_{cell}'$, while through any fixed point of $\mathfrak{G}'_{cell}$ there exist at most $\binom{|\mathfrak{G}_{cell}'|}{D_b-1}$ curves in $\Gamma_{cell}$, each containing at least $D_b$ points of $\mathfrak{G}_{cell}'$, since there exists at most 1 curve in $\Gamma$ passing through any fixed $D_b$ points in $\R^3$; therefore, there exist at most $\binom{|\mathfrak{G}_{cell}'|}{D_b-1}\cdot |\mathfrak{G}_{cell}'|$ incidences between $\mathfrak{G}_{cell}'$ and the curves in $\Gamma _{cell}$, each of which contains at least $D_b$ points of $\mathfrak{G}_{cell}'$. Thus, 
\begin{displaymath}k^{D_b/(D_b-1)} \leq I_{\mathfrak{G}_{cell}', \Gamma_{cell}} \leq (D_b-1) \cdot |\Gamma_{cell}| +\binom{k^{1/(D_b-1)}}{D_b-1} \cdot k^{1/(D_b-1)} \leq \end{displaymath} \begin{displaymath} \leq (D_b-1) \cdot |\Gamma_{cell}| +\frac{1}{(D_b-1)!}\cdot k^{D_b/(D_b-1)}\text{, so}\end{displaymath} \begin{displaymath}|\Gamma_{cell}| \gtrsim_b k^{D_b/(D_b-1)}, \end{displaymath}and thus \begin{displaymath} |\Gamma_{cell}|^2/k^{(2D_b-1)/(D_b-1)} \gtrsim_b |\Gamma_{cell}| /k.\end{displaymath} Note that this approach differs to the one applied in the case of joints formed by lines.

Now, due to our definition of $D_b$ and the fact that each of the points in $\mathfrak{G}_{cell}$ has at least $k$ curves of $\Gamma_{cell}$ passing through it, of each of which it is not an isolated point, we obtain (since $k \geq 3$, and thus $\geq 2$), 
\begin{displaymath} S_{cell} \lesssim_b |\Gamma_{cell}|^2/k^{(2D_b-1)/(D_b-1)} +|\Gamma_{cell}|/ k.\end{displaymath}
Therefore, $S_{cell} \lesssim_b |\Gamma_{cell}|^2/k^{(2D_b-1)/(D_b-1)}$, so, since we are working in a full cell, $Sd^{-3} \lesssim_b |\Gamma_{cell}|^2/k^{(2D_b-1)/(D_b-1)}$, and rearranging we see that 
\begin{displaymath} |\Gamma_{cell}| \gtrsim_b S^{1/2}d^{-3/2}k^{(2D_b-1)/(2D_b-2)}.\end{displaymath}
Furthermore, let $\Gamma_Z$ be the set of curves of $\Gamma$ which are lying in $Z$. Obviously, $\Gamma_{cell} \subset \Gamma \setminus \Gamma_Z$. Moreover, let $\Gamma_{cell}'$ be the set of curves in $\Gamma_{cell}$ such that, if $\gamma \in \Gamma_{cell}'$, there does not exist any point $x$ in the intersection of $\gamma$ with the boundary of the cell, with the property that the induced topology from $\R^3$ to the intersection of $\gamma$ with the closure of the cell contains some open neighbourhood of $x$. Finally, let $I_{cell}$ denote the number of incidences between the boundary of the cell and the curves in $\Gamma_{cell} \setminus \Gamma'_{cell}$.

Now, each of the curves in $\Gamma_{cell}\setminus \Gamma'_{cell}$ intersects the boundary of the cell at at least one point $x$, with the property that the induced topology from $\R^3$ to the intersection of the curve with the closure of the cell contains an open neighbourhood of $x$; therefore, $I_{cell}\geq |\Gamma_{cell}\setminus \Gamma'_{cell}|$ ($=|\Gamma_{cell}|-|\Gamma'_{cell}|$). Also, the union of the boundaries of all the cells is the zero set $Z$ of $p$, and if $x$ is a point of $Z$ which belongs to a curve in $\Gamma$ intersecting the interior of a cell, such that the induced topology from $\R^3$ to the intersection of the curve with the closure of the cell contains an open neighbourhood of $x$, then there exist at most $2b-1$ other cells whose interior is also intersected by the curve and whose boundary contains $x$, such that the induced topology from $\R^3$ to the intersection of the curve with the closure of each of these cells contains some open neighbourhood of $x$. So, if $I$ is the number of incidences between $Z$ and $\Gamma \setminus \Gamma_Z$, and $\mathcal{C}$ is the set of all the full cells (which, in this case, has cardinality $\sim d^3$), then
\begin{displaymath} I \geq \frac{1}{2b} \cdot \sum_{cell \;\in \;\mathcal{C}} I_{cell}\geq \frac{1}{2b}\cdot \sum_{cell \;\in\; \mathcal{C}}(|\Gamma_{cell}|-|\Gamma_{cell}'|). \end{displaymath}

Now, if $\gamma \in \Gamma_{cell}$ ($\supseteq \Gamma'_{cell}$), we consider the (unique, irreducible) complex algebraic curve $\gamma_{\C}$ in $\C^3$ which contains $\gamma$. In addition, let $p_{\C}$ be the polynomial $p$ viewed as an element of $\C[x,y,z]$, and $Z_{\C}$ the zero set of $p_{\C}$ in $\C^3$. The polynomial $p$ was constructed in such a way that $\gamma_{\C}$ intersects each of 6 complex planes, each of which contains one of the real planes in $Z$ that each contain a different face of the cube $Q$; consequently $\gamma_{\C}$ intersects $Z_{\C}$ at least once. Moreover, if $\gamma^{(1)}$, $\gamma^{(2)}$ are two distinct curves in $\Gamma$, then $\gamma^{(1)}_{\C}$, $\gamma^{(2)}_{\C}$ are two distinct curves in $\Gamma_{\C}$ (since $\gamma^{(1)}=\gamma^{(1)}_{\C}\cap \R^3$, while $\gamma^{(1)}=\gamma^{(1)}_{\C}\cap \R^3$). So, if $\Gamma_{\C}=\{\gamma_{\C}: \gamma \in \Gamma_{cell}\text{, for some cell in }\mathcal{C}\}$ and $I_{\C}$ denotes the number of incidences between $\Gamma_{\C}$ and $Z_{\C}$, it follows that 
$$I_{\C} \geq|\Gamma_{\C}|=|\Gamma|\geq\big|\bigcup_{cell \;\in\; \mathcal{C}}\Gamma_{cell}'\big|,
$$while also
$$I_{\C} \geq I.
$$
Therefore,
$$I_{\C} \geq \frac{1}{2}\cdot (I+I_{\C}) \geq
$$
\begin{displaymath}\geq \frac{1}{2} \cdot \bigg( \frac{1}{2b} \sum_{cell \;\in\; \mathcal{C}}(|\Gamma_{cell}|-|\Gamma_{cell}'|) + \big|\bigcup_{cell \;\in \;\mathcal{C}}\Gamma_{cell}'\big|\bigg) \sim_b \end{displaymath}
\begin{displaymath} \sim_b\sum_{cell \;\in\; \mathcal{C}}(|\Gamma_{cell}|-|\Gamma_{cell}'|) + \big|\bigcup_{cell \;\in \;\mathcal{C}}\Gamma_{cell}'\big|.\end{displaymath} However, each real algebraic curve in $\R^3$, of degree at most $b$, is the disjoint union of $\leq R_b$ path-connected components, for some constant $R_b$ depending only on $b$ (by Lemma \ref{4.1.15}). Therefore, 
$$\big|\bigcup_{cell \;\in \;\mathcal{C}}\Gamma_{cell}'\big| \sim_b\sum_{cell \;\in \;\mathcal{C}}|\Gamma_{cell}'|,
$$
from which it follows that
 \begin{displaymath}I_{\C}\gtrsim_b \sum_{cell \;\in\; \mathcal{C}}(|\Gamma_{cell}|-|\Gamma_{cell}'|)+\sum_{cell \;\in \;\mathcal{C}}|\Gamma_{cell}'|\sim_b \end{displaymath} \begin{displaymath}\sim_b\sum_{cell \;\in \;\mathcal{C}}|\Gamma_{cell}|\gtrsim_b \sum_{cell \;\in \;\mathcal{C}} S^{1/2}d^{-3/2}k^{(2D_b-1)/(2D_b-2)}\sim_b\end{displaymath} \begin{displaymath}\sim_b \big(S^{1/2}d^{-3/2}k^{(2D_b-1)/(2D_b-2)}\big)\cdot d^3 \sim_b S^{1/2}d^{3/2}k^{(2D_b-1)/(2D_b-2)}.\end{displaymath}
On the other hand, however, each $\gamma_{\C}\in \Gamma_{\C}$ is a complex algebraic curve in $\R^3$, of degree at most $b$, which does not lie in $Z_{\C}$, and thus intersects $Z_{\C}$ at most $b \cdot \deg p$ times. So,
$$I_{\C} \lesssim_b |\Gamma_{\C}| \cdot d \sim_b |\Gamma| \cdot d,
$$
and therefore
\begin{displaymath}S^{1/2}d^{3/2}k^{(2D_b-1)/(2D_b-2)} \lesssim_b |\Gamma| \cdot d,\end{displaymath}
which in turn gives $A_b\lesssim_b 1$. In other words, there exists some constant $C_b$, depending only on $b$, such that $A_b \leq C_b$. By fixing $A_b$ to be a constant larger than $C_b$ (and of course $\geq 2c_{0,b}$, so that $d> 1$), we have a contradiction.

Therefore, in the cellular case there exists some constant $c_{1,b}$, depending only on $b$, such that  \begin{displaymath} S \cdot N^{1/2} \leq c_{1,b} \cdot \frac{|\Gamma|^{3/2}}{k^{1/(2D_b-2)}}.\end{displaymath}

\textbf{Algebraic case:} There exist $<10^{-8}S$ points of $\mathfrak{G}$ in the union of the interiors of the cells. 

We denote by $\Gamma'$ the set of curves in $\Gamma$ each of which contains $\geq \frac{1}{100}Sk|\Gamma|^{-1}$ points of $\mathfrak{G}\cap Z$ which are not isolated points of the curve, and we continue by adapting, to this setting, the proof of Proposition \ref{1.2}, using Corollary \ref{4.1.14} and Lemmas \ref{4.2.2} and \ref{4.2.3}. 

\end{proof}

We are now able to count, with multiplicities, joints formed by a finite collection of curves in $\R^3$, parametrised by real univariate polynomials of uniformly bounded degree.

\begin{corollary} \label{jointsforcurves} Let $b$ be a positive constant and $\Gamma$ a finite collection of curves in $\R^3$, such that each $\gamma \in \Gamma$ is parametrised by $t \rightarrow \big(p^{\gamma}_1(t)$, $p_2^{\gamma}(t)$, $p_3^{\gamma}(t)\big)$ for $t \in \R$, where the $p_i^{\gamma} \in \R[t]$, for $i=1,2,3$, are polynomials not simultaneously constant, of degrees at most $b$. Let $J$ be the set of joints formed by $\Gamma$. Then,
\begin{displaymath} \sum_{x \in J}N(x)^{1/2} \leq c_b \cdot |\Gamma|^{3/2},\end{displaymath}
where $c_b$ is a constant depending only on $b$.
\end{corollary}

\begin{proof} By Corollary \ref{parametrisedcurves}, each $\gamma \in \Gamma$ is contained in a real algebraic curve in $\R^3$, of degree at most $b$. Therefore, the statement of the Corollary immediately follows from Theorem \ref{4.2.1}.

\end{proof}

\section{Multijoints}

\begin{definition} Let $\Gamma_1$, $\Gamma_2$, $\Gamma_3$ be collections of sets in $\mathcal{F}$.  Then a point $x$ in $\R^3$ is a multijoint for the three collections if

\emph{(i)} $x$ belongs to at least one of the sets in $\Gamma_i$, for all $i=1,2,3$, and \newline
\emph{(ii)} there exists at least one vector $v_i$ in $T_x^{\Gamma_i}$, for all $i=1,2,3$, such that the set $\{v_1,v_2,v_3\}$ spans $\R^3$.
\end{definition}

We will show here that, under certain assumptions on the properties of the sets in finite collections $\Gamma_1$, $\Gamma_2$ and $\Gamma_3$ in $\mathcal{F}$, the corresponding statement of Theorem \ref{theoremmult2} still holds.

Indeed, thanks to the results of Chapter \ref{5}, we are now ready to formulate and prove the following extension of Theorem \ref{theoremmult2}.

\begin{theorem} \label{theoremmult3} Let $b$ be a positive constant, and $\Gamma_1$, $\Gamma_2$, $\Gamma_3$ finite collections of real algebraic curves in $\R^3$, of degree at most $b$. Let $J$ be the set of multijoints formed by $\Gamma_1$, $\Gamma_2$ and $\Gamma_3$. Then,
\begin{displaymath}|J| \leq c_b \cdot(|\Gamma_1||\Gamma_2||\Gamma_3|)^{1/2}, \end{displaymath}
where $c_b$ is a constant depending only on $b$.
\end{theorem}

\textbf{Remark.} We would like to emphasise that we have not yet achieved a similar extension of Theorem \ref{theoremmult1}, even though we believe that one exists. The reason is that we have not yet managed to obtain computational results regarding flat curves of algebraic surfaces, a notion corresponding to the one of flat lines, which has been an essential ingredient of our proof of Theorem \ref{theoremmult1}. 

$\;\;\;\;\;\;\;\;\;\;\;\;\;\;\;\;\;\;\;\;\;\;\;\;\;\;\;\;\;\;\;\;\;\;\;\;\;\;\;\;\;\;\;\;\;\;\;\;\;\;\;\;\;\;\;\;\;\;\;\;\;\;\;\;\;\;\;\;\;\;\;\;\;\;\;\;\;\;\;\;\;\;\;\;\;\;\;\;\;\;\;\;\;\;\;\;\;\;\;\;\;\;\;\;\;\;\;\;\;\;\;\;\;\;\;\;\;\;\;\;\;\;\;\;\;\;\;\;\;\;\;\;\;\blacksquare$

In analogy to Theorem \ref{theoremmult2}, Theorem \ref{theoremmult3} is an immediate corollary of the following proposition.

\begin{proposition} Let $b$ be a positive constant, and $\Gamma_1$, $\Gamma_2$, $\Gamma_3$ finite collections of real algebraic curves in $\R^3$, of degree at most $b$.

For all $(N_1,N_2,N_3) \in \R_+^3$, let $J'_{N_1,N_2,N_3}$ be the set of multijoints formed by $\Gamma_1$, $\Gamma_2$ and $\Gamma_3$, with the property that, if $x \in J'_{N_1,N_2,N_3}$, then there exist collections $\Gamma_1(x) \subseteq \Gamma_1$, $\Gamma_2(x) \subseteq \Gamma_2$ and $\Gamma_3(x) \subseteq \Gamma_3$ of curves passing through $x$, such that $|\Gamma_1(x)|\geq N_1$, $|\Gamma_2(x)|\geq N_2$ and $|\Gamma_3(x)|\geq N_3$, and, if $\gamma_1 \in \Gamma_1(x)$, $\gamma_2 \in \Gamma_2(x)$ and $\gamma_3 \in \Gamma_3(x)$, then there exist vectors $v_1 \in T_x^{\gamma_1}$, $v_2 \in T_x^{\gamma_2}$ and $v_3 \in T_x^{\gamma_3}$ that span $\R^3$. Then,
\begin{displaymath}  |J'_{N_1,N_2,N_3}|\leq c_b \cdot \frac{(|\Gamma_1||\Gamma_2||\Gamma_3|)^{1/2}}{(N_1N_2N_3)^{1/2}}, \; \forall\;(N_1,N_2,N_3) \in\R_{+}^3,\end{displaymath}
where $c_b$ is a constant depending only on $b$.
\end{proposition}

\begin{proof} Each real algebraic curve in $\R^3$, of degree at most $b$, consists of $\leq b \lesssim_b 1$ irreducible components; we may therefore assume that each $\gamma \in \Gamma_1\cup \Gamma_2\cup\Gamma_3$ is irreducible.

The proof will be achieved by induction on the cardinalities of $\Gamma_1$, $\Gamma_2$ and $\Gamma_3$. Indeed, fix $(M_1,M_2,M_3) \in {\N^*}^3$. For $c_b$ an explicit constant $\geq b^{2}$, which depends only on $b$ and will be specified later:

(i) For any collections $\Gamma_1$, $\Gamma_2$ and $\Gamma_3$ of irreducible real algebraic curves in $\R^3$, of degree at most $b$, such that $|\Gamma_1|=|\Gamma_2|=|\Gamma_3|=1$, we have that
\begin{displaymath} |J'_{N_1,N_2,N_3}|  \leq c_b \cdot \frac{(|\Gamma_1||\Gamma_2||\Gamma_3|)^{1/2}}{(N_1N_2N_3)^{1/2}}, \; \forall\;(N_1,N_2,N_3)\in\R_{+}^3.\end{displaymath}
This is obvious, in fact, for any $c_b \geq b^{2}$, as in this case $|J'_{N_1,N_2,N_3}|=0$ for all $(N_1,N_2,N_3)$ in $\R^3_+$ such that $N_i \gneq 1$ for some $i \in \{1,2,3\}$, while, for $(N_1,N_2,N_3)$ in $\R^3_+$ such that $N_i \leq 1$ for all $i \in \{1,2,3\}$, $|J'_{N_1,N_2,N_3}|$ is equal to at most the number of intersections between the curve in $\Gamma_1$ and the curve in $\Gamma_2$, and thus equal to at most $b^2$.

(ii) Suppose that \begin{displaymath} |J'_{N_1,N_2,N_3}|  \leq c_b \cdot \frac{(|\Gamma_1||\Gamma_2||\Gamma_3|)^{1/2}}{(N_1N_2N_3)^{1/2}}, \; \forall\;(N_1,N_2,N_3) \in \R_{+}^3,\end{displaymath} for any collections $\Gamma_1$, $\Gamma_2$ and $\Gamma_3$ of irreducible real algebraic curves in $\R^3$, of degree at most $b$, such that $|\Gamma_1|\lneq M_1$, $|\Gamma_2| \lneq M_2$ and $|\Gamma_3| \lneq M_3$.

(iii) We will prove that  \begin{displaymath} |J'_{N_1,N_2,N_3}|  \leq c \cdot \frac{(|\Gamma_1||\Gamma_2||\Gamma_3|)^{1/2}}{(N_1N_2N_3)^{1/2}}, \; \forall\;(N_1,N_2,N_3) \in \R_{+}^3,\end{displaymath} for any collections $\Gamma_1$, $\Gamma_2$ and $\Gamma_3$ of irreducible real algebraic curves in $\R^3$, of degree at most $b$, such that $|\Gamma_j|= M_j$ for some $j \in \{1,2,3\}$ and $|\Gamma_i| \lneq M_i$, $|\Gamma_k| \lneq M_k$ for $\{i,k\}=\{1,2,3\}\setminus \{j\}$.

Indeed, fix such collections $\Gamma_1$, $\Gamma_2$ and $\Gamma_3$ of real algebraic curves, and $(N_1,N_2,N_3) \in\R_{+}^3$.

For simplicity, let $$\mathfrak{G}:= J'_{N_1,N_2,N_3}$$and$$S:=|J'_{N_1,N_2,N_3}|.$$

Now, the proof is completely analogous to that of Proposition \ref{multsimple}. The main differences lie at the beginning and the cellular case, we thus go on to point them out.

We assume that \begin{displaymath} \frac{|\Gamma_1|}{\lceil N_1 \rceil} \leq \frac{|\Gamma_2|}{\lceil N_2\rceil} \leq \frac{|\Gamma_3|}{\lceil N_3 \rceil}. \end{displaymath}

By the definition of the set $\mathfrak{G}$, each point of $\mathfrak{G}$ lies in $\geq \lceil N_1 \rceil$ curves of $\Gamma_1$ and $\geq \lceil N_2\rceil$ curves of $\Gamma_2$. Thus, the quantity $S\lceil N_1\rceil\lceil N_2\rceil$ is equal to at most the number of pairs of the form $(\gamma_1,\gamma_2)$, where $\gamma_1 \in \Gamma_1$, $\gamma_2 \in \Gamma_2$ and the curves $\gamma_1$ and $\gamma_2$ pass through the same point of $\mathfrak{G}$. Therefore, $S\lceil N_1\rceil\lceil N_2\rceil$ is equal to at most the number of all the pairs of the form $(\gamma_1,\gamma_2)$, where $\gamma_1 \in \Gamma_1$ and $\gamma_2 \in \Gamma_2$, i.e. to at most $|\Gamma_1||\Gamma_2|$. So,
\begin{displaymath}S\lceil N_1\rceil\lceil N_2\rceil \leq |\Gamma_1||\Gamma_2|,\end{displaymath}
and therefore
\begin{displaymath} \frac{|\Gamma_1||\Gamma_2|}{S\lceil N_1\rceil\lceil N_2\rceil}\geq 1. \end{displaymath}

Thus, $d:=A\frac{|\Gamma|_1|\Gamma|_2}{S\lceil N_1\rceil\lceil N_2\rceil}$ is a quantity $>1$ for $A>1$. We therefore assume that $A>1$, and we will specify its value later. Now, applying the Guth-Katz polynomial method for this $d>1$ and the finite set of points $\mathfrak{G}$, we deduce that there exists a non-zero polynomial $p\in \R[x,y,z]$, of degree $\leq d$, whose zero set $Z$:

(i) decomposes $\R^3$ in $\sim d^3$ cells, each of which contains $\lesssim Sd^{-3}$ points of $\mathfrak{G}$, and

(ii) contains six distinct generic planes, each of which contains a face of a fixed cube $Q$ in $\R^3$, such that the interior of $Q$ contains $\mathfrak{G}$ (and each of the planes is generic in the sense that the plane in $\C^3$ containing it intersects the smallest complex algebraic curve in $\C^3$ containing $\gamma$, for all $\gamma \in \Gamma_1 \cup \Gamma_2$);

to achieve this, we first fix a cube $Q$ in $\R^3$, with the property that its interior contains $\mathfrak{G}$ and the planes containing its faces are generic in the above sense. Then, we multiply the polynomials we end up with at each step of the Guth-Katz polynomial method with the same (appropriate) six linear polynomials, the zero set of each of which is a plane containing a different face of the cube, and stop the application of the method when we finally get a polynomial of degree at most $d$, whose zero set decomposes $\R^3$ in $\lesssim d^3$ cells (the set of the cells now consists of the non-empty intersections of the interior of the cube $Q$ with the cells that arise from the application of the Guth-Katz polynomial method, as well as the complement of the cube).

We can assume that the polynomial $p$ is square-free, as eliminating the squares of $p$ does not inflict any change on its zero set.

Let us now assume that there are $\geq 10^{-8}S$ points of $\mathfrak{G}$ in the union of the interiors of the cells; by choosing to be $A$ a sufficiently large constant depending only on $b$, we will be led to a contradiction.

Indeed, there are $\gtrsim S$ points of $\mathfrak{G}$ in the union of the interiors of the cells. However, we also know that there exist $\sim d^3$ cells in total, each with $\lesssim Sd^{-3}$ points of $\mathfrak{G}$. Therefore, there exist $\gtrsim d^3$ cells, with $\gtrsim Sd^{-3}$ points of $\mathfrak{G}$ in the interior of each. We call the cells with this property ``full cells".

Now, for every full cell,  let $\mathfrak{G}_{cell}$ be the set of points of $\mathfrak{G}$ in the interior of the cell, $\Gamma_{1,cell}$ and $\Gamma_{2,cell}$ the sets of curves in $\Gamma_1$ and $\Gamma_2$, respectively, each containing at least one point of $\mathfrak{G}_{cell}$, and $S_{cell}:=|\mathfrak{G}_{cell}|$. Now,
\begin{displaymath}S_{cell}\lceil N_1\rceil \lceil N_2\rceil  \lesssim |\Gamma_{1,cell}||\Gamma_{2,cell}|, \end{displaymath}
as the quantity $S_{cell}\lceil N_1\rceil \lceil N_2\rceil $ is equal to at most the number of pairs of the form $(\gamma_1,\gamma_2)$, where $\gamma_1 \in \Gamma_{1,cell}$, $\gamma_2 \in \Gamma_{2,cell}$ and the curves $\gamma_1$ and $\gamma_2$ pass through the same point of $\mathfrak{G}_{cell}$. Thus, $S_{cell}\lceil N_1\rceil \lceil N_2\rceil $ is equal to at most the number of all the pairs of the form $(\gamma_1,\gamma_2)$, where $\gamma_1 \in \Gamma_{1,cell}$ and $\gamma_2 \in \Gamma_{2,cell}$, i.e. to at most $|\Gamma_{1,cell}||\Gamma_{2,cell}|$.

Therefore,
\begin{displaymath}(|\Gamma_{1,cell}||\Gamma_{2,cell}|)^{1/2} \gtrsim S_{cell}^{1/2} (\lceil N_1\rceil\lceil N_2\rceil )^{1/2} \gtrsim \frac{S^{1/2}}{d^{3/2}}(\lceil N_1\rceil\lceil N_2\rceil)^{1/2}. \end{displaymath}
Furthermore, for every full cell and $i \in \{1,2\}$, let $\Gamma_{i,Z}$ be the set of curves of $\Gamma_i$ which are lying in $Z$. Obviously, $\Gamma_{i,cell} \subset \Gamma_i \setminus \Gamma_{i,Z}$. Moreover, let $\Gamma_{i,cell}'$ be the set of curves in $\Gamma_{i,cell}$ such that, if $\gamma \in \Gamma_{i,cell}'$, there does not exist any point $x$ in the intersection of $\gamma$ with the boundary of the cell, with the property that the induced topology from $\R^3$ to the intersection of $\gamma$ with the closure of the cell contains some open neighbourhood of $x$. Finally, let $I_{i,cell}$ denote the number of incidences between the boundary of the cell and the curves in $\Gamma_{i,cell}$.

Now, let $i \in \{1,2\}$. Each of the curves in $\Gamma_{i,cell}\setminus \Gamma'_{i,cell}$ intersects the boundary of the cell at at least one point $x$, with the property that the induced topology from $\R^3$ to the intersection of the curve with the closure of the cell contains an open neighbourhood of $x$; therefore, $I_{i,cell}\geq |\Gamma_{i,cell}\setminus \Gamma'_{i,cell}|$ ($=|\Gamma_{i,cell}|-|\Gamma'_{i,cell}|$). Also, the union of the boundaries of all the cells is the zero set $Z$ of $p$, and if $x$ is a point of $Z$ which belongs to a curve in $\Gamma_i$ intersecting the interior of a cell, such that the induced topology from $\R^3$ to the intersection of the curve with the closure of the cell contains an open neighbourhood of $x$, then there exist at most $2b-1$ other cells whose interior is also intersected by the curve and whose boundary contains $x$, such that the induced topology from $\R^3$ to the intersection of the curve with the closure of each of these cells contains some open neighbourhood of $x$. So, if $I_i$ is the number of incidences between $Z$ and $\Gamma_i \setminus \Gamma_{i,Z}$, and $\mathcal{C}$ is the set of all the full cells (which, in this case, has cardinality $\gtrsim d^3$), then
\begin{displaymath} I_i \geq \frac{1}{2b} \cdot \sum_{cell \;\in \;\mathcal{C}} I_{i,cell} \geq \frac{1}{2b}\cdot \sum_{cell \;\in \;\mathcal{C}}(|\Gamma_{i,cell}|-|\Gamma'_{i,cell}|).\end{displaymath}
Now, if $\gamma \in \Gamma_{i,cell}$ ($\supseteq \Gamma_{i,cell}'$), we consider the (unique, irreducible) complex algebraic curve $\gamma_{\C}$ in $\C^3$ which contains $\gamma$. In addition, let $p_{\C}$ be the polynomial $p$ viewed as an element of $\C[x,y,z]$, and $Z_{\C}$ the zero set of $p_{\C}$ in $\C^3$. The polynomial $p$ was constructed in such a way that $\gamma_{\C}$ intersects each of 6 complex planes, each of which contains one of the real planes in $Z$ that each contain a different face of the cube $Q$; consequently $\gamma_{\C}$ intersects $Z_{\C}$ at least once. Moreover, if $\gamma^{(1)}$, $\gamma^{(2)}$ are two distinct curves in $\Gamma_i$, then $\gamma^{(1)}_{\C}$, $\gamma^{(2)}_{\C}$ are two distinct curves in $\Gamma_{\C}$ (since $\gamma^{(1)}=\gamma^{(1)}_{\C}\cap \R^3$, while $\gamma^{(2)}=\gamma^{(2)}_{\C}\cap \R^3$). So, if $\Gamma_{i,\C}=\{\gamma_{\C}: \gamma \in \Gamma_{i,cell}$, for some cell in $\mathcal{C}\}$ and $I_{i,\C}$ denotes the number of incidences between $\Gamma_{i,\C}$ and $Z_{\C}$, it follows that
$$I_{i,\C}\geq |\Gamma_{i,\C}|=|\Gamma_i|\geq \big|\bigcup_{cell \;\in\; \mathcal{C}}\Gamma_{i,cell}'\big|,
$$while also
$$I_{i,\C}\geq I_i.
$$
Therefore,
$$I_{i,\C}\geq \frac{1}{2}(I_i+I_{i,\C})\geq
$$
\begin{displaymath}\geq \frac{1}{2} \cdot \bigg( \frac{1}{2b} \sum_{cell \;\in\; \mathcal{C}}(|\Gamma_{i,cell}|-|\Gamma_{i,cell}'|) + \big|\bigcup_{cell \;\in \;\mathcal{C}}\Gamma_{i,cell}'\big|\bigg) \sim_b \end{displaymath}
\begin{displaymath} \sim_b\sum_{cell \;\in\; \mathcal{C}}(|\Gamma_{i,cell}|-|\Gamma_{i,cell}'|) + \big|\bigcup_{cell \;\in \;\mathcal{C}}\Gamma_{i,cell}'\big|.\end{displaymath}
However, each real algebraic curve in $\R^3$, of degree at most $b$, is the disjoint union of $\leq R_b$ path-connected components, for some constant $R_b$ depending only on $b$ (by Lemma \ref{4.1.15}). Hence,
$$\big|\bigcup_{cell \;\in \;\mathcal{C}}\Gamma_{i,cell}'\big| \sim_b\sum_{cell \;\in \;\mathcal{C}}|\Gamma_{i,cell}'|,
$$
from which it follows that
 \begin{displaymath}I_{i,\C}\gtrsim_b \sum_{cell \;\in\; \mathcal{C}}(|\Gamma_{i,cell}|-|\Gamma_{i,cell}'|)+\sum_{cell \;\in \;\mathcal{C}}|\Gamma_{i,cell}'|\sim_b\end{displaymath}
$$\sim_b \sum_{cell \;\in \;\mathcal{C}}|\Gamma_{i,cell}|.
$$
On the other hand, however, each $\gamma_{\C}\in \Gamma_{i,\C}$ is a complex algebraic curve of degree at most $b$ which does not lie in $Z_{\C}$, and thus intersects $Z_{\C}$ at most $b \cdot \deg p$ times. Therefore,
\begin{displaymath}I_{i,\C} \lesssim_b |\Gamma_i| \cdot d.
\end{displaymath}
So, for $i \in \{1,2\}$, it holds that
$$\sum_{cell \;\in \;\mathcal{C}}|\Gamma_{i,cell}| \lesssim_b |\Gamma_i| \cdot d.
$$
Hence, from all the above we obtain
\begin{displaymath} \sum_{cell \in \mathcal{C}} \frac{S^{1/2}}{d^{3/2}}(\lceil N_1\rceil \lceil N_2\rceil )^{1/2} 
\lesssim_b \sum_{cell \in \mathcal{C}} (|\Gamma_{1,cell}||\Gamma_{2,cell}|)^{1/2} \lesssim_b \end{displaymath}
\begin{displaymath} \lesssim_b\Bigg(\sum_{cell \in\mathcal{C}} |\Gamma_{1,cell}|\Bigg)^{1/2} \Bigg(\sum_{cell \in\mathcal{C}}|\Gamma_{2,cell}|\Bigg)^{1/2}\lesssim_b \end{displaymath}
\begin{displaymath} \lesssim_b (|\Gamma_1|\cdot d)^{1/2}(|\Gamma_2|\cdot d)^{1/2} \sim_b (|\Gamma_1||\Gamma_2|)^{1/2}d.\end{displaymath}

But the full cells number $\gtrsim d^3$. Thus,
$$d^{3/2}S^{1/2}(\lceil N_1\rceil \lceil N_2\rceil )^{1/2} \lesssim_b (|\Gamma_1||\Gamma_2|)^{1/2}d,
$$
which in turn gives $A\lesssim_b 1$. In other words, there exists some constant $C_b$, depending only on $b$, such that $A \leq C_b$. By fixing $A$ to be a number larger than $C_b$ (and of course large enough to have that $d> 1$), we are led to a contradiction.

Therefore, for $A$ a sufficiently large constant that depends only on $b$, it holds that more than $(1-10^{-8})S$ points of $\mathfrak{G}$ lie in the zero set of $p$. 

The rest of the proof follows in a similar way as the proof of Proposition \ref{multsimple}.

\end{proof}

We are now able to count multijoints formed by a finite collection of curves in $\R^3$, parametrised by real univariate polynomials of uniformly bounded degree.

\begin{corollary} \label{multijointsforcurves} Let $b$ be a positive constant and $\Gamma_1$, $\Gamma_2$, $\Gamma_3$ finite collections of curves in $\R^3$, such that, for $j=1,2,3$, each $\gamma \in \Gamma_j$ is parametrised by $t \rightarrow \big(p^{\gamma}_1(t), p_2^{\gamma}(t),p_3^{\gamma}(t)\big)$ for $t \in \R$, where the $p_i^{\gamma} \in \R[t]$, for $i=1,2,3$, are polynomials not simultaneously constant, of degrees at most $b$. Let $J$ be the set of multijoints formed by $\Gamma_1$, $\Gamma_2$ and $\Gamma_3$. Then,
\begin{displaymath}|J| \leq c_b \cdot(|\Gamma_1||\Gamma_2||\Gamma_3|)^{1/2}, \end{displaymath}
where $c_b$ is a constant depending only on $b$.
\end{corollary}

\begin{proof} By Corollary \ref{parametrisedcurves}, each $\gamma \in \Gamma_j$, for $j=1,2,3$, is contained in a real algebraic curve in $\R^3$, of degree at most $b$. Therefore, the statement of the Corollary immediately follows from Theorem \ref{theoremmult3}.

\end{proof}

\newpage
\thispagestyle{plain}
\cleardoublepage

\chapter{Different field settings and higher dimensions} \label{7}

As we have already mentioned, the algebraic methods we are using to count joints and multijoints have certain limitations. More particularly, the solutions that we are providing for these problems cannot immediately be applied for more than three dimensions, as a crucial part of our proofs is that the number of critical lines of a real algebraic hypersurface in $\R^3$ is bounded, a fact which we do not know if is always true in higher dimensions. However, there is no substantial reason why the corresponding results should not hold in higher dimensions. More importantly, though, certain limitations of our techniques lie in the fact that they take advantage of the topology and the continuous nature of euclidean space, and eventually of theorems that rely on them (like the Szemer\'{e}di-Trotter and the Borsuk-Ulam theorems), rather than combinatorial estimates arising from the geometric nature of our problems. Therefore, even though, for any field $\mathbb{F}$ and any $n \geq 2$, we can naturally define a joint formed by a finite collection $\mathfrak{L}$ of lines in $\mathbb{F}^n$ as a point of $\mathbb{F}^n$ that lies in the intersection of at least $n$ lines of $\mathfrak{L}$ whose directions span $\mathbb{F}^n$, we cannot immediately apply our algebraic techniques to count joints with multiplicities in that setting.

In particular, the Guth-Katz polynomial method leads to a decomposition of $\R^n$ by the zero set of a polynomial. There are facts, though, which demonstrate an obstruction to the application of the method in the case of $\mathbb{F}^n$, where $\mathbb{F}$ is an arbitrary field. We now go ahead and discuss these facts, starting from the more technical and moving on to the more substantial:

(i) When the Guth-Katz polynomial method is applied in $\R^n$, the cells in which the zero set of the resulting polynomial decomposes $\R^n$ (i.e. the cells in the Guth-Katz decomposition theorem \cite[Corollary 4.4]{Guth_Katz_2010}, Corollary \ref{2.1.2} here) are defined as subsets of $\R^n$, on each of which certain real polynomials in $n$ variables are positive. Now, in the case of $\mathbb{F}^n$, where $\mathbb{F}$ is an arbitrary field, such a characterisation of a cell cannot necessarily be achieved, as it requires a notion of positivity in the field. Particularly in the case where $\mathbb{F}$ is a field of non-zero characteristic (in other words, a finite field or an infinite field with a finite subfield), there is no total order in the field that is compatible with the field operations, and therefore cells would have to be defined in a different way.

(ii) If $\mathbb{F}$ is a finite field with a Hausdorff (and therefore the discrete) topology, then the Boruk-Ulam theorem (on which the Guth-Katz decomposition technique is based in euclidean space) does not hold in this field setting. More particularly, for any $K$ subset of $\mathbb{F}^{n+1}$ (and thus compact with the discrete topology), such that $K=-K$, there exists a continuous and odd map $f:K \rightarrow \mathbb{F}^n$ that sends no point of $K$ to $0 \in \mathbb{F}^n$; such a map can be constructed by decomposing $K$ in the disjoint union of two sets $K_1$ and $K_2$, such that $K_2=K \cap \{-x: x \in K_1\}$, and then defining $(1,...,1) \in \mathbb{F}^n$ as the image of any point of $K_1$ through $f$, while $(-1,...,-1) \in \mathbb{F}^n$ as the image of any point of $K_2$ through $f$.

Therefore, in the case where $\mathbb{F}$ is a finite field, it is not possible to deduce an analogue of the Guth-Katz decomposition theorem \cite[Corollary 4.4]{Guth_Katz_2010} in $\mathbb{F}^n$ using similar techniques as in $\R^n$; any attempt to acquire a version of \cite[Corollary 4.4]{Guth_Katz_2010} in $\mathbb{F}^n$ would require a substantially different approach.

(iii) Even if $\mathbb{F}$ is a field of zero characteristic with an order compatible with the field operations (note that $\mathbb{Q}$ and $\R$ are examples of such fields), we are not necessarily able to deduce the Guth-Katz decomposition theorem \cite[Corollary 4.4]{Guth_Katz_2010} (i.e. Corollary \ref{2.1.2} here) following a similar procedure as the one described in Section \ref{section2.1} for the ordered field $\R$. 

Indeed, as we describe in Section \ref{section2.1}, the Guth-Katz polynomial technique in $\R^n$, which results in \cite[Corollary 4.4]{Guth_Katz_2010}, consists of successive applications of Corollary \ref{2.1.2}, which states that, for any $S_1$, ..., $S_M$ finite, disjoint sets of points in $\R^n$, where $M=\binom{d+n}{n}-1$, there exists a non-zero polynomial in $\R[x_1,...,x_n]$, of degree $\leq d$, whose zero set bisects each $S_i$. Therefore, deducing in a similar way an analogue of the Guth-Katz decomposition theorem in $\mathbb{F}^n$ would require an analogue of Corollary \ref{2.1.2} in $\mathbb{F}^n$, stating that, for any $S_1$, ..., $S_M$ finite, disjoint sets of points in $\mathbb{F}^n$, where $M=\binom{d+n}{n}-1$, there exists a non-zero polynomial in $\mathbb{F}[x_1,...,x_n]$, of degree $\leq d$, whose zero set bisects each $S_i$.

Now, Corollary \ref{2.1.2} is based on Theorem \ref{2.1.1} by Stone and Tukey, which states that, for any $U_1,$ ..., $U_M$ Lebesgue-measurable sets in $\R^n$ of finite, positive volume, where $M=\binom{d+n}{n}-1$, there exists a non-zero polynomial in $\R[x_1,...,x_n]$, of degree $\leq d$, whose zero set bisects each $U_i$. In particular, Corollary \ref{2.1.2} follows from Theorem \ref{2.1.1}, due to the fact that the usual metric in $\R^n$ is such that each point of $\R^n$ has an arbitrarily small (with respect to the metric) open neighbourhood of finite, positive Lebesgue measure (which is translation invariant), polynomials in $\R[x_1,...,x_n]$ are continuous functions, the unit sphere in $\R^n$ is compact and, if $U$ is a Lebesgue-measurable set in $\R^n$ of finite, positive volume, then the function that sends any polynomial $p$ of degree at most $d$ to the Lebesgue measure of $U \cap \{p>0\}$ is continuous. 

Theorem \ref{2.1.1} is, in turn, an immediate corollary of the Borsuk-Ulam theorem in $\R^n$, a theorem which, again, relies on the topology and the continuous nature of euclidean space.

Therefore, establishing a corresponding version of the Guth-Katz polynomial decomposition theorem in $\mathbb{F}^n$ via a similar reasoning would require the existence of a metric and a translation invariant measure in $\mathbb{F}^n$, well-defined on open balls of $\mathbb{F}^n$, with properties as above; in fact, we believe that it would be more natural for the metric and the measure to take values in $\mathbb{F}$ and not in $\R$ (this will be made more clear by the remark after paragraph (iv) that follows). Moreover, we would need a variant of the Borsuk-Ulam theorem, stating that, for all $N \in \N$, there exists a compact subset $K_N$ of $\mathbb{F}^{N+1}$, not containing $(0,...,0) \in \mathbb{F}^{N+1}$, such that $K_N=-K_N$ and, if $f:K_N \rightarrow \mathbb{F}^{N}$ is a continuous and odd map, then there exists $x \in K_N$ such that $f(x)=(0,...,0) \in \mathbb{F}^N$.

(iv) Even if the above do not hold for a field $\mathbb{F}$, preventing us from applying the Guth-Katz decomposition technique as it has been established by Guth and Katz and derive \cite[Corollary 4.4]{Guth_Katz_2010} in that way, we know of no reason why the statement of \cite[Corollary 4.4]{Guth_Katz_2010} (for some definition of a cell) should not hold in $\mathbb{F}^n$. Indeed, for any finite set $\mathfrak{G}$ of points in $\mathbb{F}^n$ and any $d \in \R_{>1}$, there could exist a non-zero polynomial $p \in \mathbb{F}[x_1,...,x_n]$, of degree at most $d$, whose zero set decomposes $\mathbb{F}^n$ in $\sim d^n$ cells, each containing at most $|\mathfrak{G}|/d^n$ points of $\mathfrak{G}$. 

However, an essential reason why the Guth-Katz decomposition theorem \cite[Corollary 4.4]{Guth_Katz_2010} is actually meaningful in $\R^n$ (and a substantial ingredient of arguments using \cite[Corollary 4.4]{Guth_Katz_2010} in $\R^n$) is that, if a line intersects the interiors of two of the cells in which the zero set of a polynomial decomposes $\R^n$, then it also intersects the zero set of the polynomial, a fact which does not necessarily hold in a general field setting.

In fact, in the particular case where $\mathbb{F}$ is a finite field, it is certain that there exists a finite set $\mathfrak{G}$ of points in $\mathbb{F}^2$ and some $d >1$, such that there does not exist a non-zero polynomial $p \in \mathbb{F}[x,y]$, of degree at most $d$, whose zero set decomposes $\mathbb{F}^2$ in $\sim d^2$ cells, each containing at most $|\mathfrak{G}|/d^2$ points of $\mathfrak{G}$, with the property that each line in $\mathbb{F}^2$ intersecting the interiors of two of the cells intersects the zero set of the polynomial. 

The reason for this is that, if this was not true, then we would be able to deduce the Szemer\'edi-Trotter theorem in $\mathbb{F}^n$, for all $n \geq 2$, using the technique explained in Chapter \ref{2} (appearing in \cite{KMS}), which proves the Szemer\'edi-Trotter theorem in $\R^n$, for all $n \geq 2$, using only \cite[Corollary 4.4]{Guth_Katz_2010} in $\R^2$ and the fact each line in $\R^2$ intersecting the interiors of two cells intersects the boundaries of the cells as well. However, the Szemer\'edi-Trotter theorem does not hold in $\mathbb{F}^n$, for any finite field $\mathbb{F}$ and any $n \geq 2$.

Indeed, let $\mathbb{F}$ be a finite field, $n \geq 2$, $\mathfrak{L}$ the set of all lines in $\mathbb{F}^n$, and $\mathcal{P}$ the set of all points in $\mathbb{F}^n$. Then, 
\begin{displaymath}|\mathfrak{L}|\sim \frac{|\mathbb{F}|^{n-1}\cdot |\mathbb{F}|^n}{|\mathbb{F}|} \sim |\mathbb{F}|^{2n-2}\text{ and }|\mathcal{P}|\sim |\mathbb{F}|^n,\end{displaymath} while the number of incidences between $\mathcal{P}$ and $\mathfrak{L}$ is 
\begin{displaymath}I_{\mathcal{P}, \mathfrak{L}} \sim |\mathfrak{L}| \cdot |\mathbb{F}| \sim |\mathbb{F}|^{2n-2}\cdot |\mathbb{F}| \sim |\mathbb{F}|^{2n-1}.\end{displaymath} 
Therefore, $|\mathcal{P}|^{2/3}|\mathfrak{L}|^{2/3}+|\mathcal{P}|+|\mathfrak{L}| \sim |\mathbb{F}|^{2n/3}|\mathbb{F}|^{\frac{2(2n-2)}{3}}+|\mathbb{F}|^n+|\mathbb{F}|^{2n-2}\sim |\mathbb{F}|^{\frac{2n+4n-4}{3}}+|\mathbb{F}|^{2n-2}\sim |\mathbb{F}|^{\frac{6n-4}{3}}+|\mathbb{F}|^{2n-2}\sim |\mathbb{F}|^{2n-\frac{4}{3}}$, and thus it does not hold that $ I_{\mathcal{P}, \mathfrak{L}} \lesssim |\mathcal{P}|^{2/3}|\mathfrak{L}|^{2/3}+|\mathcal{P}|+|\mathfrak{L}|$. In other words, the Szemer\'edi-Trotter theorem does not hold in $\mathbb{F}^n$. 

\textbf{Remark.} Note that, in order to count joints and multijoints with multiplicities, the degree of the polynomial we used to achieve the Guth-Katz decomposition was a quantity that we knew was larger than 1 thanks to the truth of the Szemer\'edi-Trotter theorem in euclidean space; our proofs were thus somehow based on the Szemer\'edi-Trotter theorem. This makes it even more unlikely to count joints and multijoints with multiplicities in finite field settings with methods similar to the ones we use in euclidean space.

$\;\;\;\;\;\;\;\;\;\;\;\;\;\;\;\;\;\;\;\;\;\;\;\;\;\;\;\;\;\;\;\;\;\;\;\;\;\;\;\;\;\;\;\;\;\;\;\;\;\;\;\;\;\;\;\;\;\;\;\;\;\;\;\;\;\;\;\;\;\;\;\;\;\;\;\;\;\;\;\;\;\;\;\;\;\;\;\;\;\;\;\;\;\;\;\;\;\;\;\;\;\;\;\;\;\;\;\;\;\;\;\;\;\;\;\;\;\;\;\;\;\;\;\;\;\;\;\;\;\;\;\;\;\blacksquare$

Now, the reason why, in $\R^n$, each line intersecting the interiors of two of the cells in which the zero set of a polynomial decomposes $\R^n$ also intersects the zero set of the polynomial is a result of the intermediate value theorem in $\R$, and, in particular, of the fact that if a polynomial $f \in \R[x]$ is such that $f(a)f(b)<0$, for $a<b$ in $\R$, then there exists some $c \in [a,b]$, such that $f(c)=0$. However, the intermediate value theorem does not necessarily hold in all fields.

\textbf{Remark.} Due to (i), (ii), (iii) and (iv), we believe that, at least as a start, it would be sensible to search extensions of results that are proved in $\R^n$ using the Guth-Katz polynomial method (such as the Szemer\'edi-Trotter theorem or our results regarding joints and multijoints in euclidean space) in the case of \textit{ordered fields}, i.e. fields with a total order compatible with the operations of the field (such as $\mathbb{Q}$).

More precisely, a field $\mathbb{F}$ is an ordered field if it is equipped with a total order relation $\leq$, such that, for all $x, y,z \in \mathbb{F}$, \newline
(a) $x \leq y \Rightarrow x+z \leq y+z$ and \newline
(b) $0 \leq x$, $0 \leq y \Rightarrow 0 \leq xy$.

Note that every ordered field has characteristic 0. In addition, it is easy to see that in an ordered field all squares are positive.

More importantly, it can be proved that every ordered field $\mathbb{F}$ has an algebraic extension $\mathbb{F}'$, with a field order that extends the order of $\mathbb{F}$, such that $\mathbb{F}'$ is a \textit{real closed field}; in fact, every ordered field has a unique \textit{real closure}, which is the smallest real closed field containing it. Now, a real closed field $\mathbb{K}$ is defined as an ordered field, such that its positive elements are exactly the squares of $\mathbb{K}$ and every polynomial  in $\mathbb{K}[x]$ of odd degree has a root in $\mathbb{K}$ (for details on real closed fields, see \cite{MR2248869} or \cite{BCR87}). 

For example, $\mathbb{R}$ is a real closed field, containing the ordered field $\mathbb{Q}$. However, $\R$ is a transcendental, not an algebraic, extension of $\mathbb{Q}$, and it can thus not be the real closure of $\mathbb{Q}$. In fact, it can be proved (see \cite{MR2248869}) that the real closure of $\mathbb{Q}$ is the set $\R_{alg}$ of real algebraic numbers, i.e. the set of real numbers that are roots of univariate polynomials with coefficients in $\mathbb{Q}$. Other examples of real closed fields are the field of hyperreal numbers and the field of Puiseux series in real coefficients.

Now, it follows from the above that the set of ordered fields is exactly the set of subfields of all real closed fields. And the behaviour of real closed fields happens to resemble that of $\R$ in many ways.

Indeed, if $\mathbb{K}$ is a real closed field, then every positive element of $\mathbb{K}$ has a unique positive root in $\mathbb{K}$. This means that we can define a type of norm $\|\cdot \|$ on $\mathbb{K}^n$, taking values in $\mathbb{K}$, such that $\|(x_1,...,x_n)\|= \sqrt{x_1^2+...+x_n^2}$, while any open ball in $\mathbb{K}^n$, centred at $x \in \mathbb{K}^n$ and with radius $r \in \mathbb{K}$ is defined as $B(x,r)=\{y \in \mathbb{K}^n:\|x-y\|<r\}$. It is easy to see that the set of open balls defined as above has the appropriate properties to be the basis of a topology, and therefore $\mathbb{K}^n$ is a topological space, with the (Hausdorff) topology generated by the set of open balls as above. Moreover, we can define the volume of any open ball $B(x,r)$ in $\mathbb{K}^n$ as the element $r^n$ of $\mathbb{K}$; this gives hope of defining a measure on the $\sigma$-algebra generated by the open sets in $\mathbb{K}^n$, taking values in $\mathbb{K}$, with properties similar to those of the Lebesgue measure in $\R^n$.

What is more, it can be proved that the ordered field $\mathbb{K}$ is real closed if and only if the intermediate value theorem holds for polynomials in $\mathbb{K}[x]$; in other words, if and only if, whenever a polynomial in $\mathbb{K}[x]$ is such that $f(a)f(b)<0$, for $a\leq b$ in $\mathbb{K}$, there exists some $c \in [a,b]$ such that $f(c)=0$.

Moreover, many of the results of Chapter \ref{5} can be extended to real closed field situations (see \cite{MR2248869} or \cite{BCR87}). 

The above, therefore, may lead to analogues of the Guth-Katz polynomial method, and, subsequently, of our results regarding joints and multijoints, in real closed field settings. And since any ordered field is a subfield of a real closed field, we believe that it is sensible to search extensions of our results in ordered field settings, each time by working in a real closed field extension of the ordered field in question (not necessarily algebraic extension).

\textbf{Example.} Let us work in the case of the ordered field $\mathbb{Q}$. Since $\mathbb{Q}$ is a subfield of $\R$, it immediately follows that the Szemer\'edi-Trotter theorem in $\R^n$, as well as our results regarding joints and multijoints in $\R^3$, hold in $\mathbb{Q}^n$ and $\mathbb{Q}^3$, respectively, as well. We would like, however, to demonstrate here the actual necessity of working in the real closed field $\R$ to extend results as above in $\mathbb{Q}^n$, at least when attempting to use similar techniques as in euclidean space.

Indeed, it is easy to see that the Borsuk-Ulam theorem does not hold in $\mathbb{Q}^n$. Indeed, if $\rho$ is a rotation of the unit sphere $S^n$ of $\R^{n+1}$, such that the north and south pole of $\rho (S^n)$ do not belong to $\mathbb{Q}^{n+1}$, and if $\pi$ is the projection of $S^n$ to $\R^n$ that sends the north and south poles of $S^n$ to $(0,...,0) \in \R^n$, then the map $\pi \circ \rho:S_{\mathbb{Q}}^n\rightarrow \R^n$, where $S_{\mathbb{Q}}^n:=\{x=(x_1,...,x_{n+1}) \in \mathbb{Q}^{n+1}:x_1^2+...+x_{n+1}^2=1\}$, is a continuous and odd map that sends no point of $S_{\mathbb{Q}}^n$ to $(0,...,0) \in \R^n$. By the density of $S_{\mathbb{Q}}^n$ in $S^n$, we only know that, if $f:S_{\mathbb{Q}}^n\rightarrow \R^n$ is a continuous and odd map, then there exists $x \in S^n$ such that $f(x)=(0,...,0) \in \R^n$. 

Now, the fact that the Borsuk-Ulam theorem does not hold in $\mathbb{Q}^n$ could be an obstruction to establishing a corresponding version of the Stone and Tukey polynomial ham sandwich theorem in $\mathbb{Q}^n$ (Theorem \ref{2.1.1}), where the polynomial with the bisecting zero set belongs to $\mathbb{Q}[x_1,...,x_n]$. And, even if such an analogue of the Stone and Tukey theorem held in $\mathbb{Q}^n$, a closer study of the proof of Corollary \ref{2.1.2} (whose successive applications constitute the Guth-Katz polynomial method in $\R^n$) shows that the same proof would fail to give an analogue of Corollary \ref{2.1.2} in $\mathbb{Q}^n$, where the polynomial whose zero set bisects finitely many disjoint, finite sets of points in $\mathbb{Q}^n$ belongs to $\mathbb{Q}[x_1,...,x_n]$.

Of course, the above do not necessarily mean that an analogue of the Guth-Katz polynomial decomposition theorem \cite[Corollary 4.4]{Guth_Katz_2010} does not hold in $\mathbb{Q}^n$. However, and despite any advantages that $\mathbb{Q}$ may enjoy due to its density in $\R$, we can still not establish a meaningful analogue of the Guth-Katz decomposition theorem \cite[Corollary 4.4]{Guth_Katz_2010} in $\mathbb{Q}^n$, without involving the real closed field $\R$ in our analysis. Indeed, the intermediate value theorem does not hold in $\mathbb{Q}$, which means that, even if, for a finite set of points $\mathfrak{G}$ in $\mathbb{Q}^n$ and some $d>1$, there exists a polynomial $p \in \mathbb{Q}[x_1,...,x_n]$, of degree $\leq d$, whose zero set decomposes $\mathbb{Q}^n$ in $\lesssim d^n$ cells, each containing $\lesssim S/d^n$ points of $\mathfrak{G}$, we would still not know if a line in $\mathbb{Q}^n$ that intersects the interiors of two of the cells in $\mathbb{Q}^n$ intersects the zero set of $p$ in $\mathbb{Q}^n$ as well; we only know that such an intersection exists in $\R^n$.

Therefore, working in the real closed field $\R$ when dealing with extensions of our results in the ordered field $\mathbb{Q}$ seems natural.

$\;\;\;\;\;\;\;\;\;\;\;\;\;\;\;\;\;\;\;\;\;\;\;\;\;\;\;\;\;\;\;\;\;\;\;\;\;\;\;\;\;\;\;\;\;\;\;\;\;\;\;\;\;\;\;\;\;\;\;\;\;\;\;\;\;\;\;\;\;\;\;\;\;\;\;\;\;\;\;\;\;\;\;\;\;\;\;\;\;\;\;\;\;\;\;\;\;\;\;\;\;\;\;\;\;\;\;\;\;\;\;\;\;\;\;\;\;\;\;\;\;\;\;\;\;\;\;\;\;\;\;\;\;\blacksquare$
 
All the above suggest that an immediate application of our methods in higher dimensions, as well as in different -- especially finite -- field settings,  seems unlikely. It is obvious in general that, unfortunately, the results arising from the algebraic techniques we are using in $\R^3$, albeit the best possible, happen to be mainly numerical, failing to give us an intuitive view of the combinatorial nature of our problems, something that could promise better results in higher dimensions and different field settings. This chapter is therefore aiming to investigate the joints problem in those situations. We start with Theorem \ref{carberyjoints} that follows, which gives an upper bound on the number of joints formed by a collection of lines in an arbitrary field setting and arbitrary dimension, we then continue with a lemma regarding the combinatorial nature of the joints problem, and we finally combine the two, obtaining an upper bound on the number of joints counted with multiplicities in $\mathbb{F}^n$, where $\mathbb{F}$ is any field and $n \geq 2$.

More particularly, Anthony Carbery observed that Quilodr\'an's argument for the solution of the joints problem without multiplicities in $\R^n$ (see \cite{MR2594983}) could be applied, to a large extent, in $\mathbb{F}^n$ as well, where $\mathbb{F}$ is any field. We eventually managed to adapt the argument in an arbitrary field setting, to obtain the following.

\begin{theorem}\label{carberyjoints} \emph{\textbf{(Carbery, Iliopoulou)}} Let $\mathbb{F}$ be any field and $n\geq 2$. Let $\mathfrak{L}$ be a finite collection of $L$ lines in $\mathbb{F}^n$, and $J$ the set of joints formed by $\mathfrak{L}$. Then,
\begin{displaymath} |J| \lesssim_n L^{\frac{n}{n-1}}. \end{displaymath}
\end{theorem}

The proof of Theorem \ref{carberyjoints} requires Lemma \ref{dvir} that follows, which is essentially Dvir's basic argument for the solution of the Kakeya problem in finite fields in \cite{MR2525780}, and whose analogous formulation in $\R^n$ is used by Quilodr\'an in \cite{MR2594983} for the solution of the joints problem in $\R^n$ (for self-containment, we include Quilodr\'an's solution of the joints problem in $\R^n$, in Theorem \ref{quilodran}). 

Note that, from now on, when we refer to a polynomial as non-zero we mean that it has a non-zero coefficient.

\textbf{Remark.} We would like to emphasise here that, while, in $\R^n$, a polynomial that has a non-zero coefficient does not vanish on the whole of $\R^n$, the same does not necessarily hold in an arbitrary field setting. For example, if $\mathbb{F}$ is a finite field, the univariate polynomial $x^{|\mathbb{F}|}- x$ in $\mathbb{F}[x]$ vanishes on the whole of $\mathbb{F}$.

In the case where $\mathbb{F}$ is a finite field, we can have some control on the number of roots of a polynomial in $\mathbb{F}[x_1,...,x_n]$ thanks to the Schwartz-Zippel lemma that follows (see \cite{Sch80} or\cite{Zip79} for a proof).

\begin{lemma}\emph{\textbf{(Schwartz, Zippel)}} Let $\mathbb{F}$ be a finite field, and $f$ a non-zero polynomial in $\mathbb{F}[x_1,...,x_n]$. Then,
\begin{displaymath}|\{x \in \mathbb{F}^n:f(x)=0\}|\leq \deg p \cdot |\mathbb{F}|^{n-1}.
\end{displaymath}

\end{lemma}

In particular, this lemma implies that, if $\mathbb{F}$ is a finite field, then a non-zero polynomial in $\mathbb{F}[x_1,...,x_n]$ of degree at most $|\mathbb{F}|-1$ does not vanish on the whole of $\mathbb{F}^n$.

In any case however, in whatever follows, whenever we refer to a polynomial as non-zero we do not simultaneously imply that it does not vanish on the whole of our space.

$\;\;\;\;\;\;\;\;\;\;\;\;\;\;\;\;\;\;\;\;\;\;\;\;\;\;\;\;\;\;\;\;\;\;\;\;\;\;\;\;\;\;\;\;\;\;\;\;\;\;\;\;\;\;\;\;\;\;\;\;\;\;\;\;\;\;\;\;\;\;\;\;\;\;\;\;\;\;\;\;\;\;\;\;\;\;\;\;\;\;\;\;\;\;\;\;\;\;\;\;\;\;\;\;\;\;\;\;\;\;\;\;\;\;\;\;\;\;\;\;\;\;\;\;\;\;\;\;\;\;\;\;\;\blacksquare$

\begin{lemma} \label{dvir} Let $\mathbb{F}$ be any field. For any set $\mathcal{P}$ of $m$ points in $\mathbb{F}^n$, there exists a non-zero polynomial in $\mathbb{F}[x_1,...,x_n]$, of degree $\lesssim_n m^{1/n}$, which vanishes at each point of $\mathcal{P}$.
\end{lemma}

\begin{proof} Given a set $\mathcal{P}$ of points in $\mathbb{F}^n$, we want to find a polynomial 
\begin{displaymath} f(x_1,...,x_n)=\sum_{\{(a_1,...,a_n) \in \N^n:\;a_1+...+a_n\leq d\}}c_{a_1,...,a_n}x_1^{a_1}\cdot ... \cdot x_n^{a_n},
\end{displaymath}
for some $d\in \N$, which vanishes on $\mathcal{P}$.

We notice that each equation of the form $f(\xi_1,...,\xi_n)=0$, for some $(\xi_1,...,\xi_n) \in \mathcal{P}$, is a linear equation with unknowns the coefficients $c_{a_1,...,a_n}$, for $(a_1,...,a_n) \in \N^n$ such that $a_1+...+a_n\leq d$; therefore, the fact that we want $f$ to vanish at each point of $\mathcal{P}$ gives rise to a system of $|\mathcal{P}|=m$ linear equations, with $\binom{d+n}{n}$ unknowns. And if the unknowns are more than the equations, i.e. $\binom{d+n}{n} >m $, then the system has a non-trivial solution, which means that there exists a polynomial $f \in \mathbb{F}[x_1,...,x_n]$, of degree at most $d$, which vanishes at each point of $\mathcal{P}$ and is non-zero.

Now, $\binom{d+n}{n}> \frac{1}{n!} d^n$, so, for the above to hold, it suffices to have $\frac{1}{n!} d^n \geq m$, or, equivalently,  $d \geq (n!)^{1/n}  m^{1/n}$. Therefore, by setting $d =\left\lfloor{(n!)^{1/n}  m^{1/n}}\right\rfloor+1 \lesssim_n m^{1/n}$, it follows that there exists a non-zero polynomial in $\mathbb{F}[x_1,...,x_n]$, of degree $\lesssim _n m^{1/n}$, which vanishes at each point of $\mathcal{P}$.

\end{proof}

Now, let us present the proof of Quilodr\'an for Theorem \ref{carberyjoints} when $\mathbb{F}=\R$.

\begin{theorem}\emph{\textbf{(Quilodr\'an, \cite{MR2594983})}}\label{quilodran} Let $\mathfrak{L}$ be a finite collection of $L$ lines in $\R^n$, and $J$ the set of joints formed by $\mathfrak{L}$. Then,
\begin{displaymath} |J| \lesssim_n L^{\frac{n}{n-1}}. \end{displaymath}
\end{theorem}

\begin{proof} Let $\mathfrak{L}$ be a finite collection of $L$ lines in $\R^n$, and $J$ the set of joints formed by $\mathfrak{L}$. We will show that there exists a line in $\mathfrak{L}$, containing $\lesssim _n |J|^{1/n}$ joints of $J$. 

Indeed, by Lemma \ref{dvir}, there exists a non-zero polynomial $f \in \R[x_1,...,x_n]$, of degree $d \lesssim_n |J|^{1/n}$, that vanishes at each point of $J$; in fact, we consider $f$ to be a polynomial of minimal degree with that property. Let us assume that each line in $\mathfrak{L}$ contains more than $d$ joints of $J$. Then, there exist $b$, $v \in \mathbb{F}^n$, such that $l=\{v+bt:t \in \mathbb{F}\}$. Now, the polynomial $f|_{l}:=f(v+bt) \in \mathbb{F}[t]$ is a polynomial in one variable with coefficients in the field $\mathbb{F}$, of degree at most $\deg f$, which has more than $\deg f$ roots (since $f$ vanishes on $l\cap J$), i.e. more roots than its degree. Therefore, $f|_{l}$ is the zero polynomial, for all $l \in \mathfrak{L}$. Now, let $x_0 \in J$. We know that there exist at least $n$ lines of $\mathfrak{L}$, $l_1$, ..., $l_n$, passing through $x_0$, whose directions $b_1$, ..., $b_n$, respectively, span $\R^n$. The polynomial $(f|_{l_i})'$ is the zero polynomial for all $i=1,...,n$, so $\nabla{f}(x_0) \cdot b_i=0$, for all $i=1,...,n$. Now, for each $i=1,...,n$, $\nabla{f}(x_0) \cdot b_i=0$ is a linear equation with unknowns the entries of the vector $\nabla{f}(x_0)$. So, since the set $\{b_1,...,b_n\}$ is linearly independent, it follows that $\nabla{f}(x_0)=0$. However, $x_0$ was an arbitrarily chosen joint from the set $J$. Therefore, $\nabla{f}$ vanishes at each point of $J$, which means that the polynomial $\frac{\partial{f}}{\partial{x_i}}$ vanishes at each point of $J$, for all $i=1,...,n$. However, for all $i=1,...,n$, $\frac{\partial{f}}{\partial{x_i}}$ is a polynomial of degree strictly smaller than $\deg f$, it therefore is the zero polynomial. Consequently,
\begin{displaymath} \nabla{f}=0. \end{displaymath} 
So, $f$ is a constant polynomial. However, $f$ vanishes on $J$, thus $f$ is the zero polynomial, which is a contradiction.

This means that our initial assumption that each line in $\mathfrak{L}$ contains more than $\deg f$ joints of $J$ was wrong, and thus there exists a line $l$ in $\mathfrak{L}$, containing fewer than $\deg f$, i.e. $\lesssim _n |J|^{1/n}$, joints of $J$.

We now take the line $l$ and the joints it contains out of our collection of lines and joints, ending up with a smaller collection $\mathfrak{L}'$ of lines, as well as a subset of $J$, which is contained in the set $J'$ of joints formed by $\mathfrak{L}'$. From what we have already shown, there exists some line $l' \in \mathfrak{L}'$, containing $\lesssim_n|J'|^{1/n} \lesssim_n |J|^{1/n}$ joints of $J'$. We take the line $l'$ and the points of $J \cap J'$ it contains out of our collection of lines and joints, and we continue in the same way, until we eliminate all the joints of our original collection $J$, something which is achieved in at most $L$ steps. In each step, we take $\lesssim_n |J|^{1/n}$ joints out of our original collection $J$, and thus
\begin{displaymath}|J| \lesssim_n L \cdot |J|^{1/n},
\end{displaymath}
which gives
\begin{displaymath}|J| \lesssim_n L^{\frac{n}{n-1}}
\end{displaymath}
after rearranging.
\end{proof}

We see that, in his proof of the joints problem in \cite{MR2594983}, Quilodr\'an used the fact that, if a polynomial $p \in \R[x_1,...,x_n]$ vanishes on a line in $\R^n$, and its gradient at a point of the line is a non-zero vector, then that vector is, in fact, perpendicular to the direction of the line. Of course we do not generally have a notion of perpendicularity in an arbitrary field setting, but we will manage to essentially follow Quilodr\'an's argument for the proof of Theorem \ref{carberyjoints}, by introducing a notion of derivative in the general field situation. This will be the Hasse derivative, as it appears in \cite{MR2648400}.

\begin{definition} Let $\mathbb{F}$ be a field and $f \in \mathbb{F}[x_1,...,x_n]$. For all $i \in \N^n$, the \emph{$i$-th Hasse derivative} $f^{(i)}$ of $f$ is defined as the element of $\mathbb{F}[x_1,...,x_n]$ that is the coefficient of $z^i$ in the expression of $f(x+z)$ as a polynomial in $z$.
\end{definition}

\textbf{Remark.} We would like to clarify here the difference between the Hasse derivative and the usual derivative. Note that from the definition of the Hasse derivative it follows that, for $f \in \mathbb{F}[x_1,...,x_n]$ and any $a \in \mathbb{F}^n$,
\begin{displaymath}f(x)=\sum_{i \in \N^n} f^{(i)}(a)(x-a)^i.
\end{displaymath}
On the other hand, for $f \in \R[x_1,...,x_n]$ and any $a \in \R^n$, it is known that
\begin{displaymath}f(x)=\sum_{i=(i_1,...,i_n) \in \N^n} \frac{1}{i_1! \cdots i_n!}\cdot\frac{\partial^{i_1+...+i_n} f}{\partial x_1^{i_1} \cdots \partial x_n^{i_n}}(a)(x-a)^i.
\end{displaymath}
Therefore, if $f \in \R[x_1,...,x_n]$, 
\begin{displaymath}f^{(i)}= \frac{1}{i_1! \cdots i_n!}\cdot\frac{\partial^{i_1+...+i_n} f}{\partial x_1^{i_1} \cdots \partial x_n^{i_n}}.
\end{displaymath}

Consequently, one could claim that we would acquire a more natural notion of derivative if we defined as the $i$-th derivative of $f \in \mathbb{F}[x_1,...,x_n]$, for any $i=(i_1,...,i_n) \in \N^n$, as $(i_1! \cdots i_n!) \cdot f^{(i)}\;\Big(:=\sum_{j=1}^{(i_1! \cdots i_n!)} f^{(i)}\Big)$. However, in a field of non-zero characteristic, this could lead to $f$ having zero $i$-th derivative, for some $i \in \N^n$ such that $f^{(i)}$ is a non-zero polynomial. Therefore, the Hasse derivative of a polynomial in $\mathbb{F}[x_1,...,x_n]$, where $\mathbb{F}$ is a field of non-zero characteristic, provides more information about the polynomial, and is thus more appropriate for the study of polynomials in that case.

Let us mention, however, that in our study we will be interested in $i=(i_1,...,i_n) \in \N^n$ such that $i_1+...+i_n=1$, in which case the Hasse derivative and the other derivative described above coincide.

$\;\;\;\;\;\;\;\;\;\;\;\;\;\;\;\;\;\;\;\;\;\;\;\;\;\;\;\;\;\;\;\;\;\;\;\;\;\;\;\;\;\;\;\;\;\;\;\;\;\;\;\;\;\;\;\;\;\;\;\;\;\;\;\;\;\;\;\;\;\;\;\;\;\;\;\;\;\;\;\;\;\;\;\;\;\;\;\;\;\;\;\;\;\;\;\;\;\;\;\;\;\;\;\;\;\;\;\;\;\;\;\;\;\;\;\;\;\;\;\;\;\;\;\;\;\;\;\;\;\;\;\;\;\blacksquare$

It is easy to see the following:

(i) If $f,g \in \mathbb{F}[x_1,...,x_n]$, then $(f+g)^{(i)}=f^{(i)}+g^{(i)}$, for all $i \in \N^n$.

(ii) Let $m(x_1,...,x_n)=c_{a_1,...,a_n}x_1^{a_1}\cdots x_n^{a_n}$ be a monomial in $\mathbb{F}[x_1,...,x_n]$. Also, for all $i=1,...,n$, suppose that $e_i$ is the vector $(0,...,0,1,0,...,0)$ in $\N^n$, with $1$ in the $i$-th coordinate. Then, \begin{displaymath}m^{(e_i)}(x_1,...,x_n)=a_i \cdot c_{a_1,...,a_n}x_1^{a_1}\cdots  x_i^{a_i-1} \cdots x_n^{a_n}\end{displaymath}
\begin{displaymath}\Bigg(:=\bigg(\sum_{k=1}^{a_i}c_{a_1,...,a_n}\bigg)x_1^{a_1}\cdots  x_i^{a_i-1} \cdots x_n^{a_i}\Bigg), \end{displaymath} for $a_i \geq 1$, and
\begin{displaymath}m^{(e_i)}(x_1,...,x_n)=0 \end{displaymath}
for $a_i=0$. 

In fact, from now on, we will write \begin{displaymath}m^{(e_i)}(x_1,...,x_n)=a_i \cdot c_{a_1,...,a_n}x_1^{a_1}\cdots  x_i^{a_i-1} \cdots x_n^{a_n}, \end{displaymath} accepting that this is the zero polynomial for $a_i=0$. 

Note that, for all $i=1,...,n$, $$m^{(e_i)}(x_1,...,x_n)=0\text{ if and only if }a_i \cdot c_{a_1,...,a_n}\;\bigg(=\sum_{k=1}^{a_i} c_{a_1,...,a_n}\bigg)=0. $$
It follows from (i) and (ii) that, for the polynomial \begin{displaymath}f(x_1,...,x_n)=\sum_{\{(a_1,...,a_n) \in \N^n:\;a_1+...+a_n\leq d\}}c_{a_1,...,a_n}x_1^{a_1}\cdots x_n^{a_n} \in \mathbb{F}[x_1,...,x_n],\end{displaymath} 
\begin{displaymath}f^{(e_i)}(x_1,...,x_n)=\sum_{\{(a_1,...,a_n) \in \N^n:\;a_1+...+a_n\leq d\}}a_i \cdot c_{a_1,...,a_n}x_1^{a_1}\cdots  x_i^{a_i-1} \cdots x_n^{a_n}, \end{displaymath}
for all $i=1,...,n$. In other words, the $(e_i)$-th Hasse derivative of a polynomial in $\mathbb{F}[x_1,...,x_n]$, for $\mathbb{F}$ an arbitrary field, is the analogue, in the general field situation, of the partial derivative of a polynomial in $\R[x_1,...,x_n]$ with respect to the variable $x_i$. This allows us to define the gradient of a polynomial in $\mathbb{F}[x_1,...,x_n]$ as follows.

\begin{definition} Let $\mathbb{F}$ be a field and $f \in \mathbb{F}[x_1,...,x_n]$. Then, the gradient of $f$ is the element $\nabla{f}$ of $(\mathbb{F}[x_1,...,x_n])^n$, defined as 
\begin{displaymath}\nabla{f}=\Big(f^{(e_1)},...,f^{(e_n)}\Big),
\end{displaymath}
where, for all $i=1,...,n$, $e_i$ is the vector $(0,...,0,1,0,...,0)$ in $\N^n$, with $1$ in the $i$-th coordinate.
\end{definition}

We would now like to be able to derive information about a polynomial $f \in \mathbb{F}[x_1,...,x_n]$ from its gradient. It would be nice to know, for example, that two polynomials in one variable with the same 1st Hasse derivative differ by a constant. However, that is not true in general. For example, the 1st Hasse derivative of the polynomial $x^{p}$ in $\mathbb{F}[x]$, where $\mathbb{F}$ is a field of characteristic $p$, is equal to $\big(\sum_{k=1}^{p}1\big) x^{p-1}=0$, i.e. it is the zero polynomial. On the other hand, the 1st Hasse derivative of the zero polynomial is also the zero polynomial, but $x^p$ and 0 do not differ by a constant as polynomials.

However, the following still holds.

\begin{lemma} \label{carbery3}Let $f \in \mathbb{F}[x_1,...,x_n]$, where $\mathbb{F}$ is a field. Suppose that $\nabla{f}=0$. Then:

\emph{(i)} If the characteristic of $\mathbb{F}$ is zero, then $f$ is a constant polynomial.

\emph{(ii)} If the characteristic of $\mathbb{F}$ is $p$, for some (prime) $p\neq 0$, then $f$ is of the form
\begin{displaymath} f(x_1,...,x_n)=\sum_{\{(a_1,...,a_n)\in \N^n:\;a_1+...+a_n\leq d\text{ and }p \mid a_i,\; \forall i=1,...,n\}} c_{a_1,...,a_n} x_1^{a_1}\cdots x_n^{a_n}, \end{displaymath} for some $d \in \N$.

\end{lemma}

\begin{proof} (i) Suppose that $f$ is not a constant polynomial. Then, there exists some $d \geq 1$, $d \in \N$, such that
\begin{displaymath} f(x_1,...,x_n)=\sum_{\{(a_1,...,a_n) \in \N^n:\;a_1+...+a_n\leq d\}} c_{a_1,...,a_n} x_1^{a_1}\cdots x_n^{a_n}, \end{displaymath}where $c_{a_1',...,a_n'} \neq 0$, for some $(a_1',...,a_n')\neq (0,...,0)$ in $\N^n$ such that $a'_1+..+a'_n \leq d$. Let $i \in \{1,...,n\}$ be such that $a_i' \neq 0$. Since $\nabla{f}=0$, it follows that the polynomial \begin{displaymath}f^{(e_i)}(x_1,...,x_n)=\sum_{\{(a_1,...,a_n) \in \N^n:\;a_1+...+a_n\leq d\}}a_i \cdot c_{a_1,...,a_n}x_1^{a_1}\cdots  x_i^{a_i-1} \cdots x_n^{a_n}\end{displaymath} is the zero polynomial, and thus $a_i \cdot c_{a_1,...,a_n}=0$, for all $(a_1,...,a_n)\in \N^n$ such that $a_1+...+a_n \leq d$. In particular,
\begin{displaymath} a_i' \cdot c_{a_1',...,a_n'}=0,
\end{displaymath}
from which we obtain
\begin{displaymath}a_i'=0\text{ or }c_{a_1',...,a_n'}=0,
\end{displaymath}as the characteristic of $\mathbb{F}$ is zero. However, this is a contradiction, and thus $f$ is a constant polynomial.

(ii) We know that
\begin{displaymath} f(x_1,...,x_n)=\sum_{\{(a_1,...,a_n) \in \N^n:\;a_1+...+a_n\leq d\}} c_{a_1,...,a_n} x_1^{a_1}\cdots x_n^{a_n}, \end{displaymath}
for some $d \in \N$.

Let us assume that $c_{a_1',...,a_n'} \neq 0$, for some $(a_1',...,a_n')$ in $\N^n$, such that $a_1+...+a_n\leq d$ and $p \nmid a_i'$ for some $i \in \{1,...,n\}$. Since $\nabla{f}=0$, it follows that the polynomial \begin{displaymath}f^{(e_i)}(x_1,...,x_n)=\sum_{\{(a_1,...,a_n) \in \N^n:\;a_1+...+a_n\leq d\}}a_i \cdot c_{a_1,...,a_n}x_1^{a_1}\cdots  x_i^{a_i-1} \cdots x_n^{a_n}\end{displaymath} is the zero polynomial, and thus $a_i \cdot c_{a_1,...,a_n}=0$, for all $(a_1,...,a_n)\in \N^n$ such that $a_1+...+a_n \leq d$. In particular,
\begin{displaymath} a_i' \cdot c_{a_1',...,a_n'}=0,
\end{displaymath}
from which we obtain
\begin{displaymath}p \mid a_i'\text{ or }c_{a_1',...,a_n'}=0,
\end{displaymath}as the characteristic of $\mathbb{F}$ is $p$. However, this is a contradiction, and thus the statement of the lemma is proved.

\end{proof}

Now, as we have already mentioned, there does not necessarily exist a notion of positivity in an arbitrary field $\mathbb{F}$, that would allow us to define an inner product on $\mathbb{F}^n\times \mathbb{F}^n$. However, for $a=(a_1,...,a_n)$ and $b=(b_1,...,b_n)$ in $\mathbb{F}^n$, we still denote by $a\cdot b$ the element $a_1b_1+...+a_nb_n$ of $\mathbb{F}$, and prove the following.

\begin{lemma}\label{nabla} Let $\mathbb{F}$ be a field and $f$ a polynomial in $\mathbb{F}[x_1,...,x_n]$. Let $l$ be a line in $\mathbb{F}^n$ with direction $b=(b_1,...,b_n) \in \mathbb{F}^n$, i.e. the set $\{v+tb:t \in \mathbb{F}\}$, for some $v=(v_1,...,v_n) \in \mathbb{F}^n$. Then, if $f|_l (t):=f(v+tb)\in \mathbb{F}[t]$ is the restriction of $f$ on $l$, we have that
\begin{displaymath} (f|_l)^{(1)}(t)=b \cdot \nabla{f}(v+tb).
\end{displaymath}
\end{lemma}

\begin{proof} We know that $(f|_l)^{(1)}\in \mathbb{F}[t]$ is the coefficient of $z^1$ in the expansion of $(f|_l)(t+z)$ as a polynomial of $z$. The polynomial $f$ is of the form
\begin{displaymath} f(x_1,...,x_n)= \sum_{\{(a_1,...,a_n) \in \N^n:\;a_1+...+a_n\leq d\}}c_{a_1,...,a_n}x_1^{a_1}\cdots x_n^{a_n}
\end{displaymath}
for some $d \in \N$, and thus
\begin{displaymath} (f|_l)(t+z)=f \big(v+(t+z)b\big)=f\big(v_1+(t+z)b_1,...,v_n+(t+z)b_n\big)= 
\end{displaymath}
\begin{displaymath} =\sum_{\{(a_1,...,a_n) \in \N^n:\;a_1+...+a_n\leq d\}}c_{a_1,...,a_n}\big(v_1+(t+z)b_1\big)^{a_1}\cdots \big(v_n+(t+z)b_n\big)^{a_n}=
\end{displaymath}
\begin{displaymath}=\sum_{\{(a_1,...,a_n) \in \N^n:\;a_1+...+a_n\leq d\}}c_{a_1,...,a_n}\big((v_1+tb_1)+zb_1\big)^{a_1}\cdots \big((v_n+tb_n)+zb_n\big)^{a_n},
\end{displaymath}

from which we easily see that
\begin{displaymath}(f|_l)^{(1)}(t)=
\end{displaymath}
\begin{displaymath}=\sum_{\{(a_1,...,a_n) \in \N^n:\;a_1+...+a_n\leq d\}}c_{a_1,...,a_n}\Bigg(a_1\cdot b_1(v_1+tb_1)^{a_1-1}(v_2+tb_2)^{a_2}\cdots (v_n+tb_n)^{a_n}+ \end{displaymath}
\begin{displaymath} +a_2\cdot b_2(v_1+tb_1)^{a_1}(v_2+tb_2)^{a_2-1}\cdots (v_n+tb_n)^{a_n}+...\end{displaymath}
\begin{displaymath} ...+a_n\cdot b_n(v_1+tb_1)^{a_1}(v_2+tb_2)^{a_2}\cdots (v_n+tb_n)^{a_n-1}\Bigg)=
\end{displaymath}
\begin{displaymath}=b_1\Bigg(\sum_{\{(a_1,...,a_n) \in \N^n:\;a_1+...+a_n\leq d\}}a_1\cdot c_{a_1,...,a_n}(v_1+tb_1)^{a_1-1}(v_2+tb_2)^{a_2}\cdots (v_n+tb_n)^{a_n}\Bigg)+
\end{displaymath}
\begin{displaymath}+b_2\Bigg(\sum_{\{(a_1,...,a_n) \in \N^n:\;a_1+...+a_n\leq d\}}a_2 \cdot c_{a_1,...,a_n}(v_1+tb_1)^{a_1}(v_2+tb_2)^{a_2-1}\cdots (v_n+tb_n)^{a_n}\Bigg)+...
\end{displaymath}
\begin{displaymath}...+b_n\Bigg(\sum_{\{(a_1,...,a_n) \in \N^n:\;a_1+...+a_n\leq d\}}a_n \cdot c_{a_1,...,a_n}(v_1+tb_1)^{a_1}(v_2+tb_2)^{a_2}\cdots (v_n+tb_n)^{a_n-1}\Bigg)=
\end{displaymath}
\begin{displaymath}=(b_1,...,b_n) \cdot \nabla{f}(v+tb)=b \cdot \nabla{f}(v+tb).
\end{displaymath}

\end{proof}

We are now ready to prove Theorem \ref{carberyjoints}.

\textit{Proof of Theorem \ref{carberyjoints}}. Let $\mathfrak{L}$ be a finite set of $L$  lines in $\mathbb{F}^n$, and $J$ the set of joints that they form. 

Let $\overline{\mathbb{F}}$ be the algebraic closure of $\mathbb{F}$. Since $J \subseteq \mathbb{F}^n$ and thus $J \subseteq\overline{\mathbb{F}}^n$, by Lemma \ref{dvir} there exists a non-zero polynomial
\begin{displaymath} f(x_1,...,x_n)=\sum_{\{(a_1,...,a_n) \in \N^n:\;a_1+...+a_n\leq d\}}c_{a_1,...,a_n}x_1^{a_1}\cdots x_n^{a_n} \end{displaymath} in $\overline{\mathbb{F}}[x_1,...,x_n]$, for some $d \lesssim _n |J|^{1/n}$, which vanishes at each point of $J$. We assume that $f$ is a polynomial of minimal degree with that property.

We will now show that there exists a line in $\mathfrak{L}$, containing $\lesssim _n |J|^{1/n}$ joints of $J$. 

Indeed, let us assume that each line in $\mathfrak{L}$ contains more than $\deg f$ joints of $J$. Let $l \in \mathfrak{L}$. Then, there exist $b$, $v \in \mathbb{F}^n$, such that $l=\{v+bt:t \in \mathbb{F}\}$. Now, for the line $\bar{l}=\{v+bt:t \in \overline{\mathbb{F}}\}$, which contains $l$ and therefore $l \cap J$, the polynomial $f|_{\bar{l}}:=f(v+bt) \in \overline{\mathbb{F}}[t]$ is a polynomial in one variable with coefficients in the field $\overline{\mathbb{F}}$, of degree at most $\deg f$, which has more than $\deg f$ roots (since $f$ vanishes on $l\cap J$), i.e. more roots than its degree. Therefore, $f|_{\bar{l}}$ is the zero polynomial, and thus $(f|_{\bar{l}})^{(1)}$ is the zero polynomial, for all $l \in \mathfrak{L}$. Now, let $x_0 \in J$. We know that there exist at least $n$ lines $l_1$, ..., $l_n$ in $\mathfrak{L}$ passing through $x_0$, whose directions $b_1$, ..., $b_n$, respectively, span $\mathbb{F}^n$. The polynomial $(f|_{\bar{l_i}})^{(1)}$ is the zero polynomial for all $i=1,...,n$, so, by Lemma \ref{nabla}, $\nabla{f}(x_0) \cdot b_i=0$, for all $i=1,...,n$. Now, for each $i=1,...,n$, $\nabla{f}(x_0) \cdot b_i=0$ is a linear equation with unknowns the entries of the vector $\nabla{f}(x_0)$. So, since the set $\{b_1,...,b_n\}$ is linearly independent in $\mathbb{F}^n$, and thus also linearly independent in $\overline{\mathbb{F}}^n$, it follows that $\nabla{f}(x_0)=0$. However, $x_0$ was an arbitrarily chosen point from the set $J$. Therefore, $\nabla{f}$ vanishes at each point of $J$, which means that the polynomial $f^{(e_i)}$ vanishes at each point of $J$, for all $i=1,...,n$. However, $f^{(e_i)}$ is a polynomial of degree strictly smaller than $\deg f$, therefore $f^{(e_i)}$ is the zero polynomial, for all $i=1,...,n$. Therefore,
\begin{displaymath} \nabla{f}=0. \end{displaymath} 
Now, if the characteristic of $\overline{\mathbb{F}}$ is zero, Lemma \ref{carbery3} implies that $f$ is a constant polynomial. However, $f$ vanishes on $J$, thus $f$ is the zero polynomial, which is a contradiction. On the other hand, if the characteristic of $\overline{\mathbb{F}}$ is not zero, then it is equal to some prime $p$, and thus Lemma \ref{carbery3} implies that 
\begin{displaymath} f(x_1,...,x_n)=\sum_{\{(a_1,...,a_n)\in \N^n:\;a_1+...+a_n\leq d\text{ and }p \mid a_i,\; \forall i=1,...,n\}} c_{a_1,...,a_n} x_1^{a_1}\cdots x_n^{a_n}. \end{displaymath}
Now, note that, since $\overline{\mathbb{F}}$ is an algebraically closed field, the polynomial equation $x^p-c=0$ has a solution in $\overline{\mathbb{F}}$, for all $c \in \overline{\mathbb{F}}$. Therefore, for all $(a_1,...,a_n) \in \N^n$ such that $a_1+...+a_n\leq d$ and $p\mid a_i$ $\forall i=1,...,n$, we denote by $c_{a_1,...,a_n}^{1/p}$ one of the solutions of the polynomial equation $x^p-c_{a_1,...,a_n}=0$ in $\overline{\mathbb{F}}$, and thus the expression 
\begin{displaymath}c_{a_1,...,a_n}^{1/p} x_1^{a_1/p}\cdots x_n^{a_n/p}
\end{displaymath} is a polynomial in $\overline{\mathbb{F}}[x_1,...,x_n]$. 

On the other hand, if $g$, $h \in \overline{\mathbb{F}}[x_1,...,x_n]$, then
\begin{displaymath}\big(g(x)+h(x)\big)^p=\sum_{k=0}^{p}\binom{p}{k}\cdot g(x)^kh(x)^{p-k}=g(x)^p+h(x)^p,
\end{displaymath}
as $p$ is the characteristic of $\overline{\mathbb{F}}$. Inductively,
\begin{displaymath}\big(f_1(x)+...+f_k(x)\big)^p=f_1(x)^p+...+f_k(x)^p, \end{displaymath}
for all $f_1$, ..., $f_k \in \overline{\mathbb{F}}[x_1,...,x_n]$, $k \in \N$.

Therefore,
\begin{displaymath}f(x_1,...,x_n)=\sum_{\{(a_1,...,a_n):\;a_1+...+a_n\leq d\text{ and }p \mid a_i,\; \forall i=1,...,n\}} c_{a_1,...,a_n} x_1^{a_1}\cdots x_n^{a_n}=\end{displaymath}
\begin{displaymath}=\sum_{\{(a_1,...,a_n):\;a_1+...+a_n\leq d\text{ and }p \mid a_i,\; \forall i=1,...,n\}}\big(c_{a_1,...,a_n}^{1/p} x_1^{a_1/p}\cdots x_n^{a_n/p}\big)^p=
\end{displaymath}
\begin{displaymath}=\Bigg(\sum_{\{(a_1,...,a_n):\;a_1+...+a_n\leq d\text{ and }p \mid a_i,\; \forall i=1,...,n\}}c_{a_1,...,a_n}^{1/p} x_1^{a_1/p}\cdots x_n^{a_n/p}\Bigg)^p=\end{displaymath}
\begin{displaymath}=g(x_1,...,x_n)^p,
\end{displaymath}
where
\begin{displaymath}g(x_1,...,x_n):=\sum_{\{(a_1,...,a_n):\;a_1+...+a_n\leq d\text{ and }p \mid a_i,\; \forall i=1,...,n\}}c_{a_1,...,a_n}^{1/p} x_1^{a_1/p}\cdots x_n^{a_n/p}
\end{displaymath}
is a polynomial in $\overline{\mathbb{F}}[x_1,...,x_n]$.

Now, for each $x_0 \in \overline{\mathbb{F}}^n$, $f(x_0)$ and $g(x_0)$ belong to the field $\overline{\mathbb{F}}$, so $f(x_0)=0$ if and only if $g(x_0)=0$. Therefore, since $f$ vanishes at each point of $J$, so does $g$. However, $f$ is a non-zero polynomial in $\overline{\mathbb{F}}[x_1,...,x_n]$ of minimal degree that vanishes on $J$, and $g$ is a non-zero polynomial in $\overline{\mathbb{F}}[x_1,...,x_n]$, whose degree is strictly smaller than $\deg f$, unless $f$ is a constant polynomial. Therefore, $f$ is a constant polynomial, and since it vanishes on $J$, it is the zero polynomial, which is a contradiction. 

This means that our initial assumption that each line in $\mathfrak{L}$ contains more than $\deg f$ joints of $J$ was wrong, and thus there exists a line $l$ in $\mathfrak{L}$, containing fewer than $\deg f$, i.e. $\lesssim _n |J|^{1/n}$, joints of $J$.

We now proceed in exactly the same way as in the solution of the joints problem in $\R^n$ by Quilodr\'an (see \cite{MR2594983}). 

Indeed, we take the line $l$ and the joints it contains out of our collection of lines and joints, ending up with a smaller collection $\mathfrak{L}'$ of lines, as well as a subset of $J$, which is contained in the set $J'$ of joints formed by $\mathfrak{L}'$. From what we have already shown, there exists some line $l' \in \mathfrak{L}'$, containing $\lesssim_n|J'|^{1/n} \lesssim_n |J|^{1/n}$ joints of $J'$. We take the line $l'$ and the points of $J \cap J'$ it contains out of our collection of lines and joints, and we continue in the same way, until we eliminate all the joints of our original collection $J$, something which is achieved in at most $L$ steps. In each step, we take $\lesssim_n |J|^{1/n}$ joints out of our original collection $J$, and thus
\begin{displaymath}|J| \lesssim_n L \cdot |J|^{1/n},
\end{displaymath}
which gives
\begin{displaymath}|J| \lesssim_n L^{\frac{n}{n-1}}
\end{displaymath}
after rearranging.

\qed

We now prove another statement, hoping to shed some light on the combinatorial nature of the joints problem in any field setting and for all dimensions. 

Indeed, let $\mathbb{F}$ be a field and $n \geq 2$. For a collection $\mathfrak{L}$ of lines in $\mathbb{F}^n$, we define as $J_N$ the set of joints formed by $\mathfrak{L}$, of multiplicity at least $N$ and smaller than $2N$. Then, the following holds.

\begin{lemma} \label{furth} Let $\mathfrak{L}$ be a finite collection of $L$ lines in $\mathbb{F}^n$, where $\mathbb{F}$ is a field and $n \geq 2$. Suppose that, whenever $n$ lines of $\mathfrak{L}$ meet at a point, they form a joint there. Then, for any constant $a_n \geq n$, there exists some positive constant $c_n$, depending only on $n$ and $a_n$, with the property that, for all $N\in \N$ such that $J_N \neq \emptyset$, at least $c_n \cdot |J_N|$ of the joints in $J_N$ are joints for some subcollection of $\mathfrak{L}$ consisting of $a_n \cdot L/N^{1/n}$ lines. \end{lemma}
\begin{proof}

Fix $a_n\geq n$. 

For every $N \in \N$ such that $J_N \neq \emptyset$, and for every $x \in J_N$, we define as $P_N(x)$ the probability that we choose at least $n$ lines through $x$, when we choose $a_n \cdot L/N^{1/n}$ lines from $\mathfrak{L}$ at random. In other words, $P_N(x)$ is the probability that $x$ is a joint for a randomly chosen subcollection of $a_n \cdot L/N^{1/n}$ lines of $\mathfrak{L}$.

We now prove the Lemma in two steps.

\textbf{Step 1:} \textit{We show that, if there exists a positive constant $c_n$, depending only on $n$, such that $P_N(x) \geq c_n$ for all $x \in J_N$ and $N \in \N$ such that $J_N \neq \emptyset$, then the Lemma is proved}. 

Indeed, fix $N \in \N$ such that $J_N \neq \emptyset$. We define as $\mathfrak{L}^N$ the set of all the subcollections of $\mathfrak{L}$ with cardinality $a_n \cdot L/N^{1/n}$. In addition, for all $x\in J_N$, let $\mathfrak{L}^N_x$ be the set of the subcollections of $\mathfrak{L}$ which have cardinality $a_n \cdot L/N^{1/n}$ and for which $x$ is a joint.

Now, the fact that $P_N(x) \geq c_n$ for all $x \in J_N$ means that $|\mathfrak{L}^N_x| \geq c_n \cdot |\mathfrak{L}^N|$ for all $x \in J_N$, so \begin{equation}\label{eq:sums} \sum_{x\in J_N}\sum_{\mathfrak{L}^N_x}1 =\sum_{x \in J_N} |\mathfrak{L}_x^N|\geq \sum_{x\in J_N} c_n \cdot |\mathfrak{L}^N| = c_n \cdot |J_N| \cdot |\mathfrak{L}^N|. \end{equation} 
However, the left-hand side of \eqref{eq:sums} is equal to \begin{displaymath} \sum_{\mathcal{L}\in \mathfrak{L}^N} \sum_{J^{\mathcal{L}}\cap J_N}1, \end{displaymath} where, for all $\mathcal{L} \in \mathfrak{L}^N$, $J^{\mathcal{L}}$ is the set of joints formed by $\mathcal{L}$, so \begin{displaymath} \sum_{\mathcal{L} \in \mathfrak{L}^N}\sum_{J_{\mathcal{L}}\cap J_N}1 \geq c_n \cdot |J_N| \cdot |\mathfrak{L}^N|,\end{displaymath} from which it follows that there exists an element of $\mathfrak{L}^N$, i.e. a subcollection of $\mathfrak{L}$ with cardinality $a_n \cdot L/N^{1/n}$, with the property that at least $c_n \cdot |J_N|$ of the elements of $J_N$ are joints for the subcollection. 

Therefore, in order to prove Lemma \ref{furth}, we only need to complete Step 2 below.

\textbf{Step 2:} \textit{We show that there exists a positive constant $c_n$, depending only on $n$, such that $P_N(x) \geq c_n$, for all $N \in \N$ such that $J_N \neq \emptyset$, and all $x \in J_N$.}

Indeed, for all $N \in \N$ such that $J_N \neq \emptyset $, and for all $x\in J_N$, we denote by $K_N(x)$ the number of lines of $\mathfrak{L}$ passing through $x$ (note that $\binom{K_N(x)}{n} \in [N,2N)$, and thus $K_N(x) \geq N^{1/n}$ -- in fact, $K_N(x) \sim N^{1/n}$ -- and $a_n \cdot L/N^{1/n}\geq n$ for $a_n \geq n$). In addition, we define $P'_N(x)$ as the probability that we choose at least $n$ lines through $x$, when we choose $a_n \cdot L/K_N(x)$ lines from $\mathfrak{L}$ at random; in other words, $P_N'(x)$ is the probability that $x$ is a joint for a randomly chosen subcollection of $a_n \cdot L/K_N(x)$ lines of $\mathfrak{L}$. 

Note that, for all $N \in \N$ such that $J_N \neq \emptyset$, and for all $x \in J_N$, $N^{1/n}\leq K_N(x)$, thus $a_n \cdot L/N^{1/n} \geq a_n \cdot L/K_N(x)$, and so, by the definition of $P_N(x)$ and $P_N'(x)$, $P_N(x)\geq P'_N(x)$. Therefore, in order to complete Step 2, it suffices to show that there exists a positive constant $c_n$, depending only on $n$, such that $P_N'(x) \geq c_n$, for all $N \in \N$ such that $J_N \neq \emptyset$, and all $x \in J_N$.

This means that it suffices to show that there exists a constant $b_n \lneq 1$, depending only on $n$, such that, for all $N \in \N$ such that $J_N \neq \emptyset$, and for all $x \in J_N$, it holds that $1-P'_N(x) \leq b_n \lneq 1$.

Indeed, let $N \in \N$ such that $J_N \neq \emptyset$, and $x \in J_N$. 

It holds that \begin{displaymath} P'_N(x)=\sum_{k=n}^{\min\Big\{K_N(x),a_n \cdot \frac{L}{K_N(x)}\Big\}} \frac{\binom{K_N(x)}{k} \cdot \binom{L-K_N(x)}{a_n \cdot \frac{L}{K_N(x)}-k}}{\binom{L}{a_n \cdot \frac{L}{K_N(x)}}},\end{displaymath} so 
\begin{equation} 1-P'_N(x)= \sum_{k=0}^{n-1}\frac{\binom{K_N(x)}{k} \cdot \binom{L-K_N(x)}{a_n \cdot \frac{L}{K_N(x)}-k}}{\binom{L}{a_n \cdot \frac{L}{K_N(x)}}}. \label{eq:sum} \end{equation}

Now, for $k\in \{0,1,...,n-1\}$, \begin{displaymath}\frac{\binom{K_N(x)}{k} \cdot \binom{L-K_N(x)}{a_n \cdot \frac{L}{K_N(x)}-k}}{\binom{L}{a_n \cdot \frac{L}{K_N(x)}}}= \end{displaymath}
\begin{displaymath}=\frac{K_N(x)!}{k! \cdot (K_N(x)-k)!}\cdot \frac{\big(L-K_N(x)-(a_n \cdot L/K_N(x)-k)+1\big)\cdots (L-K_N(x))}{(a_n \cdot L/K_N(x)-k)!}\cdot \end{displaymath} \begin{displaymath} \cdot \frac{(a_n \cdot L/K_N(x))!}{(L-a_n \cdot L/K_N(x) +1) \cdots (L-k)} \cdot \frac{1}{(L-k+1) \cdots L}= \end{displaymath}
\begin{displaymath} =\frac{1}{k!} \cdot \frac{K_N(x)!}{(K_N(x)-k)!} \cdot \frac{(a_n \cdot L/K_N(x))!}{(a_n \cdot L/K_N(x)-k)!} \cdot \frac{1}{(L-k+1) \cdots L}\cdot \end{displaymath} \begin{displaymath} \cdot \frac{L-(K_N(x)-k)-a_n \cdot L/K_N(x) +1}{L-a_N \cdot L/K_N(x) +1} \cdot \frac{L-(K_N(x)-k)-a_n \cdot L/K_N(x) +2}{L-a_n \cdot L/K_N(x) +2} \cdot\end{displaymath} \begin{equation} \cdots \frac{L-(K_N(x)-k)-a_n \cdot L/K_N(x)+(a_n \cdot L/K_N(x)-k)}{L-a_n \cdot L/K_N(x)+ (a_n \cdot L/K_N(x)-k)}. \label{eq:product} \end{equation}

Now:

a) The function \begin{displaymath} f(y)=\frac{L-(K_N(x)-k)-a_n \cdot L/K_N(x)+y}{L-a_n \cdot L/K_N(x)+y} \end{displaymath} is increasing in $y$, therefore \begin{displaymath} \frac {L-(K_N(x)-k)-a_n \cdot L/K_N(x)+1}{L-a_n \cdot L/K_N(x)+1} \cdots \end{displaymath} \begin{displaymath}\cdots \frac{L-(K_N(x)-k)-a_n \cdot L/K_N(x)+(a_n \cdot L/K_N(x)-k)}{L-a_n \cdot L/K_N(x)+ (a_n \cdot L/K_N(x)-k)} \leq \end{displaymath}
\begin{displaymath} \leq \Bigg(\frac{L-(K_N(x)-k)-a_n \cdot L/K_N(x)+(a_n \cdot L/K_N(x)-k)}{L-a_n \cdot L/K_N(x)+ (a_n \cdot L/K_N(x)-k)}\Bigg)^{a_n \cdot L/K_N(x) -k}=\end{displaymath}
\begin{displaymath} =\Bigg(\frac{L-K_N(x)}{L-k}\Bigg)^{a_n \cdot L/K_N(x) -k}= \Bigg(1-\frac{K_N(x)-k}{L-k}\Bigg)^{a_n \cdot L/K_N(x) -k}= \end{displaymath}
\begin{displaymath} =\Bigg(1-\frac{1}{\frac{L-k}{K_N(x)-k}}\Bigg)^{a_n \cdot L/K_N(x) -k}. \end{displaymath}
On the other hand, if $c_{n,k}$ is any constant larger than $k$, $a'_n$ is any positive constant and $a_n>\frac{a'_n}{1-\frac{k}{c_{n,k}}}$, then, for \begin{displaymath}K_N(x)>c_{n,k} \end{displaymath} and \begin{equation}\frac{L}{K_N(x)} \geq \frac{k}{a_n-\frac{a'_n}{1-\frac{k}{c_{n,k}}}}:=c'_{n,k},\label{eq:constraint}\end{equation} we have that \begin{displaymath} a_n \cdot \frac{L}{K_N(x)}-k \geq a'_n \cdot \frac{L-k}{K_N(x)-k}, \end{displaymath} because, under these constraints,
\begin{displaymath} \frac{k}{K_N(x)}< \frac{k}{c_{n,k}} \Rightarrow -\frac{k}{K_N(x)}>-\frac{k}{c_{n,k}} \Rightarrow 1-\frac{k}{K_N(x)}>1-\frac{k}{c_{n,k}}\;(\gneq 0), \end{displaymath}
and thus \begin{displaymath} a'_n \cdot \frac{L-k}{K_N(x)-k}=a'_n \cdot \frac{L/K_N(x)-k/K_N(x)}{1-k/K_N(x)} \leq a'_n \cdot \frac{L/K_N(x)}{1-k/K_N(x)} < a'_n \cdot \frac{L/K_N(x)}{1-k/c_{n,k}}= \end{displaymath}
\begin{displaymath} =\frac{a'_n}{1-\frac{k}{c_{n,k}}} \cdot \frac{L}{K_N(x)}, \end{displaymath} 
a quantity which is $\leq a_n \cdot \frac{L}{K_N(x)} -k$ when \eqref{eq:constraint} holds.
Therefore,
\begin{displaymath}\Bigg(1-\frac{1}{\frac{L-k}{K_N(x)-k}}\Bigg)^{a_n \cdot L/K_N(x) -k} \leq \Bigg(1-\frac{1}{\frac{L-k}{K_N(x)-k}}\Bigg)^{a'_n \cdot \frac{L-k}{K_N(x)-k}}. \end{displaymath}

And, as the function $\Big(1-\frac{1}{y}\Big)^y$ is increasing in $y$ to $e^{-1}$, the above quantity is equal to at most $e^{-a'_n}$. But $a_n>\frac{a'_n}{1-\frac{k}{c_{n,k}}}$, so, for all $M >1$, we can consider $c_{n,k}$ appropriately large, depending on $M$, so as to achieve $e^{-a'_n} \leq M \cdot e^{-a_n}$. 

Therefore, for any $M >1$ and for $K_N(x)$, $L/K_N(x)$ larger than appropriate constants depending only on $n$ and $M$, it holds that, for all $ k$ in the finite range $\{0,1,...,n-1\}$,
\begin{displaymath} \frac {L-(K_N(x)-k)-a_n \cdot L/K_N(x)+1}{L-a_n \cdot L/K_N(x)+1} \cdots \end{displaymath}
\begin{displaymath} \cdots \frac{L-(K_N(x)-k)-a_n \cdot L/K_N(x)+(a_n \cdot L/K_N(x)-k)}{L-a_n \cdot L/K_N(x)+ (a_n \cdot L/K_N(x)-k)} \leq \end{displaymath} $$ \leq M \cdot  e^{-a_n}\label{eq:result1}.$$

b) We now consider the quantity \begin{displaymath}\frac{K_N(x)!}{(K_N(x)-k)!} \cdot \frac{(a_n \cdot L/K_N(x))!}{(a_n \cdot L/K_N(x)-k)!} \cdot \frac{1}{(L-k+1) \cdots L}= \end{displaymath}
\begin{displaymath} =\frac{(K_N(x)-k+1)\cdots K_N(x)\cdot(a_n\cdot L/K_N(x)-k+1)\cdots (a_n \cdot L/K_N(x))}{(L-k+1)\cdots L}. \end{displaymath}
We see that, for all $\lambda=0,1,...,k-1$,
\begin{displaymath}\frac{(K_N(x)-\lambda)\cdot(a_n\cdot L/K_N(x)-\lambda)}{L-\lambda}\leq a_n, \end{displaymath}
as this is equivalent to
\begin{displaymath} (K_N(x)-\lambda)\cdot(a_n\cdot L/K_N(x)-\lambda)\leq a_n \cdot L-a_n \cdot \lambda \Leftrightarrow \end{displaymath}\begin{displaymath} \Leftrightarrow K_N(x)+a_n\cdot L/K_N(x)-a_n\geq\lambda, \end{displaymath} 

which is true because $K_N(x)\geq n$ (as $x$ is a joint formed by $\mathfrak{L}$ in $\R^n$) and $L/K_N(x)\geq 1$, while $\lambda <n$ (since $k \in \{0,...,n-1\}$).

Therefore,\begin{equation} \frac{(K_N(x)-k+1)\cdots K_N(x)\cdot(a_n\cdot L/K_N(x)-k+1)\cdots (a_n \cdot L/K_N(x))}{(L-k+1)\cdots L}\leq {a_n}^{k}, \label{eq:result2}\end{equation}
for all $k\in \{0,1,...,n-1\}$.

So, it follows from (a) and (b) that, for all $M >1$ and for $K_N(x)$, $L/K_N(x)$ larger than a constant, say, $\lambda_{n,M}$, which depends only on $n$ and $M$, \begin{displaymath} 1-P'_N(x)\leq M \cdot \sum_{k=0}^{n-1} \frac{{a_n}^k}{k!} \cdot\frac{1}{e^{a_n}}=:b_{n,M},\end{displaymath} which is a constant strictly smaller than 1 for $M < \frac{1}{\sum_{k=0}^{n-1} \frac{{a_n}^k}{k!} \cdot\frac{1}{e^{a_n}}}$ (in fact, $b_{n,M}$ can be taken arbitrarily small, because $a_n$ can be taken arbitrarily large and $\frac{{a_n}^k}{k! \cdot e^{a_n}} \rightarrow 0$ as $a_n \rightarrow \infty$, for all $k$ in the finite range $\{0,1,...,n-1\}$). 

This means that, for $M_0:=\frac{1}{2}\Bigg(1+\frac{1}{\sum_{k=0}^{n-1} \frac{{a_n}^k}{k!} \cdot\frac{1}{e^{a_n}}}\Bigg)$, a quantity that depends only on $n$, $b_{n,M_0}$ is a constant $\lneq 1$, depending only on $n$, such that, for all $x \in J_N$ with the property that $K_N(x)$ and $L/K_N(x)$ are larger than $\lambda_{n,M_0}$ (which is another constant depending only on $n$), \begin{displaymath} P'_N(x)\geq 1-b_{n,M_0}\gneq 0, \end{displaymath} which is exactly what we wanted to show.

For $x \in J_N$ such that $K_N(x)$ or $L/K_N(x)$ are smaller than $\lambda_{n,M_0}$, we get a similar result in the following way.

Suppose that $x \in J_N$ is such that $L/K_N(x)\leq \lambda_{n,M_0}$. Let $A_1$, ..., $A_{L/K_N(x)}$ be disjoint sets, each of cardinality $K_N(x)$. For any fixed $m \in \{1,...,L/K_N(x)\}$, we see that $P'_N(x)$ is equal to the probability that we choose at least $n$ elements of $A_m$, when we choose $a_n \cdot L/K_N(x)$ elements of $A_1\cup ... \cup A_{L/K_N(x)}$ at random. However, by the pigeon-hole principle, whenever we choose $a_n \cdot L/K_N(x)$ elements of $A_1\cup ... \cup A_{L/K_N(x)}$, there exists some $i \in \{1,...,L/K_N(x)\}$, such that at least $a_n\geq n$ elements of $A_i$ have been chosen, so
\begin{displaymath} \sum_{m=1}^{L/K_N(x)}P'_N(x)\geq 1 \Leftrightarrow L/K_N(x) \cdot P'_N(x)\geq 1, \end{displaymath}
which implies that \begin{displaymath} P'_N(x)\geq K_N(x)/L\geq 1/\lambda_{n,M_0},
\end{displaymath}
which is a positive constant, depending only on $n$.

Finally, suppose that $x \in J_N$ is such that $K_N(x) \leq \lambda_{n,M_0}$. It follows that $N\leq \binom{K_N(x)}{n} \leq K_N(x)^n \leq \lambda_{n,M_0}^n$, a constant that depends only on $n$. Therefore, $|J_N| \cdot N^{\frac{1}{n-1}}\lesssim_n |J| \lesssim_n  L^{\frac{n}{n-1}}$, which is the result whose proof we were aiming for when we introduced the Lemma. However, let us prove that, in this case as well, $P'_N(x)$ is larger than a positive constant which depends only on $n$.

Indeed, it is easy to see that, for all $k \in \big\{0,1,...,\min\{K_N(x),a_n\cdot L/K_N(x)\}\big\}$, \begin{displaymath}\frac{\binom{K_N(x)}{k} \cdot \binom{L-K_N(x)}{a_n \cdot \frac{L}{K_N(x)}-k}}{\binom{L}{a_n \cdot \frac{L}{K_N(x)}}}=\frac{\binom{a_n \cdot \frac{L}{K_N(x)}}{k} \cdot \binom{L-a_n \cdot \frac{L}{K_N(x)}}{K_N(x)-k}}{\binom{L}{K_N(x)}}, \end{displaymath} thus
\begin{displaymath} P'_N(x)=\sum_{k=n}^{\min\{K_N(x),a_n \cdot L/K_N(x)\}}\frac{\binom{K_N(x)}{k} \cdot \binom{L-K_N(x)}{a_n \cdot \frac{L}{K_N(x)}-k}}{\binom{L}{a_n \cdot \frac{L}{K_N(x)}}}=\end{displaymath} 
\begin{displaymath}=\sum_{k=n}^{\min\{K_N(x),a_n \cdot L/K_N(x)\}}\frac{\binom{a_n \cdot \frac{L}{K_N(x)}}{k} \cdot \binom{L-a_n \cdot \frac{L}{K_N(x)}}{K_N(x)-k}}{\binom{L}{K_N(x)}}.
\end{displaymath}

In other words, if $B_1$, ..., $B_{K_N(x)/a_n}$ be disjoint sets, each of cardinality $a_n \cdot L/K_N(x)$, then, for any fixed $m \in \{1,...,K_N(x)/a_n\}$, $P'_N(x)$ is equal to the probability that we choose at least $n$ elements of $B_m$, when we choose $K_N(x)$ elements of $B_1\cup ... \cup B_{K_N(x)/a_n}$ at random. However, by the pigeon-hole principle, whenever we choose $K_N(x)$ elements of $B_1\cup ... \cup B_{K_N(x)/a_n}$, there exists some $i \in \{1,...,K_N(x)/a_n\}$, such that at least $a_n\geq n$ elements of $B_i$ have been chosen, so
\begin{displaymath} \sum_{m=1}^{K_N(x)/a_n}P'_N(x)\geq 1 \Leftrightarrow K_N(x)/a_n \cdot P'_N(x)\geq 1,\end{displaymath} which implies that \begin{displaymath} P'_N(x)\geq a_n/K_N(x)\geq a_n/\lambda_{n,M_0},
\end{displaymath}
which is a positive constant, depending only on $n$.

Therefore, the second step of the proof of the Lemma is complete, and thus the statement of the Lemma is true.

\end{proof}

We now combine Theorem \ref{carberyjoints} and Lemma \ref{furth}, to obtain the following result regarding the cardinality of the set of joints of mutliplicity $\sim N$, that are formed by a finite collection $\mathfrak{L}$ of lines in $\mathbb{F}^n$, where $\mathbb{F}$ is an arbitrary field and $n \geq 2$, in the ``generic" case, i.e. in the case where, whenever $n$ lines of $\mathfrak{L}$ meet at a point of $\mathbb{F}^n$, they form a joint there.

\begin{proposition} \label{last} Let $\mathfrak{L}$ be a finite collection of $L$ lines in $\mathbb{F}^n$, where $\mathbb{F}$ is a field and $n \geq 2$. Suppose that, whenever $n$ lines of $\mathfrak{L}$ meet at a point, they form a joint there. Then,
\begin{displaymath}|J_N| \cdot N^{\frac{1}{n-1}} \lesssim_n L^{\frac{n}{n-1}},
\end{displaymath}
for all $N \in \N$.
\end{proposition}

\begin{proof} Let $N \in \N$ such that $J_N \neq \emptyset$. By Lemma \ref{furth}, there exists a positive constant $c_n$, depending only on $n$, such that, for all $N \in \N$, at least $c_n\cdot |J_N|$ of the joints in $J_N$ are joints for some subcollection $\mathfrak{L}'$ of $\mathfrak{L}$, consisting of $n \cdot L/N^{1/n}$ lines. It follows by Theorem \ref{carberyjoints} that
$$|J_N| \lesssim_n \bigg(\frac{L}{N^{1/n}}\bigg)^{\frac{n}{n-1}},
$$
which gives
$$|J_N| \cdot N^{\frac{1}{n-1}} \lesssim_n L^{\frac{n}{n-1}}
$$
after rearranging.

\end{proof}

The following is an immediate consequence of Proposition \ref{last}.

\begin{corollary}Let $\mathfrak{L}$ be a finite collection of $L$ lines in $\mathbb{F}^n$, where $\mathbb{F}$ is a field and $n \geq 2$. Suppose that, whenever $n$ lines of $\mathfrak{L}$ meet at a point, they form a joint there. Then,
\begin{displaymath}\sum_{x \in J} N(x)^{\frac{1}{n-1}} \lesssim_n \log L \cdot L^{\frac{n}{n-1}}.
\end{displaymath}

\end{corollary}

Let us mention, however, that this is a considerably weaker result than the one we expect. Indeed, we believe that, as in euclidean space, if $\mathfrak{L}$ is a finite collection of lines in $\mathbb{F}^n$, where $\mathbb{F}$ is an arbitrary field and $n \geq 2$, and $J$ is the set of joints formed by $\mathfrak{L}$, then 
\begin{displaymath}\sum_{x \in J}N(x)^{\frac{1}{n-1}} \lesssim_n  L^{\frac{n}{n-1}},
\end{displaymath}
without the additional assumption that we are in the ``generic" case.

\newpage
\thispagestyle{plain}
\cleardoublepage

\addcontentsline{toc}{chapter}{Bibliography}

\end{document}